\documentclass[12pt]{book}

\usepackage{graphicx,adjustbox}

\usepackage{color,units}

\usepackage[dvipsnames]{xcolor}

\usepackage{amsmath,amsfonts,amssymb,amsthm,relsize,mathrsfs,tikz-cd,mathtools,tensor,mathdots,slashed,epsfig,xypic,thmtools,changepage,chngcntr}
\usetikzlibrary{graphs,decorations.pathmorphing,decorations.markings}

\tikzcdset{%
	arrow style=tikz,
	diagrams={>={Straight Barb}},
	tikzcd implies/.tip={Straight Barb[scale=0.5]}
}

\usepackage[titletoc]{appendix}

\usepackage[textsize=small,prependcaption]{todonotes}
\xyoption{curve}

\usepackage{enumitem}
\setlist[enumerate]{label=(\roman*),leftmargin=0.8cm}

\usepackage{emptypage}

\usepackage{geometry}
\usepackage{pdfpages}
\usepackage{setspace}

\usepackage[pagebackref,hyperfootnotes]{hyperref}
\hypersetup{
	colorlinks,
	linkcolor={red!50!black},
	citecolor={green!50!black},
	urlcolor={blue!80!black}
}
\usepackage[capitalize,nameinlink,noabbrev]{cleveref}
\usepackage[british]{babel}
\usepackage[utf8]{inputenc}

\setuptodonotes{color=green!40,caption={TODO}}
\reversemarginpar

\newcommand{\ZZ}{\mathbb{Z}}
\newcommand{\QQ}{\mathbb{Q}}
\newcommand{\RR}{\mathbb{R}}
\newcommand{\CC}{\mathbb{C}}

\newcommand{\PP}{\mathbb{P}}
\newcommand{\HH}{\mathbb{H}}

\newcommand{\CCPP}{\mathbb{CP}}
\renewcommand{\AA}{\mathbb{A}}
\newcommand{\EE}{\mathbb{E}}

\newcommand{\calX}{\mathcal{X}}
\newcommand{\calH}{\mathcal{H}}
\newcommand{\calL}{\mathcal{L}}
\newcommand{\calP}{\mathcal{P}}
\newcommand{\calQ}{\mathcal{Q}}
\newcommand{\calM}{\mathcal{M}}
\newcommand{\calE}{\mathcal{E}}
\newcommand{\calS}{\mathcal{S}}
\newcommand{\calF}{\mathcal{F}}

\newcommand{\calZ}{\mathcal{Z}}
\newcommand{\calD}{\mathcal{D}}

\newcommand{\calT}{\mathcal{T}}
\newcommand{\calU}{\mathcal{U}}
\newcommand{\calV}{\mathcal{V}}

\newcommand{\calG}{\mathcal{G}}
\newcommand{\calA}{\mathcal{A}}

\newcommand{\scrE}{\mathscr{E}}

\newcommand{\calO}{\mathcal{O}}

\DeclarePairedDelimiter{\ip}{\langle}{\rangle}

\newcommand{\acts}{\curvearrowright}

\newcommand{\Lap}{\Delta}

\newcommand{\git}{\mathbin{
		\mathchoice{/\mkern-6mu/}
		{/\mkern-6mu/}
		{/\mkern-5mu/}
		{/\mkern-5mu/}}}

\newcommand{\diag}{\operatorname{diag}} 
\newcommand{\Hilb}{\operatorname{Hilb}} 
\newcommand{\Bl}{\operatorname{Bl}} 
\newcommand{\Ext}{\operatorname{Ext}} 
\newcommand{\codim}{\operatorname{codim}} 
\newcommand{\reldim}{\operatorname{reldim}} 

\newcommand{\DbCoh}{\mathcal{D}^b\operatorname{Coh}}

\newcommand{\YM}{\operatorname{YM}} 

\renewcommand{\deg}{\operatorname{deg}} 
\newcommand{\rk}{\operatorname{rk}} 
\newcommand{\vol}{\operatorname{vol}} 
\renewcommand{\Im}{\operatorname{Im}} 
\renewcommand{\Re}{\operatorname{Re}} 

\newcommand{\sgn}{\operatorname{sgn}}
\newcommand{\grsgn}{\operatorname{grsgn}}

\newcommand{\Lie}{\operatorname{Lie}}

\newcommand{\g}{\mathfrak{g}} 
\newcommand{\h}{\mathfrak{h}} 
\renewcommand{\k}{\mathfrak{k}} 
\newcommand{\ad}{\operatorname{ad}} 
\newcommand{\Ad}{\operatorname{Ad}} 

\newcommand{\DF}{\operatorname{DF}} 
\newcommand{\Ham}{\operatorname{Ham}} 
\newcommand{\Ric}{\operatorname{Ric}} 

\newcommand{\ch}{\operatorname{ch}} 
\newcommand{\Ch}{\operatorname{Ch}} 
\newcommand{\Td}{\operatorname{Td}} 

\newcommand{\Gr}{\operatorname{Gr}} 
\newcommand{\GrHN}{\operatorname{Gr}^\mathrm{HN}}
\newcommand{\GrJH}{\operatorname{Gr}^\mathrm{JH}}
\newcommand{\GrHNS}{\operatorname{Gr}^\mathrm{HNS}}

\newcommand{\GL}{\operatorname{GL}} 
\newcommand{\SL}{\operatorname{SL}} 
\newcommand{\U}{\operatorname{U}} 
\newcommand{\SU}{\operatorname{SU}} 
\newcommand{\PGL}{\operatorname{PGL}} 
\newcommand{\Herm}{\operatorname{Herm}} 
\newcommand{\Aff}{\operatorname{Aff}} 

\newcommand{\ind}{\operatorname{ind}}

\newcommand{\sym}{\mathrm{sym}}

\newcommand{\Met}{\operatorname{Met}}

\newcommand{\bdry}{\partial}

\newcommand{\closure}[1]{\overline{#1}}

\newcommand{\gl}{\mathfrak{gl}} 
\renewcommand{\sl}{\mathfrak{sl}}

\newcommand{\Sym}{\operatorname{Sym}} 
\newcommand{\Hom}{\operatorname{Hom}} 
\newcommand{\End}{\operatorname{End}} 
\newcommand{\Aut}{\operatorname{Aut}} 
\newcommand{\Isom}{\operatorname{Isom}} 
\newcommand{\Symp}{\operatorname{Symp}} 

\newcommand{\Proj}{\operatorname{Proj}} 

\newcommand{\Coh}{\operatorname{Coh}} 

\newcommand{\trace}{\operatorname{tr}} 
\newcommand{\image}{\operatorname{im}} 

\newcommand{\fr}{\mathcal{F}} 

\newcommand{\contr}{\Lambda}

\newcommand{\Exterior}{\mathchoice{{\textstyle\bigwedge}}%
	{{\bigwedge}}%
	{{\textstyle\wedge}}%
	{{\scriptstyle\wedge}}}

\newcommand{\G}{\mathscr{G}}

\newcommand{\del}{\partial}
\newcommand{\delbar}{\overline{\partial}}
\newcommand{\deldelbar}{\partial \overline{\partial}}
\newcommand{\delbardel}{\overline{\partial} \partial} 
\newcommand{\deriv}[2]{\frac{d #1}{d #2}} 
\newcommand{\pderiv}[2]{\frac{\partial #1}{\partial #2}} 

\newcommand{\dvol}{d \text{vol}}

\newcommand{\rest}[2]{\left.#1\right|_{#2}}

\newcommand{\isom}{\cong}

\newcommand{\id}{{\boldsymbol{1}}} 

\newtheorem{theorem}{Theorem}[section]
\newtheorem{lemma}[theorem]{Lemma}
\newtheorem{corollary}[theorem]{Corollary}
\newtheorem{proposition}[theorem]{Proposition}
\newtheorem{conjecture}[theorem]{Conjecture}

\theoremstyle{definition}
\newtheorem{definition}[theorem]{Definition}
\newtheorem{example}[theorem]{Example}

\newtheorem{remark}[theorem]{Remark}
\newtheorem{question}[theorem]{Question}

\newtheorem{principle}[theorem]{Principle}

\newtagform{Eq}{(Eq. }{)}
\usetagform{Eq}

\makeatletter
\renewcommand*{\eqref}[1]{%
	\hyperref[{#1}]{\textup{\tagform@{\ref*{#1}}}}%
}
\makeatother

\let\oldchi\chi
\renewcommand{\chi}{\protect\raisebox{\depth}{$\oldchi$}}

\let\oldepsilon\epsilon
\renewcommand{\epsilon}{\varepsilon}

\DeclarePairedDelimiter\floor{\lfloor}{\rfloor}

\newcommand{\into}{\hookrightarrow}

\newcommand{\onto}{\twoheadrightarrow}

\newcommand{\ses}[3]{\begin{tikzcd}[ampersand replacement=\&] 0 \arrow{r}\& #1 \arrow{r} \& #2 \arrow{r} \& #3 \arrow{r} \& 0\end{tikzcd}}

\allowdisplaybreaks[1]

\begin{document}
	\frontmatter
	
	\newgeometry{margin=3cm}
	\begin{singlespace}
	\begin{titlepage}
		
		\centering
		\vspace*{2cm}
		\Huge \textbf{Stability conditions and\\ canonical metrics}
		
		\vspace*{1cm}
		\Large John Benjamin McCarthy
		
		\vfill
		
		Thesis submitted for the degree of\\
		\emph{Doctor of Philosophy}
		
		\vfill
		
		Supervised by\\
		Sir Simon Donaldson\\
		\&\\
		Dr Ruadha\'i Dervan
		
		\vfill
		
		\large
		Department of Mathematics\\
		Imperial College London\\
		\&\\
		London School of Geometry and Number Theory
		
		\vspace*{0.8cm}
		
		February 2023
		
		\vspace*{1cm}
		
		\begin{figure}[h]
			\centering
			\begin{minipage}[b]{0.4\textwidth}
				\includegraphics[width=\textwidth]{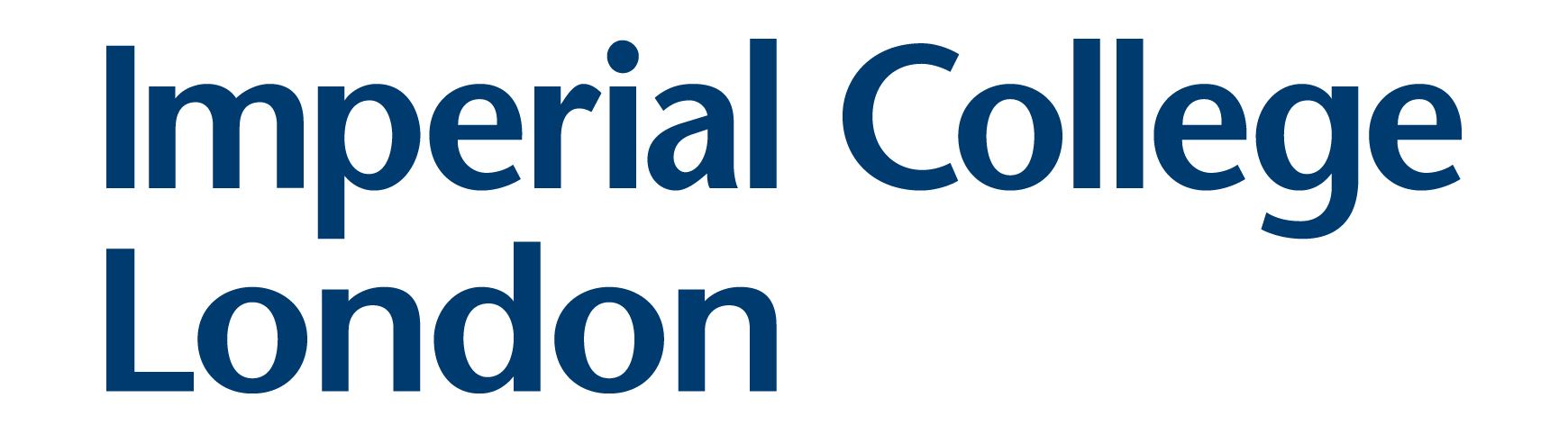}
			\end{minipage}
			\hfill
			\begin{minipage}[b]{0.4\textwidth}
				\includegraphics[width=\textwidth]{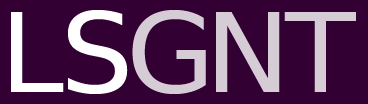}
			\end{minipage}
		\end{figure}
	\end{titlepage}
	\newpage \ \newpage
	
\thispagestyle{empty}

\vspace*{\fill}
\begin{center}
	\emph{For Diclehan'ım.}
\end{center} 
\vspace*{\fill}

\end{singlespace}
	\newpage \ \newpage
	\vspace*{\fill}
	\begin{figure}[hb]
		\centering

		\tikzset{every picture/.style={line width=0.75pt}} 
		
		\begin{tikzpicture}[x=0.75pt,y=0.75pt,yscale=-1,xscale=1]
			
			\draw  [dash pattern={on 4.5pt off 4.5pt}]  (193.06,517.08) .. controls (243.46,576.02) and (364.31,574.63) .. (471.42,553.16) ;
			\draw [color={rgb, 255:red, 0; green, 255; blue, 7 }  ,draw opacity=1 ]   (294.64,549.39) .. controls (207.29,539.59) and (174.71,521.21) .. (140.22,498.03) ;
			\draw [shift={(138.64,496.97)}, rotate = 34] [color={rgb, 255:red, 0; green, 255; blue, 7 }  ,draw opacity=1 ][line width=0.75]    (10.93,-3.29) .. controls (6.95,-1.4) and (3.31,-0.3) .. (0,0) .. controls (3.31,0.3) and (6.95,1.4) .. (10.93,3.29)   ;
			\draw    (57.13,571.44) .. controls (125.55,629.68) and (346.53,616.42) .. (460.29,571.53) ;
			\draw    (57.13,571.44) .. controls (79.74,535.48) and (108.1,505.24) .. (138.64,496.97) ;
			\draw    (138.64,496.97) .. controls (204.54,480.11) and (237.42,546.87) .. (291.51,549.39) ;
			\draw [color={rgb, 255:red, 0; green, 0; blue, 0 }  ,draw opacity=1 ]   (291.51,549.39) .. controls (347.15,554.45) and (485.21,521.18) .. (535.55,499.18) ;
			\draw    (460.29,571.53) .. controls (482.9,535.56) and (511.26,505.33) .. (541.8,497.06) ;
			\draw    (541.8,497.06) .. controls (563.34,491.09) and (579.47,494.42) .. (596.22,517.17) ;
			\draw  [dash pattern={on 4.5pt off 4.5pt}]  (138.64,496.97) .. controls (160.18,491.01) and (176.31,494.34) .. (193.06,517.08) ;
			\draw    (473.65,552.71) .. controls (506.96,546.24) and (562.35,530.54) .. (596.22,517.17) ;
			\draw  [fill={rgb, 255:red, 0; green, 0; blue, 0 }  ,fill opacity=1 ] (291.51,549.39) .. controls (291.51,547.66) and (292.91,546.26) .. (294.64,546.26) .. controls (296.37,546.26) and (297.78,547.66) .. (297.78,549.39) .. controls (297.78,551.12) and (296.37,552.52) .. (294.64,552.52) .. controls (292.91,552.52) and (291.51,551.12) .. (291.51,549.39) -- cycle ;
			\draw    (43.94,84.05) .. controls (125.09,294.15) and (243.17,361.49) .. (395.42,423.95) ;
			\draw [color={rgb, 255:red, 32; green, 0; blue, 255 }  ,draw opacity=1 ]   (294.64,549.39) .. controls (255.78,551.05) and (242.17,569.31) .. (231.51,606.64) ;
			\draw [shift={(231.02,608.36)}, rotate = 285.59] [color={rgb, 255:red, 32; green, 0; blue, 255 }  ,draw opacity=1 ][line width=0.75]    (10.93,-3.29) .. controls (6.95,-1.4) and (3.31,-0.3) .. (0,0) .. controls (3.31,0.3) and (6.95,1.4) .. (10.93,3.29)   ;
			\draw [color={rgb, 255:red, 255; green, 0; blue, 0 }  ,draw opacity=1 ]   (294.64,549.39) .. controls (347.22,549.31) and (414.51,543.02) .. (452.11,530.49) .. controls (488.76,518.27) and (478.62,523.07) .. (501.32,520) ;
			\draw [shift={(503.13,519.75)}, rotate = 171.87] [color={rgb, 255:red, 255; green, 0; blue, 0 }  ,draw opacity=1 ][line width=0.75]    (10.93,-3.29) .. controls (6.95,-1.4) and (3.31,-0.3) .. (0,0) .. controls (3.31,0.3) and (6.95,1.4) .. (10.93,3.29)   ;
			\draw    (395.42,423.95) .. controls (449.33,433.36) and (567.53,413.75) .. (603.28,399.31) ;
			\draw    (43.94,84.05) .. controls (97.85,93.46) and (226.7,95.87) .. (262.44,81.43) ;
			\draw [color={rgb, 255:red, 255; green, 0; blue, 0 }  ,draw opacity=1 ]   (294.64,286.34) .. controls (328.51,325.35) and (405.1,382.71) .. (462.01,408.54) ;
			\draw [shift={(463.73,409.31)}, rotate = 204.02] [color={rgb, 255:red, 255; green, 0; blue, 0 }  ,draw opacity=1 ][line width=0.75]    (10.93,-3.29) .. controls (6.95,-1.4) and (3.31,-0.3) .. (0,0) .. controls (3.31,0.3) and (6.95,1.4) .. (10.93,3.29)   ;
			\draw  [fill={rgb, 255:red, 0; green, 0; blue, 0 }  ,fill opacity=1 ] (291.55,285.82) .. controls (291.84,284.12) and (293.46,282.97) .. (295.16,283.25) .. controls (296.87,283.54) and (298.02,285.16) .. (297.73,286.86) .. controls (297.45,288.57) and (295.83,289.72) .. (294.12,289.43) .. controls (292.42,289.14) and (291.27,287.53) .. (291.55,285.82) -- cycle ;
			\draw [color={rgb, 255:red, 0; green, 255; blue, 7 }  ,draw opacity=1 ]   (294.64,286.34) .. controls (239.55,228.08) and (180.33,185.28) .. (149.94,95.44) ;
			\draw [shift={(149.49,94.09)}, rotate = 71.57] [color={rgb, 255:red, 0; green, 255; blue, 7 }  ,draw opacity=1 ][line width=0.75]    (10.93,-3.29) .. controls (6.95,-1.4) and (3.31,-0.3) .. (0,0) .. controls (3.31,0.3) and (6.95,1.4) .. (10.93,3.29)   ;
			\draw [color={rgb, 255:red, 32; green, 0; blue, 255 }  ,draw opacity=1 ]   (294.64,286.34) .. controls (261.08,292.14) and (241.61,295.07) .. (182.52,297.02) ;
			\draw [shift={(180.71,297.08)}, rotate = 358.15] [color={rgb, 255:red, 32; green, 0; blue, 255 }  ,draw opacity=1 ][line width=0.75]    (10.93,-3.29) .. controls (6.95,-1.4) and (3.31,-0.3) .. (0,0) .. controls (3.31,0.3) and (6.95,1.4) .. (10.93,3.29)   ;
			\draw    (262.44,81.43) .. controls (343.59,291.53) and (451.04,336.85) .. (603.28,399.31) ;
			\draw  [dash pattern={on 0.84pt off 2.51pt}]  (178.19,519.22) .. controls (234.29,584.82) and (368.81,583.28) .. (488.04,559.38) ;
			\draw  [dash pattern={on 0.84pt off 2.51pt}]  (26.88,579.73) .. controls (103.03,644.55) and (349.02,629.8) .. (475.65,579.83) ;
			\draw  [dash pattern={on 0.84pt off 2.51pt}]  (26.88,579.73) .. controls (52.05,539.7) and (83.61,506.04) .. (117.61,496.84) ;
			\draw  [dash pattern={on 0.84pt off 2.51pt}]  (475.65,579.83) .. controls (500.82,539.79) and (532.39,506.13) .. (566.39,496.93) ;
			\draw  [dash pattern={on 0.84pt off 2.51pt}]  (566.39,496.93) .. controls (590.37,490.29) and (608.31,494) .. (626.96,519.32) ;
			\draw  [dash pattern={on 0.84pt off 2.51pt}]  (117.61,496.84) .. controls (141.59,490.2) and (159.54,493.9) .. (178.19,519.22) ;
			\draw  [dash pattern={on 0.84pt off 2.51pt}]  (490.53,558.88) .. controls (527.6,551.68) and (589.25,534.2) .. (626.96,519.32) ;
			\draw  [dash pattern={on 0.84pt off 2.51pt}]  (12.86,59.9) .. controls (104.83,292.61) and (238.69,367.19) .. (411.26,436.37) ;
			\draw  [dash pattern={on 0.84pt off 2.51pt}]  (411.26,436.37) .. controls (472.36,446.8) and (606.35,425.08) .. (646.87,409.08) ;
			\draw  [dash pattern={on 0.84pt off 2.51pt}]  (12.86,59.9) .. controls (73.96,70.33) and (220.01,73) .. (260.53,57) ;
			\draw  [dash pattern={on 0.84pt off 2.51pt}]  (260.53,57) .. controls (352.51,289.71) and (474.3,339.9) .. (646.87,409.08) ;
			\draw    (354,387) -- (354,501) ;
			\draw [shift={(354,503)}, rotate = 270] [color={rgb, 255:red, 0; green, 0; blue, 0 }  ][line width=0.75]    (10.93,-3.29) .. controls (6.95,-1.4) and (3.31,-0.3) .. (0,0) .. controls (3.31,0.3) and (6.95,1.4) .. (10.93,3.29)   ;
			\draw    (43,207) -- (43,472) ;
			\draw [shift={(43,474)}, rotate = 270] [color={rgb, 255:red, 0; green, 0; blue, 0 }  ][line width=0.75]    (10.93,-3.29) .. controls (6.95,-1.4) and (3.31,-0.3) .. (0,0) .. controls (3.31,0.3) and (6.95,1.4) .. (10.93,3.29)   ;
			\draw  [color={rgb, 255:red, 0; green, 255; blue, 7 }  ,draw opacity=1 ][fill={rgb, 255:red, 0; green, 255; blue, 7 }  ,fill opacity=1 ] (133.55,68.82) .. controls (133.84,67.12) and (135.46,65.97) .. (137.16,66.25) .. controls (138.87,66.54) and (140.02,68.16) .. (139.73,69.86) .. controls (139.45,71.57) and (137.83,72.72) .. (136.12,72.43) .. controls (134.42,72.14) and (133.27,70.53) .. (133.55,68.82) -- cycle ;
			\draw  [color={rgb, 255:red, 0; green, 255; blue, 7 }  ,draw opacity=1 ][fill={rgb, 255:red, 0; green, 255; blue, 7 }  ,fill opacity=1 ] (128.55,491.82) .. controls (128.84,490.12) and (130.46,488.97) .. (132.16,489.25) .. controls (133.87,489.54) and (135.02,491.16) .. (134.73,492.86) .. controls (134.45,494.57) and (132.83,495.72) .. (131.12,495.43) .. controls (129.42,495.14) and (128.27,493.53) .. (128.55,491.82) -- cycle ;
			\draw  [color={rgb, 255:red, 255; green, 0; blue, 0 }  ,draw opacity=1 ][fill={rgb, 255:red, 255; green, 0; blue, 0 }  ,fill opacity=1 ] (519.55,434.09) .. controls (519.84,432.38) and (521.46,431.23) .. (523.16,431.52) .. controls (524.87,431.81) and (526.02,433.42) .. (525.73,435.13) .. controls (525.45,436.84) and (523.83,437.99) .. (522.12,437.7) .. controls (520.42,437.41) and (519.27,435.8) .. (519.55,434.09) -- cycle ;
			\draw  [color={rgb, 255:red, 255; green, 0; blue, 0 }  ,draw opacity=1 ][fill={rgb, 255:red, 255; green, 0; blue, 0 }  ,fill opacity=1 ] (518.55,522.09) .. controls (518.84,520.38) and (520.46,519.23) .. (522.16,519.52) .. controls (523.87,519.81) and (525.02,521.42) .. (524.73,523.13) .. controls (524.45,524.84) and (522.83,525.99) .. (521.12,525.7) .. controls (519.42,525.41) and (518.27,523.8) .. (518.55,522.09) -- cycle ;
			\draw  [color={rgb, 255:red, 32; green, 0; blue, 255 }  ,draw opacity=1 ][fill={rgb, 255:red, 32; green, 0; blue, 255 }  ,fill opacity=1 ] (224.55,622.09) .. controls (224.84,620.38) and (226.46,619.23) .. (228.16,619.52) .. controls (229.87,619.81) and (231.02,621.42) .. (230.73,623.13) .. controls (230.45,624.84) and (228.83,625.99) .. (227.12,625.7) .. controls (225.42,625.41) and (224.27,623.8) .. (224.55,622.09) -- cycle ;
			\draw  [color={rgb, 255:red, 32; green, 0; blue, 255 }  ,draw opacity=1 ][fill={rgb, 255:red, 32; green, 0; blue, 255 }  ,fill opacity=1 ] (166.55,299.2) .. controls (166.84,297.5) and (168.46,296.35) .. (170.16,296.63) .. controls (171.87,296.92) and (173.02,298.54) .. (172.73,300.24) .. controls (172.45,301.95) and (170.83,303.1) .. (169.12,302.81) .. controls (167.42,302.53) and (166.27,300.91) .. (166.55,299.2) -- cycle ;
			
			\draw (343.42,362.14) node [anchor=north west][inner sep=0.75pt]    {$\mathcal{M}$};
			\draw (327.32,514.35) node [anchor=north west][inner sep=0.75pt]    {$\mathrm{Met}( E)$};
			\draw (15.24,488.29) node [anchor=north west][inner sep=0.75pt]    {$\partial \mathrm{Met}( E)$};
			\draw (22,180.4) node [anchor=north west][inner sep=0.75pt]    {$\mathcal{M}^{\mathbf{NA}}$};
			\draw (148,38.4) node [anchor=north west][inner sep=0.75pt]    {$\mathcal{M}^{\mathbf{NA}}  >0$};
			\draw (544,441.4) node [anchor=north west][inner sep=0.75pt]    {$\mathcal{M}^{\mathbf{NA}} < 0$};
			\draw (92,303.4) node [anchor=north west][inner sep=0.75pt]    {$\mathcal{M}^{\mathbf{NA}} =0$};

		\end{tikzpicture}

	\end{figure}
	\vspace*{\fill}

	\newpage \restoregeometry\ \newpage

\chapter*{Abstract}
\label{ch:abstract}

In this thesis we study the principle that \emph{extremal objects in differential geometry correspond to stable objects in algebraic geometry}. In our introduction we survey the most famous instances of this principle with a view towards the results and background needed in the later chapters. In \cref{part:zcritical} we discuss the notion of a $Z$-critical metric recently introduced in joint work with Ruadha\'i Dervan and Lars Martin Sektnan \cite{dervan2021zcritical}. We prove a correspondence for existence with an analogue of Bridgeland stability in the large volume limit, and study important properties of the subsolution condition away from this limit, including identifying the analogues of the Donaldson and Yang--Mills functionals for the equation. In \cref{part:fibrations} we study the recent theory of optimal symplectic connections on K\"ahler fibrations in the isotrivial case. We prove a correspondence with the existence of Hermite--Einstein metrics on holomorphic principal bundles.
	
	\chapter*{Statement of originality}
	I declare that the work in this thesis is my own except where explicitly stated otherwise, and that all content has been appropriately referenced. In particular the majority of \cref{ch:zcriticalconnections,ch:correspondence} is joint work with Ruadha\'i Dervan and Lars Martin Sektnan \cite{dervan2021zcritical} and has been reproduced with an emphasis on my own contributions. 
	
	{\let\cleardoublepage\relax 
	\chapter*{Copyright declaration}

	The copyright of this thesis rests with the author. Unless otherwise indicated, its contents are licensed under a Creative Commons Attribution 4.0 International Licence (CC BY).	Under this licence, you may copy and redistribute the material in any medium or format for both commercial and non-commercial purposes. You may also create and distribute modified versions of the work. This on the condition that you credit the author.
	
	When reusing or sharing this work, ensure you make the licence terms clear to others by naming the licence and linking to the licence text. Where a work has been adapted, you should indicate that the work has been changed and describe those changes. Please seek permission from the copyright holder for uses of this work that are not included in this licence or permitted under UK Copyright Law.
	
	}
	
\chapter*{Acknowledgements}
\label{ch:acknowledgements}

{\setstretch{1.00}

First I wish to thank my supervisors, Simon Donaldson and Ruadha\'i Dervan, for their support and expertise throughout my PhD. In particular I thank Simon for providing me with some of his deep and expansive views of the subject of differential geometry and an appreciation of where to find interesting problems, and Ruadha\'i for his remarkable ability to clearly and precisely explain concepts in complex differential and algebraic geometry. The help and guidance of my supervisors has been critical for my PhD and for developing my understanding and appreciation of mathematics as a subject. 

I thank Lars Martin Sektnan for many helpful discussions throughout our joint work, my academic siblings Michael Hallam and Annamaria Ortu for companionship as we embarked on an exploration of K\"ahler geometry, and all the attendees of the complex geometry reading group for interesting and stimulating discussions and presentations. I wish to thank my fellow students at the LSGNT, including Jaime, Jordan, Jordan, Qaasim, Wendelin, Federico, Federico, Liam, Corvin, and many others for the fun and sporadically productive conversations and for surviving the pandemic as PhD students with me. 

I thank Richard Thomas for making me believe algebraic geometry is simple, even if I haven't quite figured out how yet, and for directing my attention to the deformed Hermitian Yang--Mills equation which lead to the major body of work of this thesis.

I thank my family for providing support and encouragement throughout my PhD, even if it could only come from across the globe, and for miraculously helping me find safe passage across the Australian border during the pandemic. 

Finally, thank you Diclehan for waiting patiently for me all these years. Without your kindness, joy, and endless support, my PhD may not have survived COVID-19 intact. Your commitment to me despite being a literal world apart is an inspiration and a wonder which I am truly privileged to have.

This PhD was supported by EPSRC grant number EP/L015234/1. I thank the examiners for many helpful comments and corrections.
}

	{
		\hypersetup{linkcolor=black}
		\tableofcontents
	}

	\mainmatter
	
	\counterwithin{theorem}{chapter}
	\counterwithin{principle}{chapter}
	\counterwithin{conjecture}{chapter}
	\counterwithin{corollary}{chapter}
	
\chapter{Introduction}
\label{ch:introduction}

The following fundamental principle has guided much of the research in complex geometry since the late-20\textsuperscript{th} century:

\begin{principle}\label{principle}
	\emph{Stable objects in algebraic geometry correspond to extremal objects in differential geometry.}
\end{principle}

The origins of this principle go as far back as Riemann's work on uniformisation of algebraic curves. What exactly is meant by \emph{stable} or \emph{extremal} depends on the setting, and there are now instances of \cref{principle} appearing as celebrated theorems in the context of \emph{points}, \emph{varieties}, \emph{vector bundles}, and many  variations on these geometric themes. 

In this thesis we will investigate several new directions which test \cref{principle}. In \cref{part:zcritical} we will study the principle in the context of a new and exciting development in algebraic geometry, that of \emph{Bridgeland stability conditions}. This is largely a reproduction of joint work of the author with Ruadha\'i Dervan and Lars Martin Sektnan \cite{dervan2021zcritical} with an emphasis of the contributions of the author.

The notion of a stability condition arose out of the physics of mirror symmetry, and key features of the new algebro-geometric theory include the ability to vary the choice of stability condition (given, to first approximation, by some algebraic invariant $Z$ called a \emph{central charge}), and indeed to extend this concept of stability from Abelian to triangulated categories. 

Thus the principle suggests the enticing possibility that we may see several new features from the theory of stability conditions reflected in differential geometry: the study of wall-crossing behaviour as we \emph{vary differential equations} in interesting ways aligned with the algebraic geometry of the situation, and \emph{the study of derived phenomena using differential geometry}. 

The following is the central folklore conjecture in the study of D-branes on the B-side of mirror symmetry, and the conjecture aligned with \cref{principle} in our setting. It is the mirror analogue of the (mostly) equally vague Thomas--Yau--Joyce conjecture in symplectic geometry and can be variously attributed to Douglas, Aspinwall--Douglas, Bridgeland, Leung--Yau--Zaslow, Maria\~no--Minasian--Moore--Strominger, Collins--Jacob--Yau, and the mononymous ``Physics" \cite{douglas2005stability,aspinwall2002d,bridgeland2007stability,LYZ,marino2000nonlinear,jacob2017special,collins20151,clay,aspinwall2009dirichlet}.

\begin{conjecture}[Folklore]\label{conj:folklore}
	There exists a Bridgeland stability condition $(Z,\calA)$ on the derived category $\DbCoh(X)$ of any Calabi--Yau threefold $X$, with complexified K\"ahler form $B+i\omega$, whose central charge is given by
	$$Z(E) = -\int_X e^{-i\omega} e^{-B} \Ch(E) \sqrt{\Td(X)},$$
	and a differential equation $D_Z(h)=0$ for a Hermitian metric-type structure $h$ on an object $E\in \calA$ such that $D_Z(h)=0$ admits a solution for $E$ if and only if $E$ is Bridgeland stable with respect to $(Z,\calA)$. 
\end{conjecture}

Several features of this conjecture are left to be made more precise, and we summarise the key points here:
\begin{itemize}
	\item The existence of the Bridgeland stability condition itself is a serious open problem for general Calabi--Yau threefolds, and much of the difficulty arises in constructing the Harder--Narasimhan filtrations of elements of the derived category. A programme for how to resolve this using differential geometry is being developed by Haiden--Katzarkov--Kontsevich--Pandit \cite{haidensemistability} called ``categorical K\"ahler geometry", but any such process will in fact rely on a Donaldson--Uhlenbeck--Yau-type theorem as proposed in the conjecture.
	\item The exact central charge $Z$ may need to change (either by removing $\sqrt{\Td(X)}$, changing to $\sqrt{\hat A(X)}$, or by accounting for quantum corrections).
	\item The exact differential equation is not known, and the current best guess is the \emph{deformed Hermitian Yang--Mills equation} identified by Leung--Yau--Zaslow and Maria\~no--Minasian--Moore--Strominger \cite{LYZ,marino2000nonlinear}.\footnote{Recent discussion by Li \cite{li2022thomas} suggests that even a \emph{differential equation} may be too much to hope for in general, but physical reasoning assures that such Hermitian metric-type objects will be critical points of some kind of action functional.} This has only been derived from mirror symmetry or physics in the case of a line bundle, but natural analogues in higher rank can be justified on general grounds.
\end{itemize}

In \cref{part:zcritical} we make some progress towards \cref{conj:folklore}. In particular, guided in part by \cref{principle} we do not restrict ourselves just to the central charge $Z$ above, but to a larger class of \emph{polynomial central charges} introduced by Bayer \cite{bayer}. Nor do we consider the case of a Calabi--Yau threefold, but instead any compact K\"ahler manifold. To any of Bayer's polynomial central charges $Z$ we associate a corresponding extremal metric, appearing as the solution of a differential equation on a holomorphic vector bundle. We call such an extremal metric a \emph{$Z$-critical metric},\footnote{``Zed critical"} and prove a correspondence between a stability condition and existence of solutions to the \emph{$Z$-critical equation} in an asymptotic limit, the \emph{large volume limit}.\footnote{This is the limit where the K\"ahler metric $\omega$ is replaced by $k\omega$ and we scale $k\to \infty$. In the language of physics this is the ``quantum limit" of string theory, where string length goes to zero and strings become well-approximated by point particles.} 

The main result of \cref{part:zcritical} is the following, which justifies \emph{a posteriori} our differential equation which has been chosen on general grounds.

\begin{theorem}\label{thm:maintheoremZstability}
	A simple, slope semistable holomorphic vector bundle $E\to (X,\omega)$ over a compact K\"ahler manifold, with locally-free graded object $\Gr(E)$, admits a $Z$-critical metric in the large volume limit if and only if it is asymptotically $Z$-stable.
\end{theorem}

The existence part of the above result has a particularly simple proof in the case where $E$ is actually slope stable, covered in \cref{sec:step1}, and such bundles give the first examples of solutions to the deformed Hermitian Yang--Mills equation on higher rank vector bundles. In \cref{sec:stabilityimpliesexistence} we will give the full details of the proof of \cref{thm:maintheoremZstability} in the case where $\Gr(E)$ has two components, which demonstrates most clearly how the algebro-geometric condition of asymptotic $Z$-stability enters into the analysis, and refer the reader to \cite[\S 4]{dervan2021zcritical} for the details of the more general case. The case we prove in \cref{sec:stabilityimpliesexistence} is sufficient to produce new ``non-trivial" examples of $Z$-critical metrics on strictly semistable rank 3 vector bundles over $\CCPP^2$ described in \cref{ex:zcritical}.

Additionally in \cref{ch:zcriticalconnections} we will study the basic geometric theory of these new equations away from the large volume limit. This leads to a new notion of a \emph{subsolution} which has not previously appeared in the context of non-Abelian gauge theory in the literature, and we study the geometric consequences of it. In particular we prove:

\begin{theorem}
	When restricted to the locus of subsolutions inside the space of Chern connections on a holomorphic vector bundle $E\to (X,\omega)$, the $Z$-critical equation is elliptic and arises from an infinite-dimensional moment map construction with respect to the action of the gauge group.
\end{theorem}

Obstructions to existence of solutions to the $Z$-critical equation are expected to arise from both subsheaves and subvarieties, and the latter is a new phenomenon in gauge theory, but familiar in the literature of stability conditions in algebraic geometry. This has been recently studied in the context of the \emph{deformed Hermitian Yang--Mills equation} on line bundles, and the \emph{J-equation}, both of which can be described as $Z$-critical equations for specific choices of $Z$.\footnote{With degenerate stability vector, in the case of the J-equation.} It has been shown in those cases how algebro-geometric obstructions due to subvarieties are captured by the notion of subsolution, and we demonstrate this further in the case of complex surfaces.

\begin{theorem}
	For a holomorphic line bundle $L\to (X,\omega)$ over a compact K\"ahler surface and a choice of polynomial central charge $Z$ satisfying the volume hypothesis, the following are equivalent:
	\begin{enumerate}
		\item $L$ admits a solution to the $Z$-critical equation,
		\item $L$ admits a subsolution to the $Z$-critical equation,
		\item $L$ is $Z$-stable with respect to analytic curves $C\subset X$. 
	\end{enumerate}
\end{theorem}

We begin the investigation of this phenomenon in higher rank for the first time, giving the first explanation for how obstructions due to subvarieties may be captured in a non-Abelian gauge-theoretic context using the subsolution condition.

\begin{theorem}
	If a holomorphic vector bundle $E\to (X,\omega)$ admits a subsolution to the $Z$-critical equation then $E$ is $Z$-stable with respect to irreducible analytic divisors $D\subset X$. Moreover if $E$ admits a strong subsolution, then it is $Z$-stable with respect to all irreducible analytic subvarieties.
\end{theorem}

In \cref{sec:variationalfunctional} we give further evidence of the importance of the subsolution condition by constructing Donaldson and Yang--Mills-type functionals for the $Z$-critical equation, and showing that restricted to the locus of subsolutions, the Donaldson-type functional is convex and has critical points precisely given by $Z$-critical metrics, which are also absolute minima of the Yang--Mills-type functional.

Curiously, the algebro-geometric stability condition implied by the subsolution condition for the $Z$-critical equation is not the same as the one predicted by Bridgeland stability, as has been observed by Collins--Yau \cite{collins2018moment}. Indeed due to this discrepancy, a counterexample to \cref{conj:folklore} has been constructed on (the non-Calabi--Yau) $\Bl_p \CCPP^2$ by Collins--Shi \cite{collins2020stability}. This suggests it may be necessary to modify the definition of Bridgeland stability in order to accurately reflect this discrepancy, and a possible (rudimentary) modification implied by the Freed--Kapustin--Witten anomaly has already been suggested in the physics literature, but appears to have gone unnoticed in the mathematical literature. In \cref{sec:Zcriticalfuturedirections} we briefly investigate this modification, and discuss several other future directions including how the notion of quasi-isomorphism may be reflected using the differential geometry of complexes of vector bundles.

In \cref{part:fibrations}, our focus is shifted to a completely different geometric context in which \cref{principle} appears, \emph{K\"ahler fibrations}. This new setting interpolates between the setting of vector bundles (seen in this context as arising from projective bundles, particular examples of K\"ahler fibrations) and varieties (K\"ahler fibrations over a point). These two settings have been well-studied in complex geometry, with \cref{principle} manifesting as the celebrated Donaldson--Uhlenbeck--Yau theorem and Chen--Donaldson--Sun theorem respectively, both of which we shall review in \cref{ch:preliminaries}. 

In the interpolating setting of fibrations, a hybrid of the stability conditions of bundles and varieties has been introduced by Dervan--Sektnan \cite{dervan2019moduli}, who also identified a corresponding notion of canonical metric on the fibration, an \emph{optimal symplectic connection} \cite{dervan2019optimal}. 

The principle predicts a correspondence between these notions, and we investigate this in the case of \emph{isotrivial K\"ahler fibrations}, where the complex structure of the fibres does not vary over the base. All such fibrations arise as associated bundles to holomorphic principal bundles, and in this geometric setting we prove the following.

\begin{theorem}\label{thm:introductionOSCHE}
	Suppose an isotrivial K\"ahler fibration $(X,\omega_X) \to (B,\omega_B)$ arises through the associated bundle construction from a holomorphic principal bundle $P\to (B,\omega_B)$ with connection $A$. Then the natural induced symplectic connection $\omega_X$ is an optimal symplectic connection if and only if $A$ is a Hermite--Einstein connection on $P$.
\end{theorem}

In particular this implies a version of \cref{principle} where we utilise the already-existing theory of stability of holomorphic principal bundles:

\begin{corollary}
	If a holomorphic principal bundle $P\to (B,\omega_B)$ is polystable, then the associated isotrivial K\"ahler fibration $(X,\omega_X) \to (B,\omega_B)$ admits an optimal symplectic connection. 
\end{corollary}

This statement becomes an \emph{if and only if} when the principal bundle under consideration is the bundle of relative automorphisms of the isotrivial K\"ahler fibration.

Through the work of Dervan--Sektnan about existence of constant scalar curvature K\"ahler metrics in adiabatic K\"ahler classes on the total space of K\"ahler fibrations $(X,\omega_X) \to (B,\omega_B)$, the existence result for optimal symplectic connections above allows us to construct many new examples of cscK metrics on the total space of isotrivial K\"ahler fibrations.

\begin{corollary}
	Suppose $P\to (B,\omega_B)$ is a simple, stable holomorphic principal $G$-bundle. Suppose that $(B,\omega_B)$ is cscK with discrete automorphisms. If $G$ acts holomorphically on a cscK manifold $(Y,\omega_Y)$ and the maximal compact subgroup $K\subset G$ acts by holomorphic isometries, then the total space of the associated bundle $X=P\times_G Y$ admits cscK metrics in adiabatic K\"ahler classes.
\end{corollary}

We describe an example of such a such a fibre bundle admitting a cscK metric on the total space, using a principal $\SL(2,\CC)$-bundle and model fibre given by the Mukai--Umemura threefold, a K\"ahler--Einstein Fano threefold admitting an action of $\SL(2,\CC)$. Many such principal bundles exist, so this gives a large class of new examples of fibrations admitting cscK metrics in adiabatic K\"ahler classes.

Further to this isotrivial setting, in \cref{sec:Fibrationsfuturedirections} we will discuss the stability of isotrivial fibrations, particularly in the case of projective bundles which has been understood by Ross--Thomas using slope K-stability \cite{ross2006obstruction}, and we will discuss a description of non-isotrivial fibrations in terms of principal bundles.

	\counterwithin{theorem}{section}
	\counterwithin{principle}{section}
	\counterwithin{conjecture}{section}
	\counterwithin{corollary}{section}
	
\chapter{Preliminaries}
\label{ch:preliminaries}

In this chapter we recall some of the preliminaries necessary to understand both $Z$-critical connections in \cref{part:zcritical} and K\"ahler fibrations in \cref{part:fibrations}, and indulge in a survey of the most famous instances of \cref{principle}.

We will give a brief overview of geometric invariant theory and its relationship to K\"ahler geometry through the Kempf--Ness theorem, which will not be used but serves as the philosophical background for the principle.

We will then recall the theory of slope stability and canonical metrics on holomorphic vector bundles over compact K\"ahler manifolds, where \cref{principle} is made precise in the celebrated Donaldson--Uhlenbeck--Yau theorem (\cref{thm:DUY}), emphasising the necessary Hermitian geometry required in later chapters. This is important in understanding the behaviour of $Z$-critical metrics \emph{at} the large volume limit $k=\infty$, and for understanding the existence of canonical metrics on isotrivial K\"ahler fibrations.

Finally we detour slightly into the corresponding stability theory for varieties, K-stability, and work towards stating the correspondence there, the Yau--Tian--Donaldson conjecture (\cref{conj:YTD}). This theory places the notion of a \emph{degeneration} as a focal point, and this language is used to understand the stability of fibrations in terms of so-called \emph{fibration degenerations} as we will describe in \cref{part:fibrations}. The same language has also been used in the study of the deformed Hermitian Yang--Mills equation by Collins--Yau \cite{collins2018moment}. 

As part of our exposition of K-stability, we will investigate an interesting new example of a singular variety whose deformation theory within the Fano deformation class $V_{14}$ can be explicitly understood in terms of the representations of $\SL(2,\CC)$. 

\section{Geometric invariant theory\label{sec:GIT}}

We will begin with the simplest, finite-dimensional manifestation of \cref{principle}, which has been either directly applied or used as a guide for most subsequent instances of the principle (except, perhaps, for Bridgeland stability). The theory of stability is known as \emph{geometric invariant theory} (GIT), and the principle becomes the Kempf--Ness theorem.

Geometric invariant theory is the language developed by Mumford to define quotients of algebraic varieties by group actions \cite{mumford1994geometric}. In accordance with the revolution of algebraic geometry in the mid 20th century, which placed the ring of functions of a variety at the centre of focus, GIT concerns itself with invariant elements of commutative rings under group actions. This problem was considered to varying extents by Hilbert in his \emph{invariant theory}, whence the name comes.

The basic observation of Mumford\footnote{To an extent, this was already implicit in the work of Hilbert, where the concept of a semi-stable point (a "Nullform") was characterised using a weight inequality \cite[p. 359]{hilbert1893vollen}.} was the following: in order for the quotient of an algebraic variety $X$ by a group action of $G$ to be a well-behaved algebraic space, it is necessary to restrict to a locus $X^s\subset X$ of well-behaved, \emph{stable} points. 

Let us restrict to the setting of a reductive group $G$ acting on a polarised variety $(X,L)$. We will only consider varieties over $\CC$, although the theory has been developed over more general fields. We require that the $G$ action on $X$ lifts to an action on the ample line bundle $L\to X$,\footnote{Equivalently if we view $X\into \PP^N$ for some $N$ then $G\subset \PGL(N,\CC)$ must lift to $G\subset \GL(N+1,\CC)$ and act on $\CC^{N+1}$.} and hence acts on all its powers $L^k$ and on the spaces of global sections $H^0(X,L^k)$. 

We have a homogeneous coordinate ring
$$R(X,L) = \bigoplus_{k\ge 0} H^0(X,L^k)$$
upon which $G$ acts, and the ring of invariants is
$$R(X,L)^G = \bigoplus_{k \ge 0} H^0(X,L^k)^G.$$
This is a graded $\CC$-algebra, and when $G$ is reductive it is finitely generated \cite{hilbert1890ueber,nagata1963invariants}. Thus we can define a projective variety
$$X\git G = \Proj R(X,L)^G,$$
the geometric invariant theory quotient of $(X,L)$ by $G$. The first key observation is that this quotient variety misses some points of $X$. That is, there is a natural rational map $X\dashrightarrow X\git G$ induced by the inclusion $R^G\into R$, and if $x\in X$ is a point such that $f(x)=0$ for every invariant section $f\in H^0(X,L^k)^G$ for every $k$, then the morphism $X\dashrightarrow X\git G$ is not defined at $x$. Such points are called \emph{unstable} by Mumford (or ``nullforms" by Hilbert). We have the following further characterisations of points in $X$:

\begin{definition}
	A point $x\in (X,L)$ is called:
	\begin{itemize}
		\item \emph{semistable} if there exists some $f\in H^0(X,L^k)^G$ such that $f(x)\ne 0$. Call the locus of semistable points $X^{ss}$.
		\item \emph{polystable} if the orbit $G\cdot x$ is closed in $X$. Call the locus of polystable points $X^{ps}$.
		\item \emph{stable} if the orbit is closed and $x$ has finite stabiliser $G_x$. Call the locus of stable points $X^{s}$.
		\item \emph{unstable} if it is not semistable. 
	\end{itemize}
\end{definition}

By definition, the rational morphism $X\dashrightarrow X\git G$ becomes a genuine morphism when restricting to the semistable locus, $X^{ss}$. Furthermore if two semistable points $x,y\in X^{ss}$ have orbits whose closures intersect, $\closure{G\cdot x} \cap \closure{G\cdot y} \ne \emptyset$, then every $G$-invariant section takes the same value on $x$ and $y$ so they become identified in $X\git G$. This relation is called \emph{S-equivalence} (after Seshadri, who investigated its consequences for the construction of moduli spaces of vector bundles). Conversely the invariant functions will separate points whose closures do not intersect, and thus the GIT quotient $X^{ps} \to X^{ps}\git G$ is actually an orbit space $X^{ps}\git G = X^{ps}/G$. In fact it will be a consequence of the Kempf--Ness theorem that the orbit space $X^{ps}/G$ is equal to the GIT quotient $X\git G$.\footnote{We note however that $X^{ps}$ does not admit any nice description as an algebraic subvariety of $X$, but is instead only a constructible subset of $X$. Therefore to obtain an algebraic quotient which is also an orbit space, one must further restrict to the stable locus $X^s$.} Thus we obtain the famous diagram of GIT quotients:
\begin{figure}[h]
	\centering
	\begin{tikzcd}
		X \arrow[dashed]{r} & X\git G \arrow[equal]{d}\\
		X^{ss} \arrow{r} \arrow[hook]{u} & X^{ss}\git G\\
		X^{ps} \arrow{r} \arrow[hook]{u} & X^{ps}/G \arrow[equal]{u}\\
		X^{s} \arrow{r} \arrow[hook]{u} & X^s/G \arrow[hook]{u}
	\end{tikzcd}
\end{figure}

It can be shown that the GIT quotient $X\git G$ satisfies the axioms of a \emph{categorical quotient} for the action of $G$ on $X$, and therefore is the optimal solution to taking quotients by reductive groups in the category of algebraic varieties (or schemes).

\begin{remark}
	There are alternative notions of GIT quotient when $X$ is affine or non-projective and still twisted by a line bundle $L\to X$. This is important in, for example, the construction of $\CCPP^n$ as a twisted affine GIT quotient of $\CC^{n+1}$, or the construction of moduli of representations as in the work of King \cite{king1994moduli}.
\end{remark}

\subsection{The Hilbert--Mumford Criterion}

The criteria to determine the stability of a point $x\in X$ above are often difficult to verify in practice, as they require information about the entire $G$-orbit of $x$. Mumford identified a simpler criterion which only requires one to know about the one parameter subgroups (1-PS) $\lambda: \CC^* \into G$ of $G$, called the \emph{Hilbert--Mumford criterion}.

Consider again the setting of a reductive group $G$ acting on a polarised variety $(X,L)$. Given a point $x\in X$ and a 1-PS $\lambda$, the limit point 
$y:= \lim_{t\to 0} \lambda(t) \cdot x$
is a fixed point of the $\CC^*$ action of $\lambda$ on $X$. Thus since the $G$ action lifts to $L$, the fibre $L_y$ over $y$ is also fixed by $\lambda$, and so comes with a weight $\mu(x,\lambda)$ such that $\lambda(t) \cdot v = t^{\mu(x,\lambda)} v$ for every $v\in L_y$. This is the \emph{Hilbert--Mumford weight} of $(x,\lambda)$. 

The Hilbert--Mumford criterion is the following \cite{mumford1994geometric,thomas2005notes}:

\begin{theorem}[Hilbert--Mumford criterion]\label{thm:HMcriterion}
	A point $x\in X$ is:
	\begin{itemize}
		\item semistable if for every 1-PS $\lambda$ in $G$, $\mu(x,\lambda) \ge 0$.
		\item polystable if $x$ is semistable and whenever $\mu(x,\lambda)=0$, $\lambda$ is in the stabiliser $G_x$ of $x$.
		\item stable if $x$ is polystable and has finite stabiliser, so that $\mu(x,\lambda)=0$ only if $\lambda$ is the trivial 1-PS.
		\item unstable if for at least one 1-PS $\lambda$ in $G$, $\mu(x,\lambda) < 0.$
	\end{itemize}
\end{theorem}

\begin{figure}[h]
	\centering
	
	\tikzset{every picture/.style={line width=0.75pt}} 
	
	\begin{tikzpicture}[x=0.75pt,y=0.75pt,yscale=-1,xscale=1]
		
		\draw    (151.55,278.07) .. controls (236.68,230.77) and (381.15,229.91) .. (467.15,280.65) ;
		\draw    (115.43,329.67) .. controls (231.52,261.73) and (369.11,262.59) .. (508.42,332.25) ;
		\draw    (77.57,116.11) .. controls (116.35,147.42) and (258.95,177.43) .. (309.42,175.59) ;
		\draw [shift={(75.87,114.68)}, rotate = 41.38] [color={rgb, 255:red, 0; green, 0; blue, 0 }  ][line width=0.75]    (10.93,-4.9) .. controls (6.95,-2.3) and (3.31,-0.67) .. (0,0) .. controls (3.31,0.67) and (6.95,2.3) .. (10.93,4.9)   ;
		\draw  [dash pattern={on 4.5pt off 4.5pt}]  (63.83,110.38) .. controls (179.07,12.35) and (447.37,11.49) .. (540.24,98.34) ;
		\draw    (63.83,110.38) -- (151.55,278.07) ;
		\draw    (540.24,98.34) -- (467.15,280.65) ;
		\draw  [fill={rgb, 255:red, 0; green, 0; blue, 0 }  ,fill opacity=1 ] (306.34,175.59) .. controls (306.34,173.89) and (307.72,172.51) .. (309.42,172.51) .. controls (311.12,172.51) and (312.5,173.89) .. (312.5,175.59) .. controls (312.5,177.29) and (311.12,178.67) .. (309.42,178.67) .. controls (307.72,178.67) and (306.34,177.29) .. (306.34,175.59) -- cycle ;
		\draw  [fill={rgb, 255:red, 0; green, 0; blue, 0 }  ,fill opacity=1 ] (308.06,279.29) .. controls (308.06,277.59) and (309.44,276.21) .. (311.14,276.21) .. controls (312.84,276.21) and (314.22,277.59) .. (314.22,279.29) .. controls (314.22,280.99) and (312.84,282.37) .. (311.14,282.37) .. controls (309.44,282.37) and (308.06,280.99) .. (308.06,279.29) -- cycle ;
		\draw    (309.42,175.59) .. controls (357.45,176.59) and (506.27,146.25) .. (532.6,104.76) ;
		\draw [shift={(533.36,103.5)}, rotate = 119.74] [color={rgb, 255:red, 0; green, 0; blue, 0 }  ][line width=0.75]    (10.93,-4.9) .. controls (6.95,-2.3) and (3.31,-0.67) .. (0,0) .. controls (3.31,0.67) and (6.95,2.3) .. (10.93,4.9)   ;
		\draw    (309.42,175.59) .. controls (413.83,181.76) and (458.55,195.16) .. (491.22,221.81) ;
		\draw    (309.42,175.59) .. controls (441.35,186.06) and (473.17,239.37) .. (467.15,280.65) ;
		\draw  [fill={rgb, 255:red, 0; green, 0; blue, 0 }  ,fill opacity=1 ] (488.14,221.81) .. controls (488.14,220.11) and (489.52,218.73) .. (491.22,218.73) .. controls (492.93,218.73) and (494.31,220.11) .. (494.31,221.81) .. controls (494.31,223.52) and (492.93,224.9) .. (491.22,224.9) .. controls (489.52,224.9) and (488.14,223.52) .. (488.14,221.81) -- cycle ;
		\draw  [fill={rgb, 255:red, 0; green, 0; blue, 0 }  ,fill opacity=1 ] (464.06,280.65) .. controls (464.06,278.95) and (465.44,277.57) .. (467.15,277.57) .. controls (468.85,277.57) and (470.23,278.95) .. (470.23,280.65) .. controls (470.23,282.35) and (468.85,283.73) .. (467.15,283.73) .. controls (465.44,283.73) and (464.06,282.35) .. (464.06,280.65) -- cycle ;
		\draw  [dash pattern={on 0.84pt off 2.51pt}]  (309.78,176.6) .. controls (205.08,182.5) and (156.8,199.22) .. (124.03,224.75) ;
		\draw  [dash pattern={on 0.84pt off 2.51pt}]  (309.78,176.6) .. controls (177.48,186.62) and (145.58,237.68) .. (151.61,277.21) ;
		\draw    (318.38,297.85) .. controls (367.84,297.01) and (416.49,314.16) .. (439.62,326.17) ;
		\draw [shift={(441.35,327.09)}, rotate = 208.3] [color={rgb, 255:red, 0; green, 0; blue, 0 }  ][line width=0.75]    (10.93,-3.29) .. controls (6.95,-1.4) and (3.31,-0.3) .. (0,0) .. controls (3.31,0.3) and (6.95,1.4) .. (10.93,3.29)   ;
		\draw  [fill={rgb, 255:red, 0; green, 0; blue, 0 }  ,fill opacity=1 ] (454.25,309.38) .. controls (454.25,307.68) and (455.63,306.3) .. (457.33,306.3) .. controls (459.03,306.3) and (460.41,307.68) .. (460.41,309.38) .. controls (460.41,311.09) and (459.03,312.47) .. (457.33,312.47) .. controls (455.63,312.47) and (454.25,311.09) .. (454.25,309.38) -- cycle ;
		\draw  [fill={rgb, 255:red, 0; green, 0; blue, 0 }  ,fill opacity=1 ] (375.13,287.03) .. controls (375.13,285.32) and (376.51,283.94) .. (378.21,283.94) .. controls (379.92,283.94) and (381.29,285.32) .. (381.29,287.03) .. controls (381.29,288.73) and (379.92,290.11) .. (378.21,290.11) .. controls (376.51,290.11) and (375.13,288.73) .. (375.13,287.03) -- cycle ;
		
		\draw (312.52,258.29) node [anchor=north west][inner sep=0.75pt]    {$x$};
		\draw (315.96,155.95) node [anchor=north west][inner sep=0.75pt]    {$v$};
		\draw (349.84,314.18) node [anchor=north west][inner sep=0.75pt]    {$\lambda \rightarrow 0$};
		\draw (378.26,267.75) node [anchor=north west][inner sep=0.75pt]    {$\lambda \cdot x$};
		\draw (463.01,290.1) node [anchor=north west][inner sep=0.75pt]    {$y$};
		\draw (475.05,272.91) node [anchor=north west][inner sep=0.75pt]    {$0$};
		\draw (541.77,85.44) node [anchor=north west][inner sep=0.75pt]    {$\infty $};
		\draw (103.55,220.02) node [anchor=north west][inner sep=0.75pt]    {$L$};
		\draw (102.55,303.43) node [anchor=north west][inner sep=0.75pt]    {$X$};
		\draw (544.74,108.52) node [anchor=north west][inner sep=0.75pt]    {$\mu ( x,\lambda )  >0$};
		\draw (507.74,213.52) node [anchor=north west][inner sep=0.75pt]    {$\mu ( x,\lambda ) =0$};
		\draw (489.74,281.52) node [anchor=north west][inner sep=0.75pt]    {$\mu ( x,\lambda ) < 0$};

	\end{tikzpicture}

	\caption{The Hilbert--Mumford criterion. If $v$ lies over $x$, then as $\lambda \to 0$, $\lambda(t)\cdot v\in L_{\lambda(t)\cdot x}$ will go to infinity for $\mu(x,\lambda)>0$ and zero for $\mu(x,\lambda)<0$. Notice that the orbit of $\tilde{x}$ is half-closed if $\mu(x,\lambda)>0$, but half-open when $\mu(x,\lambda) = 0$. One can check the other ($\lambda \to \infty$) end of the orbit by replacing the 1-PS with its inverse. Combining the statement for these two 1-PS gives the Hilbert-Mumford criterion relating the asymptotic weight of $\lambda$ to the closedness of the orbit $\lambda \cdot x$.}\label{fig:hilbertmumford}
	
\end{figure}
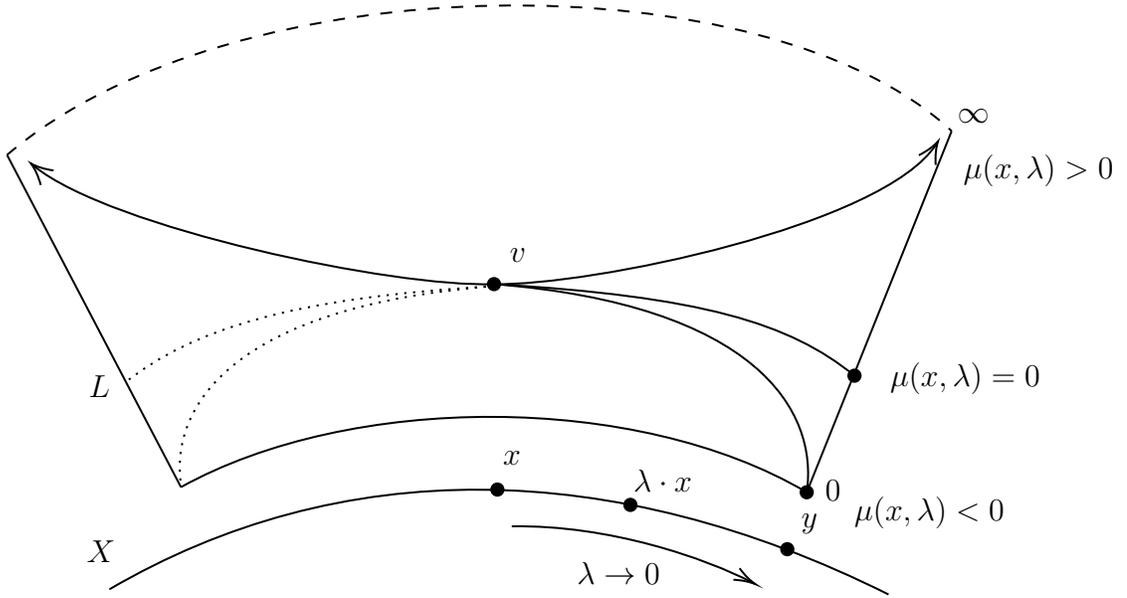

A famous pictorial proof of the Hilbert--Mumford criterion as it relates to the closedness of $\CC^*$-orbits of points is given in \cref{fig:hilbertmumford}.

We will not utilize the Hilbert--Mumford criterion directly following this, however the notion of a 1-parameter degeneration of an object is central to the understanding of stability for both bundles and varieties as we will subsequently explain.

\subsection{The Kempf--Ness theorem\label{sec:kempfness}}

According to \cref{principle} corresponding to the GIT stability of points there should be an extremal notion, which we now explain.

Consider now the setting of $(X,\omega)$ a compact symplectic manifold, and suppose a real compact Lie group $K$ acts on $X$ by Hamiltonian symplectomorphisms. That is, for each $\xi\in \k$, the induced vector field $v_\xi$ is Hamiltonian. 

\begin{definition}
	A \emph{moment map} $\mu: X \to \k^*$ for the Hamiltonian action of $K$ on $(X,\omega)$ is a $K$-equivariant map with respect to the coadjoint action on $\k^*$ such that 
	$$d\langle \mu, \xi\rangle = - i_{v_\xi} \omega$$
	for all $\xi \in \k$, where $\langle -,- \rangle$ is the natural pairing between $\k^*$ and $\k$.. The map $\k \to C^{\infty}(X)$ sending $\xi \mapsto \langle \mu, \xi\rangle$ is called the \emph{comoment map}, which is sometimes denoted $\mu^*$.
\end{definition}

The \emph{symplectic} or \emph{Marsden--Weinstein quotient} of $(X,\omega)$ by $K$ is given by
$$X\git K := \mu^{-1}(0) / K.$$
If $0$ is a regular value of $\mu$ and $K$ acts freely on $\mu^{-1}(0)$, then the symplectic quotient is a smooth manifold and naturally inherits a symplectic structure $\omega_{\mathrm{red}}$ from $X$ such that
$$\pi^* \omega_{\mathrm{red}} = \iota^* \omega$$
where $\iota: \mu^{-1}(0) \into X$ is the inclusion and $\pi: \mu^{-1}(0) \to X\git K$ the projection.

Let us now return to the case of a projective variety $(X,L)$. Fix a K\"ahler metric $\omega \in c_1(L)$ and suppose a group $G$ acts linearly on $(X,L)$. Furthermore, suppose a maximal compact subgroup $K\subset G$ acts by Hamiltonian symplectomorphisms on $(X,\omega)$. Then we now have two possible quotients of $X$, one arising from GIT and the other from symlectic geometry. The Kempf--Ness theorem relates these.

\begin{theorem}[Kempf--Ness \cite{kempf1979length}, Guillemin--Sternberg \cite{guillemin1984convexity}, Kirwan \cite{kirwan1984cohomology}]\label{thm:kempfness}
	Let $G$ be a reductive group with maximal compact subgroup $K$ acting linearly on a projective variety $(X,L)$, such that $K$ acts by Hamiltonians on $(X,\omega)$ where $\omega\in c_1(L)$ is a K\"ahler form. Suppose $\mu: X \to \k^*$ is a moment map for the $K$ action. A point $x\in X$ is GIT polystable for the $G$ action if and only if $G\cdot x$ intersects $\mu^{-1}(0)$.\footnote{Exactly which level set of the moment map one takes is critical in the statement of the Kempf--Ness theorem, as shifting the level set to some $\xi \in \mathfrak{z}(\mathfrak{k}^*)$ corresponds to twisting the $\CC^*$ linearisation on the ample line bundle $L$ by a character. See \cite[Rmk. 5.22]{szekelyhidi2014introduction}.}  Furthermore, if $x$ is polystable then $G\cdot x$ intersects $\mu^{-1}(0)$ is a unique $K$-orbit.
\end{theorem}

It was first proved by Kirwan \cite{kirwan1984cohomology} that this gives a homeomorphism of quotient spaces. Indeed further to \cref{thm:kempfness} one has:
\begin{itemize}
	\item $G\cdot \mu^{-1}(0) = X^{ps}$,
	\item A point $x\in X$ is semistable if and only if $\closure{G\cdot x}$ meets $\mu^{-1}(0)$ in a single $K$-orbit.
	\item $0$ is a regular value of $\mu$ if and only if $X^{ss}=X^s$.
	\item The inclusion $\mu^{-1}(0)\into X^{ss}$ induces a homeomorphism of quotients
	$$X\git_0 K \to (X,L)\git G.$$
\end{itemize}
The homeomorphism between quotients follows from the fact that a continuous bijection from a compact space to a Hausdorff space is a homeomorphism. The continuous inverse
$$(X,L)\git G \to X\git_0 K$$
can be found by following the gradient flow of the norm squared of the moment map,
$$V = \|\mu\|^2.$$
In particular the gradient flow of $V$ takes a semistable point $x\in X^{ss}$ and flows down to a point $y\in X^{ps}$ in the unique polystable orbit in the closure of $G\cdot x$ \cite{kirwan1984cohomology}. 

The Kempf--Ness theorem may be proven working in a fixed orbit of $G$ and considering the so-called \emph{Kempf--Ness functional},
$$\calM : G/K \to \RR.$$
The Kempf--Ness functional $\calM_x$ on an orbit $G\cdot x$ satisfies (or may be defined as the unique functional satisfying) the characteristic variational property that
$$\left.\deriv{}{t}\right|_{t=0} \calM([\exp(it\xi) g]) = -2\langle \mu(g\cdot x), \xi\rangle$$
for $\xi \in \k$. In particular a point $g\cdot x\in G\cdot x$ is a zero of the moment map $\mu(g\cdot x)=0$ if and only if $[g]$ is a critical point of the Kempf--Ness functional. To prove the theorem, one goes further to establish that $\calM$ is convex along geodesics in $G/K$, which is a complete negatively curved metric space with respect to the induced Riemannian metric after choosing a bi-invariant metric on $G$. One deduces that the orbit of $g\cdot v$ is closed for $v$ some lift of $x$ if and only if $\calM_x$ is proper on $G/K$, which by the convexity of $\calM$ occurs precisely when it has a critical point \cite{szekelyhidi2014introduction,georgoulas2022moment}.

The relationship to the GIT picture can be made even more explicit. Indeed one has the following.
\begin{lemma}[See for example {\cite[Lem. 5.2]{georgoulas2022moment}}]\label{lem:limitslopeGIT}
	Let $\xi\in \k$ be rational and let $\lambda :\CC^* \into G$ be the 1-PS generated by $\xi$. Then the Hilbert--Mumford weight $\mu(x,\lambda)$ is the slope at infinity of the Kempf--Ness functional:
	$$\mu(x,\lambda) = \lim_{t\to \infty} \frac{\calM_x([\lambda(t)])}{t}.$$
\end{lemma}
The density of rational directions in $\g/\k \isom \k$ and the above lemma produce the following philosophical statement: Provided $\calM$ has good enough behaviour, it suffices to verify properness over all of $G/K$ just by checking the limiting slopes are positive along rationally defined geodesics. This is the fundamental fact which underpins \cref{principle}, and explains how just \emph{algebraic} (i.e. \emph{rational}) information is strong enough to imply \emph{analytical} existence theorems (see for example \cite[\S 1]{boucksom2018variational} for a discussion of this principle in the finite-dimensional setting and how it relates to the variational approach to the study of K-stability).

\section{Holomorphic vector bundles\label{sec:bundles}}

The study of stability and extremal metrics on holomorphic vector bundles began with the work of Narasimhan and Seshadri relating the slope stability of an algebraic vector bundle over a compact Riemann surface, in the sense of Mumford--Takemoto, with existence of a projective unitary representation of the fundamental group of the surface \cite{narasimhan1965stable}. This work was a development in relation to ideas of Andr\'e Weil, who had already understood this correspondence for line bundles and suggested a similar picture in higher rank. 

The theorem of Narasimhan and Seshadri was placed within a much broader context by Atiyah and Bott \cite{atiyahbott}, who emphasized the role played by \emph{Yang--Mills connections}, which in the case of compact Riemann surfaces are just projectively flat connections, whose holonomy therefore produces the representations of the fundamental group. Variously, Atiyah--Bott introduced an interpretation of the Yang--Mills equation as an infinite-dimensional moment map, bringing the conceptual picture of geometric invariant theory into focus, as well as providing new techniques in equivariant Morse theory to compute the topology of these symplectic quotients --- the moduli of stable vector bundles on a curve (a theory subsequently developed by Kirwan for many types of symplectic quotients). At the same time an independent proof of the Narasimhan--Seshadri theorem using the language of Yang--Mills connections was provided by Donaldson \cite{donaldson1983new}, and interpreted in the language of Atiyah and Bott this proof could be interpreted as an infinite-dimensional Kempf--Ness theorem for the curvature as a moment map.

Over the 1980s this theory was generalised to higher dimensional bases by Donaldson, Uhlenbeck and Yau, and others, culminating in a correspondence bearing their names. We now describe this theory.

\subsection{Stability\label{sec:stability}}

Let us first turn to the stability theory of holomorphic vector bundles, and in particular slope stability and (briefly) Gieseker stability. 

\subsubsection{Slope stability}

The notion of slope stability of a vector bundle was determined by Mumford shortly after his development of geometric invariant theory \cite{mumford1963projective}. Mumford considered the problem over a curve, and the condition for higher dimensional bases was determined by Takemoto  \cite{takemoto1972stable}, so this notion is often called Mumford--Takemoto stability. We will state the definition in the setting of compact K\"ahler manifolds, which includes the projective settings of Mumford and Takemoto.

\begin{definition}
	The slope of a coherent analytic sheaf $\calE$ over a compact K\"ahler manifold $(X,\omega)$ is defined as
	$$\mu(\calE) = \frac{\deg \calE}{\rk \calE} = \frac{ (c_1(\calE) . [\omega]^{n-1}) [X]}{\rk \calE}$$
	whenever $\rk \calE > 0$, and $\mu(\calE) = +\infty$ if $\rk \calE = 0$.
\end{definition}

\begin{definition}
	A coherent analytic sheaf $\calE$ on a compact K\"ahler manifold $(X,\omega)$ is \emph{slope semistable} if for all non-trivial coherent analytic subsheaves $\calS\into \calE$, we have the inequality
	$$\mu(\calS) \le \mu(\calE).$$
	Furthermore we say $\calE$ is \emph{slope stable} if, whenever $\rk \calS < \rk \calE$, we have
	$$\mu(\calS) < \mu(\calE).$$
	We say the sheaf is \emph{slope polystable} if it is a direct sum of slope stable sheaves of the same slope. The sheaf is \emph{slope unstable} if it is not slope semistable.
\end{definition}
\begin{remark}
	Typically stability is defined for a \emph{torsion-free} coherent analytic sheaf. It is an instant consequence of our definition that a slope (semi)stable coherent sheaf of positive rank is torsion-free. In general a well-behaved theory exists for pure coherent sheaves of a fixed dimension, as detailed in \cite[\S 1.6]{huybrechts-lehn}, but the slope of torsion sheaves must be defined as the quotient of leading order coefficients of the Hilbert polynomial in that case.
\end{remark}

Slope stability can be intuitively understood when $\calE$ is a holomorphic vector bundle $E$. In particular, it is a topological condition on the holomorphic geometry of $E$, which asks that $E$ does not have sitting inside it holomorphic subsheaves which are ``more twisted" than $E$ itself (after appropriate normalization by dividing by ranks). See \cref{fig:unstablebundle}.

\begin{figure}[h]
	\centering

	\tikzset{every picture/.style={line width=0.75pt}} 
	
	\begin{tikzpicture}[x=0.75pt,y=0.75pt,yscale=-1,xscale=1]
		
		\draw [color={rgb, 255:red, 65; green, 117; blue, 5 }  ,draw opacity=1 ]   (422.5,236.5) .. controls (480.4,241.33) and (562,273.73) .. (554.8,294.53) ;
		\draw [shift={(554.8,294.53)}, rotate = 154.09] [color={rgb, 255:red, 65; green, 117; blue, 5 }  ,draw opacity=1 ][line width=0.75]    (-5.59,0) -- (5.59,0)(0,5.59) -- (0,-5.59)   ;
		\draw  [color={rgb, 255:red, 255; green, 0; blue, 0 }  ,draw opacity=0.6 ][fill={rgb, 255:red, 255; green, 0; blue, 0 }  ,fill opacity=0.3 ] (71,64) .. controls (71,64) and (295.33,150.67) .. (283,186) .. controls (270.67,221.33) and (283.33,185.47) .. (283,186) .. controls (282.67,186.53) and (110,169) .. (110,171) .. controls (110,173) and (71,64) .. (71,64) -- cycle ;
		\draw  [color={rgb, 255:red, 255; green, 0; blue, 0 }  ,draw opacity=0.6 ][fill={rgb, 255:red, 255; green, 0; blue, 0 }  ,fill opacity=0.3 ] (283,186) .. controls (285,185) and (537,217) .. (538,218) .. controls (539,219) and (509,329) .. (509,327) .. controls (509,325) and (262,234) .. (275,207) .. controls (288,180) and (281,187) .. (283,186) -- cycle ;
		\draw [line width=1.5]    (37,64) -- (37,171) ;
		\draw [line width=1.5]  [dash pattern={on 1.69pt off 2.76pt}]  (37,171) -- (144,171) ;
		\draw [line width=1.5]    (37,64) -- (144,64) ;
		\draw [line width=1.5]  [dash pattern={on 1.69pt off 2.76pt}]  (144,64) -- (144,171) ;
		\draw [line width=1.5]    (470,219) -- (470,326) ;
		\draw [line width=1.5]    (470,326) -- (577,326) ;
		\draw [line width=1.5]    (470,219) -- (577,219) ;
		\draw [line width=1.5]    (577,219) -- (577,326) ;
		\draw [line width=1.5]    (144,64) .. controls (371,113) and (357,219) .. (470,219) ;
		\draw [line width=1.5]    (37,171) .. controls (144.64,213.29) and (226.95,227.53) .. (317.71,247.34) ;
		\draw [line width=1.5]    (398.97,199.27) .. controls (479.62,201.47) and (511.65,203.89) .. (577,219) ;
		\draw [line width=1.5]  [dash pattern={on 1.69pt off 2.76pt}]  (144,171) .. controls (239,198) and (337,200) .. (398.97,199.27) ;
		\draw [line width=1.5]  [dash pattern={on 1.69pt off 2.76pt}]  (317.71,247.34) .. controls (403,265) and (470,271) .. (577,326) ;
		\draw  [dash pattern={on 4.5pt off 4.5pt}]  (90.5,117.5) -- (523.5,272.5) ;
		\draw  [color={rgb, 255:red, 0; green, 0; blue, 0 }  ,draw opacity=1 ][fill={rgb, 255:red, 0; green, 0; blue, 0 }  ,fill opacity=1 ] (88,117.5) .. controls (88,116.12) and (89.12,115) .. (90.5,115) .. controls (91.88,115) and (93,116.12) .. (93,117.5) .. controls (93,118.88) and (91.88,120) .. (90.5,120) .. controls (89.12,120) and (88,118.88) .. (88,117.5) -- cycle ;
		\draw  [color={rgb, 255:red, 0; green, 0; blue, 0 }  ,draw opacity=1 ][fill={rgb, 255:red, 0; green, 0; blue, 0 }  ,fill opacity=1 ] (521,272.5) .. controls (521,271.12) and (522.12,270) .. (523.5,270) .. controls (524.88,270) and (526,271.12) .. (526,272.5) .. controls (526,273.88) and (524.88,275) .. (523.5,275) .. controls (522.12,275) and (521,273.88) .. (521,272.5) -- cycle ;
		\draw [color={rgb, 255:red, 255; green, 0; blue, 0 }  ,draw opacity=0.6 ]   (71,64) -- (110,171) ;
		\draw [color={rgb, 255:red, 255; green, 0; blue, 0 }  ,draw opacity=0.6 ]   (538,218) -- (509,327) ;
		\draw [color={rgb, 255:red, 65; green, 117; blue, 5 }  ,draw opacity=1 ]   (58.6,151) .. controls (238.12,213.22) and (352.09,226.02) .. (426.79,237.02) ;
		\draw [shift={(58.6,151)}, rotate = 64.12] [color={rgb, 255:red, 65; green, 117; blue, 5 }  ,draw opacity=1 ][line width=0.75]    (-5.59,0) -- (5.59,0)(0,5.59) -- (0,-5.59)   ;
		\draw  [color={rgb, 255:red, 65; green, 117; blue, 5 }  ,draw opacity=1 ][fill={rgb, 255:red, 65; green, 117; blue, 5 }  ,fill opacity=1 ] (420,236.5) .. controls (420,235.12) and (421.12,234) .. (422.5,234) .. controls (423.88,234) and (425,235.12) .. (425,236.5) .. controls (425,237.88) and (423.88,239) .. (422.5,239) .. controls (421.12,239) and (420,237.88) .. (420,236.5) -- cycle ;
		\draw [color={rgb, 255:red, 189; green, 16; blue, 224 }  ,draw opacity=1 ]   (82.4,97.33) .. controls (134.4,140.33) and (124,153) .. (267,181) ;
		\draw [shift={(82.4,97.33)}, rotate = 84.59] [color={rgb, 255:red, 189; green, 16; blue, 224 }  ,draw opacity=1 ][line width=0.75]    (-5.59,0) -- (5.59,0)(0,5.59) -- (0,-5.59)   ;
		\draw  [color={rgb, 255:red, 189; green, 16; blue, 224 }  ,draw opacity=1 ][fill={rgb, 255:red, 189; green, 16; blue, 224 }  ,fill opacity=1 ] (111,125.5) .. controls (111,124.12) and (112.12,123) .. (113.5,123) .. controls (114.88,123) and (116,124.12) .. (116,125.5) .. controls (116,126.88) and (114.88,128) .. (113.5,128) .. controls (112.12,128) and (111,126.88) .. (111,125.5) -- cycle ;
		\draw [color={rgb, 255:red, 189; green, 16; blue, 224 }  ,draw opacity=1 ]   (267,181) .. controls (557,253) and (391,243) .. (521,286) ;
		\draw [shift={(521,286)}, rotate = 63.3] [color={rgb, 255:red, 189; green, 16; blue, 224 }  ,draw opacity=1 ][line width=0.75]    (-5.59,0) -- (5.59,0)(0,5.59) -- (0,-5.59)   ;
		\draw  [color={rgb, 255:red, 189; green, 16; blue, 224 }  ,draw opacity=1 ][fill={rgb, 255:red, 189; green, 16; blue, 224 }  ,fill opacity=1 ] (453,247.5) .. controls (453,246.12) and (454.12,245) .. (455.5,245) .. controls (456.88,245) and (458,246.12) .. (458,247.5) .. controls (458,248.88) and (456.88,250) .. (455.5,250) .. controls (454.12,250) and (453,248.88) .. (453,247.5) -- cycle ;
		\draw    (90.5,235.5) -- (523.5,390.5) ;
		\draw  [color={rgb, 255:red, 65; green, 117; blue, 5 }  ,draw opacity=1 ][fill={rgb, 255:red, 65; green, 117; blue, 5 }  ,fill opacity=1 ] (420,354) .. controls (420,352.62) and (421.12,351.5) .. (422.5,351.5) .. controls (423.88,351.5) and (425,352.62) .. (425,354) .. controls (425,355.38) and (423.88,356.5) .. (422.5,356.5) .. controls (421.12,356.5) and (420,355.38) .. (420,354) -- cycle ;
		\draw  [color={rgb, 255:red, 189; green, 16; blue, 224 }  ,draw opacity=1 ][fill={rgb, 255:red, 189; green, 16; blue, 224 }  ,fill opacity=1 ] (453,365.5) .. controls (453,364.12) and (454.12,363) .. (455.5,363) .. controls (456.88,363) and (458,364.12) .. (458,365.5) .. controls (458,366.88) and (456.88,368) .. (455.5,368) .. controls (454.12,368) and (453,366.88) .. (453,365.5) -- cycle ;
		\draw  [color={rgb, 255:red, 189; green, 16; blue, 224 }  ,draw opacity=1 ][fill={rgb, 255:red, 189; green, 16; blue, 224 }  ,fill opacity=1 ] (111,244) .. controls (111,242.62) and (112.12,241.5) .. (113.5,241.5) .. controls (114.88,241.5) and (116,242.62) .. (116,244) .. controls (116,245.38) and (114.88,246.5) .. (113.5,246.5) .. controls (112.12,246.5) and (111,245.38) .. (111,244) -- cycle ;
		\draw    (564,355) -- (564,368) ;
		\draw [shift={(564,371)}, rotate = 270] [fill={rgb, 255:red, 0; green, 0; blue, 0 }  ][line width=0.08]  [draw opacity=0] (8.93,-4.29) -- (0,0) -- (8.93,4.29) -- cycle    ;
		\draw [line width=1.5]    (37,64) .. controls (323,144) and (242,264) .. (470,326) ;
		\draw    (278,252.33) -- (278,283) ;
		\draw [shift={(278,286)}, rotate = 270] [fill={rgb, 255:red, 0; green, 0; blue, 0 }  ][line width=0.08]  [draw opacity=0] (8.93,-4.29) -- (0,0) -- (8.93,4.29) -- cycle    ;
		
		\draw (528,275.9) node [anchor=north west][inner sep=0.75pt]    {$0$};
		\draw (557,335.4) node [anchor=north west][inner sep=0.75pt]    {$E$};
		\draw (535,189.4) node [anchor=north west][inner sep=0.75pt]  [color={rgb, 255:red, 255; green, 0; blue, 0 }  ,opacity=0.6 ]  {$F$};
		\draw (556,376.4) node [anchor=north west][inner sep=0.75pt]    {$X$};
		\draw (190.8,198.4) node [anchor=north west][inner sep=0.75pt]  [color={rgb, 255:red, 65; green, 117; blue, 5 }  ,opacity=1 ]  {$s$};
		\draw (155.2,156.4) node [anchor=north west][inner sep=0.75pt]  [color={rgb, 255:red, 189; green, 16; blue, 224 }  ,opacity=1 ]  {$t$};
		\draw (79,257.4) node [anchor=north west][inner sep=0.75pt]    {$t=0$};
		\draw (420,376.4) node [anchor=north west][inner sep=0.75pt]    {$t=0$};
		\draw (376,357.4) node [anchor=north west][inner sep=0.75pt]    {$s=0$};
		\draw (77.2,120.07) node [anchor=north west][inner sep=0.75pt]    {$0$};

	\end{tikzpicture}
	
	\caption{An unstable bundle $E$ has holomorphic subbundles $F$ which are more twisted than $E$.}
	\label{fig:unstablebundle}
\end{figure}
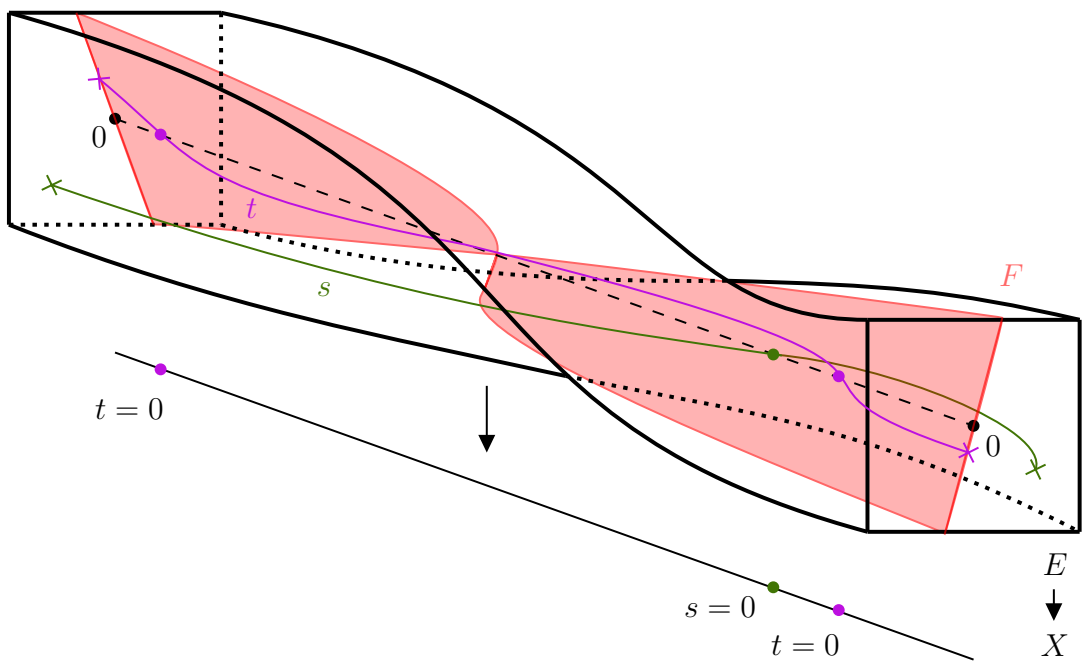

This criterion limiting the complexity of the holomorphic geometry of $E$ has various consequences, in particular for the existence of automorphisms of $E$. We recall:
\begin{proposition}[See {\cite[Prop. 5.7.11, Cor. 5.7.12]{kobayashi1987differential}}]\label{prop:stabilitysimplicity}
	Suppose $\calE,\calE'$ are slope semistable coherent sheaves over a compact K\"ahler manifold $(X,\omega)$. Then
	\begin{enumerate}
		\item If $\mu(\calE) > \mu(\calE')$ then $\Hom(\calE,\calE') = 0$.
		\item Suppose $\calE, \calE'$ are slope stable, $\mu(\calE) = \mu(\calE')$, and $u:\calE\to \calE'$ is a non-zero morphism. Then $u$ is injective and $\rk \image u = \rk \calE'$ (that is, $u$ is generically surjective). In particular $u$ is an isomorphism away from a locus of codimension $\ge 2$. If $\calE$ and $\calE'$ have the same Chern numbers, that is
		$$c_i(\calE) . [\omega]^{n-i} = c_i(\calE') . [\omega]^{n-i}$$
		for all $i$, then $u$ is in fact an isomorphism, since $\calE'/\image u$ must have vanishing Chern numbers but if it is not zero then $\Ch_k(\calE'/\image u).[\omega]^{n-k} > 0$ where $k$ is the codimension of the support of $\calE'/\image u$.\footnote{When $\calE, \calE'$ are locally free it is sufficient that $\rk \calE = \rk \calE'$ and $\deg \calE = \deg \calE'$. See \cite[Cor. 5.7.12]{kobayashi1987differential} for the proof of this case.} Furthermore, in this case after making an identification $\calE = \calE'$ then $u$ is always a constant multiple of the identity morphism $u=\lambda \id_\calE$. 
	\end{enumerate}
\end{proposition}

The second point above gives a strong restriction on the automorphisms of slope stable coherent sheaves. We recall the following definition:
\begin{definition}
	A coherent sheaf $\calE$ on $X$ is called \emph{simple} if $$H^0(X,\End(\calE)) \isom \CC\cdot \id_\calE.$$ 
\end{definition}
Thus the second point of \cref{prop:stabilitysimplicity} asserts that a slope stable coherent sheaf is simple. Every such sheaf has automorphisms given by the constant multiples of the identity morphism, so a simple sheaf is one with automorphism group \emph{as small as possible}.\footnote{One perspective is to consider the projectivised gauge group $\calG(E)/\CC^*\cdot \id_E$, under which stable bundles have trivial automorphism group.}  As observed in our discussion of geometric invariant theory, it is often the automorphisms of an object (the stabilisers of the corresponding point in the Quot scheme) which are problematic for the construction of well-behaved quotients. In particular, one expects that after restricting to (semi)stable sheaves one should obtain a better-behaved moduli space. The moduli space of slope (semi)stable sheaves is particularly well understood in one and two dimensions \cite[\S 8.2]{huybrechts-lehn}, and has more recently been understood in general \cite{greb2017compact}.

A considerable theory of the algebraic geometry of stable bundles and stable sheaves now exists, and a good survey is the text of Huybrechts and Lehn \cite{huybrechts-lehn}. We will emphasise just a few further points about stability of bundles which we will need later. Firstly, we will recall the following see-saw property for short exact sequences of sheaves. We state the see-saw lemma explicity in the case of torsion sheaves, in order to emphasize the comparison with the notion of $Z$-stability with respect to subvarieties appearing in \cref{ch:zcriticalconnections}.

\begin{lemma}[See-saw lemma; See {\cite[Prop. 5.7.6]{kobayashi1987differential}}]
	\label{lem:seesawslopestability}
	Let 
	\begin{center}
		\ses{\calS}{\calE}{\calQ}
	\end{center}
	be a short exact sequence of coherent sheaves. Suppose $\rk \calE > 0$. Then
	$$\mu(\calS) \le \mu(\calE) \quad \iff \quad \mu(\calE) \le \mu(\calQ)$$
	and
	$$\mu(\calS) \ge \mu(\calE) \quad \iff \quad \mu(\calE) \ge \mu(\calQ).$$
	Furthermore:
	\begin{itemize}
		\item If $0 < \rk \calS < \rk \calE$ then
		$$\mu(\calS) < \mu(\calE) \quad \iff \quad \mu(\calE) < \mu(\calQ)$$
		and
		$$\mu(\calS) > \mu(\calE) \quad \iff \quad \mu(\calE) > \mu(\calQ).$$
		\item If $\rk \calS = 0$ then $\mu(\calS) > \mu(\calE)$ and $\mu(\calE) \ge \mu(\calQ)$ with equality whenever $\calS$ is supported in codimension $\ge 2$.
		\item If $\rk \calS = \rk \calE$ then $\mu(\calS) \le \mu(\calE)$ with equality whenever $\calQ$ is supported in codimension $\ge 2$, and $\mu(\calE) < \mu(\calQ)$.
	\end{itemize} 
\end{lemma}
\begin{proof}
	The proof follows quickly from noting that rank $\rk$ and degree $\deg$ are additive in short exact sequences of coherent sheaves and we omit it here. We give a similar proof for asymptotic $Z$-stability in \cref{lem:seesawazs}. The final considerations follow from the fact that if $\calT$ is a torsion sheaf, then $\deg \calT \ge 0$, and furthermore if $\calT$ is supported in codimension $\ge 2$, then $\deg \calT = 0$ (see for example \cite[Lem. 5.7.5]{kobayashi1987differential}).
\end{proof}
The see-saw property for the slope of sheaves allows us to rephrase the characteristic inequality for stability in several ways: Suppose $\calS\subset \calE$ is a coherent subsheaf with quotient $\calE\to \calQ$. Then the following notions of stability are equivalent (see \cite[Prop. 5.7.6]{kobayashi1987differential}):
\begin{enumerate}
	\item Stable if $\mu(\calS) \le \mu(\calE)$ for all coherent subsheaves $\calS\into \calE$, with equality only if $\rk \calS = \rk \calE$.
	\item Stable if $\mu(\calE) \le \mu(\calQ)$ for all coherent quotients $\calE \onto \calQ$, with equality only if $\rk \calQ = \rk \calE$.
	\item Stable if $\mu(\calS) < \mu(\calQ)$ for all short exact sequences $S\into \calE \onto \calQ$.
\end{enumerate}

We will see later in \cref{part:zcritical} that $Z$ stability is often more naturally stated in terms of condition (ii) above, as obstructions due to subvarieties $V\subset X$ appear via quotient sheaves $E\to E\otimes \calO_V$. Furthermore we will see in \cref{ch:correspondence} that it is actually condition (iii) which naturally manifests in the analysis of the $Z$-critical equation, as opposed to the equivalent but less symmetric condition (i). 

We will see a geometric manifestation of \cref{lem:seesawslopestability} in the next section, in terms of the principle of Hermitian geometry that \emph{curvature (generically) decreases in subbundles and increases in quotients}. 

\subsubsection{Filtrations\label{sec:filtrations}}

It will be useful to us to recall several kinds of filtrations which naturally appear for holomorphic vector bundles. 

First, given a semistable holomorphic vector bundle $E\to X$ there are a natural class of so-called \emph{Jordan--H\"older filtrations} of $E$. 

\begin{definition}
	A \emph{Jordan--H\"older filtration} $\scrE$ of a semistable vector bundle $E\to X$ is a filtration
	$$\scrE: \quad 0 = \calE_0 \subset \calE_1 \subset \cdots \subset \calE_k = E$$
	of $E$ by torsion-free coherent subsheaves such that for every $i$, $\calE'_i := \calE_i/\calE_{i-1}$ is slope stable and $\mu(\calE'_i) = \mu(E)$. 
	The polystable sheaf
	$$\Gr(E) := \bigoplus_i \calE'_i$$
	is called the \emph{associated graded object of $E$ with respect to $\scrE$}.
\end{definition}

There is a basic existence and uniqueness result for Jordan--H\"older filtrations of $E$. This comes with a subtlety that the graded object is not strictly unique. Since slope stability is ignorant to torsion sheaves of codimension $\ge 2$, the graded object $\Gr(E)$ is only unique up to codimension 2. This may be resolved by passing to the \emph{reflexive hull} $\Gr(E)^{**}$ of $\Gr(E)$.\footnote{A \emph{reflexive sheaf} is a sheaf $\calE$ isomorphic to its double dual $\calE \isom \calE^{**}$. The \emph{reflexive hull} of a sheaf $\calE$ is its double dual $\calE^{**}$, which is always reflexive. Reflexive sheaves are torsion-free and are close enough to being locally free to do analysis with.}

\begin{proposition}[See {\cite[Thm. 1.6.7, Prop. 1.6.10]{huybrechts-lehn}} or {\cite[Thm. 5.7.18]{kobayashi1987differential}}]
	Any semistable vector bundle admits a Jordan--H\"older filtration. The graded objects of any two Jordan--H\"older filtrations are isomorphic outside a set of codimension $\ge 2$. The reflexive hulls $\Gr(E)^{**}$ of any two Jordan--H\"older filtrations are isomorphic, and the hull is therefore uniquely defined.
\end{proposition}

\begin{remark}
	Due to our notion of slope stability implying torsion-freeness, it is an automatic consequence of the stability of the factors of the graded object $\Gr(E)$ is torsion-free for any semistable sheaf $E$. The existence of a filtration with this property is automatic when $E$ is torsion-free as the maximal destabilising subsheaf of a semistable torsion-free sheaf is always saturated. Moreover, a filtration with the the property that the successive quotients are only destabilised by torsion sheaves of codimension $\ge 2$ can be upgraded to a Jordan--H\"older filtration simply by replacing $\calE_i$ with the saturation of $\calE_i$ inside $\calE_{i+1}$ at every step. The resulting filtration has the same numerical properties, and each successive quotient is slope stable in our sense.
	
	The same remark applies to the existence of the Harder--Narasimhan filtration in \cref{thm:HNfiltration} (except that the summands are semistable and torsion-free). In that case the full details that the maximal destabilising sheaf is torsion-free can be found in \cite[Thm. 5.7.15]{kobayashi1987differential}.
\end{remark}

The key properties of slope stability which are used in the construction and uniqueness of Jordan--H\"older filtrations are those of \cref{prop:stabilitysimplicity} and \cref{lem:seesawslopestability}. In particular any slope-type stability function which satisfies these properties will produce an analogous Jordan--H\"older filtration theory for semistable sheaves (for example one could use Gieseker or $Z$-stability instead of slope stability). 

\begin{remark}
	Even if $E$ is a vector bundle, there is no guarantee that the polystable degeneration $\Gr(E)$ is locally free, or even reflexive. However in \cref{ch:correspondence} we will work under the assumption that $\Gr(E)$ is locally free.
	
	We will make use of the existence of a Hermite--Einstein metric on $\Gr(E)$ in the proof of the correspondence for $Z$-critical metrics. The Donaldson--Uhlenbeck--Yau theorem admits an extension to polystable reflexive sheaves (see \cite{bando1994stable}) and so the proof of the correspondence in \cref{ch:correspondence} should generalise to the setting where $\Gr(E)$ is not necessarily locally free, and a singular Hermite--Einstein metric on the polystable reflexive sheaf $\Gr(E)^{**}$ is utilized in the perturbation argument instead. The key difficulty then is to understand the behaviour of the HE metric on the singular set of $\Gr(E)^{**}$ (i.e. the analytic subvariety forming the support of the torsion sheaf $\Gr(E)^{**}/\Gr(E)$, which has been well-studied in the theory of bubbling phenomena in gauge theory).
\end{remark}

One may obtain the vector bundle $E$ from its graded object $\Gr(E)$ by a sequence of extensions. After fixing a Jordan--H\"older filtration $\scrE$ of $E$, we note that $E$ fits into a short exact sequence
\begin{center}
	\ses{\calE_{k-1}}{E}{\calE'_k}
\end{center}
and so $E$ corresponds to an extension of the graded factor $\calE'_k$ by the subsheaf $\calE_{k-1}$, classified by a class $e_k(\scrE)\in \Ext^1(\calE'_k, \calE_{k-1})$. Indeed we can repeat this process to build a sequence of extension classes $e_i(\scrE) \in \Ext^1(\calE'_i, \calE_{i-1})$ which reproduce $E$ from $\Gr(E)$ given $\scrE$. 

\begin{remark}\label{rmk:turningoffextension}
	This process should be seen as a deformation of complex structure from $\Gr(E)$ to $E$, or conversely a degeneration from $E$ to $\Gr(E)$, through the process of ``turning off an extension" (see \cite[Rmk. 5.14]{ross2006obstruction}). In particular given any one-step filtration $0\subset \calF \subset E$ of a vector bundle $E$ (not necessarily the Jordan--H\"older filtration of a semistable bundle), we obtain a short exact sequence
	\begin{center}
		\ses{\calF}{E}{E/\calF}
	\end{center}
	which defines an extension class $e\in \Ext^1(E/\calF, \calF)$. By scaling $\lambda e$ for $\lambda \in \CC$ we obtain a family $\calE \to \CC$ for which the generic fibre $\calE_t\isom E$ for $t\ne 0$, and the central fibre splits into the direct sum $\calF \oplus E/\calF$. 
	
	In \cref{part:zcritical} we will take the viewpoint of deformation of complex structure for this construction, in particular in the proof of the correspondence in \cref{ch:correspondence}.
	
	In \cref{part:fibrations} we will take the viewpoint of bundle degenerations, and this gives a different interpretation of the stability of $E$ as an object, which is more aligned with the language used in the study of K-stability of varieties and of geometric invariant theory. See for example \cite{donaldson2005lower} for a discussion of this perspective.
\end{remark}

For the sake of completeness of discussion, we will also mention the existence of a filtration for \emph{any} holomorphic vector bundle $E\to X$, where $E$ need not be slope semistable. This is the famous \emph{Harder--Narasimhan filtration} of $E$ \cite{harder1975cohomology}. 

\begin{definition}
	A Harder--Narasimhan filtration of a holomorphic vector bundle $E\to X$ is a filtration 
	$$0 = \calE_0 \subset \calE_1 \subset \cdots \subset \calE_k = E$$
	by coherent subsheaves such that each quotient $\calE'_i = \calE_i / \calE_{i-1}$ is slope semistable and the slopes strictly decrease: $\mu(\calE'_i) > \mu(\calE'_{i+1})$ for all $i$.
\end{definition}

In contrast to Jordan--H\"older filtrations which only have uniqueness of the associated graded object, the Harder--Narasimhan filtration of $E$ \emph{is} unique. Indeed we have:

\begin{theorem}[See for example {\cite[Thm. 5.7.15]{kobayashi1987differential}}.]\label{thm:HNfiltration}
	Any holomorphic vector bundle $E\to X$ over a compact K\"ahler manifold admits a Harder--Narasimhan filtration, and there is a unique filtration with each successive quotient $\calE_i'$ torsion-free.
\end{theorem}

Associated to the HN filtration of $E$ is a graded object $\GrHN(E)$ whose summands are torsion-free and slope semistable of decreasing slope. For each summand $\calE'_i$ we may then take the Jordan--H\"older filtration $\GrJH(\calE'_i)$ to obtain a double filtration of $\calE$. Since the JH filtrations of each summand of $\GrHN(E)$ are not unique, the double filtration of $E$ is not unique, but its double graded object associated to the \emph{Harder--Narasimhan--Seshadri filtration}, denoted $\GrHNS(E)$, \emph{is} unique.\footnote{More precisely, the reflexive hull of the double graded object is unique.}

\subsubsection{Gieseker stability\label{sec:giesekerstability}}

Let us briefly review the notion of Gieseker stability, which will be useful by comparison to $Z$-stability as it will appear in \cref{part:zcritical}. Gieseker stability was introduced by Gieseker and Maruyama \cite{gieseker1977moduli,maruyama1977moduli,maruyama1978moduli} and is an alternative to slope stability which is more closely related to the geometric invariant theory of sheaves. First let us recall the definition of the Hilbert polynomial of a coherent sheaf. 

\begin{definition}
	Let $(X,L)$ be a smooth polarised variety. The \emph{Hilbert polynomial} of a coherent sheaf $\calE$ on $X$ is given by
	$$\calP_\calE(k) := \chi(\calE\otimes L^k)=\sum_{i>0} (-1)^i \dim H^i(X,\calE)$$
	where $\chi$ is the holomorphic Euler characteristic of $\calE \otimes L^k$. 
\end{definition}
By the Riemann--Roch formula we can compute the Euler characteristic as
$$\chi(\calE\otimes L^k) = \int_X \Ch(\calE\otimes L^k) \Td(X)$$
which is manifestly a polynomial in $k$. If $L$ is ample (by definition for $(X,L)$ polarised), then by Serre vanishing we see $H^i(X,\calE\otimes L^k) = 0$ for $i>0$ and $k\gg 0$. Thus for $k$ sufficiently large, we have
$$\chi(\calE\otimes L^k) = \dim H^0(X,\calE \otimes L^k).$$
Note that the latter is not a polynomial for all $k$, but for $k\gg 0$ the Riemann--Roch formula implies that it is polynomial. 

\begin{definition}\label{def:giesekerstability}
	A torsion-free coherent sheaf $\calE$ on $(X,L)$ is called \emph{Gieseker stable} (resp. \emph{Gieseker semistable}) if for all proper, non-zero coherent subsheaves $\calF \subset \calE$ we have
	$$\frac{\calP_\calF(k)}{\rk \calF} < \frac{\calP_\calE(k)}{\rk \calE}\quad (\text{resp. } \le)$$
	for all $k\gg 0$. 
\end{definition}

The quantity $\frac{\calP_\calE(k)}{\rk \calE}$ is the \emph{Gieseker slope} of $\calE$. Gieseker stability is an asymptotic form of stability for a sheaf near a ``large volume" limit for $(X,L)$ (where $L$ is replaced by $L^k$), and the first key property of Gieseker stability is the following, which follows readily from the Riemann--Roch formula applied to the Gieseker slope.

\begin{lemma}\label{lem:gieskerimplications}
	To leading order in $k$, the Gieseker slope is given (up to unimportant geometric factors depending on $(X,L)$) by the slope $\mu(E)$. Thus a slope stable vector bundle is Gieseker stable. In particular there are implications
	
	{\centering Slope stable $\implies$ Gieseker stable $\implies$ Gieseker semistable $\implies$ Slope semistable.}
\end{lemma}

One can go on to observe that the analogues of \cref{prop:stabilitysimplicity} and \cref{lem:seesawslopestability} hold for the Gieseker slope also, and so the discussion of \cref{sec:filtrations} could have been equally carried out for Gieseker stability. For a detailed presentation of this theory, see \cite[Ch. 1]{huybrechts-lehn}.

\subsection{Yang--Mills connections}

In view of \cref{principle}, corresponding to the preceding notion of a slope stable vector bundle there should be a notion of an extremal object in differential geometry. The first indications of what this extremal object is were provided by Donaldson and Atiyah--Bott in their rephrasing of the theorem of Narasimhan--Seshadri in terms of gauge theory \cite{donaldson1983new,atiyahbott}.

In particular Atiyah and Bott made the following remarkable insight: If $A$ is a Chern connection on a Hermitian vector bundle $(E,h)$ over a compact Riemann surface then the map
$$\mu: A \mapsto F(A)$$
is a moment map for the action of the unitary gauge group $\calG$ acting on $E$. Indeed we recall that $A$ lives in the space $\calA(h)$ of integrable $h$-unitary connections, which is an affine space modelled on $\Omega^1(X,\End_{SH}(E,h))$ where $\End_{SH}(E,h)$ denotes the skew-Hermitian endomorphisms of $E$ with respect to $h$. The curvature $F_A\in \Omega^{1,1}(X,\End_{SH}(E,h))$ may be viewed as an element of the formal dual of the Lie algebra $\g=\Omega^0 (X, \End_{SH}(E,h))$ of $\calG$ under the natural pairing
$$(\varphi, \psi) \mapsto -\int_X \trace \varphi \psi$$
for $\varphi \in \g$, $\psi \in \Omega^2(X,\End_{SH}(E,h))$. Here the underlying symplectic structure on $\calA(h)$ is given by the \emph{Atiyah--Bott symplectic form},
$$\Omega_{AB}(a,b) = - \int_X \trace(a\wedge b).$$

Primarily through observations of Donaldson \cite[\S 4]{donaldson1985anti}, this understanding of the curvature as a moment map on compact Riemann surfaces admits an upgrade to compact K\"ahler manifolds.\footnote{Here the Atiyah--Bott symplectic form simply picks up a $\wedge \omega^{n-1}$ in the integrand.} Associated to this moment map construction are two natural functionals which have finite-dimensional analogues we have seen in \cref{sec:GIT}. The first is the \emph{Yang--Mills functional} $\|\mu\|^2$, given by
\begin{equation}\label{eq:ymfunctional}\YM(A) = \int_X |F_A|^2 \dvol.\end{equation}
To identify the second functional, let us now recall that any Chern connection $A$ for $h$ is determined uniquely by a Dolbeault operator $\nabla_A^{0,1} = \delbar_A$ satisfying the integrablity condition $\delbar_A^2 = F_A^{0,2} = 0$. Such Dolbeault operators are in one-to-one correspondence with holomorphic structures on the underlying smooth vector bundle $E$, so using this identification of the space of Chern connections $\calA(h)$ with holomorphic structures we transfer the action of the complex gauge group $\calG^\CC$ onto $\calA(h)$, and we obtain the formal Kempf--Ness picture from finite dimensions analogous to \cref{sec:GIT}. 

On a fixed $\calG^\CC$ orbit inside $\calA(h)$ we may consider the ``Kempf--Ness functional" of our problem, just as in \cref{sec:kempfness}. Indeed in this setting we switch our perspective from varying the holomorphic structure of $(E,h)$ to fixing a holomorphic structure $\delbar_E$ and varying the metric, as follows: In a fixed complex gauge orbit $\calG^\CC \cdot \delbar_E$ inside $\calA(h)$, pulling back the fixed Hermitian metric $h$ from $(E,g\cdot \delbar_E)$ to $(E,\delbar_E)$ gives a new Hermitian metric $g^*h$ on $(E,\delbar_E)$ (modulo the unitary gauge transformations $g\in \calG$), which produces a Chern connection $A(g^* h)$. This construction gives a one-to-one correspondence between Dolbeault operators $\delbar_A \in \calG^\CC \cdot \delbar_E$ and Hermitian metrics $h$ on $(E,\delbar_E)$, and the orbit in $\calA(h)$ may now be described as the quotient $\calG^\CC/\calG \isom \Herm(E)$.

The Kempf--Ness functional in this setting is known as the \emph{Donaldson functional}, and was introduced by Donaldson in \cite{donaldson1985anti}. Given a fixed reference metric $h$ on $E$, the Donaldson functional is the unique functional
$$\calM: \calG^\CC/\calG \to \RR$$
satisfying $\calM(h) = 0$ and with first variation given by
$$\deriv{}{t} \calM(h_t) = \int_X \trace (h_t^{-1} \del_t h_t \circ (\contr_\omega i F(h_t) - \lambda \id_E)) \omega^n.$$

\subsubsection{Hermite--Einstein metrics\label{sec:hermiteeinsteinmetrics}}

As is typical in a Kempf--Ness formalism for \cref{principle}, there are two types of extremal objects we can study:

\begin{enumerate}
	\item Critical points of the functional $\|\mu\|^2$ on $\calA(h)$.
	\item Critical points of $\calM$ on $\calG^\CC/\calG\isom \Herm(E)$. 
\end{enumerate}

The latter critical points correspond to the \emph{absolute minima} of $\|\mu\|^2$ on $\calA(h)$ inside $\G^\CC$ orbits, but one also has higher critical points of $\|\mu\|^2$ which we will remark on later.

An analysis of the variation of the Yang--Mills functional $\YM$ or Donaldson functional $\calM$ reveals that the absolute minima of the Yang--Mills functional $\YM = \|\mu\|^2$ are given by so-called \emph{Hermite--Einstein metrics} (see \cite[Thm. 4.3.9]{kobayashi1987differential}).

\begin{definition}\label{def:hermiteeinstein}
	A Hermitian metric $h$ on a holomorphic vector bundle $E\to (X,\omega)$ over a compact K\"ahler manifold is called \emph{Hermite--Einstein} (or \emph{Hermitian Yang--Mills}) if 
	\begin{equation}
		\label{eq:hermiteeinstein}
		iF(h)\wedge \omega^{n-1} = \lambda \id_E\otimes \omega^n
	\end{equation}
	for some constant $\lambda \in \RR$. 
	
	If a metric $h$ satisfies the Hermite--Einstein equation where $\lambda=f \in C^\infty(X)$ is a non-constant function, we call $h$ a \emph{weak Hermite--Einstein metric with function $f$}.
\end{definition}

This is a second order elliptic partial differential equation in the metric $h$, which is linear in the curvature $F(h)$. Let us make several remarks:

\begin{remark}
	We note that the Einstein constant $\lambda$ is purely topological, and by integrating and using Chern--Weil theory we deduce
	$$\lambda = \frac{2\pi}{n! \vol(X)} \mu(E).$$
	When instead we have a non-constant Einstein function $\lambda = f$ the above calculation simply determines the topological average of $f$ over $X$.
\end{remark}
\begin{remark} 
	After a conformal change of metric any weak Hermite--Einstein metric may be transformed into a genuine Hermite--Einstein metric, so without loss of generality we may always take $\lambda$ to be constant (see \cite[\S 4.2]{kobayashi1987differential}). In particular the relationship of stability with existence applies to weak Hermite--Einstein metrics also.
	
	Weak Hermite--Einstein metrics will naturally appear in the large volume limit of the $Z$-critical equation in \cref{part:zcritical}.
\end{remark}

\begin{remark}\label{rmk:highercriticalym}
	The higher critical points of the Yang--Mills functional $\YM$ correspond to Hermite--Einstein-type metrics with a block diagonal matrix whose coefficients depend on the Harder--Narasimhan--Seshadri type of $E$ (see \cite[Thm. 4.3.27]{kobayashi1987differential}). In particular suppose $E\to (X,\omega)$ is a holomorphic vector bundle with HNS type $\nu(E)$. If a Hermitian metric $h$ on $E$ is Yang--Mills, that is $h$ is a critical point of the Yang--Mills functional \eqref{eq:ymfunctional}, then $E$ admits a holomorphic orthogonal decomposition 
	$$E = E'_1 \oplus \cdots \oplus E'_k$$
	with summands $E'_i$ each vector subbundles precisely so $E \isom \GrHNS(E)$ biholomorphically. One can characterise when the Hermitian metric $h$ on $E$ is Yang--Mills by a differential equation similar to the Hermite--Einstein equation. The metric is Yang--Mills if and only if
	$$iF(h) \wedge \omega^{n-1} = \frac{2\pi}{n! \vol(X)} \diag (\nu(E)) \otimes \omega^n$$
	with respect to this orthogonal decomposition of $E$, where $\diag(\nu(E))$ has value $\mu(E'_i)$ on the summand $E'_i$. 
\end{remark}

We will not delay any longer in stating the celebrated correspondence between Hermite--Einstein metrics and slope stability. 

\begin{theorem}[Donaldson--Uhlenbeck--Yau]\label{thm:DUY}
	A holomorphic vector bundle $E\to (X,\omega)$ over a compact K\"ahler manifold admits a Hermite--Einstein metric if and only if it is slope polystable.
\end{theorem}

This correspondence was conjectured independently by Kobayashi and Hitchin \cite{kobayashi1982curvature,hitchin1979non} at the beginning of the 1980s. The ``easy" direction that existence implies stability was proven by Kobayashi and L\"ubke \cite{kobayashi1980first,lubke1983stability}. In the reverse direction, the case of compact Riemann surfaces is essentially the theorem of Narasimhan--Seshadri \cite{narasimhan1965stable} reinterpreted using Donaldson's proof \cite{donaldson1983new}. The case of algebraic surfaces was proven by Donaldson \cite{donaldson1985anti}, and the full correspondence for compact K\"ahler manifolds using a continuity method by Uhlenbeck and Yau the following year \cite{uhlenbeck1986existence}. Donaldson afterwards gave a new proof of the theorem for all projective manifolds using a different technique to Uhlenbeck and Yau, with an inductive argument along the lines of the case of algebraic surfaces \cite{donaldson1987infinite}. 

There are generalisations of \cref{thm:DUY} known for non-K\"ahler manifolds, proved for surfaces by Buchdahl \cite{buchdahl1986hermitian} and in general by Li--Yau \cite{li1987hermitian}. Additionally we recall that a version of the correspondence is known when $E$ is instead a reflexive sheaf over $X$ due to Bando--Siu, which in particular applies in some settings to the graded object $\Gr(E)$ of a semistable bundle $E$, which is always polystable but not necessarily locally free \cite{bando1994stable}. 

As is true for the slope stability of vector bundles, we note that there is a significant body of work using the Hermite--Einstein equation to construct moduli of metrics on bundles, and indeed a rich interplay between the algebraic and differential-geometric techniques to describe the same moduli spaces. We will not comment further on these results, and refer for example to \cite{greb2021complex} for more details.

\subsubsection{Hermitian geometry of Chern connections\label{sec:chernconnections}}
In this section we will analyse the Hermite--Einstein equation and in doing so collect a number of facts in the Hermitian geometry of complex vector bundles which will be important in \cref{part:zcritical,part:fibrations}. As part of this we will observe many of the basic properties of stability which appeared in \cref{sec:stability} manifesting in differential geometry, and in particular we will prove the easy direction of the Donaldson--Uhlenbeck--Yau theorem (at least when the subobject is also locally free). 

Let us begin by linearising the Hermite--Einstein equation \eqref{eq:hermiteeinstein}. To do so, let us recall the tangency structure of the space of Hermitian metrics $\Herm(E)$ on a holomorphic vector bundle. The isomorphism $\Herm(E) \isom \calG^\CC / \calG(E,h)$ for a fixed reference metric $h$ shows that any nearby metric (equivalently, any nearby Chern connection in the same gauge orbit of $A(h)$) is given by 
$$h_t = \exp(tV)\cdot h = h(-,\exp(tV)-)$$
where $V\in \Omega^0(X,\End_{H}(E,h))$ is a Hermitian endomorphism with respect to $h$. This has the effect of transforming
$$\delbar_{h_t} = \delbar_h = \delbar_E$$
and
$$\del_{h_t} = \exp(-tV) \circ \del_h \circ \exp(tV).$$
Indeed we have:

\begin{lemma}\label{lem:linearisationcurvatureMetric}
	The curvature $F(h)$ of a Hermitian metric transforms as
	$$F(h_t) = F(h) + \delbar \del_h V t + O(t^2).$$
\end{lemma}
\begin{proof}
	Recall in local coordinates the $(1,0)$-part of the Chern connection takes the form
	$$A = H^{-1} \del H$$
	where $H$ is the local Hermitian matrix representing $h$. In local coordinates $H_t = H\exp(tV)$ so $H_t^{-1} = \exp(-tV)H^{-1}$. Then we compute
	\begin{align*}
		A_t &= H_t^{-1} \del H_t\\
		&= \exp(-tV)H^{-1} (\del H \exp(tV) + H t\del V \exp (tV))\\
		&= \exp(-tV) A \exp(tV) + t \exp(-tV) \del V \exp(tV)\\
		&= A + (-VA + AV + \del V)t + O(t^2)\\
		&= A + \del_h V t + O(t^2).
	\end{align*}
	The Chern curvature is further computed as $\delbar A_t$ which produces
	\begin{equation*}F(h_t) = F(h) + \delbar \del_h V t + O(t^2).\qedhere\end{equation*}
\end{proof}
Instead of working with a fixed holomorphic structure on $E$ and varying $h$ inside $\Herm(E)$, one could instead fix $h$ and vary the Chern connection $A$ within its complex gauge orbit inside $\calA(h)$. Acting by a unitary gauge transformation $g\in \G^\CC$ (i.e. such that $g^* g = \id_E$) conjugates the Chern connection, and in particular $F_A$ will satisfy a gauge-invariant equation if and only if $F_{g\cdot A}$ will. Therefore if we wish to capture a genuine change of the curvature inside a complex gauge orbit, we should work orthogonal to the unitary gauge transformations. The orthogonal complement to the unitary transformations are the Hermitian transformations $g\in \G^\CC$ such that $g^* = g$ (in the sense that the Hermitian endomorphisms are the orthogonal complement to the unitary Lie algebra of $\G$ with respect to the natural trace pairing). The well-known formula for the complex gauge group acting on Chern connections is given by
\begin{align}
	g\cdot d_A &= g^* \circ \delbar_A \circ g^{*-1} + g^{-1} \circ \del_A \circ g\label{eq:gaugeaction}\\
	&= g \circ \delbar_A \circ g^{-1} + g^{-1} \circ \del_A \circ g\nonumber
\end{align}
where in the second line we have used that $g=g^*$ in our case.\footnote{Some sources use the convention that $g\cdot \delbar_E = g^{-1} \circ \delbar_E \circ g$ which produces a sign difference in the computation of the linearisation, but no other differences. Our convention matches the transformation of the Hermitian metric and follows for example \cite{donaldson1985anti}, at the cost of treating the action on $\Herm(E)$ as a right action and on $\calA(h)$ as a left action.} To produce a small perturbation of a Chern connection $A$ within its complex gauge orbit, we act by $g_t=\exp (tV)$ where $V$ is a Hermitian endomorphism to produce $d_{A_t} = g_t \cdot d_A$. Then we have the following.

\begin{lemma}\label{lem:linearisationcurvatureChern}
	The curvature $F_A$ of a Chern connection for a fixed Hermitian metric $h$ transforms as
	$$F_{A_t} = F_A + \left( \delbar_{A^{\End E}} \del_{A^{\End E}} - \del_{A^{\End E}} \delbar_{A^{\End E}} \right) V t + O(t^2).$$
\end{lemma}
\begin{proof}
	Using the formula $F_{A_t} = d_{A_t}^2$ and the fact that $\delbar_A^2 = \del_A^2 = 0$ since the Chern curvature has type $(1,1)$, we have
	\begin{align}
		F_{A_t} &= (g_t \circ \delbar_A \circ g_t^{-1} + g_t^{-1} \circ \del_A \circ g_t) \circ (g_t \circ \delbar_A \circ g_t^{-1} + g_t^{-1} \circ \del_A \circ g_t)\nonumber\\
		&= g_t \delbar_A^2 g_t^{-1} + g_t \delbar_A g_t^{-2} \del_A  g_t \nonumber\\
		&\quad + g_t^{-1} \del_A g_t^2 \delbar_A g_t^{-1} + g_t^{-1} \del_A^2  g_t\nonumber\\
		&= g_t \delbar_A g_t^{-2} \del_A g_t + g_t^{-1} \del_A g_t^2 \delbar_A g_t^{-1}\nonumber\\
		&= \delbar_A \del_A + \del_A \delbar_A + t(V \delbar_A \del_A - 2\delbar_A V \del_A + \delbar_A \del_A V\nonumber\\
		&\quad -V \del_A \delbar_A + 2\del_A V \delbar_A - \del_A \delbar_A V) + O(t^2).\label{eq:curvatureexpansionchern}
	\end{align}
	Recalling the endomorphism connections
	$$\del_{A^{\End E}} V = \del V + [A^{1,0},V],\quad \delbar_{A^{\End E}} V = \delbar V + [A^{0,1},V]$$
	and expanding $\del_A = \del + A^{1,0}$ and $\delbar_A = \delbar + A^{0,1}$ we have 
	$$\del_A V = \del_{A^{\End E}} V - A^{1,0} V,\quad \delbar_A V = \delbar_{A^{\End E}} V - A^{0,1} V,$$
	and substituting these into \eqref{eq:curvatureexpansionchern} one may compute
	\[F_{A_t} = F_A + \left( \delbar_{A^{\End E}} \del_{A^{\End E}} - \del_{A^{\End E}} \delbar_{A^{\End E}} \right) V t + O(t^2).\qedhere \]
\end{proof}

\begin{remark}
	One can see that the two ways of linearising the curvature are equivalent. Indeed recall that the isomorphism between $\G^\CC \cdot d_A$ and $\Herm(E,\delbar_E)$ is given by sending $(h,g\cdot \delbar_E)$ to $(g\cdot h,\delbar_E)$ where $g$ is a Hermitian endomorphism. One can show 
	$$F(h,g^{1/2} \cdot \delbar_E) = g^{1/2} \circ F(g\cdot h, \delbar_E) \circ g^{-1/2}$$
	and setting $g=\exp(tV)$ and computing the first order expressions in $t$, one observes that the expression in \cref{lem:linearisationcurvatureChern} appearing on the left-hand side is transformed into the expression from \cref{lem:linearisationcurvatureMetric} appearing on the right.
\end{remark}

As a consequence of these computations, we can now compute the linearisation of the Hermite--Einstein equation.

\begin{proposition}\label{prop:linearisationHermiteEinstein}
	The linearisation $P$ of the Hermite--Einstein operator
	$$D: A \mapsto iF_A \wedge \omega^{n-1} - \lambda \id_E \otimes \omega^n$$
	under the action of the Hermitian gauge transformations\footnote{Instead we could work with Hermitian metrics on a fixed bundle in an orbit.} is (up to a constant factor) given by the $\End E$-Laplacian $\Lap_{A^{\End E}} = d_A^* d_A$.
\end{proposition}
\begin{proof}
	Let $A_t = \exp(tV) \cdot A$ be a small perturbation of a Chern connection, where $V$ is Hermitian. Then by \cref{lem:linearisationcurvatureChern} we have
	\begin{align*}
		\left. \deriv{}{t}\right|_{t=0} D(A_t) &= i \omega^{n-1} \wedge \left( \delbar_{A^{\End E}} \del_{A^{\End E}} - \del_{A^{\End E}} \delbar_{A^{\End E}} \right) V\\
		&= \frac{i}{n} \contr_{\omega} \left(\left( \delbar_{A^{\End E}} \del_{A^{\End E}} - \del_{A^{\End E}} \delbar_{A^{\End E}} \right) V\right) \omega^n\\
		&= \frac{1}{n} (\Lap_{A^{\End E}} V) \omega^n.
	\end{align*}
	The last equality follows from the Nakano identities 
	$$\contr_{\omega} \del_{A^{\End E}} = i \delbar_{A^{\End E}}^*,\quad \contr_{\omega} \delbar_{A^{\End E}} = -i \del_{A^{\End E}}^*.$$
	See for example \cite[Prop. 5.22]{ballmann2006lectures}.
\end{proof}

Let us explicitly record some useful corollaries of our computation of the linearisation, which we will use in \cref{part:zcritical}. These follow immediately from well-known properties of the bundle Laplacian.

\begin{corollary}\label{cor:HermiteEinsteinelliptic}
	The Hermite--Einstein equation is elliptic, and the kernel of the linearisation is given by the global holomorphic endomorphisms
	$$\ker P_A = H^0(X,\End (E,d_A^{0,1})).$$
\end{corollary}

\begin{remark}
	The linearisation of the Hermite--Einstein operator as a map 
	$$D: h\mapsto iF(h) \wedge \omega^{n-1} - \lambda \id_E \otimes \omega^n$$
	produces the $\del_h$-Laplacian $\Lap_{\del_h}$ by the Nakano identity
	$$-i \del_h^* = \contr \delbar.$$
	This differs from the bundle Laplacian $\Lap_{A^{\End E}}$ in \cref{prop:linearisationHermiteEinstein} by a commutator involving the contraction of the curvature:
	$$2 \Lap_{\del_h} V = \Lap_{A^{\End E}} V + [i\contr_{\omega} F_h, V].$$
	Thus \emph{at a solution} to the Hermite--Einstein equation, the linearisations agree and are both the standard bundle Laplacian. We will only need the linearisation at a solution, so there is no subtlety introduced here.
\end{remark}

We will now briefly recall the theory of second fundamental forms and the Hermitian geometry of subbundles, which produces the fundamental principle that ``curvature (generically) decreases in holomorphic subbundles" which underpins the easy direction of the Donaldson--Uhlenbeck--Yau theorem, and which we will make extensive use of in \cref{thm:existencestabilitysubbundles}.

Let $A$ be a Chern connection for a Hermitian metric $h$ on a holomorphic vector bundle $E\to X$. Suppose $S\subset E$ is a holomorphic subbundle, producing a short exact sequence
\begin{center}
	\ses{S}{E}{Q}
\end{center}
where $Q=E/S$ is the quotient. With respect to the smooth (but not necessarily holomorphic!) splitting 
$$E\isom S\oplus Q$$
induced by the metric $h$, the Chern connection may be written in block matrix form as
\begin{equation}
	A_E=\begin{pmatrix}
		A_S & \beta\\
		-\beta^* & A_Q
	\end{pmatrix}\label{eqn:connectionsubbundle}
\end{equation}
where $\beta \in \Omega^{0,1}(X,\Hom(Q,S))$ is the \emph{second fundamental form} of the subbundle $S\subset E$. Further, with respect to this same splitting the curvature $F$ of $A$ may be computed as
\begin{equation}
	F_E = \begin{pmatrix}
		F_S - \beta \wedge \beta^* & \del_{A^{\Hom(Q,S)}} \beta\\
		-\delbar_{A^{\Hom(S,Q)}} \beta^* & F_Q - \beta^* \wedge \beta
	\end{pmatrix}.\label{eqn:curvaturesubbundle}
\end{equation}

\begin{remark}\label{rmk:secondfundamentalformatiyah}
	The second fundamental form can be given an algebraic interpretation using the Dolbeault isomorphism theorem, for example following a computation similar to \cite[Prop. 4]{atiyah1957complex}.\footnote{Atiyah's calculation is for the particular problem of defining a holomorphic connection on a principal bundle, but the particular sequence (the ``Atiyah sequence") being split is unimportant to his proof.} Namely, if $e\in \Ext^1(Q,S)$ is the extension class of the short exact sequence, then using the identification $\Ext^1(Q,S) \isom H^1(X, \Hom(Q,S))$ when $Q$ and $S$ are locally free, and the Dolbeault isomorphism
	$$H^1(X,\Hom(Q,S)) \isom H_{\delbar_E}^{0,1} (X, \Hom(Q,S))$$
	one has $e=[\beta]$. In particular if $\beta$ is zero in Dolbeault cohomology, then the short exact sequence splits. Indeed in that case the Chern connection $A$ can be transformed by a change of gauge to one in which $\beta=0$, and then $S$ is invariant under the covariant derivative $d_A$, so the orthogonal complement is too (since $d_A$ is unitary), making $Q$ a holomorphic subbundle of $E$. 
\end{remark}

To finish off our discussion of Hermitian geometry of subbundles, let us demonstrate the principle that curvature (generically) decreases in subbundles and increases in quotients. This is the analogue in Hermitian geometry of the see-saw property \cref{lem:seesawslopestability} for slope stability, and it quickly produces a proof of one direction of \cref{thm:DUY} (at least when the subobject is a subbundle). 

\begin{proposition}
	Suppose $E\to(X,\omega)$ is indecomposable and admits a Hermite--Einstein metric $h$. Then $E$ is slope stable with respect to holomorphic subbundles $S\subset E$.
\end{proposition}
\begin{proof}
	Here we use the decomposition of the Chern curvature and the properties of the Hermite--Einstein metric. We know
	$$iF(h) \wedge \omega^{n-1} = \lambda \id_E \otimes \omega^n.$$
	Suppose $S\subset E$ is any proper non-zero holomorphic subbundle and $Q=E/S$ is the quotient bundle. Then with respect to the $C^\infty$ direct sum decomposition $E=S\oplus Q$ the $\End(S)$ component of the Hermite--Einstein equation is
	$$i(F_S - \beta\wedge \beta^*) \wedge \omega^{n-1} = \frac{2\pi}{n! \vol(X)}\mu(E) \id_S \otimes \omega^n.$$
	Taking trace and integrating we note that
	$$-\int_X \trace (\beta \wedge \beta^*)\wedge \omega^{n-1} =  \frac{i}{2\pi} \|\beta\|^2$$
	(suitably normalised) and thus we obtain
	$$2\pi (\deg S + \|\beta\|^2) = 2\pi \mu(E) \rk S$$
	and using that $\|\beta\|^2>0$ since $E$ is indecomposable (otherwise $E=S\oplus Q$ holomorphically as $\beta=0$ implies the sequence splits) we obtain
	$$\mu(S) < \mu(E).$$
	Thus $E$ is slope stable with respect to $S$.
\end{proof}

\subsubsection{Almost Hermite--Einstein metrics\label{sec:AHE}}

In order to complete our background on holomorphic vector bundles, we will discuss the differential-geometric analogue of Gieseker stability discussed in \cref{sec:giesekerstability}. This theory was developed by Leung in \cite{leung1997einstein,leung1998symplectic}, and shares a number of ideas with those in \cref{part:zcritical}.

Let us take the moment map approach to the problem. As was observed by Atiyah--Bott and Donaldson, the Hermite--Einstein equation can be obtained as the moment map equation for the symplectic form
$$\Omega_{AB}(a,b) = -\int_X\trace(a\wedge b) \wedge \omega^{n-1}$$
on the space of integrable unitary connections $\calA(h)$ on a Hermitian vector bundle $(E,h)\to (X,\omega)$. Donaldson constructed a determinant line bundle $\calL\to \calA(h)$ for which the Atiyah--Bott symplectic form is the curvature of a Hermitian metric \cite{donaldson1987infinite}. Donaldsons construction uses a local version of the index theorem and ideas due to Bismut--Freed, and indeed going back to Quillen, to identify the correct determinant bundle \cite{quillen1985determinants,bismut1986analysis,bismut1986analysisII}. As part of this construction, one considers the universal bundle
$$\EE \to \calA(h) \times X$$
which has a universal connection $\AA$ with the property that $\rest{\AA}{\{A\}\times X} = A$. This connection has a universal curvature $F_\AA \in \Omega^2(\calA(h)\times X)$. It follows from Bismut--Freed's local version of the Atiyah--Singer index theorem for families that
$$\Omega := c_1(\ind \EE) = \left\{ \int_X \exp\left(\frac{i}{2\pi} F_\AA\right) \Td(X) \right\}_{(2)}$$
where on the right-hand side one takes the degree two component of the form on $\calA(h)$. 

The symplectic structure which arises from this form on $\calA(h)$ was studied by Leung \cite{leung1998symplectic}. In particular Leung considers the setting where $E\to X$ is replaced by $E\otimes L^k$ where $L\to X$ is an ample line bundle and $\omega\in c_1(L)$ is a representative K\"ahler form. Leung computes 
\begin{equation}
	\Omega_k(A)(a,b) = \int_X \trace \left[ \exp\left(\frac{i}{2\pi} F_A + k \omega \otimes \id_E\right), a, b\right]_{\sym} \Td(X)\label{eq:Leungsymplecticform}
\end{equation}
where the symmetric bracket is defined as in \cref{def:symmetrisation}. With respect to the action of the gauge group $\calG^\CC$ on $\calA(h)$, one obtains a moment map equation
\begin{equation}\label{eq:almosthermiteeinstein}
	\left\{\exp\left(\frac{i}{2\pi} F_A + k\omega \otimes \id_E\right) \widetilde \Td(X)\right\}^{(2n)} = \frac{1}{\rk(E)} \chi(E\otimes L^k) \frac{\omega^n}{n!} \otimes \id_E
\end{equation}
where $\widetilde \Td(X)$ is given by the Chern--Weil representatives of the Chern classes of $X$ with respect to the Levi-Civita connection induced by the K\"ahler metric $\omega$. The form of this moment map is highly suggestive of a link to Gieseker stabiliy. A connection $A$ satisfying \eqref{eq:almosthermiteeinstein} is called an \emph{almost Hermite--Einstein connection} (or for a Hermitian metric $h$, an \emph{almost Hermite--Einstein metric}).\footnote{Not to be confused with an ``approximate Hermite--Einstein metric" as defined by Kobayashi \cite{kobayashi1987differential}, which is any sequence of Hermitian metrics which are $\epsilon$-close to being Hermite--Einstein in an appropriate norm. Indeed almost HE metrics are examples of approximate HE metrics, as are the $Z_k$-critical metrics for $k\gg 0$ which will appear in \cref{part:zcritical}.} The almost Hermite--Einstein equation converges to the Hermite--Einstein equation in the large volume limit, and Leung proves the following.

\begin{theorem}[\cite{leung1997einstein}]\label{thm:leungAHEgieseker}
	Let $E\to (X,\omega)$ be a simple, slope semistable vector bundle with locally free graded object $\Gr(E)$. Then $E$ admits an almost Hermite--Einstein metric for all $k\gg 0$ if and only if $E$ is Gieseker stable. 
\end{theorem}

\begin{remark}
	\cref{thm:leungAHEgieseker} is the direct analogue of our main theorem \cref{thm:maintheoremZstability} for Gieseker stability as opposed to asymptotic $Z$-stability. The idea of the proof is similar, in that Leung applies a perturbation result around the limit $k\to \infty$ utilising the Donaldson--Uhlenbeck--Yau theorem and the existence of a Hermite--Einstein metric on the graded object. However the details of the proof are different. In particular Leung's analysis of the linearised operator does not take into account possible automorphisms arising from symmetries of the graded object, which introduces difficulties in bounding the inverse of the linearisation in the proof of \cref{thm:existenceimpliesstability} in the general case.\footnote{We note that these terms do not appear in the proof in \cref{ch:correspondence} where we restrict to the simpler case where $\Gr(E)$ has two components.}
\end{remark}

\section{Varieties\label{sec:varieties}}

Having described in some detail the background of stability on holomorphic vector bundles, we will now discuss the corresponding theory for varieties. Here the notion of \emph{degenerations} takes the central role over the \emph{subobject} perspective which was most geometrically meaningful in the case of bundles. Such degenerations are called \emph{test configurations}, and we will make use of this language in \cref{part:fibrations} to describe the notion of a \emph{fibration degeneration}. 

\subsection{K-stability}

The theory of stability of varieties grew out of the success of the Donaldson--Uhlenbeck--Yau theorem \cref{thm:DUY} in predicting the existence of solutions to geometric PDEs on bundles and a search for a criterion which would imply the existence of K\"ahler--Einstein metrics on Fano manifolds (the cases of Calabi--Yau and general type having been resolved by Yau's proof of the Calabi conjecture). Indeed it was known by work of Matsushima and Lichnerowicz that the Lie algebra of the automorphism group of any K\"ahler--Einstein Fano manifold must be reductive \cite{matsushima1957structure,lichnerowicz1958geometrie}. However the example of $X=\Bl_p \CCPP^2$ is Fano with non-reductive automorphism group.\footnote{The automorphism group is isomorphic to $\Aff(2,\CC)$ which is not reductive as it contains a factor of the additive group $\CC_+^2$.} 

Yau conjectured that the existence of a K\"ahler--Einstein metric on a smooth Fano variety should be equivalent to a stability condition analogous to slope stability of bundles, and suggested perhaps that existence should be equivalent to the slope polystability of the tangent bundle \cite{yau1993open}. The correct criterion was identified by Tian, who called this condition \emph{K-stability} \cite{tian1997kahler}. The ``K" was labelled after the K-energy functional introduced by Mabuchi \cite{mabuchi1986k}, and stands for \emph{Kinetic} or \emph{Kanonisch} (and not \emph{K\"ahler}!).\footnote{This was confirmed in a private correspondence between Ruadha\'i Dervan and Mabuchi.} Tian's stability condition was rephrased purely algebro-geometrically and expanded to include the case of all polarised varieties by Donaldson \cite{donaldson2002scalar}. We will describe the notion of K-stability in Donaldson's language.

K-stability is defined in direct analogy with the Hilbert--Mumford criterion for stability of points on a variety with group action. In order to make this analogy precise, one must therefore define:
\begin{itemize}
	\item a notion of one-parameter degeneration of a polarised variety $(X,L)$, and 
	\item a notion of weight computed on the limiting fibre of that one-parameter degeneration.
\end{itemize}

First we will describe the one-parameter degenerations.

\begin{definition}[Test configuration]\label{def:testconfiguration}
	A test configuration (see \cref{fig:testconfiguration}) $(\calX,\calL)$ with exponent $r$ of a polarised variety $(X,L)$ is a scheme $\calX$ with a flat morphism $\pi: \calX \to \CC$ and a relatively ample line bundle $\calL$ over it, such that
	\begin{itemize}
		\item There is a $\CC^*$ action on the polarised family $(\calX,\calL)$ covering the standard $\CC^*$ action on $\CC$ making $\pi$ equivariant.
		\item For every $t\in \CC^*$, $(\calX_t,\calL_t)\isom (X,L^r)$.
	\end{itemize}
	We say a test configuration is a \emph{product} if $\calX \isom X \times \CC$ and \emph{trivial} if it is a product and the $\CC^*$ action on the $X$ factor is trivial.
\end{definition}

\begin{figure}[h]
	\centering

	\tikzset{every picture/.style={line width=0.75pt}} 
	
	\begin{tikzpicture}[x=0.75pt,y=0.75pt,yscale=-1,xscale=1]
		
		\draw    (91.7,313) -- (481.23,313) ;
		\draw    (118.91,84.85) -- (118.91,271.27) ;
		\draw    (454.18,84.85) -- (454.18,271.27) ;
		\draw    (370.71,84.85) -- (370.71,271.27) ;
		\draw    (202.84,84.85) -- (202.84,271.27) ;
		\draw    (329.44,85.32) -- (329.44,271.73) ;
		\draw    (411.05,84.85) -- (411.05,271.27) ;
		\draw    (245.5,85.78) -- (245.5,272.2) ;
		\draw    (162.5,85.32) -- (162.5,271.73) ;
		\draw    (285.85,271.27) .. controls (327.27,227.83) and (250.29,240.82) .. (305.01,190.74) ;
		\draw    (304.09,163.84) .. controls (273.48,159.2) and (254,201.86) .. (305.01,214.85) ;
		\draw    (284.61,88.72) .. controls (274.41,110.97) and (296.67,161.06) .. (305.94,126.74) .. controls (315.22,92.43) and (290.17,110.97) .. (278.93,124.28) .. controls (267.68,137.58) and (281.42,161.48) .. (304.09,163.84) ;
		\draw  [fill={rgb, 255:red, 0; green, 0; blue, 0 }  ,fill opacity=1 ] (285.85,313) .. controls (285.85,311.55) and (287.02,310.38) .. (288.47,310.38) .. controls (289.93,310.38) and (291.1,311.55) .. (291.1,313) .. controls (291.1,314.45) and (289.93,315.63) .. (288.47,315.63) .. controls (287.02,315.63) and (285.85,314.45) .. (285.85,313) -- cycle ;
		\draw  [fill={rgb, 255:red, 0; green, 0; blue, 0 }  ,fill opacity=1 ] (408.08,312.85) .. controls (408.08,311.4) and (409.26,310.22) .. (410.71,310.22) .. controls (412.16,310.22) and (413.34,311.4) .. (413.34,312.85) .. controls (413.34,314.3) and (412.16,315.48) .. (410.71,315.48) .. controls (409.26,315.48) and (408.08,314.3) .. (408.08,312.85) -- cycle ;
		\draw    (62.33,212.07) -- (62.33,273.99) ;
		\draw [shift={(62.33,276.99)}, rotate = 270] [fill={rgb, 255:red, 0; green, 0; blue, 0 }  ][line width=0.08]  [draw opacity=0] (8.93,-4.29) -- (0,0) -- (8.93,4.29) -- cycle    ;
		\draw  [dash pattern={on 0.84pt off 2.51pt}] (91.7,85.01) -- (482.16,85.01) -- (482.16,271.27) -- (91.7,271.27) -- cycle ;
		\draw  [color={rgb, 255:red, 255; green, 0; blue, 0 }  ,draw opacity=1 ][fill={rgb, 255:red, 255; green, 3; blue, 3 }  ,fill opacity=1 ] (286.85,209) .. controls (286.85,207.55) and (288.02,206.38) .. (289.47,206.38) .. controls (290.93,206.38) and (292.1,207.55) .. (292.1,209) .. controls (292.1,210.45) and (290.93,211.63) .. (289.47,211.63) .. controls (288.02,211.63) and (286.85,210.45) .. (286.85,209) -- cycle ;
		\draw  [color={rgb, 255:red, 255; green, 0; blue, 0 }  ,draw opacity=1 ][fill={rgb, 255:red, 255; green, 3; blue, 3 }  ,fill opacity=1 ] (301.46,163.84) .. controls (301.46,162.39) and (302.64,161.21) .. (304.09,161.21) .. controls (305.54,161.21) and (306.71,162.39) .. (306.71,163.84) .. controls (306.71,165.29) and (305.54,166.47) .. (304.09,166.47) .. controls (302.64,166.47) and (301.46,165.29) .. (301.46,163.84) -- cycle ;
		\draw  [color={rgb, 255:red, 255; green, 0; blue, 0 }  ,draw opacity=1 ][fill={rgb, 255:red, 255; green, 3; blue, 3 }  ,fill opacity=1 ] (281.85,119) .. controls (281.85,117.55) and (283.02,116.38) .. (284.47,116.38) .. controls (285.93,116.38) and (287.1,117.55) .. (287.1,119) .. controls (287.1,120.45) and (285.93,121.63) .. (284.47,121.63) .. controls (283.02,121.63) and (281.85,120.45) .. (281.85,119) -- cycle ;
		
		\draw (56.53,171.85) node [anchor=north west][inner sep=0.75pt]    {$\mathcal{X}$};
		\draw (55.6,302.16) node [anchor=north west][inner sep=0.75pt]    {$\mathbb{C}$};
		\draw (283.4,327.04) node [anchor=north west][inner sep=0.75pt]    {$0$};
		\draw (277.7,57.7) node [anchor=north west][inner sep=0.75pt]    {$\mathcal{X}_{0}$};
		\draw (386.56,57.7) node [anchor=north west][inner sep=0.75pt]    {$\mathcal{X}_{t} \cong X$};
		\draw (406.95,324.57) node [anchor=north west][inner sep=0.75pt]    {$t$};

	\end{tikzpicture}

	\caption{A test configuration. The generic fibre $\calX_t$ is isomorphic to $X$, but the central fibre $\calX_0$ may have singular points or multiple components.}\label{fig:testconfiguration}
\end{figure}
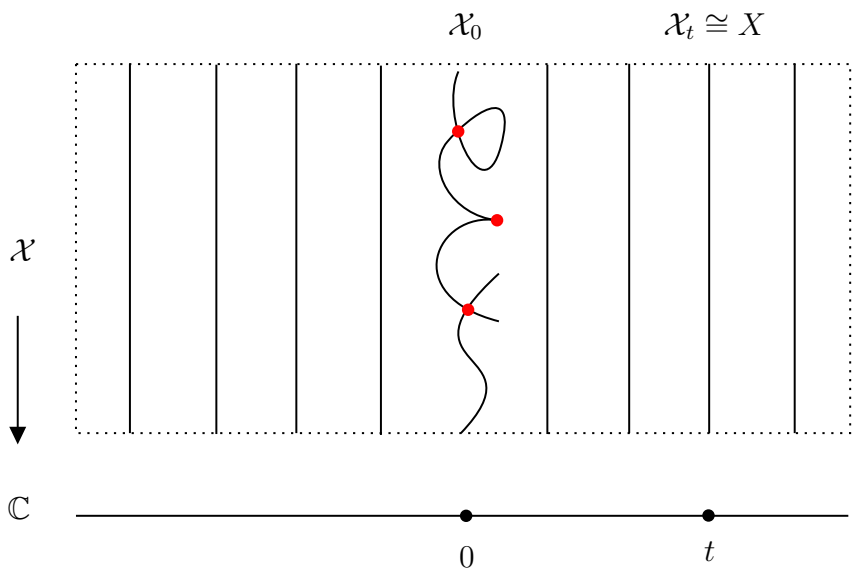

One consequence of the flatness of the morphism $\pi$ is that the Hilbert polynomial $\chi(\calX_t, \calL_t) = \chi(X,L)$ agrees with $\chi(\calX_0,\calL_0)$. Indeed, through more than just an analogy, test configurations \emph{literally} arise as (closures of) orbits of one-parameter subgroups applied to $(X,L^r)$ as a point inside a Hilbert scheme $\Hilb_{\calP(r)}$ for $\calP(r) = \chi(X,L^r)$ (see \cite[Prop. 3.7]{ross2007study}).

The weight associated to a test configuration is defined as follows. The central fibre $(\calX_0,\calL_0)$ is fixed by the $\CC^*$ action on $(\calX,\calL)$ and therefore this action lifts to $H^0(\calX_0,\calL_0^k)$ for $k\gg 0$. A $\CC^*$ action on a vector space splits it into a sum of weight spaces, and the \emph{total weight} is the sum of the weights on each factor. Equivalently we define the total weight function $w(k)$ as the weight of the induced $\CC^*$ action on $\det H^0(\calX_0, \calL_0^k)$. Let us explicitly write
$$\calP(k) = \dim H^0(\calX_0,\calL_0^k) = a_0 k^n + a_1 k^{n-1} + O(k^{n-2}),\quad w(k) = b_0 k^{n+1} + b_1 k^n + O(k^{n-2}).$$
\begin{definition}[Donaldson--Futaki invariant]\label{def:donaldsonfutakiinvariant}
	The weight or \emph{Donaldson--Futaki invariant} of a test configuration $(\calX,\calL)$ for $(X,L^r)$ is given by
	$$\DF(\calX,\calL) := \frac{a_0b_1 - b_0a_1}{a_0}.$$
\end{definition}

An alternative computation of the Donaldson--Futaki weight in terms of intersection theory is now known. Indeed an intersection-theoretic formula was already understood by Tian in terms of the CM line bundle when he introduced his original condition for K-stability \cite{tian1997kahler}, and the equivalence of these notions was later proven \cite{paultian,phong2008deligne} (see also \cite{wang2012height,odaka2013generalization}). Compactify the test configuration trivially at $\infty$ to a family $(\bar \calX,\bar \calL)\to \CCPP^1$. Then the DF invariant or $\mathrm{CM}$-weight is given by
\begin{equation}\label{eq:DFintersectiontheory}\DF(\calX,\calL) := \frac{1}{2(n+1) L^n} \left( n \mu(X,L) (\bar \calL)^{n+1} + (n+1) K_{\bar \calX/\CCPP^1} . \bar \calL^n\right)\end{equation}
where $\mu(X,L) = -(L^{n-1}.K_X)/L^n$ is the \emph{slope} of $(X,L)$. The intersection-theoretic formula will be useful in the definition of fibration stability in \cref{part:fibrations}.

\begin{definition}\label{def:Kstability}
	Let $(X,L)$ be a polarised variety. We say the pair is
	\begin{itemize}
		\item \emph{K-semistable} if $\DF(\calX,\calL) \ge 0$ for any test configuration $(\calX,\calL)$.
		\item \emph{K-polystable} if it is K-semistable and $\DF(\calX,\calL) = 0$ only if $(\calX,\calL)$ normalises to a product test configuration.
		\item \emph{K-stable} if $\DF(\calX,\calL) > 0$ for any test configuration which does not normalise to the trivial test configuration.
		\item \emph{K-unstable} if it is not K-semistable.
	\end{itemize}
\end{definition}

\begin{remark}\label{rmk:slopeKstability}
	An attempt to define a notion of K-stability more in line with Yau's original prediction of a stability condition analogous to slope stability was carried out by Ross--Thomas \cite{ross2006obstruction,ross2007study}. The notion of \emph{slope K-stability} is useful for finding obstructions to K-stability, but except in the case of curves it is not equivalent to K-stability.\footnote{However in that case it essentially provides the first purely algebraic proof of the K-stability of smooth curves, which can only be slope K-destabilised by singular points, therefore allowing the uniformization theorem to take on an interpretation as an instance of \cref{principle}.} Indeed Panov--Ross showed that $\Bl_{p,q} \CCPP^2$ is slope K-stable but not K-stable \cite[Ex. 7.8]{panov2009slope}. The notion of slope K-stability considers test configurations arising from the deformation to the normal cone of a subscheme $Z\subset X$ (see \cite[\S 4]{ross2006obstruction}), but to get a theory equivalent to K-stability one must instead consider \emph{flags} of subschemes. This was already noted by Mumford \cite{mumford1977stability} and has been discussed in detail by Ross--Thomas \cite[\S 3]{ross2007study}. We will discuss certain test configurations arising out of flags of subschemes in the study of stability of fibrations in \cref{part:fibrations}.
\end{remark}

\begin{remark}
	The definition of K-stability is expected to need modification in the case where $(X,L)$ is not Fano in order to be strong enough to imply the existence of a constant scalar curvature K\"ahler metric. A modification was proposed by Sz\'ekelyhidi based on ideas of Witt-Nystr\"om \cite{szekelyhidi2015filtrations,nystrom2012test}, known as \emph{filtration K-stability}. This notion has been expanded upon greatly in the wake of work by Berman--Boucksom--Jonsson (see \cite{berman2021variational}) on a variational approach to K-stability, and the exciting work of Li on the YTD conjecture for polarised varieties \cite{li2020geodesic}.
\end{remark}

\subsection{Canonical K\"ahler metrics\label{sec:kahlermetrics}}

Repetitiously, guided by \cref{principle} we now introduce the corresponding notion of extremal object to K-stability. Indeed K-stability arose directly from its relationship to the existence of canonical K\"ahler metrics, and we emphasise Donaldson's point of view on \emph{constant scalar curvature K\"ahler metrics}. The Ricci form associated to a K\"ahler manifold $(X,\omega)$ is given by
$$\Ric(\omega) = - \frac{i}{2\pi} \deldelbar \log \det \omega$$
and the scalar curvature by
$$S(\omega) = \contr_{\omega} \Ric (\omega).$$
The average value of $S$ is determined topologically and given by
$$\hat S = \frac{n c_1(X) . [\omega]^{n-1}}{[\omega]^n} = n \mu(X,L)$$
where $[\omega]=c_1(L)$ if $\omega$ is integral and a polarising ample line bundle has been chosen.

\begin{definition}
	A K\"ahler metric is \emph{K\"ahler--Einstein} if
	$$\Ric(\omega) = \lambda \omega$$
	for some $\lambda \in \RR$. It is \emph{constant scalar curvature K\"ahler} (cscK) if 
	$$S(\omega) = \hat S.$$
\end{definition}

It was proposed by Calabi to search for examples of canonical K\"ahler metrics satisfying the above property \cite{calabi1982extremal}, which ultimately lead to Calabi's notion of an extremal metric (which we will not comment on further). The instance of \cref{principle} in this setting is a conjecture, proposed for Fano manifolds by Yau and Tian, and for polarised manifolds by Donaldson (although the idea that such a correspondence should hold in this generality was already implicit in the work of Fujiki \cite{fujiki1992moduli}).

\begin{conjecture}[Yau--Tian--Donaldson]\label{conj:YTD}
	A smooth polarised variety $(X,L)$ is K-polystable if and only if it admits a constant scalar curvature K\"ahler metric $\omega\in c_1(L)$. 
\end{conjecture}

The case of Fano manifolds, where a cscK metric is K\"ahler--Einstein, is known in the smooth case due to Chen--Donaldson--Sun \cite{chen2015kahlerI,chen2015kahlerII,chen2015kahlerIII}, with later proofs provided by Tian, and using a variety of different approaches by Datar--Sz\'ekelyhidi, Chen--Sun--Wang, and Berman--Boucksom--Jonsson \cite{tian2015k,datar2016kahler,chen2018kahler,berman2021variational}. There are now versions of the YTD conjecture known for singular Fanos in great generality due to the work of Li--Tian--Wang, Li, and Liu--Xu--Zhuang, whose work combines to show any log Fano pair $(X,\Delta)$ admits a weak K\"ahler--Einstein metric if and only if it is K-polystable \cite{li2021uniform,li2022g,liu2021finite}.

Let us complete our short survey of \cref{principle} for varieties by emphasising the importance of the Kempf--Ness-type functional in this setting also. Here the functional was introduced by Mabuchi, and is otherwise known as the K-energy \cite{mabuchi1986k}. As usual, we specify the functional by its first variation. If $\omega$ is a K\"ahler metric and $\omega_\varphi = \omega + i \deldelbar \varphi$ is any metric in the same K\"ahler class, then the Mabuchi functional is the unique functional $\calM$ on the space of K\"ahler metrics in the class $[\omega]$ such that $\calM(\omega) = 0$ and 
$$\deriv{}{t} \calM(\varphi_t) = \int_X \dot \varphi_t (\hat S - S(\omega_t)) \omega_t^n.$$
By construction, a critical point of the Mabuchi functional is necessarily a cscK metric. A version of \cref{lem:limitslopeGIT} for the Mabuchi functional on a Fano manifold, computing the limit slope as the Donaldson--Futaki invariant of a test configuration corresponding to a geodesic ray, is now understood well in terms of the variational framework (see \cite{boucksom2019uniform} and the references therein). A proof of the YTD conjecture utilising this functional, which follows the approach of the original Kempf--Ness theorem, was subsequently carried out by Berman--Boucksom--Jonsson \cite{berman2021variational}. There is now a great body of work proving the analogous properties to the Kempf--Ness functional for arbitrary polarised manifolds, particularly by Chen and Cheng \cite{chen2021constant,chen2021constant2}.

\subsection{Deformations of Fano threefolds with $\SL(2,\CC)$ actions}

In this section we will briefly describe the explicit K-stability problem for the Mukai--Umemura threefold as exposited by Donaldson \cite[\S 5]{donaldson2008kahler}, and describe a new, singular threefold with large automorphism group for which the same techniques may apply. This singular threefold has not appeared in the literature, and we suggest how its deformation theory may interact with K-stability. 

We will make use of the Mukai--Umemura threefold to construct examples of isotrivial fibrations associated to principal $\SL(2,\CC)$-bundles which admit optimal symplectic connections in \cref{ch:isotrivialfibrations}.

These threefolds with automorphism group $\SL(2,\CC)$ are constructed as common zero-loci of sections of bundles over Grassmannians, and we will describe three such examples. In order to find threefolds with automorphism group $G=\SL(2,\CC)$, we make use of the basic representation theory of $\SL(2,\CC)$ which we now recall.

\begin{theorem}[See for example {\cite[Ch. 11]{fulton2013representation}}]
	For each non-negative integer $k\in \ZZ_{\ge 0}$, there exists a unique, irreducible, finite-dimensional representation of $\SL(2,\CC)$ of dimension $k+1$, denoted $s^k$. This representation can be explicitly described as 
	$$s^k = \CC[X,Y]_k,$$
	the homogeneous polynomials of degree $k$ in two variables acted upon in the standard way.
\end{theorem}

\begin{lemma}
	Given two irreducible representations $s^k, s^l$ of $\SL(2,\CC)$ with $k\ge l$, we have decompositions
	\begin{align*}
		s^k \otimes s^l &\isom s^{k+l} \oplus s^{k+l-2}\oplus \cdots \oplus s^{k-l+2} \oplus s^{k-l},\\
		\Sym^2 s^k &\isom s^{2k} \oplus s^{2k-4}\oplus \cdots \oplus s^4 \oplus s^0,\\
		\Exterior^2 s^k &\isom s^{2k-2} \oplus s^{2k-6} \oplus \cdots \oplus s^6 \oplus s^2.
	\end{align*}
\end{lemma}

These decompositions can be obtained by analysing the highest weights of the tensor products of irreducible representations, and using the fact that each irreducible representation $s^k$ of $\SL(2,\CC)$ is acted on with highest weight $k$. One may also use the Cayley--Sylvester formula by viewing $s^k$ as the symmetric product $\Sym^k s^1$ of the standard representation $s^1 = \CC^2 \isom \CC[X,Y]_1.$

\subsubsection{Mukai--Umemura threefold\label{sec:mukaiumemura}}

In order to construct the Mukai--Umemura threefold as a subset of a Grassmannian, one considers the 7-dimensional vector space $V=s^6$, and the Grassmannian $\Gr_3(V)$ of 3-planes in $V$. This is $3\cdot (7-3) = 12$ dimensional. Begin by choosing a 3-dimensional subspace $\Pi \subset \Exterior^2 V^*$, and define
$$X_\Pi := \{P\in \Gr_3(V) \mid \rest{\omega}{P} = 0 \text{ for all } \omega \in \Pi\}.$$
We can alternatively view $X_\Pi$ as a zero locus of sections of bundles on the Grassmannian as follows. Denote by $$\calV\to \Gr_3(V)$$
the tautological bundle of 3-planes over $\Gr_3(V)$, and let $\omega_1,\omega_2,\omega_3$ be a basis of $\Pi\subset \Exterior^2 V^*$. Then define sections $\sigma_1,\sigma_2,\sigma_3$ of $\Exterior^2 \calV^*$ by $\sigma_i(P) := \rest{\omega_i}{P}$ and observe 
$$X_\Pi = Z(\sigma_1) \cap Z(\sigma_2) \cap Z(\sigma_3) \subset \Gr_3(V).$$
The resulting variety only depends on the three-plane $\Pi$, and a dimension count shows $X_\Pi$ is a threefold for a generic choice of $\Pi$. Indeed by perturbing the 2-forms $\omega_i$ we have a Zariski-open subspace $\calU \subset \Gr_3(\Exterior^2 V^*)$ of 3-planes $\Pi$ such that $X_\Pi$ is a smooth threefold, and a result of Mukai shows that the inequivalent 3-planes lie in separate orbits of the action of $\SL(V)$ on $\calU$. Therefore the possible such threefolds $X_\Pi$ are classified by the quotient space
$$\calM = \calU / \SL(V).$$

\begin{lemma}
	The threefold $X_\Pi$ is Fano for any $\Pi\in \calU$.
\end{lemma}
\begin{proof}
	Let $H=\Exterior^2 \calV^*$. Then $H=\calV \otimes L$ where $L=\det \calV^*$. Therefore we have
	$$\det H = L^{\otimes 2}.$$
	The tangent bundle of $\Gr_3(V)$ is given
	$$T \Gr_3(V) = \calV^* \otimes ((V\otimes \calO)/\calV),$$
	so we see $\det T \Gr_3(V) = L^{\otimes 7}$. Given some $\Pi \in \calU$, the tangent bundle to $X_\Pi$ is the kernel of a surjective map $T \Gr_3(V) \to H\oplus H \oplus H$. Therefore we obtain
	$$\det TX_\Pi \isom L^{\otimes 7} \otimes 3 L^{\otimes -2} \isom L$$
	so $K_{X_\Pi}^{-1} \isom L > 0$ is ample.
\end{proof}

The Mukai--Umemura threefold arises as a particular choice of $\Pi$. Namely if $V=s^6$ then $V^* \isom V$ since the 7-dimensional irreducible representation of $\SL(2,\CC)$ is unique, and consequently
$$\Exterior^2(V^*) = s^{10} \oplus s^6 \oplus s^2.$$
The $s^2$ summand defines a 3-plane $\Pi_0$ in $\Gr_3(\Exterior^2 V^*)$, and the corresponding threefold $X_0 = X_{\Pi_0}$ is the \emph{Mukai--Umemura threefold}. It was shown by Mukai--Umemura that $X_0$ is non-singular, or in other words that $s^2 \in \calU \subset \Gr_3(\Exterior^2 V^*)$ \cite[\S 2]{mukaiumemura}. Since $\Pi_0$ is $\SL(2,\CC)$-invariant, one obtains a natural $\SL(2,\CC)$ action on $X_0$. 

Donaldson computed explicitly Tian's alpha invariant for the Mukai--Umemura threefold, which is a criterion implying the existence of a K\"ahler--Einstein metric on a smooth Fano manifold \cite{tian1987kahler}. Indeed one obtains:

\begin{proposition}[\cite{donaldson2008kahler}]
	The Mukai--Umemura threefold $X_0$ is smooth and admits a K\"ahler--Einstein metric. In particular it is K-polystable.
\end{proposition}

Further to this, one has an $\SL(2,\CC)$-equivariant embedding
$$X_0 \into \PP(H^0(X_0,K_{X_0}^{-1}))$$
where $H^0(X_0,K_{X_0}^{-1}) \isom s^0 \oplus s^{12}.$ The divisor at $\infty$ of $X_0$ is given by the zero set of the section $\sigma\in H^0(X_0,K_{X_0}^{-1})$ generating the $s^0\isom \CC$ factor. The geometry of $X_0$ can be explicitly understood in terms of this divisor, and for example one may observe that $X_0\backslash Z(\sigma)$ is isomorphic to $\SL(2,\CC)/\Gamma$ where $\Gamma$ is the group of symmetries of an icosahedron, and $X_0$ is therefore a natural compactification of this space.

Furthermore the deformation theory of $X_0$ can be computed in terms of the representation theory of $\SL(2,\CC)$. Recall that the variety $X_0$ is identified with a point $\Pi_0 \in \calU \subset \Gr_3(\Exterior^2 V^*)$. The versal deformation space at $[\Pi_0] \in \calM = \calU / \SL(V)$ can be computed in the following way.\footnote{By a result of Mukai all smooth Fano threefolds of rank 1 and degree 22 appear as sections of the Grassmannian and are parametrised by $\calU$, so $\calM$ is a representative for the moduli of Fanos in this deformation family.} We have $\Exterior^2 V^* = s^{10} \oplus s^6 \oplus s^2$ so 
$$T_{\Pi_0} \Gr_3(\Exterior^2 V^*) = (s^{10} \oplus s^6) \otimes s^2 = s^{12} \oplus s^{10} \oplus 2 s^8 \oplus s^6 \oplus s^4.$$
There is an action of $\SL(V) = \SL(s^6)$ on $\Gr_3(\Exterior^2 V^*)$ whose kernel is given by the stabiliser at $\Pi_0=s^2$, which is $\SL(2,\CC)$. The Lie algebra of $\GL(s^6) \isom \End (s^6) \isom s^6\otimes s^6$ is given by
$$\gl(s^6) = s^6\otimes s^6 = s^{12} \oplus s^{10} \oplus s^8 \oplus s^6 \oplus s^4 \oplus s^2 \oplus s^0$$
and therefore $\mathfrak{sl}(s^6) \isom s^{12} \oplus \cdots \oplus s^2$ after quotienting out the factor $s^0$ generating $\CC^* \cdot \id \in \GL(s^6)$. The induced action on $T_{\Pi_0} \Gr_3(s^6)$ does not see the stabiliser $s^2$ factor. One therefore obtains that the Zariski tangent space $T_{[\Pi_0]} \calM$ is given by the quotient of the morphism
$$\frac{\mathfrak{sl}(s^6)}{s^2} \to T_{\Pi_0} \Gr_3(s^6)$$
which by the above computation is
$$T_{[\Pi_0]} \calM = s^8.$$
Note that the generic expected dimension of the smooth locus of $\calM$ is 6, and the point $[\Pi_0]\in \calM$ is expected to be a singular point, corresponding to the fact that $X_0$ has 3-dimensional automorphisms in comparison to the generic member of $\calM$ which has trivial automorphisms (and so the Zariski tangent space is 6+3=9 dimensional, as computed).

The GIT stability of points in a representation $s^k$ of $\SL(2,\CC)$ is well-understood and can be phrased in terms of the zeroes of the associated homogenous polynomials on $\CCPP^1$ (see for example \cite{thomas2005notes}). This lead to the enticing conjecture of Donaldson relating this GIT stability to the K-stability of deformations of $X_0$ in $\calU$ \cite{donaldson2008kahler}. This conjecture was resolved due to the work of Sz\'ekelyhidi and Br\"onnle \cite{szekelyhidi2010kahler,bronnle2011deformation}, whose work applies in far greater generality to deformations of cscK manifolds with automorphisms.

\begin{theorem}
	A small deformation $X_\Pi$ of $X_0$ in $\calU$ is K-(semi/poly)stable if and only if the associated point $v_\Pi\in T_{[\Pi_0]} \calM$ is GIT (semi/poly)stable for the action of $\SL(2,\CC)=\Aut(X_{\Pi_0}, L)$ on $T_{[\Pi_0]} \calM$.
\end{theorem}

Computations of Tian show that manifolds arising from GIT unstable or strictly GIT semistable orbits above are not K\"ahler--Einstein \cite{tian1987kahler}. Deformations of $X_0$ will break the $\SL(2,\CC)$ action, but some may preserve the $\CC^*$ action, and there is a one-parameter family $X(\tau)$ of such threefolds near $X_0$, with $X_0$ corresponding to $\tau = 1$. Donaldson made a refined conjecture in this case \cite[\S 4.1]{donaldson2015stability} where K-polystability should be characterised by properties of the singularities of the divisor, and Cheltsov--Shramov have resolved this conjecture showing that $X(\tau)$ is K-polystable except for a small number of values of $\tau$ \cite{cheltsov2018kahler}. 

\subsubsection{$V_5$ manifold}

One can construct another example of a smooth Fano threefold with an action of $\SL(2,\CC)$ as a zero-locus of sections inside a Grassmannian, this time in the deformation family $V_5$. In this case one uses the irreducible representation $V=s^4$ and considers a triple intersection in $\Gr_2(V)$, which is $2\cdot(5-2)=6$-dimensional. Indeed if $\calV \to \Gr_2(V)$ is the tautological bundle, then $\Exterior^2\calV^*$ is a line bundle and a triple of sections defined by a plane $\Pi \subset \Exterior^2 V^*$ produces a Fano threefold, and a similar formal theory exists as in the case of $V_{22}$ above and the Mukai--Umemura threefold.

Indeed one may compute 
$$\Exterior^2 V^* = s^6\oplus s^2$$
and setting $\Pi_0 = s^2$ one obtains a smooth Fano threefold $Y_0$ with $\SL(2,\CC)$ action. The algebraic geometry of the manifold $Y_0$ has been described in \cite[\S 6]{przyjalkowski2019fano} \cite{sanna2014rational}. In particular the alpha invariant of $Y_0$ can be computed and is greater than $\frac{5}{6}$, so $Y_0$ admits a K\"ahler--Einstein metric and is K-polystable \cite{cheltsov2009extremal}.

However a computation of the versal deformation space for $Y_0$ reveals $\calM = \{Y_0\}$ and $Y_0$ is rigid. Therefore the same deformation picture which exists for the Mukai--Umemura threefold $X_0$ cannot be studied for $Y_0$.

\subsubsection{A singular threefold in the class $V_{14}$}

There is a third deformation family of Fano threefolds which may be constructed as zeroes of sections of a Grassmannian in a straight-forward way, class $V_{14}$. Here one takes 5-fold intersections of forms in $\Gr_2(\CC^6)$. There are 15 dimensions of local moduli of such smooth Fano threefolds. Their K-stability is completely understood.

\begin{theorem}[{\cite[Ex. 4.1.7]{calabiproblem}}]
	Every smooth Fano threefold in the deformation class $V_{14}$ has discrete automorphism group and is K-stable.
\end{theorem}

We will now describe a \emph{singular} threefold in this deformation class with large automorphism group. Consider now $V=s^5$ the 6-dimensional irreducible representation of $\SL(2,\CC)$. As usual let $\calV \to \Gr_2(V)$ denote the tautological bundle, and consider a quintuple of sections of $\Exterior^2 \calV^*$ defined by a 5-plane $\Pi\subset \Exterior^2 V^*$. We have
$$\Exterior^2 s^5 = s^8 \oplus s^4 \oplus s^0$$
and take $\Pi_0 := s^4\subset \Exterior^2 s^5$. Then set
$$Z_0 := Z_{\Pi_0} \subset \Gr_2(V).$$
Since $\Pi_0$ is $\SL(2,\CC)$-invariant, $Z_0$ comes with an $\SL(2,\CC)$ action. Unlike in the cases of the Mukai--Umemura threefold and $V_5$ manifold (which can be shown to be smooth due to explicit descriptions as compactifications of $\SL(2,\CC)/\Gamma$ for $\Gamma$ a finite subgroup), the above classification of smooth $V_{14}$ manifolds shows that $Z_0$ must be singular.\footnote{It would be interesting to understand if $Z_0$ may be identified as a singular compactification of $\SL(2,\CC)/\Gamma$ for some finite subgroup $\Gamma$. The subgroups admitting smooth compactifications were identified by Mukai and Umemura.} Indeed we will explicitly describe $Z_0$ and identify two singular points.

\begin{remark}
	Whilst we cannot prove $Z_0$ itself is Fano, we expect this to be the case. A similar argument to the smooth case of $V_5$ or $V_{22}$ shows that the ample bundle $\det \calV^*$ is generically isomorphic to $K_{Z_0}^{-1}$, so what remains is a more detailed understanding of how the singular locus of $Z_0$ changes its anticanonical bundle. Nevertheless $Z_0$ is a degeneration of smooth Fano threefolds, and therefore has nef anticanonical bundle.
\end{remark}

Consider the Pl\"ucker embedding
$$\Gr_2(V) \into \PP(\Exterior^2 V).$$
Under the canonical identification $\Exterior^2 (s^5)^* \isom \Exterior^2 s^5$ the $s^4$-summand is sent to itself. Indeed $Z_0$ can be identified with the 2-planes $P\in \Gr_2(V)$ which are orthogonal to $s^4\subset \Exterior^2 s^5$. Thus the Pl\"ucker embedding of $\Gr_2(s^5)$ restricts to an $\SL(2,\CC)$-equivariant embedding
$$Z\into \PP(s^0 \oplus s^8).$$
The $s^0$ summand defines an $\SL(2,\CC)$-invariant divisor $D$ at $\infty$ of $Z_0\subset \PP(s^8\oplus s^8)$ corresponding to the intersection $D=Z_0\cap \PP(s^8)\subset \PP(s^0\oplus s^8)$. This summand is generated by a section of the ample line bundle $\det \calV^*$ restricted to $Z_0$, which we will explicitly identify.

Let us now explicitly describe the embedding. First we will describe the representations of $\SL(2,\CC)$ explicitly. Fix a basis $x^i y^{5-i}$ of $s^5 = \CC[x,y]_5$. Then the generators of $\sl(2)$ are given by
$$X=x\pderiv{}{y},\quad Y = y\pderiv{}{x},\quad H = [X,Y] = x\pderiv{}{x} - y \pderiv{}{y}.$$
The subspace $s^4\subset s^8\oplus s^4\oplus s^0 = \Exterior^2 s^5$ is a 5-dimensional space spanned by the forms
\begin{multline*}
	x^2 y^3 \wedge x^5 - 3 x^3 y^2 \wedge x^4 y,\quad 2xy^4 \wedge x^5 - 4 x^2 y^3 \wedge x^4 y,\quad y^5\wedge x^5 + xy^4 \wedge x^4 y - 8 x^2y^3 \wedge x^3 y^2,\\
	3 y^5 \wedge x^4 y - 6 x y^4 \wedge x^3 y^2,\quad y^5 \wedge x^3 y^2 - 3 x y^4 \wedge x^2 y^3.
\end{multline*}
The 10-dimensional complement $s^8\oplus s^0$ is spanned by the forms
\begin{multline}
	x^4 y \wedge x^5,\quad x^3 y^2 \wedge x^5,\quad 3x^2 y^3 \wedge x^5 + 5 x^3 y^2 \wedge x^4 y,\\
	x y^4 \wedge x^5 + 5 x^2 y^3 \wedge x^4 y,\quad y^5\wedge x^5 + 15 x y^4 \wedge x^4 y + 20 x^2 y^3 \wedge x^3 y^2,\\
	y^5 \wedge x^5 - 5 xy^4 \wedge x^4 y + 10 x^2 y^3 \wedge x^3 y^2,\quad y^5 \wedge x^4 y + 5 x y^4 \wedge x^3 y^2,\quad 3 y^5 \wedge x^3 y^2 + 5 x y^4 \wedge x^2 y^3,\\
	y^5 \wedge x^2 y^3,\quad y^5 \wedge x y^4,\label{eq:s8s0basis}
\end{multline}
where the irreducible component $s^0$ is spanned by the form
$$\tilde \sigma = y^5 \wedge x^5 - 5 xy^5 \wedge x^4 y + 10 x^2 y^3 \wedge x^3 y^2.$$
Thus the invariant section $\sigma$ can be identified with the induced section of the ample line bundle $\Exterior^2 \calV^*$ from the form $\tilde \sigma$ defined by sending
$$P \mapsto \rest{\tilde \sigma}{P}$$
for all $P\in \Gr_2(s^5)$ such that $\rest{\omega}{P} = 0$ for all $\omega$ in the span of $s^4$. The points of $Z_0 \subset \PP(s^0 \oplus s^8)$ can be identified with the elements $P=a\wedge b$ for $a,b,\in s^5$ with are orthogonal to every 2-form in $\Pi_0$. In particular we may single out four points which all lie along the $\SL(2,\CC)$-invariant divisor $D=Z(\sigma) = \PP(s^8) \cap Z_0 \subset Z_0$ at $\infty$ of $Z_0$:
\begin{equation}\label{eq:fixedpoints}x^4 y \wedge x^5, \quad x^3 y^2 \wedge x^5, \quad y^5\wedge x^2 y^3,\quad  y^5\wedge xy^4.\end{equation}
We may now translate these representations into an explicit embedding $Z_0 \into \Gr_2(s^5) \into \PP(\Exterior^2 s^5))$ of the variety $Z_0$. Let us change notation slightly to denote $x^5\in s^5$ by $e_1$, $x^4 y = e_2,\dots, y^5 = e_6$. Then if $W=a\wedge b \in \Exterior^2 s^5$ is a 2-plane in $s^5$ with $a=\sum a_i e_i, b = \sum b_i e_i$, the condition that $W$ lies in $Z_0$ is given by the following relations:
\begin{align*}
	(a_4 b_1 - a_1 b_4) -3 (a_3 b_2 - a_2 b_3) &= 0\\
	2(a_5 b_1 - a_1 b_5) - 4 (a_4 b_2 - a_2 b_4) &=0\\
	(a_6 b_1 - a_1 b_6) + (a_5 b_2 - a_2 b_5)-8(a_4 b_3 - a_3 b_4)&=0\\
	3(a_6 b_2 - a_2 b_6) - 6 (a_5 b_3 - a_3 b_5) &= 0\\
	(a_6 b_3 - a_3 b_6) - 3 (a_5 b_4 - a_4 b_5) &= 0.
\end{align*}
In Pl\"ucker coordinates for $W$ these relations become
\begin{align*}
	-W_{14} + 3 W_{23} &= 0\\
	-2W_{15} + 4 W_{24} &= 0\\
	-W_{16} - W_{25} + 8 W_{34} &= 0\\
	-3W_{26} + 6 W_{35} &= 0\\
	-W_{36} + 3 W_{45} &= 0.
\end{align*}

Combined with the Pl\"ucker relations, this produces an explicit representation of $Z_0$ embedded in $\PP(\Exterior^2 s^5)$ as the intersection of the Grassmannian $\Gr_2(s^5)$ with 5 hyperplanes in $\Exterior^2 s^5$.

\begin{remark}
	At this point it becomes clear that despite the non-genericity of $s^4 \subset \Exterior^2 V$ (e.g. $Z_0$ is singular), the variety $Z_0$ has the expected dimension 3, being the intersection of the 8-dimensional $\Gr_2(\CC^6)$ with 5 hyperplanes.
\end{remark}

Let us now describe the $\SL(2,\CC)$-equivariant embedding $Z_0 \into \PP(s^8 \oplus s^0)$. The variety $Z_0$ can be identified with the locus of totally decomposable two-forms $a\wedge b\in s^8 \oplus s^0$. Using the basis of \eqref{eq:s8s0basis} we can distil the embedding of $Z_0$ from the above embedding as follows. Let $\alpha \in s^8 \oplus s^0$ be a two-form. Then $\alpha$ is specified as a vector
$$\alpha = \alpha_i v_i$$
where $v_1,\dots,v_{10}$ are the basis vectors listed in \eqref{eq:s8s0basis}. Such a vector is totally decomposable if it satisfies the Pl\"ucker relations for $\Exterior^2 s^5$. We express the basis of $s^8\oplus s^0$ in terms of the standard basis of $s^5$:
\begin{multline}
	e_2 \wedge e_1,\quad e_3 \wedge e_1,\quad 3e_4 \wedge e_1 + 5 e_3 \wedge e_2,\\
	e_5 \wedge e_1 + 5 e_4 \wedge e_2,\quad e_6\wedge e_1 + 15 e_5 \wedge e_2+ 20 e_4 \wedge e_3,\\
	e_6 \wedge e_1 - 5 e_5 \wedge e_2 + 10 e_4 \wedge e_3,\quad e_6 \wedge e_2 + 5 e_5 \wedge e_3,\quad 3 e_6 \wedge e_3 + 5 e_5 \wedge e_4,\\
	e_6 \wedge e_4,\quad e_6\wedge e_5,\label{eq:s8s0ebasis}
\end{multline}
Then the element $\alpha$ has Pl\"ucker coordinates given by
\begin{align*}
	W_{12} &= - \alpha_1\\
	W_{13} &= -\alpha_2\\
	W_{14} &= -3\alpha_3\\
	W_{15} &= -\alpha_3\\
	W_{16} &= -\alpha_5 - \alpha_6\\
	W_{23} &= -5\alpha_3\\
	W_{24} &= -5\alpha_4\\
	W_{25} &= -15\alpha_5 + 5\alpha_6\\
	W_{26} &= -\alpha_7\\
	W_{34} &= -20\alpha_5 - 10\alpha_6\\
	W_{35} &= -5\alpha_7\\
	W_{36} &= -3\alpha_8\\
	W_{45} &= -5\alpha_8\\
	W_{46} &= -\alpha_9\\
	W_{56} &= -\alpha_{10}.
\end{align*}
Using the Pl\"ucker relations gives the homogeneous ideal defining $Z_0$
\begin{multline*}
	\alpha_{8}^{2}-5\,\alpha_{7}\alpha_{9}+20\,\alpha_{5}\alpha_{10}+10\,\alpha_{6}\alpha_{10},\\
	\,5\,\alpha_{7}\alpha_{8}-15\,\alpha_{5}\alpha_{9}+5\,\alpha_{6}\alpha_{9}+5\,\alpha_{4}\alpha_{10},\\
	\,5\,\alpha_{7}^{2}-45\,\alpha_{5}\alpha_{8}+15\,\alpha_{6}\alpha_{8}+5\,\alpha_{3}\alpha_{10},\\
	\,20\,\alpha_{5}\alpha_{7}+10\,\alpha_{6}\alpha_{7}-15\,\alpha_{4}\alpha_{8}+5\,\alpha_{3}\alpha_{9},\\
	\,5\,\alpha_{5}\alpha_{8}+5\,\alpha_{6}\alpha_{8}-\alpha_{3}\alpha_{9}+3\,\alpha_{3}\alpha_{10},\\
	\,5\,\alpha_{5}\alpha_{7}+5\,\alpha_{6}\alpha_{7}-3\,\alpha_{3}\alpha_{8}+\alpha_{2}\alpha_{10},\\
	\,15\,\alpha_{5}^{2}+10\,\alpha_{5}\alpha_{6}-5\,\alpha_{6}^{2}-\alpha_{3}\alpha_{7}+\alpha_{1}\alpha_{10},\\
	\,20\,\alpha_{5}^{2}+30\,\alpha_{5}\alpha_{6}+10\,\alpha_{6}^{2}-9\,\alpha_{3}\alpha_{8}+\alpha_{2}\alpha_{9},\\
	\,5\,\alpha_{4}\alpha_{5}+5\,\alpha_{4}\alpha_{6}-3\,\alpha_{3}\alpha_{7}+\alpha_{1}\alpha_{9},\\
	\,5\,\alpha_{3}\alpha_{5}+5\,\alpha_{3}\alpha_{6}-\alpha_{2}\alpha_{7}+3\,\alpha_{1}\alpha_{8},\\
	\,300\,\alpha_{5}^{2}+50\,\alpha_{5}\alpha_{6}-50\,\alpha_{6}^{2}-25\,\alpha_{4}\alpha_{7}+25\,\alpha_{3}\alpha_{8},\\
	\,20\,\alpha_{3}\alpha_{5}+10\,\alpha_{3}\alpha_{6}-15\,\alpha_{3}\alpha_{7}+5\,\alpha_{2}\alpha_{8},\\
	\,5\,\alpha_{3}\alpha_{4}-45\,\alpha_{3}\alpha_{5}+15\,\alpha_{3}\alpha_{6}+5\,\alpha_{1}\alpha_{8},\\
	\,5\,\alpha_{3}^{2}-15\,\alpha_{2}\alpha_{5}+5\,\alpha_{2}\alpha_{6}+5\,\alpha_{1}\alpha_{7},\\
	\,15\,\alpha_{3}^{2}-5\,\alpha_{2}\alpha_{4}+20\,\alpha_{1}\alpha_{5}+10\,\alpha_{1}\alpha_{6}
\end{multline*}
inside $\CC[\alpha_1,\dots,\alpha_{10}]$.

One can observe directly from computing the Jacobian of partial derivatives that:
\begin{lemma}
	The variety $Z_0$ is singular at 
	$$P=x^4 y \wedge x^5,\quad Q = y^5 \wedge xy^4.$$
\end{lemma}
\begin{proof}
	The generic rank of the Jacobian of the defining ideal of $Z_0$ is $6$, and at the points $P$ and $Q$ drops to $5$. Note that the other points lying along the invariant divisor $D$ in \eqref{eq:fixedpoints} are non-singular points of $Z_0$ where the Jacobian has rank $6$.
	
	In fact, a less computational argument can be given that $P$ and $Q$ are singular. Following the discussion in \cite[Prop. A.4]{kuznetsov2003derived} we note that a 2-dimensional subspace $P\subset V$ such that $P\in Z_0$ is a singular point of $Z_0$ if and only if $P\subset \ker \alpha$ for some 2-form $\alpha \in \Pi_0$. Now see that $P=x^4 y \wedge x^5$ is contained in the kernel of the form
	$$\alpha = y^5 \wedge x^3 y^2 - 3xy^5 \wedge x^2 y^3\in s^4$$
	and $Q= y^5 \wedge xy^4$ lies entirely within the kernel of the form
	\[\beta = x^2 y^3 \wedge x^5 - 3x^3 y^2 \wedge x^4 y\in s^4.\qedhere\]
\end{proof}

Notice that the $\CC^*$ action of $\SL(2,\CC)$ is infinitesimally generated by the commutator $H=[X,Y]$, and planes $P\in Z_0$ of the form $a\wedge b$ for $a,b$ eigenvectors of $H$ will be fixed by the $\CC^*$ action on $Z_0$. Indeed the above four points \eqref{eq:fixedpoints} on $D$ are fixed points of the action of $\CC^* \subset \SL(2,\CC)$.

Following the same process of geometrically describing the Mukai--Umemura threefold by Donaldson \cite[\S 5.2]{donaldson2008kahler}, one could proceed to attempt to describe further the structure of the divisor $D$ at $\infty$ and the symplectic geometry of $Z_0$. Indeed choosing an invariant symplectic structure on $Z_0$ for example by pulling back the $\SL(2,\CC)$-invariant Fubini--Study form via the $\SL(2,\CC)$-equivariant Pl\"ucker embedding $Z_0 \into \PP(s^8\oplus s^0)$, one can consider a moment map for the $S^1\subset \CC^*$ action for which we have four fixed points. This should act like a Morse function for the $\SL(2,\CC)$-invariant divisor at $\infty$ (except for the fact that $Z_0$ is now singular), and for example the index of this Morse function at the $\CC^*$ fixed points should be given by the weight of the $\CC^*$ action on the fibre of $\det \calV^*$ over those points. In particular we have weights
$$w(x^4 y \wedge x^5) = 8, \quad w(x^3 y^2 \wedge x^5) = 6, \quad w(y^5 \wedge x^2 y^3) = -6,\quad w(y^5 \wedge xy^4) = -8.$$

We see the points $P$ and $Q$ should be thought of as the highest and lowest critical points of the Morse function. As in \cite{donaldson2008kahler} we predict that by by studying the level sets of the Morse function on the divisor $D$, one may build a fairly explicit description of the geometry of $Z_0$ around the divisor $D$, and for example recover its homotopy type, and explicitly describe the local defining equation for the divisor $D$ aiding in the computation of the alpha invariant.

We predict that the points identified above are the only singular points of $Z_0$.

\begin{conjecture}
	The points $P, Q$ are the only singular points of $Z_0$.
\end{conjecture}

Let us now turn to the versal deformation space of $Z_0$, similarly to the case of the Mukai--Mumemura threefold. To vary our Fano threefold within the deformation class, we can vary the choice of 5-plane of two-forms on $V=s^5$, so we are concerned with the Grassmannian $\Gr_5(\Exterior^2 V^*)$ which is 50-dimensional. The tangent space at $\Pi=s^4$ is given by
$$T_{s^4} \Gr_5 (\Exterior^2 s^5) = s^4 \otimes (s^8 \oplus s^0) = s^{12} \oplus s^{10} \oplus s^8 \oplus s^6 \oplus 2s^4$$
and there is an action of $\SL(V)$ on $\Gr_3(\Exterior^2 V^*)$ which is infinitesimally described by
$$\sl(s^5) = \frac{s^5\otimes s^5}{s^0} = s^{10} \oplus s^8 \oplus s^6 \oplus s^4 \oplus s^2.$$
The stabiliser at $s^4$ for this action of $\SL(s^5)$ is given by the $s^2$ generating the $\SL(2,\CC)$ action, and thus we obtain an 18-dimensional versal deformation space
$$T_{[s^4]} \calM = s^{12} \oplus s^4.$$
\begin{remark}
	Note that this deformation space is larger than the dimension of the moduli space of smooth Fanos of type $V_{14}$. The higher-dimensional Zariski tangent space to the moduli space at the point $[s^4]$ indicates the singular nature of $\calM$ at this point, and the discrepancy is captured precisely by the extra 3-dimensional stabiliser at $[s^4]$ which is not present for the smooth points, all of which have discrete automorphism groups.
\end{remark}

Motivated by the analogous picture in the smooth case of the Mukai--Umemura threefold, we conjecture the following:

\begin{conjecture}\label{conj:deformationstable}
	The singular threefold $(Z_0,-K_{Z_0})$ is Fano and K-polystable. Furthermore a small deformation $Z_x$ of $Z_0$ in $\calM$ is K-(un/semi/poly)stable if and only if the corresponding point $x\in T_{[s^4]} \calM$ is GIT (un/semi/poly)stable for the $\SL(2,\CC)$ action on this vector space.
\end{conjecture}

As in the case of the Mukai--Umemura threefood, the K-polystability of $Z_0$ should be verifiable just by analysing the behaviour around the single $\SL(2,\CC)$-invariant divisor $D=Z(\sigma)$. Indeed if we assumed that the equivariant $\alpha_G$-invariant was applicable in this setting, this invariant divisor would produce the only contribution to the calculation.

We expect that the recent computational techniques of the beta or delta invariant, which are highly effective in computations of K-stability for singular Fanos, will be applicable to this setting. The formal structure of the divisor appears similar to the case of $X_0$ or $Y_0$ and if its local description is similarly given by a union of tangent lines to a rational curve with a cusp singularity of the form $z_1^2 = z_2^3$ it is reasonable to expect the same bound $\alpha_G(Z_0)\ge \frac{5}{6}$ holds in this case.

The conjecture \cref{conj:deformationstable} is in fact a theorem for $\QQ$-Gorenstein Fano varieties by the work of Spotti--Sun--Yao \cite{spotti2016existence}, who generalised the techniques of Sz\'ekelyhidi and Br\"onnle to the case of smooth deformations of singular log Fanos in that case. If the $Z_0$ is $\QQ$-Gorenstein and K-polystable, then the results of Spotti--Sun--Yao would resolve the conjecture.

To be more precise about the conjectural K-polystability of deformations, let us recall the structure of GIT of $\SL(2,\CC)$ acting on $s^p$.
\begin{proposition}[See for example {\cite[Thm. 3.10]{thomas2005notes}}]\label{prop:sl2GIT}
	The orbits of $\SL(2,\CC)$ on $s^p$ come in one of five types:
	\begin{enumerate}
		\item The trivial orbit $\{0\}$. 
		\item The orbits of polynomials in $s^p$ having no zero of multiplicity $\ge p/2$. These orbits are closed in $s^p$.
		\item The orbits of polynomials having two distinct zeroes, each of multiplicity $=p/2$. These orbits are closed and have positive-dimensional stabiliser $\CC^* \in \SL(2,\CC)$.
		\item The orbits of polynomials having one zero of multiplicity $=p/2$ and all other zeroes with smaller multiplicity. Then orbits are not closed, but contain an orbit of type (iii) in their closure.
		\item The orbits of polynomials with a zero of multiplicity $>p/2$. These orbits are not closed and $0$ is contained in their closure.
	\end{enumerate}
\end{proposition}
In the sense of GIT these orbits can be categorized as follows. Orbits of type (i) and (v) are GIT-unstable. Those of type (iv) are semistable, of type (iii) and (ii) polystable. The orbits of type (iii) are strictly polystable.

Let us now use this description to describe the orbits inside the direct sum $s^{12}\oplus s^4$.

\begin{proposition}
	An orbit inside $s^{12}\oplus s^4$ is GIT (un/semi/poly)stable if and only if the union of zeros of the two polynomials is (un/semi/poly)stable in the sense of \cref{prop:sl2GIT} for $p=16$.
\end{proposition}
\begin{proof}
	This can be seen most directly from the corresponding moment map description. The moment map for the action of $\SL(2,\CC)$ on $s^p$ is given by the centre of mass of the configuration of $p$ points on $\CCPP^1 = S^2$. The moment map for the action on $s^{12}\oplus s^4$ is given by the sum of moment maps, corresponding to the centre of mass of the combined system of 16 zeroes.
\end{proof}

Let us denote $M:=s^{12}\oplus s^4$. Let us denote by $M^s$ the locus inside $M$ corresponding to smooth representatives of $V_{14}$ near $Z_0$. Then the fact that every smooth representative of $V_{14}$ has trivial automorphism group reveals that $M^s$ must be contained inside the union of orbits of stable points in $M$. Since smoothness is an open condition, one expects $M^s$ to define an 18-dimensional Zariski-open subset inside $M$ of GIT-stable points, and the quotient $M^s/\SL(2,\CC)$ to provide a model for the 15-dimensional moduli of smooth representatives of class $V_{14}$. 

\subsubsection{General remarks}

The singular threefold above should be a typical example in the study of K-stability of varieties. Namely, the study of K-stabiliy of polarised varieties $(X,L)$ in terms of test configurations $(\calX,\calL)$ implies a menagerie of mildly singular polarised varieties with large automorphism groups. Indeed the central fibre of any test configuration admits at least a $\CC^*$ action.

In the case where $(X,L)$ is K-semistable there is conjecturally an optimal degeneration to a possibly singular K-polystable variety, so generally one expects a wealth of examples of singular K-polystable varieties with non-discrete automorphism groups. Existence results for K\"ahler--Einstein metrics on singular varieties have strengthened in recent years, particularly in the Fano case. The study of K-stability of deformations of K-polystable varieties (see for example \cite{szekelyhidi2010kahler,spotti2016existence}, and in the case of bundles for example \cite{buchdahl2020polystability}) suggests that an effective method of understanding the local moduli of K-polystable varieties, and indeed of K\"ahler--Einstein or cscK manifolds, is to look at the GIT stability of smooth loci near singular points in the moduli space which have large automorphism groups. 

The above example in the deformation class $V_{14}$ appears to be a to fit into the existing framework of this deformation theory, and serves as an excellent test case for this approach to local moduli, since the structure of the smooth deformations of $Z_0$ are completely understood due to the strong classification results of K-stable Fano manifolds \cite{calabiproblem}.

	\part{$Z$-critical metrics\label{part:zcritical}}
	
	\chapter{Background\label{ch:zcriticalbackground}}
	
	In this chapter we will survey the origins of the folkore \cref{conj:folklore}, Bridgeland stability, the deformed Hermitian Yang--Mills equation, and the $Z$-critical equation as it relates to mirror symmetry. The material of this chapter will not be directly used in the subsequent chapters, but we hope that it provides an updated survey of the place the dHYM or $Z$-critical equation fits into the mirror symmetry picture. For previous surveys which emphasize the role of stability conditions as they relate to physics, see \cite{clay,aspinwall2004d,aspinwall2009dirichlet}.

	\section{Mirror symmetry and BPS branes}
	
	In this section we briefly summarise the physical origins of BPS D-branes and their incarnations in various models of string theory. See \cite{aspinwall2004d} for a survey of D-branes on Calabi--Yau manifolds.
	
	In superstring theory the model for spacetime $\calS$ is 10-dimensional, and a common toy model is to consider the product structure
	$$\calS= \RR^{3,1} \times X$$
	where $X$ is a 6-dimensional manifold. In order to macroscopically reflect the regular 4-dimensional spacetime of (super)gravity, one assumes the manifold $X$ is compact, with a diameter roughly near the Planck length. 
	
	By analysing the supersymmetry condition for $\calS$, Candelas--Horowitz--Strominger--Witten deduced that the manifold $X$ must have, ignoring quantum corrections (i.e. to leading order in the string tension $\alpha'$), $\SU(3)$ holonomy \cite{candelas1985vacuum}. That is, $X$ is a Calabi--Yau threefold with K\"ahler metric $\omega$ and holomorphic volume form $\Omega$ satisfying the equation
	$$\frac{\omega^n}{n!} = c \Omega \wedge \bar \Omega$$
	for some constant $c\in \RR$. 
	
	Inside the spacetime $\calS$, open string worldsheets $\iota: \Sigma \into \calS$ are required to satisfy boundary conditions.\footnote{A \emph{string worldsheet} is the surface $\iota: \Sigma \into \calS$ inside spacetime which is swept out by the string as it moves through time. When such a string is open, its worldsheet can be modelled by a Riemann surface $\Sigma$ with boundary, embedded in $\calS$.} The Dirichlet-type boundary conditions lead to the notion of a ``D-brane". This consists of the following data:
	\begin{itemize}
		\item A submanifold $L\subset \calS$ of spacetime on which the D-brane is ``wrapped." This submanifold is required to satisfy certain geometric criteria depending on the model of string theory considered.
		\item A ``Chan--Paton bundle" $E$ over the submanifold with a gauge field (connection) on it. This bundle (with connection) is often viewed as a geometric representative of the \emph{charge} of the D-brane, which takes values in the K-theory (possibly K-theory with connection, i.e. differential K-theory) of the submanifold.
	\end{itemize}
	The D-brane serves as a boundary condition by requiring that $\iota(\bdry \Sigma) \subset L$ and additional couplings between the gauge fields over $\Sigma$ and that on $E$. The particular geometry of the submanifolds $L\subset X$ of spacetime and the bundles $E\to L$ over them depends on the model of superstring theory being considered. In Type II string theory one has two different models for D-branes, arising from the A and B topological twists of the theory. In either model spacetime is compactified on a Calabi--Yau manifold $(X,\omega,\Omega)$. 
	
	In the A-model, D-branes are wrapped on Lagrangian submanifolds $L\subset (X,\omega)$ and the Chan--Paton bundle is a flat unitary bundle over $L$. In the B-model the D-branes are wrapped on complex submanifolds $L\subset X$ and the Chan--Paton bundle admits a unitary connection with curvature of type $(1,1)$, a Chern connection. Therefore the Chan--Paton bundle is a holomorphic bundle over $L$. In fact the process of taking sums 
	$$\calE = \bigoplus_{n\in \ZZ} (E_n,L_n)$$
	of D-branes wrapped on different submanifolds, combined with a study of deformations of such objects, shows that one should enlarge the D-branes in Type IIB string theory to include \emph{complexes} of holomorphic bundles \cite[\S 5.3]{aspinwall2009dirichlet}. This homological description of D-branes was first predicted by Kontsevich \cite{kontsevich1995homological} before the precise notion of a D-brane had been distilled in the physical literature.
	
	\subsection{Large volume limit}
	
	The study of superstring theory is generally considered \emph{perturbatively} around the limit in which string length $\ell_s$ goes to zero.\footnote{The string length is often replaced in the physical literature by the string tension $\alpha'\sim\ell_s^2$. This factor of $\alpha'$ always appears as a perturbative constant in front of any curvature term $F$ for the gauge field on the Chan--Paton bundle $E$ over a D-brane. The limit $\alpha'\to 0$ is equivalent to $\ell_s \to 0$ and in \cref{ch:correspondence} where we consider the scaling of the $Z$-critical equation by $\varepsilon^2 = 1/k$ we note $\varepsilon\sim \ell_s$ is a similar scaling limit.} This is equivalent to the \emph{large volume limit} on the Calabi--Yau manifold $(X,\omega,\Omega)$ in which $\omega\mapsto k\omega$ and $k\to \infty$.\footnote{Either think of the string length become very small compared to the size of the spacetime, or the spacetime becoming very large compared to the fixed string size.} The validity of the predictions of superstring theory and mirror symmetry which do not depend only on the topological A or B model (i.e. metric considerations, or those considerations which require understanding both the K\"ahler $\omega$ and complex $\Omega$ structures on the Calabi--Yau) should be understood as only approximate, depending on corrections arising from quantum and stringy effects as one moves away from the large volume limit (or its mirror, the \emph{large complex structure limit}).
	
	For example as has been recently discussed by Li \cite{li2020metric} (see also \cite[\S 7]{aspinwall2009dirichlet} for an earlier discussion) the existence of an SYZ fibration by special Lagrangian tori should only be expected in some generic region of the Calabi--Yau threefold $X$. The mass of the region on which the SYZ fibration exists should approach the full volume of $X$ as one approaches the large complex structure limit. 
	
	Similarly the study of BPS D-branes discussed in the next section, which is a non-topological\footnote{Here ``topological" is used in the physical sense, and refers to those properties of the geometry which do not depend on the Riemannian metric. Namely this is not taken to exclude effects depending on the complex structure or symplectic structure, so is certainly not only topological in the mathematical sense.} aspect of string theory depending on both the K\"ahler and complex moduli of $X$, should experience similar corrections away from the large volume limit. Near the large volume limit understanding of D-branes, $Z$-critical metrics, and stability conditions is closer to the non-derived theory occurring at $k=\infty$ (which correponds to the traditional slope stability and Hermite--Einstein metrics discussed in \cref{sec:stability}). As we move away from this limit, more features of the derived category and influences from corrections arising out of enumerative geometry should manifest. In some sense this explains the effectiveness of the asymptotic assumptions appearing in \cref{ch:zcriticalconnections,ch:correspondence} for studying $Z$-critical connections, or the comparative ease through which Bayer's polynomial Bridgeland stability conditions may be studied (see \cref{sec:polynomialbridgelandstability}) compared to genuine Bridgeland stability.
	
	\subsection{BPS branes}
	
	Critical to superstring theory is that the D-branes under consideration respect supersymmetry. Such branes are known as \emph{BPS branes} and given the choice of Type IIA or IIB there are alternative perspectives on what geometric conditions a pair $(E,L)$ must satisfy to be BPS.
	
	\begin{itemize}
		\item In Type IIA string theory the condition for a D-brane to be BPS was first derived in \cite{becker1995fivebranes} and requires that the submanifold $L$ should be a \emph{special Lagrangian} in the sense that
		$$\Im(e^{-i\varphi} \rest{\Omega}{L}) = 0$$
		where $\Omega$ is the holomorphic volume form on $X$ (which is therefore a top-degree form on $L$) and $\varphi$ is a constant argument.
		\item Inspired by \cref{principle} and the success of \cref{thm:DUY}, Thomas and Thomas--Yau conjectured an alternative characterisation of BPS D-branes in the A-model in terms of stability of Lagrangians \cite{thomas2001moment,thomas2001special}, now known as the \emph{Thomas--Yau conjecture}. This has since been upgraded to the derived Fukaya category by Joyce \cite{joyce2014conjectures} and so sometimes obtains the suffix \emph{--Joyce}. In this case the stability condition should have central charge
		$$Z(L) = \int_L \Omega$$
		where $L\into X$ is the Lagrangian and $\Omega$. See \cite{li2022thomas} for a recent discussion of the status of this proposal. 
		\item On the B-side a BPS condition was proposed on physical grounds by Douglas, called \emph{$\Pi$-stability} \cite{douglas2005stability,douglas-icm}. This was distilled into the mathematical concept of \emph{Bridgeland stability} by Bridgeland \cite{bridgeland2007stability}. We will discuss this notion in more detail in \cref{sec:Bridgelandstability}. 
		\item A characterisation of the BPS condition in terms of the associated gauge fields in Type IIB string theory was carried out directly in \cite{marino2000nonlinear} for Abelian D-branes (that is, $E\to L$ being a line bundle\footnote{One can only apologise for the confusing notation here, where $L$ is the submanifold not the line bundle.}). A derivation, again for a line bundle, using a semiflat model of SYZ mirror symmetry was carried out in \cite{LYZ}. On the A-side the special Lagrangian equation is transformed into the \emph{deformed Hermitian Yang--Mills equation} which we will review in \cref{sec:dhym}. 
	\end{itemize}

	Guided by \cref{principle} and the physical origins of $\Pi$-stability and the dHYM equation, one is lead to the folklore \cref{conj:folklore} dicussed in the introduction. Combined with mirror symmetry, which predicts that the BPS branes in the A and B model should be interchanged \cite[Conj. 1.4]{aspinwall2009dirichlet}, we get a larger body of conjectured correspondences represented in \cref{fig:mirrorsymmetryprinciple}.
	
	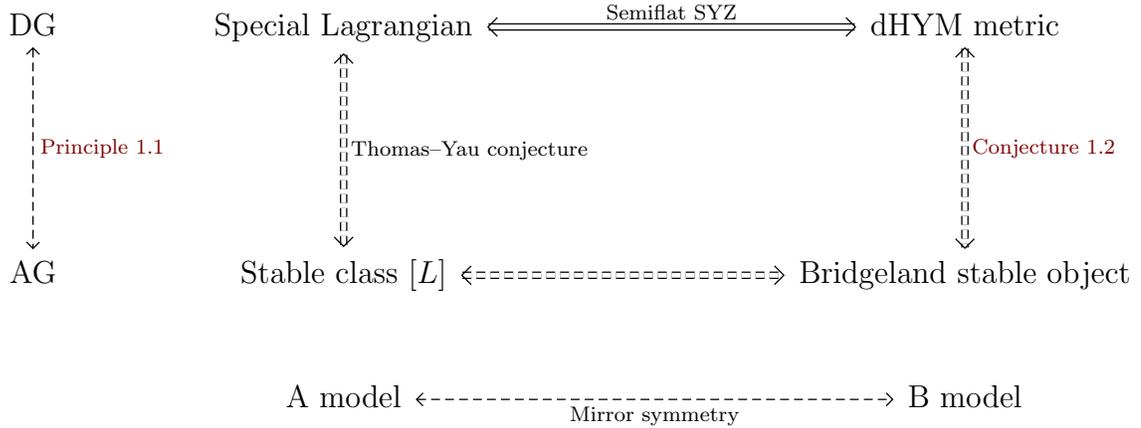
\begin{figure}[h]
		\centering
		\begin{tikzcd}[
			column sep={10em,between origins},
			row sep={4em,between origins},
			]
			\text{DG} \arrow[leftrightarrow,dashed]{dd}{\text{\cref{principle}}} & \text{Special Lagrangian} \arrow[Leftrightarrow]{rr}{\text{Semiflat SYZ}} \arrow[Leftrightarrow,dashed]{dd}{\text{Thomas--Yau conjecture}}&  & \text{dHYM metric} \arrow[Leftrightarrow,dashed]{dd}{\text{\cref{conj:folklore}}}\\
			&  &  & \, \\
			\text{AG} & \text{Stable class $[L]$} \arrow[Leftrightarrow,dashed,swap]{rr} & \, & \text{Bridgeland stable object}\\
			& \text{A model} \arrow[leftrightarrow,dashed,swap]{rr}{\text{Mirror symmetry}} & & \text{B model}
		\end{tikzcd}
		\caption{The grand mirror symmetry/\cref{principle} conjecture.}\label{fig:mirrorsymmetryprinciple}
	\end{figure}

	\section{$\Pi$-stability and Bridgeland stability\label{sec:Bridgelandstability}}
	
	The notion of $\Pi$-stability was introduced by Douglas \cite{douglas2005stability} as a criteria for a D-brane in Type IIB string theory to be a BPS brane. See \cite{aspinwall2002d} for a discussion of the physical interpretations of $\Pi$-stability. The essential principle is to regard a short exact sequence
	\begin{center}
		\ses{F}{E}{Q}
	\end{center}
	as a decay of the D-brane $E$ to the pair $F$ and $Q$ (or conversely to view $E$ as a ``bound state" of $F$ and $Q$). In order for $E$ to be BPS, it must be extremal in the sense that it has maximal supersymmetry charge density. One criteria for determining this is that $E$ cannot decay into any D-brane $F$ with \emph{larger} supersymmetry charge density than $E$ itself. This translates to the criterion that
	$$\mu_Z(F) < \mu_Z(E)$$
	where $Z$ is the supersymmetry charge (and $\mu_Z$ is the charge density or ``normalised charge," given by the slope). 
	
	This criterion closely resembles the traditional slope stability discussed in \cref{sec:stability}, but as mentioned previously, a deeper understanding of D-branes in light of homological mirror symmetry reveals that one must consider the case where $E,F$, and $Q$ are \emph{complexes of coherent sheaves}. Here the notion of a short exact sequence fails (instead being replaced by an \emph{exact triangle}) and it is unclear which complex out of $F,E$ and $Q$ is meant to be decaying into the other two. In order to consistently keep track of this, one must introduce a \emph{grading} on complexes.\footnote{Or on the central charge, although as discussed in \cite[\S 8]{aspinwall2002d} this introduces other difficulties.} This physical theory was made precise by Bridgeland \cite{bridgeland2007stability}. Let us now recall the notion of a Bridgeland stability condition (see \cite{macri2017lectures} for a survey of the many interesting aspects of this theory in algebraic geometry). We will consider just the setting where the triangulated category $\calD = \DbCoh(X)$ is the bounded derived category of coherent sheaves on a manifold $X$.
	
	In order to emphasise the links with the Abelian theory appearing in \cref{ch:zcriticalconnections}, we will first introduce the notion of a stability condition on an Abelian category $\calA$, which we will use to define a Bridgeland stability condition.
	
	\begin{definition}[See for example \cite{rudakov1997stability}]\label{def:abelianstabilitycondition}
		Let $\calA$ be an Abelian category. A \emph{stability condition} $Z$ on $\calA$ is an additive group homomorphism $$Z: K_0(\calA) \to \CC$$
		such that
		\begin{itemize}
			\item For every non-zero $E\in \calA$, we have $\Im Z(E) \ge 0$ and if equality occurs then $\Re Z(E) < 0$.\footnote{Here we can actually ask just that $Z(E) \in e^{i\theta} \cdot \HH$ for every non-zero $E$, where $\theta\in [0,2\pi)$ is some angle depending only on $Z$, which is the same for every $E$. Any such stability function can be rotated into one landing in the upper-half plane without changing the theory.}
			\item Any non-zero $E\in \calA$ admits a \emph{Harder--Narasimhan filtration},
			$$0 = E_0 \subset E_1 \subset \cdots \subset E_{\ell-1} \subset E_\ell = E$$
			of objects $E_i\in \calA$ with strictly decreasing generalised slope such that $E_i' = E_i/E_{i-1}$ is semistable with respect to $Z$, in the following sense.
		\end{itemize}
		Define the generalised slope of $E$ by
		$$\mu_Z(E) := - \frac{\Re Z(E)}{\Im Z(E)}$$
		when $\Im Z(E) >0$ and $+\infty$ when $\Im Z(E) = 0$. Then an object $E\in \calA$ is \emph{(semi)stable} if for all proper, non-zero subobjects $0\to F \to E$ we have
		$$\mu_Z(F) < \mu(E) \quad \text{(resp. }\le \text{)}.$$
		Equivalently we can use the argument $\arg Z(F) < \arg Z(E)$ instead of the generalised slope $\mu_Z$, since they lie in the same half-plane (see \cref{lem:equivalentstabilitycondition}).
	\end{definition}

	Using the notion of a stability condition on an Abelian category, we will give a definition of Bridgeland stability which is slightly different to the standard presentation, although completely equivalent.
	
	\begin{definition}[\cite{bridgeland2007stability}]\label{def:bridgelandstabilitycondition}
		A \emph{Bridgeland stability condition} on $\DbCoh(X)$ consists of a choice of surjective group homomorphism onto a finite-rank lattice
		$$v: K_0(X) \onto \Lambda,$$
		and a pair $\sigma = (Z,\calA)$ of
		\begin{itemize}
			\item a \emph{heart} $\calA$ of a bounded $t$-structure on $\DbCoh(X)$, which is a full Abelian subcategory of $\DbCoh(X)$, and
			\item a stability condition $Z$ on $\calA$ such that $$Z: K_0(X) = K_0(\calA) = K_0(\DbCoh(X)) \to \CC$$ factors through $v$, and hence only depends on the image of $[E]\in K_0(X)$ in $\Lambda$. 
		\end{itemize}
		These must satisfy the support property (originally due to Kontsevich--Soibelman \cite{kontsevich2008stability}, but equivalent to the \emph{full locally-finite} condition of Bridgeland):
		\begin{itemize}
			\item Let $\|\cdot\|$ be any norm on $\Lambda_\RR$ (all being equivalent), then
			$$\inf \left\{ \left.\frac{|Z(v(E))|}{\|v(E)\|} \,\right| \,0\ne E \in \calA \text{ is semistable} \right\} > 0.$$
		\end{itemize}
	\end{definition}

	\begin{remark}\label{rmk:slicing}
		This differs from the standard approach to defining a stability condition in terms of a \emph{slicing} of $\DbCoh(X)$, which is more aligned with the notion of grading mentioned above for elements of the derived category. A slicing $\calP$ of $\DbCoh(X)$ is essentially an assignment of a lift $\phi$ of the phase angle $\varphi$ for any $E\in \DbCoh(X)$, such that if 
		$$Z(E) = r e^{i\pi \varphi}$$
		then $\varphi = \phi \mod \ZZ$. The full subcategories $\calP(\phi)\subset \DbCoh(X)$ of semistable objects of phase $\phi$ must satisfy
		\begin{itemize}
			\item $\calP(\phi)[1] = \calP(\phi+1)$,
			\item if $\phi_1 > \phi_2$ then $\Hom(A,B)=0$ for any $A\in \calP(\phi_1), B\in \calP(\phi_2)$, and
			\item any object $E\in \DbCoh(X)$ admits a \emph{Harder--Narasimhan filtration} whose successive quotients have decreasing phase $\phi_i$ and live in $\calP(\phi_i)$. 
		\end{itemize}
		Given such a slicing, the subcategory $\calA := \calP((0,1])$ defines a heart of a bounded $t$-structure. Notice that the second condition above demonstrates how this grading encodes possible decays of objects $E$ to $F,Q$ appearing in exact triangles $F \to E \to Q \to F[1]$ by requiring that $\phi(F) < \phi(E)$. 
	\end{remark}

	The standard example of a heart in $\DbCoh(X)$ is given by $\calA = \Coh(X)$. A choice of finite rank lattice and homomorphism $v$ is typically given by the Chern character
	$$v=\Ch: K_0(X) \to H^*(X,\QQ)$$
	of a complex $E\in \DbCoh(X)$. An ideal choice of stability condition (although one that would demonstrate that Bridgeland stability is no more complicated than the slope stability of \cref{sec:stability}) is
	$$Z(E) := -\deg E + i \rk E$$
	on $\Coh(X)$. Stability with respect to $Z$ reproduces slope stability (since the generalised slope $\mu_Z$ is simply the slope). On a compact Riemann surface $X=\Sigma$ the pair $(Z,\calA)$ so described gives a Bridgeland stability condition on $\DbCoh(\Sigma)$. However in general we have the following important observation:
	\begin{proposition}[See for example {\cite[Ex. 4.3]{macri2017lectures}}]
		There exists no stability condition on $\DbCoh(X)$ with heart $\calA=\Coh(X)$ or central charge $Z=-\deg + i \rk$ when $X$ has dimension at least two.
	\end{proposition} 
	This observation reveals the fundamental difficulty in the study of Bridgeland stability conditions: their existence. In order to find adequate hearts, one considers the process of \emph{tilting} of $t$-structures, and candidate central charges $Z$ arise primarily out of physical input. The difficulty in understanding when a heart $\calA$ is a candidate for some Bridgeland stability condition comes in proving the existence of Harder--Narasimhan filtrations in the sense of \cref{def:abelianstabilitycondition}. To do so, one often requires strong Chern number inequalities known as \emph{generalied Bogomolov--Gieseker inequalities}.
	
	The central charges of primary interest arising out of string theory (see \cite[\S 2.3]{aspinwall2004d}) are the following:
	\begin{equation}\label{eq:centralcharges}Z_{\mathrm{dHYM}}(E) = \int_X e^{-i\omega} \Ch^B(E),\quad Z_{\Td}(E) = \int_X e^{-i\omega} \Ch^B(E) \sqrt{\Td(X)}\end{equation}
	where $\Ch^B(E) = e^{-B} \Ch(E)$ for some class $B\in H^{1,1}(X,\RR)$ and $[\omega]$ is some K\"ahler class on $X$. We will return to these central charges in more depth in \cref{ch:zcriticalconnections}.
	
	\begin{remark}
		As has been computed using physical techniques (see \cite[\S 5.1.4]{aspinwall2009dirichlet} and the references therein), the expected central charge of a D-brane $(L,E)$ where $\iota: L\into X$ is some submanifold is given by the formula
		$$Z(L,E) = -\int_X e^{-i\omega} \Ch^B(\iota_*E) \sqrt{\hat A(X)}$$
		where $\hat A(X)$ is the A-hat genus, which is related to the Todd class by $\Td = \exp(\frac{1}{2} c_1) \hat A$. This suggests that in general for non-Calabi--Yau varieties one should also consider central charges of this form (which will be genuinely different to $Z_{\Td}$ above since $c_1(X)$ will not necessarily vanish). In particular when computed \emph{on} the submanifold $L$ we expect a correction term relating to $c_1(L)$ to appear. This has been predicted physically by the so-called ``Freed--Kapustin--Witten anomaly" and seems to at least partly resolve the discrepancy between $Z_V(L)$ and $Z(L\otimes \calO_V)$ which occurs in the study of the deformed Hermitian Yang--Mills equation. We will discuss this in more detail in \cref{sec:counterexample}.
	\end{remark}
	
	In the case of complex surfaces stability conditions with the above central charges have been constructed (see for example \cite{bridgeland2008stability} for stability conditions with central charge $Z_{\Td}$ constructed on K3 surfaces). 
	
	For strict Calabi--Yau threefolds until recently not a single example of a stability condition with one of the above central charges was known.\footnote{Due to techniques developed in work of Bayer--Macr\'i--Toda \cite{bayer2014bridgeland} existence of stability conditions on threefolds such as $\PP^3$, Abelian threefolds \cite{maciocia2015I,maciocia2015II} and crepant resolutions of finite quotients thereof \cite{bayer2016space} are known.} However we now have
	\begin{theorem}[\cite{li2019stability}]
		There is a family of stability conditions $(Z,\calA)$ on each smooth quintic threefold $X$ with central charge given by the first central charge in \eqref{eq:centralcharges} for varying choice of class $B,[\omega]$.		
	\end{theorem}
	This was proven using earlier work of Bayer--Macr\'i--Toda and Bayer--Macr\'i--Stellari \cite{bayer2014bridgeland,bayer2016space} which shows that, if a generalised Bogmolov--Gieseker-type inequality for sheaves holds on a given threefold $X$, then tiltings of the heart $\Coh(X)$ admit stability conditions with the above central charges. 

	We remark that there is a rich body of work in studying the \emph{space} of stability conditions, which was already identified by Bridgeland to be a complex manifold \cite{bridgeland2007stability}. We are primarily concerned with the study of stable objects for a given stability condition in the following, so we refer the reader to \cite{bridgeland2009spaces} which in addition contains an interesting discussion of how the space of space of stability conditions admits an interpretation in terms of the moduli space of string backgrounds in superstring theory.

	\subsection{Polynomial stability conditions\label{sec:polynomialbridgelandstability}}
	
	Let us briefly discuss a limiting version of the above theory introduced by Bayer \cite{bayer}, which is closely aligned with the study of $Z$-critical metrics in \cref{ch:zcriticalconnections}. The main alteration to the definition of Bridgeland stability in \cref{def:bridgelandstabilitycondition} is that the central charge
	$$Z:K_0(X) \to \CC$$
	is now polynomial-valued,
	$$Z_k : K_0(X) \to \CC[k].$$
	This produces a phase $\varphi_k = \arg Z_k$ or generalised slope $\mu_{Z_k}$ on a heart $\calA$ which depends on $k$. The polynomial central charge $Z$ is required to be compatible with a \emph{polynomial phase function} $\phi=\phi_k$ on $\calA$, corresponding to a choice of slicing on $\calA$ in the sense of \cref{rmk:slicing} which is compatible with the limiting behaviour $k\to \infty$ in the sense that
	$$Z_k(E) \in \RR_{>0} \cdot e^{i\pi \phi_k(E)}$$
	for any non-zero $E$ in $\calA$ and $k\gg 0$. Such stability conditions are closely related to the notion of polynomial stability on an Abelian category introduced by Rudakov \cite[\S 2]{rudakov1997stability} as a generalisation of Gieseker stability.
	
	The asymptotic nature of polynomial Bridgeland stability enables the existence of such stability conditions on any projective manifold to be proven more easily than genuine Bridgeland stability conditions. Namely by considering hearts of $t$-structures defined by perverse sheaves, the leading order terms in the polynomial central charge can be compared in the large volume limit $k\to \infty$ allowing strong Artinian and Noetherian properties to be proven, from which the existence of Harder--Narasimhan filtrations follows. We will meet the central charges of most interest to Bayer in \cref{ch:zcriticalconnections}.

	\section{Deformed Hermitian Yang--Mills equation\label{sec:dhym}}
	
	In this section we will briefly recall the existing literature on the \emph{deformed Hermitian Yang--Mills equation}. This equation, which characterises BPS D-branes in Type IIB string theory, is derived through SYZ mirror symmetry in the semi-flat limit from the special Lagrangian equation \cite{LYZ}, or by a direct analysis of the supersymmetry and critical points of the Dirac--Born--Infield action as in \cite{marino2000nonlinear}. 
	
	Let $(X,\omega)$ be a compact K\"ahler manifold and $[\alpha] \in H^{1,1}(X,\RR)$ a real class of type $(1,1)$. The dHYM equation asks for a canonical representative $\alpha \in [\alpha]$ satisfying the equation
	\begin{equation}\label{eq:dhym} \Im(e^{-i\hat \theta} (\omega + i \alpha)^n) = 0\end{equation}
	where $\hat \theta = \arg \hat z$ and 
	$$\hat z = \int_X (\omega + i \alpha)^n.$$
	One typically subjects $\alpha$ to the positivity condition
	\begin{equation}\label{eq:dhymalmostcalibrated} \Re(e^{-i\hat \theta} (\omega + i \alpha)^n) > 0\end{equation}
	in the sense of positivity of $(n,n)$-forms. This is called the \emph{almost calibrated} condition. A typical example is the case where $[\alpha] = c_1(L)$, and then given a Hermitian metric $h$ on $L$ producing a Chern curvature $F(h)$, any other form in the class $c_1(L)$ can be identified with $\frac{i}{2\pi} F(e^{\varphi} h)$ for a smooth function $\varphi: X\to \RR$. 
	
	The analytical study of \eqref{eq:dhym} was initiated by Jacob--Yau \cite{jacob2017special} and continued in a series of works \cite{collins20151,collins2018deformed,collins2018moment,collins2020stability}. Indeed already it was observed by Jacob--Yau that \eqref{eq:dhym} is a fully non-linear elliptic partial differential equation for $\alpha$, and an intimate relationship to some algebro-geometric stability condition was identified. One key observation is that if $\lambda_i$ denote the eigenvalues of $\alpha$ with respect to $\omega$, then \eqref{eq:dhym} can be rewritten as
	\begin{equation}\label{eq:dhymlocal} \Theta := \sum_{i=1}^n \arctan(\lambda_i) = \hat \theta \mod 2\pi \end{equation}
	where we take the branch $\arctan: \RR\to (-\pi/2,\pi/2)$. Therefore a solution of the dHYM equation produces a canonical lift of $\hat \theta$ to a real-valued phase for $\hat z$ (see \cite[\S 2]{collins2018deformed} for more details). The behaviour of the dHYM equation is highly sensitive to what region $\Theta$ lies inside. The maximum possible value of $\Theta$ is given by $\frac{n\pi}{2}$. Values near that maximum are known as the critical phase range.
	
	\begin{definition}
		The \emph{hypercritical phase range} is given by $$\Theta \in \left(\frac{(n-1)\pi}{2}, \frac{n\pi}{2}\right)$$
		and the \emph{supercritical phase range} is given by
		 $$\Theta \in \left(\frac{(n-2)\pi}{2}, \frac{n\pi}{2}\right)$$
	\end{definition}
	
	The central conjecture proposed by Collins--Jacob--Yau relating to the dHYM equation, given as a rudimentary formulation of \cref{conj:folklore} in this setting, is the following.
	
	\begin{conjecture}[{\cite[Conj. 1.4]{collins20151}}]\label{conj:collinsjacobyaudHYM}
		Let $(X,\omega)$ be a compact K\"ahler manifold and $[\alpha]\in H^{1,1}(X,\RR)$ a real $(1,1)$-class. Let
		$$\Theta_V := \arg \int_V (\alpha + i \omega)^{\dim V}.$$
		There exists a solution to the dHYM equation \eqref{eq:dhymlocal} if and only if for every proper irreducible analytic subvariety $V\subset X$ we have
		$$\Theta_V > \Theta_X - (n-\dim V)\frac{\pi}{2}.$$
	\end{conjecture}

	\begin{remark}
		After rewriting the dHYM equation in terms of the central charge $Z_{\mathrm{dHYM}}$ as in \cref{ex:dhymcentralcharge}, the above stability condition is precisely the $Z$-stability condition \cref{def:Zstabilitysubvariety}, as was already noted for example in \cite{collins2018moment}.
	\end{remark}

	In dimension two \cref{conj:collinsjacobyaudHYM} was resolved already by Jacob--Yau \cite{jacob2017special} by appealing to the Demailly--Pǎun theorem and Yau's proof of the Calabi conjecture. We will prove an existence result for $Z$-critical metrics on line bundles over surfaces following the same idea in \cref{thm:existencesurfaces}. 
	
	In general \cref{conj:collinsjacobyaudHYM} is now well-understood in the supercritical phase range. Let us call \eqref{eq:dhym} with $\Theta$ in the supercritical phase range the ``supercritical deformed Hermitian Yang--Mills equation." 
	
	\begin{theorem}[Chen {\cite{chen2021j}}, Datar--Pingali \cite{datar2021numerical}, Chu--Lee--Takahashi \cite{chu2021nakai}]\label{thm:dhymexistencesupercritical}
		Let $(X,\omega)$ be a projective K\"ahler manifold where $\omega \in c_1(L)$ for some ample line bundle $L$. There exists a solution in the class $[\alpha]$ to the supercritical dHYM equation if and only if $[\alpha]$ is stable in the sense of Collins--Jacob--Yau.
	\end{theorem}

	This was shown using a uniform version of the stability condition by Chen, and extended to the non-uniform case by Datar--Pingali and Chu--Lee--Takahashi. These works also prove a form of \cref{thm:dhymexistencesupercritical}  for arbitrary compact K\"ahler manifolds under slightly more analytical stability hypotheses, which can be strengthened to the purely algebro-geometric criteria conjectured by Collins--Jacob--Yau in the projective case.
	
	\subsection{Higher rank}
	
	There is no derivation in general of the dHYM equation for higher rank vector bundles, either from mirror symmetry or directly from the analysis of the Dirac--Born--Infield action for non-Abelian gauge theory.\footnote{In \cite{marino2000nonlinear} the Abelian case is derived from the Dirac--Born--Infield action directly, as opposed to \cite{LYZ}, but the higher rank version of this action is not fully understood with various proposals existing in the literature. It would be interesting to choose the most popular (using a symmetrised trace \cite{tseytlin1997non}) and attempt to derive the higher rank dHYM following similar arguments to the Abelian case.} The higher rank equation should appear as the mirror of the special Lagrangian equation for a Lagrangian multisection of a semiflat SYZ fibration \cite{LYZ}. However, the na\"ive derivation using the techniques of Leung--Yau--Zaslow produces a $\U(1)^n$-dHYM equation as opposed to a $\U(n)$-dHYM equation, because it fails to take into account monodromy around ramification points of the multisection. 
	
	Nevertheless Collins--Yau proposed on aesthetic grounds a higher rank analogue of \eqref{eq:dhym} for a Hermitian metric $h$ on a vector bundle $E\to(X,\omega)$ \cite[\S 8.1]{collins2018moment}.
	\begin{align}
		\Im\left(e^{-i\hat \theta} \left(\omega \otimes \id_E - \frac{F(h)}{2\pi}\right)^n \right) &=0\\
		\trace \left( \Re\left(e^{-i\hat \theta} \left(\omega \otimes \id_E - \frac{F(h)}{2\pi}\right)^n \right) \right) &>0
	\end{align}
	where $\hat \theta = \arg \int_X \trace (\omega \otimes \id_E - \frac{i}{2\pi} F(h))^n$ and $\Re$ and $\Im$ refer to the $h$-Hermitian and skew-Hermitian parts of the endomorphism respectively.
	
	This equation is expected to be one of the first instances in non-Abelian gauge theory of an equation with algebro-geometric obstructions to solutions arising both from subsheaves (as in the case of slope stable bundles discussed in \cref{sec:bundles}), \emph{and} subvarieties (and indeed sheaves supported on subvarieties). 
	
	A further derivation in the case of \emph{complexes} of vector bundles so as to make contact with the upgraded notion of D-brane in Type IIB string theory is also lacking. We will make some remarks about this problem in \cref{sec:Zcriticalfuturedirections}. In general we do not necessarily expect a \emph{differential equation} in the traditional sense (an observation which has also been discussed by Li \cite{li2022thomas}), but generally one expects an equation of the above tensorial form including some singular or current-like objects related to the geometry of the underlying complex.
	
\chapter{$Z$-critical connections}
\label{ch:zcriticalconnections}

In this chapter we introduce the notion of a $Z$-critical connection (and a $Z$-critical metric) and investigate their basic properties. $Z$-critical metrics are generalisations of dHYM metrics which allow one to vary the choice of stability input, and in light of the folklore \cref{conj:folklore} one expects according to \cref{principle} a correspondence with a stability condition for \emph{any} choice of $Z$-critical equation. 

The discussion of this chapter is geared towards understanding the correspondence in \cref{ch:correspondence} and so remains well within the world of metrics on holomorphic vector bundles. Any subtleties of the derived category are suppressed here, but we will remark on various shadows hinting at a deeper theory as they arise. We will discuss the possibility of a theory on complexes in \cref{sec:Zcriticalfuturedirections}. 

The majority of the content of this chapter with the exception of \cref{sec:stabilitysubvarieties,sec:variationalfunctional,sec:Zcriticalfuturedirections} is joint work with Ruadha\'i Dervan and Lars Sektnan \cite{dervan2021zcritical}. This work has been reproduced here with an emphasis on the contributions of the author.

\section{Asymptotic $Z$-stability\label{sec:AZS}}

In this section we define the key notion of stability which will be relevant for the correspondence in \cref{ch:correspondence}. In particular this is a limiting form of the stability conditions one sees arise in the derived category, and we also fix a simple geometric background.

The primary input of this notion of stability is a \emph{central charge}, an assignment of an algebraic invariant to each sheaf over a compact K\"ahler manifold. We work with a particular explicit class of such central charges identified by Bayer \cite{bayer}, known as polynomial central charges, which we have briefly discussed in \cref{sec:polynomialbridgelandstability}. This class is large enough to include the deformed Hermitian Yang--Mills equation and interesting alternatives to or perturbations of that equation. 

First we need to define the notion of a stability vector.\footnote{In Bayer's work one also introduces the the \emph{perversity function}, which is important in identifying a particular t-structure consisting of perverse sheaves for his notion of polynomial stability condition. We will not need that technology here.}

\begin{definition}[Stability vector]
	A \emph{stability vector} $\rho \in (\CC^*)^{n+1}$ consists of a collection $\rho=(\rho_0,\dots,\rho_n)$ of non-zero complex numbers such that $\Im(\rho_n)>0$ and 
	$$\Im\left(\frac{\rho_d}{\rho_d+1}\right) > 0$$
	for $d=0,\dots,n-1$.
\end{definition}

\begin{figure}[h]
	\centering

	\tikzset{every picture/.style={line width=0.75pt}} 
	
	\begin{tikzpicture}[x=0.75pt,y=0.75pt,yscale=-1,xscale=1]
		
		\draw    (191.67,244.33) -- (191.67,80) ;
		\draw [shift={(191.67,77)}, rotate = 90] [fill={rgb, 255:red, 0; green, 0; blue, 0 }  ][line width=0.08]  [draw opacity=0] (8.93,-4.29) -- (0,0) -- (8.93,4.29) -- cycle    ;
		\draw    (102.5,169.67) -- (277.83,169.67) ;
		\draw [shift={(280.83,169.67)}, rotate = 180] [fill={rgb, 255:red, 0; green, 0; blue, 0 }  ][line width=0.08]  [draw opacity=0] (8.93,-4.29) -- (0,0) -- (8.93,4.29) -- cycle    ;
		\draw  [dash pattern={on 0.84pt off 2.51pt}]  (191.67,169.67) -- (250.67,143.67) ;
		\draw [shift={(250.67,143.67)}, rotate = 336.22] [color={rgb, 255:red, 0; green, 0; blue, 0 }  ][fill={rgb, 255:red, 0; green, 0; blue, 0 }  ][line width=0.75]      (0, 0) circle [x radius= 3.35, y radius= 3.35]   ;
		\draw  [dash pattern={on 0.84pt off 2.51pt}]  (191.67,169.67) -- (232.67,103.67) ;
		\draw [shift={(232.67,103.67)}, rotate = 301.85] [color={rgb, 255:red, 0; green, 0; blue, 0 }  ][fill={rgb, 255:red, 0; green, 0; blue, 0 }  ][line width=0.75]      (0, 0) circle [x radius= 3.35, y radius= 3.35]   ;
		\draw  [dash pattern={on 0.84pt off 2.51pt}]  (191.67,169.67) -- (176.67,120.33) ;
		\draw [shift={(176.67,120.33)}, rotate = 253.09] [color={rgb, 255:red, 0; green, 0; blue, 0 }  ][fill={rgb, 255:red, 0; green, 0; blue, 0 }  ][line width=0.75]      (0, 0) circle [x radius= 3.35, y radius= 3.35]   ;
		\draw  [dash pattern={on 0.84pt off 2.51pt}]  (191.67,169.67) -- (290.67,265) ;
		\draw [shift={(290.67,265)}, rotate = 43.92] [color={rgb, 255:red, 0; green, 0; blue, 0 }  ][fill={rgb, 255:red, 0; green, 0; blue, 0 }  ][line width=0.75]      (0, 0) circle [x radius= 3.35, y radius= 3.35]   ;
		\draw  [dash pattern={on 0.84pt off 2.51pt}]  (191.67,169.67) -- (171.33,223) ;
		\draw [shift={(171.33,223)}, rotate = 110.87] [color={rgb, 255:red, 0; green, 0; blue, 0 }  ][fill={rgb, 255:red, 0; green, 0; blue, 0 }  ][line width=0.75]      (0, 0) circle [x radius= 3.35, y radius= 3.35]   ;
		
		\draw (261.33,131.07) node [anchor=north west][inner sep=0.75pt]    {$\rho _{0}$};
		\draw (162,233.73) node [anchor=north west][inner sep=0.75pt]    {$\rho _{1}$};
		\draw (163.33,92.4) node [anchor=north west][inner sep=0.75pt]    {$\rho _{2}$};
		\draw (237.33,77.73) node [anchor=north west][inner sep=0.75pt]    {$\rho _{3}$};
		\draw (292.67,268.4) node [anchor=north west][inner sep=0.75pt]    {$\rho _{4}$};

	\end{tikzpicture}
	
	\caption{A stability vector consists of a sequence of non-zero complex numbers with angles moving clockwise around the origin.}
\end{figure}
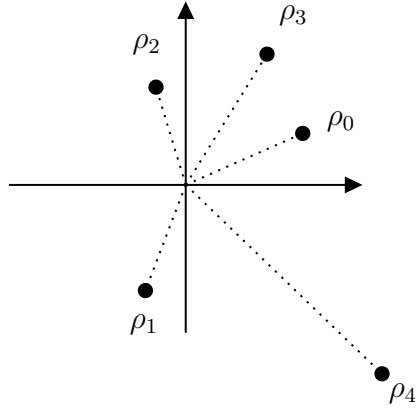

\begin{remark}
	The condition $\Im(\rho_n)>0$ is a normalisation condition which ensures that in the large volume limit the central charge lands in the upper-half plane $\HH$. As discussed in \cref{sec:Bridgelandstability} what is actually important is that the algebraic invariants of sheaves land in the \emph{same} half-plane of $\CC$, and indeed any stability vector may be rotated so that $\Im(\rho_n)>0$ without changing the condition on successive quotients of $\rho_d$ or the corresponding notion of stability.
\end{remark}

Let us now define a polynomial central charge.

\begin{definition}[Polynomial central charge]
	\label{def:polynomialcentralcharge}
	A \emph{polynomial central charge} $$Z: K(X) \to \CC$$ on a compact K\"ahler manifold $X$ consists of the data of:
	\begin{itemize}
		\item A \emph{stability vector} $\rho=(\rho_0,\dots,\rho_n)$ where $n=\dim X$.
		\item A K\"ahler class $[\omega]$.
		\item A real unipotent class $U\in H^*(X,\RR)$ such that $U=1+N$ where $N\in H^{>0}(X,\RR)$.
	\end{itemize}
	The \emph{density of the central charge} $[Z]: K(X) \to H^*(X,\RR)$ is the inhomogeneous class
	$$[Z(E)] := \left( \sum_{d=0}^n \rho_d  [\omega]^d \right) . \Ch(E) . U \in H^*(X,\RR).$$
	The central charge $Z: K(X) \to \CC$ is defined by
	$$Z(E) := [Z(E)] ([X]) = \int_X \sum_{d=0}^n \rho_d  [\omega]^d . \Ch(E) . U.$$
\end{definition}

The term \emph{polynomial} arises for the following reason: If the K\"ahler class $[\omega]$ is replaced by $k[\omega]$ for an indeterminant $k>0$, then one obtains a polynomial map $Z_k : K(X) \to \CC[k]$ and a polynomial central charge in the sense of \cref{sec:polynomialbridgelandstability}. The limit $k\to \infty$ is the large volume limit, and in the following we will denote by $Z_k$ a central charge of the above form with K\"ahler class replaced by $k[\omega]$. 

To define stability, we use $Z$ to produce a suitable notion of \emph{slope} in the sense of the traditional slope stability. Alternatively we can use the \emph{phase}, as is more traditionally used in the Bridgeland stability literature.

\begin{definition}[Slope and phase]
	Let $Z$ be a polynomial central charge. Suppose $E$ is a bundle with $Z(E)\ne 0$. Define the \emph{$Z$-slope} as
	$$\mu_{Z}(E) := - \frac{\Re(Z(E))}{\Im(Z(E))}$$
	and $\mu_Z(E) = +\infty \quad (\text{resp. }-\infty)$ if $\Im Z(E) = 0$ and $\Re(Z(E)) < 0 \quad (\text{resp. }>0)$. Define the \emph{$Z$-phase} of $E$ as
	$$\varphi(E) := \arg Z(E)$$
	where we use the principal branch $\arg: \CC^* \to (-\pi,\pi]$.	If $Z(E) = 0$ define $\mu_Z(E) = +\infty$ and $\varphi(E) = \pi$. 
\end{definition}

We now define our key notion of stability for the correspondence in \cref{ch:correspondence}.

\begin{definition}[Asymptotic $Z$-stability]
	A holomorphic vector bundle $E\to X$ over a compact K\"ahler manifold with $Z_k(E)\ne 0$ for $k\gg 0$ is said to be \emph{asymptotically $Z$-stable} with respect to a polynomial central charge $Z$ if, for every proper, non-zero coherent subsheaf $\calF \subset E$, one has
	$$\varphi_k(\calF) < \varphi_k(E)$$
	for all $k\gg 0$.
\end{definition}

We also explicitly define the alternative notion of $Z$-stability away from the large volume limit.

\begin{definition}[$Z$-stability]\label{def:Zstability}
	A bundle $E\to X$ with $Z(E) \ne 0$ is called $Z$-stable if for every proper, non-zero coherent subsheaf, one has
	$$\varphi(\calF) < \varphi(E).$$
\end{definition}

Let us now make a series of remarks about the definiton of $Z$-stability.

\begin{remark}
	Notice that asymptotic $Z$-stability is not the same as $Z_k$-stability for all $k\gg 0$, as the choice of $k$ may depend on the particular coherent subsheaf $\calF$ in asymptotic stability. One expects that $Z$-stability should be the relevant stability condition away from the large volume limit, but here one will certainly need to take into account the full derived category for most purposes. 
\end{remark}

\begin{remark}
	We could, similarly to the settings of slope stability and Gieseker stability in \cref{sec:stability}, define (asymptotic) $Z$-(semi/poly)stability, the existence of Jordan--H\"older and Harder--Narasimhan-type filtrations with respect to the $Z$-slope and so on. These constructions will not be relevant to our work, and indeed one expects them to be poorly behaved away from the large volume limit, requiring the technology of the derived category.
\end{remark}

\begin{remark}\label{rmk:ze0}
	The assumption $Z(E)\ne 0$ should not be serious, as in the setting of Bridgeland stability this must imply $E$ is numerically trivial. In our more simplified setting, it has already been observed by Collins--Jacob--Yau that if a line bundle has $Z(L)=0$ then it cannot admit any solution to the deformed Hermitian--Yang--Mills equation, so one might include the case $Z(E)=0$ in the definition of a $Z$-unstable bundle, for the purposes of proving a correspondence with existence of solutions to the $Z$-critical equation.
	
	On the other hand by definition if $E$ admits a subsheaf with $Z(\calF)=0$ then $E$ is $Z$-destabilised by $\calF$ and in particular an (asymptotically) $Z$-stable never contains coherent subsheaves with $Z(\calF) = 0$. This is motivated by the analogous assumption for slope stability (that a torsion sheaf supported in codimension $\ge 2$ has slope $+\infty$) and the definition of a weak stability condition \cite{todaweakstability} and weak polynomial stability condition \cite{lo2020weight}. 
\end{remark}

\begin{remark}
	As explained in \cite{bayer} it is possible to embed the theory of Gieseker stability in the language of polynomial stability conditions. Whilst for a small class of such central charges it follows from, for example \cite{rudakov1997stability}, that Gieseker stability coincides with asymptotic $Z$-stability, in general we allow for enough freedom in the choice of $Z$ to explore genuinely different stability conditions.
\end{remark}

The inequality of phases given in the definition of asymptotic stability is cumbersome, and in the setting where all phases of bundles and subbundles lie in the same half-plane, it can be replaced by a simpler condition.

\begin{lemma}\label{lem:equivalentstabilitycondition}
	Let $E\to X$ be a holomorphic vector bundle and $\calF\subset E$ a proper, non-trivial coherent subsheaf. Then consider the following conditions:
	\begin{enumerate}
		\item $\varphi_k(\calF) < \varphi_k(E)\text{ for all $k\gg 0$}$
		\item $\Im\left( \frac{Z_k(\calF)}{Z_k(E)} \right) < 0\text{ for all $k\gg 0$}$
		\item $\mu_{Z_k}(\calF) < \mu_{Z_k}(E)\text{ for all $k\gg 0$}$
	\end{enumerate}
	If $Z_k(\calF) \ne 0$ for all $k\gg 0$ then the above conditions are equivalent. If $Z_k(\calF) = 0$ then conditions (i) and (iii) are equivalent.
\end{lemma}
\begin{proof}
	The normalisation condition on the stability vector $\rho$ of the polynomial central charge $Z$ and a computation of the leading order coefficient using Riemann--Roch ensures that the leading order term in $k$ has positive imaginary part, $\Im(\rho_n) > 0$. Working in the large volume limit we thus have for $k$ sufficiently large, 
	$$Z_k(E), Z_k(\calF) \in \HH.$$
	For two non-zero complex numbers $z=re^{i\theta},z'=r'e^{i\theta'}$, working with the branch $\arg: \CC^* \to (-\pi,\pi]$ we have
	$$\Im\left(\frac{z}{z'}\right) = \frac{r}{r'} \sin(\theta-\theta').$$
	If $|\theta-\theta'| <\pi $ so that $z$ and $z'$ lie in the same half-plane, then the sign of $\sin x$ is the same as that of $x$, so
	$$\Im\left(\frac{z}{z'}\right) < 0$$
	if and only if $\theta < \theta'$. Thus (i) is equivalent to (ii) except when $Z_k(\calF) = 0$ for all $k\gg 0$.
	
	That (ii) is equivalent to (iii) follows immediately from the definition of the $Z$-slope.
\end{proof}

The equivalent stability condition appearing in \cref{lem:equivalentstabilitycondition} (ii) given in terms of the $Z$-slopes is the one which will naturally appear from the analysis of the $Z$-critical equation in the proof of the correspondence in \cref{ch:correspondence}. It also naturally appears in the context of subsolutions and stability with respect to subvarieties.

To that end, let us also explicitly define $Z$-stability with respect to subvarieties. Let $\iota: V\subset X$ be a subvariety and $E\to X$ a vector bundle. Define the central charge of $E$ over $V$ by 
\begin{equation}\label{eq:Zslopesubvariety}Z_V(E) := \iota^* [Z(E)]([V]) = \int_V \left(\sum_{d=0}^n \rho_d[\omega]^d\right) . \Ch(E) . U.\end{equation}

\begin{remark}
	This defintion of $Z_V(E)$ differs from the central charge $Z(E\otimes \calO_V)$ of the coherent sheaf $E\otimes \calO_V$ on $X$, which would be the natural quotient object in Bridgeland stability. We will return to this point in \cref{sec:counterexample}.
\end{remark}

Note that with the above notation, $Z(E) = Z_X(E)$. Let us now define $Z$-stability with respect to subvarieties.

\begin{definition}[$Z$-stability with respect to subvarieties]\label{def:Zstabilitysubvariety}
	A holomorphic vector bundle $E\to (X,\omega)$ is $Z$-stable with respect to an irreducible subvariety $V\subset X$ with $Z_V(E)\ne 0$ if
	$$\Im\left( \frac{Z_V(E)}{Z_X(E)}\right) > 0.$$
\end{definition}

Again we remark that this notion of stability is not well-behaved away from the large volume limit, as we need to guarantee that every $Z_V(E)$ lands in the same half-plane as $Z_X(E)$ (see \cref{rmk:jacobsheu} where it is observed one sometimes needs to ask for the above inequality to be flipped depending on the codimension of $V\subset X$). Indeed this type of stability will not actually be relevant for us near the large volume limit, and we will discuss this more in \cref{sec:subsolutions} when we study subsolutions to the $Z$-critical equation.

Indeed $Z$-stability in the form \cref{def:Zstabilitysubvariety} is equivalent to the stability condition of Collins--Jacob--Yau in \cref{conj:collinsjacobyaudHYM} for $Z=Z_{\mathrm{dHYM}}$ on a line bundle, where by \cref{thm:dhymexistencesupercritical} it is a necessary and sufficient condition in the supercritical phase of the dHYM equation. Under this phase assumption it is automatic that the charges $Z_V(E)$ for all $V\subset X$ lie in the same half-plane.\footnote{This seems to suggest that there is a strong link between the assumptions of phase ranges for the dHYM equation and hearts of $t$-structures.}

Let us now discuss the basic properties of asymptotic $Z$-stability. Firstly, we have the following relation to slope stability.

\begin{lemma}\label{lemma:slopesemistable}
	If a coherent sheaf $E$ is asymptotically $Z$-stable, then it is slope semistable. 
	
\end{lemma}
\begin{proof}
	By \cref{lem:equivalentstabilitycondition}, if $\calF$ is any proper, non-zero coherent subsheaf (which must therefore satisfy $Z_k(\calF) \ne 0$ for $k\gg 0$ by the definition of asymptotic $Z$-stability), we have
	\begin{equation}\label{eq:imsuchandsuch}\Im\left(\frac{Z_k(\calF)}{Z_k(E)}\right) < 0.\end{equation}
	Consider the expansion
	\begin{equation}Z_k(E) = \rho_n k^n [\omega]^n \rk(E) + \rho_{n-1} k^{n-1} [\omega]^{n-1}.(c_1(E) + \rk(E) U_2) + O(k^{n-2})\label{eq:Zkexpansion}\end{equation}
	where $U_2$ denotes the degree two component of the class $U\in H^*(X, \RR).$ For simplicity let us write 
	\begin{equation}\label{degU}\deg_U (E) := [\omega]^{n-1}.(c_1(E) + \rk(E) U_2).\end{equation}
	Then dividing we see
	$$\frac{Z_k(\calF)}{Z_k(E)} = \frac{\rk(\calF)}{\rk(E)} + k^{-1} \frac{\rho_{n-1}}{\rho_n} \left( \frac{\deg_U \calF \rk E - \deg_U E \rk \calF}{([\omega]^n \rk E)^2}\right) + O(k^{-2}).$$
	Since $\rho$ is a stability vector we have
	$$\Im \left(\frac{\rho_{n-1}}{\rho_n}\right) > 0.$$
	As \eqref{eq:imsuchandsuch} holds for all $k$, the leading order term of its expansion in $k$ is non-positive. Thus by taking imaginary parts, this inequality produces
	$$\deg_U \calF \rk E - \deg_U E \rk \calF \leq 0$$
	which is equivalent to asking
	$$\frac{\deg_U \calF}{\rk \calF} \leq \frac{\deg_U E}{\rk E}.$$
	Since $\deg_U (E) = \deg E + \rk E [\omega]^{n-1}.U_2$, this can be further simplified to the usual slope inequality
	$$\frac{\deg \calF}{\rk \calF} \leq \frac{\deg E}{\rk E}.$$ \end{proof}

\begin{remark}\label{Ucondition}The condition that $U$ (or rather its degree two component) is a \emph{real} operator is used in the proof above; in general, for $U$ complex, one needs to impose a further topological hypothesis on the imaginary part of its degree two component, which seems slightly unnatural. For this reason we included $U$ being real in the definition of a central charge. This includes the examples of most interest which we discuss in the next section.\end{remark}

The following shows that the considerations of \cref{rmk:ze0} are not important in the large volume regime.

\begin{corollary}\label{non-trivialphase} For an arbitrary torsion-free coherent sheaf $\calE$ and polynomial central charge $Z_k$, the value $Z_k(\calE) \neq 0$ does not vanish for $k \gg0$.
\end{corollary}

Let us also note explicitly the following, which is the direct analogy of \cref{lem:gieskerimplications} in the case of asymptotic $Z$-stability.

\begin{corollary}\label{cor:stabilityimpliesazs}
	If $E\to X$ is slope stable, then it is asymptotically $Z$-stable for any $Z$. In particular there are implications
	
	\begin{center} Slope stable $\implies$ asymptotic $Z$-stability $\implies$ Slope semistable.\end{center}
\end{corollary}

Let us also observe that asymptotic $Z$-stability satisfies the see-saw property just as slope stability does in \cref{lem:seesawslopestability}.

\begin{lemma}\label{lem:seesawazs}
	Consider a short exact sequence of coherent sheaves
	\begin{center}
		\ses{\calS}{\calE}{\calQ}
	\end{center}
	with $\rk \calE > 0$. Then
	$$\mu_{Z_k}(\calS) \le \mu_{Z_k}(\calE) \text{ for all $k\gg 0$}\quad \iff \quad \mu_{Z_k}(\calE) \le \mu_{Z_k}(\calQ)\text{ for all $k\gg 0$}$$
	and 
	$$\mu_{Z_k}(\calS) \ge \mu_{Z_k}(\calE) \text{ for all $k\gg 0$}\quad \iff \quad \mu_{Z_k}(\calE) \ge \mu_{Z_k}(\calQ)\text{ for all $k\gg 0$}.$$
	Furthermore:
	\begin{itemize}
		\item If $Z_k(\calS) \ne 0$ and $Z_k(\calQ) \ne 0$ then
		$$\mu_{Z_k}(\calS) < \mu_{Z_k}(\calE) \text{ for all $k\gg 0$}\quad \iff \quad \mu_{Z_k}(\calE) < \mu_{Z_k}(\calQ)\text{ for all $k\gg 0$}$$
		and 
		$$\mu_{Z_k}(\calS) > \mu_{Z_k}(\calE) \text{ for all $k\gg 0$}\quad \iff \quad \mu_{Z_k}(\calE) > \mu_{Z_k}(\calQ)\text{ for all $k\gg 0$}.$$
		\item If $Z_k(\calS) = 0$ then $\mu_{Z_k}(\calS) > \mu_{Z_k}(\calE)$ and $\mu_{Z_k}(\calE) = \mu_{Z_k}(\calQ).$
		\item If $Z_k(\calQ) = 0$ then $\mu_{Z_k}(\calS) = \mu_{Z_k}(\calE)$ and $\mu_{Z_k}(\calE) < \mu_{Z_k}(\calQ).$
	\end{itemize}
	The same conclusions hold if we replace the $Z$-slope by the phase $\varphi_k$ by \cref{lem:equivalentstabilitycondition}.
\end{lemma}
\begin{proof}
	Since $Z$ is additive in short exact sequences, we have
	$$Z_k(\calE) = Z_k(\calS) + Z_k(\calQ)$$
	for all $k\gg 0$. Therefore $\Re Z$ and $\Im Z$ are also additive in short exact sequences. 
	
	Since $\rk \calE > 0$, we have by the expansion \cref{eq:Zkexpansion} that $\Im Z_{k}(E) > 0$ for all $k\gg 0$. Suppose then that $\mu_{Z_k}(\calS) < \mu_{Z_k}(\calE)$ for all $k\gg 0$. If this occurs either $\Im Z_k(\calS) > 0$ for all $k\gg 0$ or $\Im Z_k(\calS) = 0$ and $\Re Z_k(\calS) > 0$ for all $k\gg 0$, (in the latter case if $\Re Z_k(\calS) < 0$ then $\mu_{Z_k}(\calS) = +\infty$ but $\mu_{Z_k}(\calE)$ is finite). If $\Im Z_k(\calS) > 0$ then 
	\begin{align*}
		0 &> \Im Z_k(\calS) (\mu_{Z_k} (\calS) - \mu_{Z_k}(\calE))\\
		&= -\Re Z_k (\calS) + \frac{\Im Z_k(\calS) \Re Z_k(\calE)}{\Im Z_k(\calE)}\\
		&= \Re Z_k(\calQ) - \Re Z_k(\calE) + \frac{(\Im Z_k(\calE) - \Im Z_k(\calQ)) \Re Z_k(\calE)}{\Im Z_k(\calE)}\\
		&= \Re Z_k(\calQ) - \Im Z_k(\calQ) \mu_{Z_k}(\calE).
	\end{align*}
	If $\Im Z_k(\calQ) > 0$ then we conclude $\mu_{Z_k}(\calE) < \mu_{Z_k}(\calQ)$. If $\Im Z_k(\calQ) = 0$ then we conclude $\Re Z_k(\calQ) < 0$ so $\mu_{Z_k}(\calQ) = +\infty$ and again $\mu_{Z_k}(\calE) < \mu_{Z_k}(\calQ)$. 
	Now consider the case where $\Im Z_k(\calS) = 0$ and $\Re Z_k(\calS) > 0$ for all $k \gg 0$. Then we must have $\Im Z_k(\calQ) = \Im Z_k(\calE) > 0$ for all $k\gg 0$. Therefore 
	$$\Im Z_k(\calQ) (\mu_{Z_k}(\calQ) - \mu_{Z_k}(\calE)) = \Re Z_k(\calS) > 0$$
	from which we conclude $\mu_{Z_k}(\calE) < \mu_{Z_k}(\calQ)$ for all $k\gg 0$. We can repeat the same argument in the reverse direction to conclude $\mu_{Z_k}(\calE) < \mu_{Z_k}(\calQ)$ for all $k\gg 0$ implies $\mu_{Z_k}(\calS) < \mu_{Z_k}(\calE)$ for all $k\gg 0$ also.
	
	On the other hand, repeating the same argument with inequalities reversed we obtain the analogous statement that $\mu_{Z_k}(\calS) > \mu_{Z_k}(\calE)$  for all $k \gg 0$ if and only if $\mu_{Z_k}(\calE) > \mu_{Z_k}(\calQ)$  for all $k \gg 0$.
	
%

	Now suppose $\mu_{Z_k}(\calS) = \mu_{Z_k}(\calE)$ for all $k\gg 0$. Since $\Im Z_k(\calE) > 0$ and $\mu_{Z_k}(\calE) < \infty$ we must have $\Im Z_k(\calS) > 0$. We have
	$$\Im Z_k(\calE) = \Im Z_k(\calS) \frac{\Re Z_k(\calE)}{\Re Z_k(\calS)} > 0.$$
	Now we compute
	$$0 = \Im Z_k(\calS) (\mu_{Z_k}(\calS) - \mu_{Z_k}(\calE)) = -\Re Z_k(\calQ) - \Im Z_k(\calQ) \mu_{Z_k}(\calE).$$
	From this we conclude either $\mu_{Z_k}(\calQ) = \mu_{Z_k}(\calE)$ or $Z_k(\calQ) = 0$ for all $k \gg 0$, in which case $\mu_{Z_k}(\calE) < \mu_{Z_k}(\calQ) = +\infty$. The same argument resolves the case $\mu_{Z_k}(\calE) = \mu_{Z_k}(\calQ)$. 
\end{proof}

In particular just as for slope stability we obtain the simplicity of an asymptotically $Z$-stable sheaf, which will be critical in our proof that stability implies existence of a $Z$-critical metric in \cref{ch:correspondence}.

\begin{lemma}\label{lem:azsimpliessimple}
	An asymptotically $Z$-stable vector bundle $E$ is simple.
\end{lemma}
\begin{proof}
	This only uses the see-saw property, and so follows from the standard argument for slope stable or Gieseker stable sheaves \cite[Cor. 1.2.8]{huybrechts-lehn}. Let $u: E \to E$ be a non-zero sheaf endomorphism and suppose $\ker u \ne 0$. Then by asymptotic $Z$-stability since $\ker u \into E$ we have
	$$\varphi_k(\ker u) < \varphi_k(E)$$
	for all $k\gg 0$. By the see-saw property \cref{lem:seesawazs} this gives
	$$\varphi_k(E) < \varphi_k(E/\ker u) = \varphi_k(\image u)$$
	which contradicts asymptotic $Z$-stability with respect to the subsheaf $\image u \into E$. Thus $\ker u = 0$. Thus $u$ is an injective morphism from $E$ to $E$, and therefore induces an injective morphism $u: \det E \to \det E$, which is therefore an isomorphism (since it is a non-zero map of line bundles of the same rank and $c_1$). Thus $u$ is an isomorphism too.
	
	It is then a standard result that $H^0(X,\End E) = \CC$. Indeed this vector space is a finite-dimensional division algebra over $\CC$, and thus must be isomorphism to $\CC$.
\end{proof}

\subsection{Examples}

Here we give several important examples of central charges we will be interested in. For the purposes of stating the central charges in the same form that has appeared in the literature, we will ignore the normalisation condition $\Im(\rho_n)>0$ on the stability vector for now.

\begin{example}[dHYM central charge]\label{ex:dhymcentralcharge}
	Fix a class $B\in H^{1,1}(X,\RR)$ (the class of a ``B-field"). Let $(X,\omega)$ be a compact K\"ahler manifold of dimension $n$. Define data
	$$\rho_d = -\frac{(-i)^d}{d!},\quad U=e^{-B}.$$
	Then using the given K\"ahler form $\omega$ on $X$ as the K\"ahler class, one obtains a central charge
	$$Z_{\mathrm{dHYM}}(E) = -\int_X e^{-i[\omega]-B} \Ch(E),$$
	which we call the \emph{deformed Hermitian Yang--Mills central charge}. 
\end{example}

\begin{example}[String theory central charge]\label{ex:stringtheorycentralcharge}
	Fix $\rho_d$ as in the above example, but now set $U=e^{-B} \sqrt{\Td(X)}.$ The resulting central charge
	$$Z_{\Td}(E) = -\int_X e^{-i[\omega]-B} \Ch(E) \sqrt{\Td(X)}$$
	is the one most relevant for considerations in string theory. We will also be interested in the variant with $U=e^{-B} \sqrt{\hat A(X)}$, which we denote $Z_{\hat A}(E)$. 
\end{example}

Let us now discuss a detailed example where we can easily understand asymptotic $Z$-stability for the above central charges. In order to be interesting, such an example should consist of a strictly slope-semistable bundle which is not slope polystable, and such an example was considered by Maruyama.

\begin{example}\label{ex:zcritical}
	We consider a simple rank three bundle on $\CCPP^2$ considered by Maruyama in the context of Gieseker stability, and we follow its presentation in \cite[p.96]{okonekschneiderspindler80}. Thus let $F$ be a slope stable vector bundle of rank $2$ on $\CCPP^2$ with $c_1(F) =0$ and $H^1 (\CCPP^2, F ) \neq 0$. Given $\tau \in H^1 (\CCPP^2, F )  \setminus \{ 0 \}$, define an extension 
	$$ 
	0 \to F \to E \to \calO_{\CCPP^2} \to 0
	$$
	using $\tau.$ The Chern characters of $E$ are then given by $$c_1 (E) = 0 \qquad c_2(E) = c_2(F).$$ Note by the Bogomolov inequality that $c_2 (F) \geq 0$, since $F$ is a slope stable bundle with vanishing first Chern class. The bundle $E$ is then not slope stable, since $\mu(F) = \mu(E)$ and hence destabilises $E$, but $E$ is in fact Gieseker stable. It follows that $E$ is simple and slope semistable, with
	$$0 \subset F \subset E	$$ 
	giving a Jordan--H\"older filtration of $E$, meaning $E$ has associated graded object 
	$$\Gr (E) = F \oplus \calO_{\CCPP^2}.$$
	Since this graded object is locally free, the assumptions of \cref{thm:maintheoremZstability} apply. First, consider the dHYM central charge with B-field, given by
	$$
	Z_{\mathrm{dHYM}} (E) = - \int_{\CCPP^2} e^{- i k [\omega]} \ch (E) e^{-B},
	$$
	for some B-field $B\in H^2(\CCPP^2, \RR).$ Then since $c_1(E)=0$, we have $\ch(E) = \rk(E) - c_2(E)$. Let us denote $$\sigma := \int_{\CCPP^2} c_2(E) = \int_{\CCPP^2} c_2(F) \ge 0.$$ 
	Denoting $h=c_1(\calO(1)) = [\omega]$, one can compute
	$$Z_{\mathrm{dHYM}}(E) = \sigma + \frac{\rk(E)}{2} k^2 - ik\rk(E) B.h - \frac{\rk{E}}{2} B^2,$$
	and so the imaginary part of $Z(F)\big/Z(E)$ is a positive constant multiple of 
	\begin{align*}
		&(\sigma + \frac{\rk(E)}{2} (k^2-B^2) ) \cdot (-k\rk(F) B.h) - (\sigma + \frac{\rk(E)}{2} (k^2-B^2) ) \cdot (-k\rk(E) B.h)\\
		&= k \sigma B.h (\rk(E)-\rk(F)).
	\end{align*}
	
	When $\sigma>0$ is positive,  since $\rk(E) - \rk(F) >0$, we have $\varphi_k(F) < \varphi_k(E)$ whenever the B-field is chosen such that $B.h<0$. Thus with an appropriately chosen $B$-field, \cref{thm:maintheoremZstability} implies $E$ admits deformed Hermitian Yang-Mills connections for all $k \gg 0$; note that this requires the $B$-field to be non-trivial.  In the remaining case $B.h\ge 0$ or $\sigma=0$, which includes the case of trivial $B$-field, we have $\varphi_k(F)\ge \varphi_k(E)$, which implies $E$ is not asymptotically $Z_{\mathrm{dHYM}}$-stable, and therefore by \cref{thm:maintheoremZstability} $E$ cannot admit deformed Hermitian Yang--Mills connections for $k \gg 0$. Note that we have only used the fact that $c_2(E)=c_2(F)$ in this calculation. Considering the same question for the dual $E^*$, we have that $c_2(E)=c_2(E^*)$, but the subbundle $F\subset E$ induces a quotient $E^* \to F^*$, and so the opposite conclusion will hold for $E^*$. That is, $E^*$ fails to be asymptotically $Z_{\mathrm{dHYM}}$-stable whenever $\sigma B.h\le 0$ and is stable otherwise. 
	
	Indeed one may actually compute that if $E$ admits a $Z_{\mathrm{dHYM},B}$-critical metric for some $B$ with $B.h<0$ then the Chern connection for the induced Hermitian structure on $E^*$ satisfies the $Z_{\mathrm{dHYM},-B}$-critical equation, which combined with \cref{thm:maintheoremZstability} recovers the above algebro-geometric observation. It is interesting to interpret the above duality by noting that $Z_{\mathrm{dHYM},-B}$ is the ``dual polynomial stability condition" to $Z_{\mathrm{dHYM},B}$ in the sense of Bayer \cite[\S 3.3]{bayer} provided we abusively ignore whether the heart of our bounded $t$-structure contains the bundle $E$. There it is shown that if $E$ is $Z_{\mathrm{dHYM},B}$-stable then $E^*$ is $Z_{\mathrm{dHYM},-B}$-stable.\footnote{Technically Bayer considers the dual object $\mathbb{D}(E)$ which in our case would be $E^* \otimes \omega_{\PP^2}$, and the dual stability condition involves transforming $e^{-B}$ to $\Ch(\omega_{\PP^2})^{-1} \cdot e^B$, so the numerics work out to align exactly with the above calculations in our example.}
	%
	
	Next, we consider the other central charge of relevance to string theory
	$$
	Z_{\Td} (E)= - \int_{\CCPP^2} e^{- i k [\omega]} \ch (E) \sqrt{\Td (\CCPP^2) } e^{-B};
	$$
	that is, where the unipotent operator $U=\sqrt{\Td (\CCPP^2) } e^{-B}$ contains both the B-field and the square root of the Todd class. The Todd class of $\CCPP^2$ is given by 
	$$
	\Td(\CCPP^2) = 1 + \frac{3}{2} h + h^2,
	$$
	and so 
	$$
	\sqrt{\Td (\CCPP^2) } = 1 + \frac{3}{4}h + \frac{7}{32}h^2.
	$$
	By the same computations above we obtain that the imaginary part of $\frac{Z_{\Td}(F)}{Z_{\Td}(E)}$ is a positive constant multiple of
	$$\frac{3k}{4} \sigma (\rk E - \rk F) \left(B.h - \frac{3}{4}\right).$$
	Thus $\varphi_k(F)<\varphi_k(E)$ holds whenever $B.h < \frac{3}{4}$, producing $Z_{\Td}$-critical connections, and in the case $B.h\ge \frac{3}{4}$ asymptotic $Z_{\Td}$-stability is violated and such connections cannot exist. As before, the opposite conclusions will hold when taking $E^*$ instead of $E$. 
	

	As mentioned above, as a consequence of the proof of \cref{thm:maintheoremZstability} it suffices to check asymptotic $Z$-stability only with respect to $F$. Since $c_2(E) \geq 0$ always holds by the slope stability of $F$, we can summarise our findings as follows:
	\begin{enumerate}
		\item If $c_2(E)>0$ then $E$ is asymptotically $Z_{\mathrm{dHYM}}$-stable whenever we have chosen a B-field with $B.h<0$, and asymptotically $Z_{\mathrm{dHYM}}$-unstable if $B.h\ge 0$. Therefore $E$ admits a deformed Hermitian Yang--Mills metric with B-field for $B.h<0$, but does not admit deformed Hermitian Yang--Mills connections when $B.h\ge 0$ in the large volume regime. In particular $E$ does not admit a deformed Hermitian Yang--Mills connection with vanishing B-field.
		\item The opposite conclusions as above hold when $E$ is replaced by $E^*$. That is, $E^*$ admits deformed Hermitian Yang--Mills connections  whenever the B-field is chosen such that $B.h>0$, and does not admit deformed Hermitian Yang--Mills connections when $B.h\le 0$. In particular neither $E$ or $E^*$ admit dHYM metrics with vanishing B-field.
		\item If $c_2(E)>0$ then $E$ is asymptotically $Z_{\Td}$-stable whenever $B.h<\frac{3}{4}$, and is unstable otherwise. Therefore $E$ admits a $Z_{\Td}$-critical connetions whenever $B.h<\frac{3}{4}$. In particular this includes the case with vanishing $B$-field. 
		\item The opposite conclusion as above holds for $E^*$. That is, $E^*$ admits a $Z_{\Td}$-critical metric whenever $B.h>\frac{3}{4}$. In particular this does \emph{not} include the case with vanishing B-field, in contrast to $E$.
		\item If $c_2(E)=0$ then $E$ is asymptotically $Z$-semistable with respect to either central charge.
		
	\end{enumerate}
\end{example}

\begin{remark}
	The example above demonstrates a simple example of a wall-crossing phenomenon for $Z$-stability and the existence of $Z$-critical connections. In particular as the stability conditions $Z_{\mathrm{dHYM}}$ and $Z_{\Td}$ are varied by replacing $B$ with $t B$ and letting $t\in \RR$ vary, we see a jump from stability to instability across critical thresholds ($t=0$ and $t=\frac{3}{4}$ respectively). 
	
	One can observe similar wall-crossing phenomena if we choose to vary the stability vector for our stability conditions instead. Let us work just with $Z_{\mathrm{dHYM}}$. If we change our stability vector to $\rho = (\rho_0, \rho_1, -1)$ to obtain a new stability condition $Z$, having normalised $\rho_2=-1$ for convenience, then regardless of the B-field chosen we obtain $Z$-stablity whenever $\Im \rho_0 < 0$, and also when $\Im \rho_0 = 0$ and $B.h<0$. Instability follows whenever $\Im \rho_0 > 0$ or $\Im \rho_0 = 0$ and $B.h\ge 0$. Similar conclusions will follow for $Z_{\Td}$, and one can also replace $E$ by $E^*$ and flip the inequality $B.h<0$ to $B.h>0$. 
\end{remark}

\begin{remark}
	These examples are the first non-trivial solutions to the dHYM equation and $Z$-critical equation in higher rank. Here non-trivial means where $E\to X$ is not itself slope stable, all such bundles admitting solutions in the large volume limit by \cref{thm:stabilityimpliesexistencestable}, although even the existence in that case had not been previously observed.
\end{remark}

\section{$Z$-critical connections}

Let $Z$ be a polynomial central charge. In this section we will define a differential equation the existence of solutions to which formally aligns with $Z$-stability. To do so we must specify some differential-geometric representative data for the polynomial central charge as follows:

\begin{itemize}
	\item Let $\omega \in [\omega]$ be a fixed K\"ahler form in the K\"ahler class of the polynomial central charge $Z$.
	\item Let $\tilde U \in \Omega^*(X,\RR)$ be any differential form representing the unipotent class $U$.
	\item Let $h$ be a Hermitian metric on $E$.
\end{itemize}

Given a Hermitian metric $h$ on $E$ one obtains \emph{Chern--Weil representatives} of the characteristic classes of $E$ by the expression
$$\tilde \Ch(h) = \exp\left(\frac{i}{2\pi} F(h)\right) \in \Omega^*(X, \End(E))$$
where $F(h)$ is the Chern curvature of $h$. For example one obtains
$$\tilde \Ch_0(h) = \id_E,\quad \tilde\Ch_1(h) = \frac{i}{2\pi} F(h), \quad \tilde \Ch_2(h) = -\frac{1}{8\pi^2} F(h)^2$$
and so on. These define representatives of the Chern classes in the sense that
$$\left[ \trace \tilde \Ch_i(h)\right] = \Ch_i(E).$$
Using the above data, we define the differential-geometric representative $\tilde Z_k(h)$ of the density of the polynomial central charge $[Z_k(E)]$ by

$$\tilde Z_k(h) := \left( \sum_{d=0}^n \rho_d k^d \omega^d\right) \wedge \tilde \Ch(h) \wedge \tilde U.$$

Let us now define a $Z$-critical metric. Given an endomorphism $T\in \Gamma(\End(E))$ then with respect to the Hermitian metric $h$ one obtains a splitting
$$T= \Re(T) + i \Im(T)$$
where $\Re(T), \Im(T)$ are Hermitian endomorphisms, and therefore $i\Im(T)$ is skew-Hermitian. Indeed $\Re(T) = (T+T^*)/2$ and $\Im(T) = (T-T^*)/2i$. Note that with respect to a metric $h$, the Chern curvature $F(h)$ is an imaginary endomorphism-valued form, so that $i F(h)$ is real.\footnote{Indeed recall that the Chern connection is unitary so its curvature $F(h)$ has values in a $\mathfrak{u}(n)$-bundle which consists of skew-Hermitian (i.e. ``imaginary") endomorphisms.}

\begin{definition}[$Z$-critical metric and connection]
	Let $h$ be a Hermitian metric on a holomorphic vector bundle $E\to (X,\omega)$ over a compact K\"ahler manifold. Let $Z$ be a polynomial central charge with data $[\omega], U, \rho$ and fixed representative data $\omega, \tilde U$. Suppose $Z(E)\ne 0$. Then we say $h$ is a \emph{$Z$-critical metric} if
	$$\Im(e^{-i\varphi(E)} \tilde Z(h)) = 0.$$
	If $A$ is the Chern connection of $h$ we say $A$ is a \emph{$Z$-critical connection}. We call the above equation the \emph{$Z$-critical equation}.
\end{definition}

In the following we will freely interchange perspectives between $Z$-critical metrics or connections, depending on which language is most convenient at the time. As has already been discussed in \cref{sec:chernconnections} there is no loss of generality in taking either perspective.

\begin{remark}\label{rmk:AlternativeZcriticalequation}
	Under the assumption $Z(E)\ne 0$ we could alternatively define a $Z$-critical metric as one satisfying the equation
	\begin{equation}\label{eq:AlternativeZcriticalequation}\arg \tilde Z(h) = e^{i\theta} \id_E\end{equation}
	where $\arg \tilde Z(h)$ is the unitary part of $\tilde Z(h)$ with respect to the polar decomposition induced by $h$ and $\theta$ is a constant.\footnote{Note the unitary part of the polar decomposition is not uniquely defined unless the positive semi-definite part is in fact strictly positive-definite. This is precisely the assumption of being $Z$-almost calibrated.} This would be more adapated to definition of $Z$-stability in terms of phases \cref{def:Zstability}. An equation of this form has been previously proposed for metrics on quiver representations by Kontsevich \cite{kontsevich2015homological} inspired by King's stability and moment map criterion \cite{king1994moduli}. The author thanks Pranav Pandit for pointing this reference out to them.	
	
	On a line bundle the two equations are equivalent, and in general \eqref{eq:AlternativeZcriticalequation} implies the $Z$-critical equation. It would be interesting to understand the relationship between these different formulations in general. One expects that the $Z$-almost calibrated condition \cref{def:almostcalibrated} may play some role as it does in the case of line bundles.\footnote{The equations should not be equivalent in general, since after passing to higher rank taking real and imaginary parts does not interact with the polar decomposition in the same way as for line bundles.}
\end{remark}

Note that this equation makes sense on the level of cohomology. That is, we have
\begin{equation}\label{eq:Zcriticaltraceintegral}\int_X \trace \Im(e^{-i\varphi(E)} \tilde Z(h)) = 0.\end{equation}

Let us state the $Z$-critical equation for \cref{ex:dhymcentralcharge,ex:stringtheorycentralcharge}. 

\begin{example}[Deformed Hermitian Yang--Mills equation]
	Consider the example \cref{ex:dhymcentralcharge} of the dHYM central charge
	$$Z_{\mathrm{dHYM}}(E) = - \int_X e^{-i[\omega]-B} \Ch(E).$$
	Choose representatives $\omega \in [\omega]$ and $\beta \in B$, and let $h$ be a Hermitian metric on a vector bundle $E\to X$.
	We compute
	\begin{align*}
		-e^{-i\omega -\beta} \tilde{\Ch}(E) &= -\sum_{j=0}^n (-i)^j \frac{(\omega-i\beta)^j}{j!} \wedge \tilde{\ch}_{n-j} (E)\\
		&= -\sum_{j=0}^n (-i)^j \frac{(\omega-i\beta)^j}{j! (n-j)!} \wedge \left( \frac{i}{2\pi} F(h)\right)^{n-j}\\
		&= -\frac{1}{i^n n!} \sum_{j=0}^n {n\choose j} (\omega-i\beta)^j \wedge i^{n-j} \left( \frac{i}{2\pi} F(h)\right)^{n-j}\\
		&= \frac{(-1)^{n+1} i^n}{n!} \left( \omega \otimes \id_E - i\beta \otimes \id_E - \frac{F(h)}{2\pi}\right)^n.
	\end{align*}
	The phase is given by
	\begin{align*}
		e^{-i\varphi(E)} &= \frac{r(E)}{Z(E)}\\
		&= r(E) \left( \frac{(-1)^{n+1} i^n}{n!} \int_X \trace \left( \omega \otimes \id_E - i\beta \otimes \id_E -  \frac{F(h)}{2\pi}\right)^n\right)^{-1}.
	\end{align*}
	If we define $$\hat z = \int_X \trace \left( \omega \otimes \id_E - i\beta \otimes \id_E - \frac{F(h)}{2\pi}\right)^n$$
	and $\hat z = \hat r e^{i\hat \phi}$, then since $r/\hat r = \frac{1}{n!}$ it follows that
	$$\Im (e^{-i\varphi} \tilde{Z}(E)) = \frac{1}{n!} \Im \left(e^{-i\phi} \left( \omega \otimes \id_E - i\beta \otimes \id_E - \frac{F(h)}{2\pi}\right)^n\right).$$ This is the deformed Hermitian Yang--Mills equation with B-field in higher rank, as discussed in \cref{sec:dhym}.
	
	In the literature on the dHYM equation where $E$ is a line bundle one normally writes $\alpha = \frac{i}{2\pi} F(h) - \beta$ which is a real $(1,1)$-form, and the dHYM equation for $\alpha$ is written
	$$\Im(e^{-i\phi} (\omega + i\alpha)^n) = 0.$$
\end{example}

\begin{example}[String theory central charge equation]
	Consider the central charge $Z_{\Td}$ or $Z_{\hat A}$ of \cref{ex:stringtheorycentralcharge}. Then again choosing representatives $\omega \in [\omega], \beta \in B$, and now representatives $\widetilde{\sqrt{\Td(X)}}$ of $\sqrt{\Td(X)}$ or $\widetilde{ \sqrt{\hat A(X)}}$, one produces the equations
	$$\Im(e^{-i\varphi(E)} \tilde Z_{\Td}(h)) = 0$$
	and 
	$$\Im(e^{-i\varphi(E)} \tilde Z_{\hat A}(h)) = 0.$$
	The former equation was introduced in the physics literature by Enger--L\"utken \cite{enger-lutken}, who proved a version of the correspondence in \cref{ch:correspondence} on a Calabi--Yau threefold. 
	
	We note that the most natural choice of representative forms for $\Td(X)$ or $\hat A(X)$ are those arising as the Chern--Weil representatives of $TX$ with respect to the Levi-Civita connection of the K\"ahler metric $\omega$. Indeed in works studying coupled dHYM metrics such as \cite{schlitzer2021deformed} it is likely very important to choose these particular representatives (although in their work they consider the standard dHYM equation rather than a version involving $\sqrt{\Td(X)}$ or $\sqrt{\hat A(X)}$). Nevertheless which particular representatives are chosen does not seem to be relevant to our work.
\end{example}

It is interesting to consider the $Z$-critical equation on a line bundle, to point out the analogy with the forms of the dHYM equation which have previously appeared in the literature.

\begin{example}
If $L\to X$ is a line bundle, then $\tilde Z(h)$ is a complex $(n,n)$-form on $X$. The $Z$-critical equation can be formulated as
$$\Im \left( e^{-i\varphi(L)} \exp (i \tilde \varphi(h))\right)= 0$$
where we define
$$\tilde \varphi(h) = \arctan\left(\frac{\Im \tilde Z(h)}{\Re \tilde Z(h)}\right)$$
as the phase function 
$$\tilde \varphi(h) = \arg \frac{\tilde Z(h)}{\omega^n} : X \to (-\pi, \pi).$$
Thus the $Z$-critical equation is equivalent to the condition
\begin{equation}
	\tilde \varphi(h) = \varphi(L) \mod 2\pi \ell.\label{eq:zcriticallinebundle}
\end{equation}
This scalar form of the $Z$-critical equation is similar to that which has appeared previously for the dHYM equation \cite[\S 2]{jacob2017special}. In general if one defines the eigenvalues of $F(h)$ relative to $\omega$ as $\lambda_1,\dots,\lambda_n$, then there is a function $f_Z$ such that
$$\tilde \varphi(h) = f_Z(\lambda_1,\dots,\lambda_n).$$
Equations in the form \eqref{eq:zcriticallinebundle} have appeared in the context of Sz\'ekelyhidi's work on fully non-linear equations \cite{szekelyhidi2018fully}, and the behaviour of such equations is sensitive to the properties of the function $f_Z$. In particular for the J-equation one has
$$f_Z(\lambda_1,\dots,\lambda_n) = \sum_i \frac{1}{\lambda_i}$$
and for the dHYM equation one has
$$f_Z(\lambda_1,\dots,\lambda_n) = \sum_i \arctan(\lambda_i).$$
The symmetry of the function $f_Z$ in these examples is important in the analysis of those equations, and we will explain this in \cref{sec:subsolutions}.
\end{example}

Related to the algebraic statement \cref{lem:azsimpliessimple} showing that asymptotic $Z$-stability approaches slope stability, there is a corresponding property of the $Z$-critical equation in the large volume limit.

\begin{lemma}\label{lem:largevolume}
	In the large volume regime $k \gg 0$, the leading order condition for $h$ to be a $Z$-critical metric for a polynomial stability condition is given by the weak Hermite--Einstein equation. More precisely, there is an expansion in $k$ of the form \begin{align*}\Im \big( &e^{-i\varphi_k(E)} \tilde{Z}_{k}(h) \big)=  \\ & ck^{2n-1}\left([\omega]^n \rk (E) \omega^{n-1} \wedge \left( \frac{i}{2\pi} F(h) + \tilde U_2 \id_E\right) - \deg_U( E)  \omega^n\otimes \id_E\right) + O(k^{2n-2}),\end{align*} 
	where $c\in \RR_{>0}$ is a positive constant depending on $\rho$,  $\tilde U_2$ is the degree two part of the differential form $\tilde U$, and $\deg_U(E)$ is the degree defined as Equation \eqref{degU}. 
	To leading order therefore, the $Z$-critical equation is equivalent to  $$F(h) \wedge \omega^{n-1} = - 2\pi i \left( \frac{\deg_U(E)}{ [\omega]^n \rk(E)} - \frac{1}{n}\contr_{\omega} \tilde U_2 \right) \id_E \otimes \omega^n,$$ which is the weak Hermite--Einstein equation (\cref{def:hermiteeinstein}) with function $$f = -2\pi i \left( \frac{\deg_U(E)}{(n-1)! [\omega]^n\rk(E)} - \frac{1}{n}\contr_{\omega} \tilde U_2 \right).$$
\end{lemma}

\begin{proof}
	This follows from computing the leading order terms directly. Writing $Z_k(E) = r_k (\cos \varphi_k + i \sin \varphi_k)$, one computes
	\begin{align*}
		r_k \cos \varphi_k &= \Re \rho_n k^n [\omega]^n \rk E + \Re \rho_{n-1} k^{n-1} [\omega]^{n-1}.(c_1(E) + \rk (E) U_2) + O(k^{n-2}),\\
		r_k \sin \varphi_k &=  \Im \rho_n k^n [\omega]^n \rk E + \Im \rho_{n-1} k^{n-1} [\omega]^{n-1}.(c_1(E) + \rk (E) U_2) + O(k^{n-2}).
	\end{align*}
	We also have
	\begin{align*}
		\Re \tilde Z_k(h) &= \Re \rho_n k^n \omega^n \otimes \id_E + \Re \rho_{n-1} k^{n-1} \omega^{n-1} \wedge \left( \frac{i}{2\pi} F(h) + \tilde U_2 \id_E\right) + O(k^{n-2}),\\
		\Im \tilde Z_k(h) &= \Im \rho_n k^n \omega^n \otimes \id_E + \Im \rho_{n-1} k^{n-1} \omega^{n-1} \wedge \left( \frac{i}{2\pi} F(h) + \tilde U_2 \id_E\right) + O(k^{n-2}).
	\end{align*}
	The condition for $h$ to be $Z_k$-critical is thus
	$$r_k \cos \varphi_k \Im \tilde Z_k(h) - r_k \sin \varphi_k \Re \tilde Z_k(h) = 0$$
	Computing the induced expansion in $k$ gives
	\begin{align*}
		&\Im \big( e^{-i\varphi_k(E)} \tilde{Z}_{k}(E) \big) \\
		&= k^{2n-1}\bigg((\Re \rho_n [\omega]^n \rk E)\left(\Im \rho_{n-1} \omega^{n-1} \wedge \left( \frac{i}{2\pi} F(h) + \tilde U_2 \id_E\right)\right)\\
		&\quad+ (\Re \rho_{n-1} \deg_U(E)) (\Im \rho_n \omega^n \otimes \id_E)\\
		&\quad- (\Im \rho_n [\omega]^n \rk E ) \Re \rho_{n-1} k^{n-1} \omega^{n-1} \wedge \left( \frac{i}{2\pi} F(h) + \tilde U_2 \id_E\right)\\
		&\quad-  (\Im \rho_{n-1} \deg_U(E)) (\Re \rho_n \omega^n \otimes \id_E)\bigg) + O(k^{2n-2}).
	\end{align*}
	We simplify by dividing by the factor $$c = \Re \rho_n \Im \rho_{n-1} - \Im \rho_n \Re \rho_{n-1},$$ which is non-zero, and in fact positive, by the stability vector assumption that $\Im (\rho_{n-1}/\rho_n) >0.$ Thus the $k^{2n-1}$-coefficient becomes
	$$[\omega]^n \rk (E) \omega^{n-1} \wedge \left( \frac{i}{2\pi} F(h) + \tilde U_2 \id_E\right) - \deg_U (E)  \omega^n\otimes \id_E,$$
	as desired. 
\end{proof}

\begin{remark}
	For the purposes of proving the correspondence between asymptotic $Z$-stability and existence of $Z$-critical metrics in the large volume limit, the above and \cref{lem:azsimpliessimple} only use the condition
	$$\Im\left( \frac{\rho_{n-1}}{\rho_n}\right) > 0$$
	and the lower order conditions on the stability vector $\rho$ are not necessary. One can therefore allow certain degenerate examples of polynomial central charges as examples for $Z$-critical equations without effecting the strength of the main result. These should include examples of so-called ``weak" stability conditions for which $Z(E)=0$ does not imply $E$ is numerically trivial. In particular slope stability satisfies all the necessary conditions for our theory to apply. 
\end{remark}

For completeness let us define the almost calibrated condition for the $Z$-critical equation, which was proposed by Collins--Yau \cite[\S 8.1]{collins2018moment}. 

\begin{definition}\label{def:almostcalibrated}
	Let $E\to (X,\omega)$ be a vector bundle and $Z$ a polynomial stability condition. Define the space of \emph{$Z$-almost calibrated} metrics by
	$$\calH_Z := \left\{ h \mid \Re \left( \trace \left( e^{-i\varphi(E)} \tilde Z(h)\right) \right) > 0\right\}.$$
\end{definition}

Part of the importance of the almost calibrated condition is that it implies the positivity of a natural $L^2$ inner product on endomorphisms $\phi,\psi$ of $E$, given by
$$(\phi,\psi) \mapsto \int_X \phi \psi \Re \left( \trace \left( e^{-i\varphi(E)} \tilde Z(h)\right) \right).$$
This defines a positive Riemannian metric on the space $\calH_Z$, and the geodesics with respect to this metric were utilized by Collins--Yau to obtain strong analytical results for the dHYM equation. It is likely that this will be important for the $Z$-critical equation in the future. We will note see the almost calibrated condition near the large volume limit, due to the following lemma.

\begin{lemma}
	A Hermitian metric $h$ on $E\to (X,\omega)$ is $Z$-almost calibrated for all $k\gg 0$.
\end{lemma}
\begin{proof}
	One simply verifies
	$$\Re\left(\trace \left( e^{-i\varphi(E)} \tilde Z(h)\right) \right) = k^{2n} (\rk E [\omega]^n (|\rho_n|^2) \omega^n + O(k^{2n-1})$$
	which is positive for $k\gg 0$ since $\rho_n \ne 0$. 
\end{proof}

We note that the $Z$-almost calibrated condition should be related to certain phase assumptions for Hermitian metrics, just as for similar almost calibrated conditions in the study of Lagrangian submanifolds or the dHYM equation. Indeed working with the principal branch of $\arg$ we observe:\footnote{In general one expects a decomposition of the space of $Z$-almost calibrated metrics into pieces consisting of shifts of the argument bound in \cref{prop:almostcalibrated} similarly to those observed by Collins--Yau \cite{collins2018moment} for the dHYM equation on a line bundle. We will make no further comment on this important consideration.}

\begin{proposition}\label{prop:almostcalibrated}
	Suppose $h$ is a $Z$-almost calibrated metric on $E\to(X,\omega)$. Then 
	$$\left| \arg \trace \left(\tilde Z(h)\right) - \arg Z(E) \right| < \frac{\pi}{2}.$$
\end{proposition}

The above inequality is straightforward from noting that 
$$\Re\left(\trace \left( e^{-i\varphi(E)} \tilde Z(h)\right) \right) = r \cos\left(\trace \left( e^{-i\varphi(E)} \tilde Z(h)\right)\right)$$ where $r: X \to \RR_{>0}$ is a positive function, and the positivity of $\cos(x)$ within the principal branch of $\arg, x\in (-\pi,\pi)$ occurs whenever $|x|<\frac{\pi}{2}.$

\section{Subsolutions\label{sec:subsolutions}}

To begin our study of subsolutions, we will first take a detour and investigate the linearisation of the $Z$-critical equation.

\subsection{Linearisation and ellipticity}

First we will define the symmetrised product of a collection of endomorphisms. 
\begin{definition}\label{def:symmetrisation}
	Let $B_1,\dots,B_j$ be such a collection of endomorphisms. We define
	$$[B_1\cdots B_j]_{\mathrm{sym}} := \frac{1}{j!} \sum_{\sigma \in S_j} B_{\sigma(1)} \cdots B_{\sigma(j)}.$$
	When the $B_i$ are endomorphism-valued forms, we must take the graded symmetrisation. Define the graded sign $\grsgn$ of a permutation $\sigma$ as follows. If $\sigma$ is an adjacent transposition $(i\,\,\, i+1)$ then 
	$$\grsgn \sigma = \deg B_i \deg B_{i+1} \mod 2.$$
	In general any permutation $\sigma \in S_j$ can be written as a product of adjacent transpositions $\sigma = \sigma_1\cdots \sigma_\ell$, and we set
	$$\grsgn \sigma = \sum_{i=1}^\ell \grsgn \sigma_i \mod 2.$$
	Then the symmetrisation is defined by
	$$[B_1\wedge\cdots\wedge B_j]_{\sym} = \frac{1}{j!} \sum_{\sigma \in S_j} (-1)^{\grsgn \sigma} B_{\sigma(1)} \wedge \cdots \wedge B_{\sigma(j)}.$$
\end{definition}

The graded symmetrisation is constructed so that, for example, the graded symmetrisation of even degree endomorphism-valued forms is just the regular symmetrisation, and for endomorphism-valued one-forms is just the standard graded symmetrisation with sign changing by $\sgn \sigma$. In the following we will have expressions involving symmetrisations of two-forms and one-forms. We note that the graded sign of a permutation of $B_1\wedge \cdots \wedge B_j$ does not depend on the choice of decomposition into adjacent transpositions.

\begin{remark}
	We must caution that it is not the same to consider $[B_1 B_2 B_3]_{\sym}$ and $[(B_1B_2) B_3]_{\sym}$. In general such expressions will only agree when all the $B_i$ commute. This will for example be the case when $E$ is a line bundle so the curvature is just an imaginary two-form, which explains why the formalism above is not necessary in that setting. In the following we will always treat each curvature term $F_A$ appearing in a factor of the $Z$-critical equation as a separate endomorphism of $E$. 
\end{remark}

Let us now work in the convention of fixing the Hermitian metric $h$ and varying the Chern connection $A$. To compute the linearisation of the $Z$-critical equation, let us recall \cref{lem:linearisationcurvatureChern} which computed the linearisation of the curvature as
$$F_{A_t} = F_A + (\delbardel - \deldelbar) V t + O(t^2)$$
for $A_t = \exp(tV) \cdot A$ where $V$ is a Hermitian endomorphism of $(E,h)$. 

Let $D_Z$ denote the $Z$-critical operator
$$D_Z: A \mapsto \Im(e^{-i\varphi(E)} \tilde Z(A)).$$
Each term appearing in $D_Z$ is some real constant multiple of a form
\begin{equation}\label{eq:DZterm}\omega^d \wedge \left(\frac{i}{2\pi} F_A\right)^j \wedge \tilde U_\ell\end{equation}
such that $d+j+\ell = n$. Perturbing $A$ to $A_t=\exp(tV)\cdot A$, we linearise \eqref{eq:DZterm} as
\begin{align}
	&\left.\deriv{}{t}\right|_{t=0} \omega^d \wedge \left(\frac{i}{2\pi} F_{A_t}\right)^j \wedge \tilde U_\ell\nonumber\\
	&=  \omega^d \wedge j \left[ \underbrace{\left(\frac{i}{2\pi} F_A\right) \wedge \cdots \wedge \left(\frac{i}{2\pi} F_A\right)}_{j-1 \text{ times}} \wedge \frac{i}{2\pi} (\delbardel - \deldelbar) V \right]_{\sym} \wedge \tilde U_{\ell}.\label{eq:DZlinearisationoneterm}
\end{align}

\begin{definition}
	Define the derivative $\tilde Z'(A)$ to be the $\End E$-valued $(n-1,n-1)$-form given by taking the formal derivative with respect to $\frac{i}{2\pi} F_A$, where $d (\frac{i}{2\pi} F_A) = \id_E$. 
\end{definition}

\begin{proposition}\label{prop:linearisationDZ}
	The linearisation $P_Z$ of $D_Z$ at $A$ is given by
	$$P_Z(A)(V) = \left[\Im(e^{-i\varphi(E)} Z'(A))\wedge \frac{i}{2\pi} (\delbardel - \deldelbar) V\right]_{\sym}.$$
\end{proposition}
\begin{proof}
	This follows from applying the calculation in \eqref{eq:DZlinearisationoneterm} to every term appearing in the $Z$-critical equation. 
\end{proof}

In order to understand the ellipticity of our equation and the resulting notion of a subsolution, let us define a notion of positivity for $\End$-valued forms.

\begin{definition}\label{def:positivityEndEforms}
	A real $\End E$-valued $(n-1,n-1)$-form 
	$$T=\sum B_1 \wedge\cdots \wedge B_j$$
	is \emph{positive} if at every $p\in X$ and every $u\in T_{0,1}^* X_p \otimes \End E_p$ with $u\ne 0$ one has
	$$\trace i \sum \left[ {B_1}_p \wedge \cdots \wedge {B_j}_p \wedge u^* \wedge u \right]_{\sym} > 0.$$
\end{definition}

This definition of positivity recovers the regular notion of positivity such as in \cite{song2008convergence} or \cite[Ch. III \S 1]{demailly1997complex} when $T$ is just an $(n-1,n-1)$-form. It enters our discussion through the definition of a subsolution.

\begin{definition}[Subsolution]\label{def:subsolution}
	A Chern connection $A$ on $E\to (X,\omega)$ is a \emph{subsolution} of the $Z$-critical equation if 
	$$\Im(e^{-i\varphi(E)} Z'(A)) > 0$$
	in the sense of $\End E$-valued $(n-1,n-1)$-forms \cref{def:positivityEndEforms}.
\end{definition}

The relevance of this definition to the linearisation computed above is the following.

\begin{proposition}\label{prop:subsolutionimpliesellipticity}
	Suppose $A$ is a subsolution to the $Z$-critical equation. Then the $Z$-critical operator $D_Z$ is elliptic at $A$.
\end{proposition}
\begin{proof}
	By definition, a non-linear differential operator $D_Z$ is elliptic at $A$ if the linearisation $P_Z(A)$ is a linear elliptic differential operator. Consider the linearisation \cref{prop:linearisationDZ} as an operator
	$$P_Z(A): \Gamma(\End_H(E,h)) \to \Gamma(\End_H(E,h))$$
	on the space of Hermitian endomorphisms of $E$, defined by
	$$P_Z(A): V \mapsto \frac{ \left[\Im(e^{-i\varphi(E)} Z'(A))\wedge \frac{i}{2\pi} (\delbardel - \deldelbar) V\right]_{\sym}}{\omega^n}.$$
	As the notation already suggests, we may ignore any dependence of $\del$ and $\delbar$ on the Chern connection $A$, as these occur only at first order and below in $\delbardel$ (so these terms are irrelevant for demonstrating ellipticity). 
	
	To compute the symbol, notice that if we have a test differential form $\eta = a_i dx^i + b_j dy^j$ then
	$$\sigma\left(\pderiv{}{z^i}\right) = \frac{1}{2} (a_i - i b_i), \quad \sigma\left(\pderiv{}{\bar z^j}\right) = \frac{1}{2} (a_j + i b_j).$$
	In particular if we define $\xi := \xi_i d\bar z^i = (a_i + i b_i) d\bar z^i\in T_{0,1}^* X_p$ then in the operator
	$$V\mapsto \delbardel V = \sum_{j,k} \pderiv{}{\bar z^k} \pderiv{}{z^j}(V) d\bar z^k \wedge dz^j$$
	we compute the symbol by formally replacing
	$$\pderiv{}{\bar z^k} \mapsto \xi_k, \quad \pderiv{}{z^j} \mapsto \bar \xi_j.$$
	Thus the principal symbol $\sigma_\xi (P_Z(A)): \End E \to \End E$ is defined by
	$$\sigma_\xi(P_Z(A))(V) =  \frac{ \left[\Im(e^{-i\varphi(E)} Z'(A))\wedge \frac{i}{2\pi} 2 \xi \wedge \bar \xi \otimes V \right]_{\sym}}{\omega^n}.$$
	Ellipticity holds when $\sigma_\xi(P_Z(A))$ is invertible for any $\xi\ne 0$. Since $A$ is a subsolution, we have
	$$\trace i \left[ \Im(e^{-i\varphi(E)} Z'(A)) \wedge u^* \wedge u \right]_{\sym} > 0.$$
	If we choose $u=\xi \otimes V$ for some non-zero endomorphism $V$ then using the tracial property we may cyclically permute one of the $V$'s in each term in the above expression to the end, and we obtain
	$$\trace (\sigma_\xi(P_Z(A))(V) V) = \frac{1}{2} \trace i \left[ \Im(e^{-i\varphi(E)} Z'(A)) \wedge u^* \wedge u \right]_{\sym} \ne 0.$$
	It follows that $\sigma_\xi(P_Z(A))(V)\ne 0$ for any $V\ne 0$, so $\sigma_\xi(P_Z(A))$ has no kernel and $D_Z$ is elliptic.
\end{proof}

\begin{remark}
	The subsolution condition does not appear in the large volume limit of critical interest in \cref{ch:correspondence} as the leading order term in $k$ has coefficient $[\omega]^{n-1}$ and so dominates the other terms for $k\gg 0$. Indeed \cref{prop:subsolutionimpliesellipticity} is not necessary in the large volume limit, where ellipticity is easily inherited from the Hermitian Yang--Mills equation at $k=\infty$, using the property that invertibility of the symbol is an open condition. Let us record this in the following proposition.
\end{remark}

\begin{proposition}\label{lem:linearisation}
	In the large volume limit every Chern connection $A$ is a subsolution of the $Z$-critical equation, and in particular the $Z$-critical equation is elliptic in the large volume limit. Furthermore there is an expansion of the linearisation
	$$P_{Z_k}(A) = C (\rk E) [\omega]^n k^{2n-1} \Lap_{A^{\End E}} + O(k^{2n-2})$$
	for some constant $C\ne 0$.
\end{proposition}
\begin{proof}
	The claim about subsolutions is immediate from the expansion \cref{lem:largevolume} and computing the derivative $\tilde Z'(A)$. Indeed to leading order $k^{2n-1}$ the subsolution condition consists of a form
	$$c i \trace u^* \wedge u \wedge \omega^{n-1}$$
	for a constant $c>0$ depending on the stability vector $\rho$, which dominates all other terms in the large volume limit.
	
	The expansion of the linearisation is a straight-forward application of the calculation \cref{lem:linearisationcurvatureChern} of the linearisation for the Hermite--Einstein equation applied to the expansion \cref{lem:largevolume} of the $Z$-critical equation.
\end{proof}

For posterity we also record one more property of the linearisation of the $Z$-critical equation unrelated to the subsolution condition.

\begin{proposition}\label{prop:selfadjoint}
	The linearisation $P_Z(A)$ of the $Z$-critical operator is self-adjoint as an operator on the space of smooth sections of the bundle of Hermitian endomorphisms of $E$. In particular $P_Z(A)$ extends to a symmetric unbounded operator from the $L^2$ sections of the bundle of Hermitian endomorphisms to itself.
\end{proposition}
\begin{proof}
	This follows simply from integration by parts. Indeed we consider
	$$P_Z(A): V \mapsto \frac{ \left[\Im(e^{-i\varphi(E)} Z'(A))\wedge \frac{i}{2\pi} (\delbardel - \deldelbar) V\right]_{\sym}}{\omega^n}$$
	for a Hermitian endomorphism $V$. Then we wish to show
	$$\langle U, P_Z(A)(V) \rangle = \langle V, P_Z(A) (U) \rangle$$
	where the inner product is given by the pairing on smooth Hermitian endomorphism $U,V$ defined by
	$$\langle U,V \rangle = \int_X \trace (UV) \omega^n.$$
	Integrating by parts with respect to $\delbar_A$ and $\del_A$ and using the graded symmetrisation and tracial property gives the result, which in fact holds individually for any term of the form
	$$T \wedge \frac{i}{2\pi} (\delbar_A \del_A - \del_A \delbar_A) V$$
	where $T$ is any $d_A$-closed endomorphism-valued $(n-1,n-1)$-form. 
\end{proof}

\subsection{Stability with respect to subvarieties\label{sec:stabilitysubvarieties}}

Let us now reveal the importance of the subsolution condition for the algebraic geometry of the bundle $E$. 

First we emphasise the following somewhat remarkable property of the $Z$-critical equation, which is fundamentally due to the Chern--Weil theory of the Chern character.

\begin{lemma}\label{lem:Zderivativerestriction}
	Let $D\subset X$ be an irreducible divisor, and $E\to X$ a vector bundle. Then
	$$Z_D(E) = \int_D \trace \rest{\tilde Z'(A)}{D}$$
	where $A$ is a Chern connection on $E$.
\end{lemma}
\begin{proof}
	Recall that by definition of the $Z$-slope of a subvariety \eqref{eq:Zslopesubvariety}, we have
	$$Z_D(E) = \int_D [Z(E)]$$
	where $[Z(E)]$ is the density of the central charge $Z$. To verify the claim, it suffices to check
	$$[Z(E)]^{(n-1,n-1)} = [\trace Z'(A)].$$
	Indeed
	$$[Z(E)] = \left( \sum_{d=0}^n \rho_d \omega^d \right) . \Ch(E) . U$$
	so the degree $(n-1,n-1)$ component is
	\begin{align*}
		[Z(E)]^{(n-1,n-1)} &= \sum_{d+j+\ell = n-1} \rho_d \omega^d . \Ch_j(E) . U_j\\
		&=\sum_{d+j+\ell = n} \rho_d \omega^d . \Ch_{j-1}(E) . U_j
	\end{align*}
	where we interpret $\Ch_{-1}(E) = 0$. Now we use the fact that the Chern--Weil representative of $\Ch(E)$ with respect to a connection $A$ is
	$$\exp\left( \frac{i}{2\pi} F_A\right)$$
	and the formal derivative of $\tilde \Ch_j(A)$ with respect to $\frac{i}{2\pi} F_A$ is $\tilde \Ch_{j-1}(A)$, since $\deriv{}{x} \exp(x) = \exp(x).$

	Therefore we have
	$$[\trace Z'(A)] = [Z(A)]^{(n-1,n-1)}.$$
\end{proof}

\begin{remark}
	Let us take the perspective that other gauge-theoretic equations, such as the Hermite--Einstein, J-equation (and higher rank J-equation), and dHYM equation can be written as $Z$-critical equations possibly for degenerate choices of stability condition. Then \cref{lem:Zderivativerestriction} manifests for each of these equations, and demonstrates that if the equation can be expressed in terms of the Chern character, then (due to the Chern--Weil representation in terms of $\exp$) positivity of successive linearisations of the equation can be related to algebro-geometric stability with respect to subvarieties. 
\end{remark}

Let us now use \cref{lem:Zderivativerestriction} to deduce a stability criterion implied by the subsolution condition.

\begin{theorem}\label{thm:subsolutiondivisorstability}
	Suppose $E\to (X,\omega)$ admits a subsolution to the $Z$-critical equation. Then $E$ is $Z$-stable with respect to quotients $E\onto E\otimes \calO_D$ for $D\subset X$ an irreducible analytic divisor.
\end{theorem}
\begin{proof}
	Let $D\subset X$ be an irreducible codimension one analytic subvariety of $X$, and let $A$ be the subsolution of the $Z$-critical equation on $E$. Let $f \in \calO_X(U)$ be a holomorphic function on an open subset $U\subset X$ such that $\rest{D}{U} = \{f = 0\}.$ Set $u=d\bar f \otimes \id_E$ at a smooth point $p\in D\cap U$. The subsolution condition implies
	$$\trace i\left[ \Im(e^{-i\varphi(E)} Z'(A)), df\otimes \id, d\bar f\otimes \id_E\right]_{\sym} > 0$$
	at $p$. Since $d\bar f \otimes \id_E$ commutes with other endomorphisms, this is equivalent to the condition
	$$\trace \left( \Im(e^{-i\varphi(E)} \tilde Z'(A))\right)\wedge i df \wedge d\bar f>0.$$
	This implies that the $(n-1,n-1)$-form 
	$$\trace \left( \Im(e^{-i\varphi(E)} \tilde Z'(A))\right)$$
	is a positive volume form on $T_p D$ at $p\in D$. Indeed choosing submanifold coordinates $(z^1,\dots,z^n)$ on $U$ such that $z^1 = f$, then expanding 
	$$\alpha = \rest{\trace \left( \Im(e^{-i\varphi(E)} \tilde Z'(A))\right)}{U}$$
	in local coordinates $(z^2,\dots,z^n)$ on $D\cap U$, we observe that 
	$$\trace \left( \Im(e^{-i\varphi(E)} \tilde Z'(A))\right)\wedge i df \wedge d\bar f = \alpha \wedge i df \wedge d\bar f$$
	on $U$. Since the kernel of $idf \wedge d\bar f$ is generated by the coordinate vector fields $\pderiv{}{z^i}$ for $i=2,\dots,n$ but $\alpha \wedge i df \wedge d \bar f > 0$, we must have that $\alpha$ is positive in these directions.	This result also follows from applying \cite[Ch. III \S 1 (1.6)]{demailly1997complex}, which we have observed directly here in our case.
	
	Using the subsolution condition for all points $p$ in the smooth locus of $D$, we obtain that the top-degree form
	$$\rest{\trace \left( \Im(e^{-i\varphi(E)} \tilde Z'(A))\right)}{D}$$
	is positive on the smooth locus of $D$, and conclude
	$$\int_D \rest{\trace \left( \Im(e^{-i\varphi(E)} \tilde Z'(A))\right)}{D} > 0,$$
	where the integral only depends on the smooth locus since this form extends, being the restriction of a globally-defined smooth form from $X$.	By \cref{lem:Zderivativerestriction} this is equal to
	$$\Im(e^{-i\varphi(E)} Z_D(E)) > 0$$
	which occurs if and only if 
	$$\Im\left( \frac{Z_D(E)}{Z_X(E)}\right) > 0.$$
	Thus $E$ is $Z$-stable with respect to $D$ in the sense of \cref{def:Zstabilitysubvariety}. 
\end{proof}

\begin{remark}
	In the case where $Z = -\deg + i \rk$ is the degenerate central charge which reproduces the Hermite--Einstein equation, the above condition that $\trace \left( \Im(e^{-i\varphi(E)} \tilde Z'(A))\right)$ restricts to a volume form on any divisor is the simple fact that 
	$$\int_D \omega^{n-1} > 0$$
	for any divisor $D\subset X$ combined with the fact that $\rk E > 0$. In particular translating to generalised slopes this is equivalent to the fact that 
	$$\mu_Z(E) < \mu_Z(E\otimes \calO_D) = +\infty,$$
	that is, $E$ is never slope-destabilised by a torsion sheaf with support in codimension one. Later on \cref{thm:subsolutionsubvarietystability} in this setting will be vacuous (since the further derivatives of the central charge $Z$ will vanish) which manifests in the fact that the quotient $E\onto E\otimes \calO_D$ with $\codim D \ge 2$ has a kernel $F\into E$ which is a coherent subsheaf of rank $\rk F = \rk E$ and equal slope, and such subsheaves are irrelevant by definition of slope stability.
\end{remark}

In the case of the deformed Hermitian Yang--Mills equation and J-equation the relationship between stability with respect to subvarieties and the subsolution condition is well-understood. Let us consider just the example of the J-equation
$$\omega \wedge \alpha^{n-1} = c \alpha^n$$
for a K\"ahler metric $\alpha$ with fixed input $\omega$. Then the subsolution condition identified by Song--Weinkove \cite{song2008convergence} is given by 
\begin{equation}
	nc \alpha^{n-1} - (n-1) \omega \wedge \alpha^{n-2} > 0\label{eqn:jequationsubsolution}
\end{equation}
in the sense of $(n-1,n-1)$-forms. Writing the eigenvalues of $\alpha$ relative to $\omega$ as $\lambda_1,\dots,\lambda_n$, the J-equation takes the form
$$\sum_{i=1}^n \frac{1}{\lambda_i} = nc$$
and the subsolution condition becomes
\begin{equation}
	\sum_{i=1 \\ i\ne j}^n \frac{1}{\lambda_i} < nc,\quad j=1,\dots,k. \tag{\ref*{eqn:jequationsubsolution}} \label{eqn:jequationsubsolutionlambda}
\end{equation}
One may exploit the inherent symmetry of the J-equation with respect to the eigenvalues $\lambda_i$ to deduce from the subsolution condition \eqref{eqn:jequationsubsolutionlambda} the bound
$$\frac{1}{\lambda_i} < \frac{nc}{n-1}$$
for all $i=1,\dots,n$. In particular as was noted by Lejmi--Sz\'ekelyhidi \cite{lejmi-szekelyhidi}, this quickly implies the stability condition
$$\int_V c\alpha^p - p\alpha^{p-1} \wedge \omega > 0$$
for all $p$-dimensional subvarieties $V\subset X$.\footnote{Lejmi--Sz\'ekelyhidi conjectured the converse, that this stability condition implies the existence of a subsolution, and a uniform version was proven for any compact K\"ahler manifold by Chen \cite{chen2019j}. This was strengthened to the full conjecture in the projective case by Datar--Pingali \cite{datar2021numerical} and in general by Song \cite{song2020nakai}.}

A similar analysis can be done for the dHYM equation which has the symmetric expression
$$\sum_{i=1}^n \arctan(\lambda_i) = \hat \theta,$$
provided one works near the hypercritical phase range.

To summarise, the inherent symmetry of the J-equation and dHYM equation allow one to pass from a subsolution condition just on an $(n-1,n-1)$-form such as in \cref{def:subsolution} to stability criterion like \cref{thm:subsolutiondivisorstability} for all higher codimension analytic subvarieties. In general we therefore expect the following notion of subsolution to be necessary and stronger than \cref{def:subsolution}.

\begin{definition}[Strong subsolution]
	A Hermitian metric $h$ on $E\to (X,\omega)$ is a \emph{strong subsolution} for the $Z$-critical equation if 
	$$\Im(e^{-i\varphi(E)} \tilde Z^{(p)} (h))>0$$
	in the sense of $\End E$-valued $(n-p,n-p)$-forms for all $p=1,\dots,n$. Here $\tilde Z^{(p)}(h)$ denotes the $p$th formal matrix derivative of $\tilde Z(h)$ with respect to $iF(h)/2\pi$.
\end{definition}

Here we say an $\End E$-valued $(n-p,n-p)$-form $T$ is positive if for all $x\in X$ and all non-zero, linearly independent $u_1,\dots,u_p\in T_{0,1}^* X_x \otimes \End E_x$, we have
$$\trace i^p \left[ \rest{T}{x} \wedge u_1^* \wedge u_1 \wedge \cdots \wedge u_p^* \wedge u_p\right]_{\sym} > 0.$$

This is a generalisation to bundles of the positivity of $(n-p,n-p)$-forms in \cite[Ch. III \S 1]{demailly1997complex}.

Based on the above discussion, the dHYM equation in higher rank should still exhibit some of the symmetries of the line bundle setting, albeit with added difficulty due to the curvature $F_A$ being matrix-valued. Nevertheless let us propose a possibly very strong conjecture, which we do not expect to hold in the pointwise sense except in the case where $Z$ exhibits particular symmetries as described above.

\begin{conjecture}
	A subsolution $h$ of the higher rank deformed Hermitian Yang--Mills equation is a strong subsolution.
\end{conjecture}

We note that similarly to the case of the dHYM equation, we may only expect the above conjecture to hold in some ``critical phase range" (what exactly is meant by the phase for the higher rank $Z$-critical equation is not clear). 

In any case, after assuming the strong subsolution condition, we can produce an analogue of \cref{thm:subsolutiondivisorstability} for any codimensional subvariety.

\begin{theorem}\label{thm:subsolutionsubvarietystability}
	Suppose $E\to (X,\omega)$ admits a strong subsolution to the $Z$-critical equation. Then $E$ is $Z$-stable with respect to all irreducible analytic subvarieties $V\subset X$.
\end{theorem}
\begin{proof}
	The proof is a straightforward adaptation of \cref{thm:subsolutiondivisorstability} and \cref{lem:Zderivativerestriction} to the case of higher derivatives. Namely the exact same argument as in \cref{lem:Zderivativerestriction} applied to successive derivatives of $\tilde Z(A)$ implies 
	$$Z_V(E) = \int_V \trace \rest{\tilde Z^{(p)} (A)}{V}$$
	where $\codim V = p$. Now supppose $A$ is a strong subsolution and let $f_1,\dots,f_p \in \calO_X(U)$ be a collection of holomorphic functions on some open subset $U\subset X$ such that
	$$\rest{D}{U} = \{f_1=\cdots=f_p=0\}.$$
	We apply the subsolution condition with $u_1=d\bar f_1 \otimes \id_E, \dots, u_p = d\bar f\otimes \id_E$. Then just as in \cref{thm:subsolutiondivisorstability} we see that the $(n-p,n-p)$-form
	$$\trace \Im(e^{-i\varphi(E)} \tilde Z^{(p)}(A))$$
	is a positive volume form when restricted to $V$, and so 
	$$\int_V \Im(e^{-i\varphi(E)} \tilde Z^{(p)}(A)) > 0.$$
	But by the analogue of \cref{lem:Zderivativerestriction} explained above, this is equivalent to 
	$$\Im(e^{-i\varphi(E)} Z_V(E)) > 0$$
	which is equivalent to 
	$$\Im\left( \frac{Z_V(E)}{Z_X(E)}\right) > 0$$
	so $E$ is $Z$-stable with respect to $V$. 
\end{proof}

\begin{remark}\label{rmk:jacobsheu}
	As observed in the work of Jacob--Sheu \cite{jacob2020deformed} and in \cite[Rmk. 1.10]{chen2021j}, when working far from the large volume limit and under ``lower phase" assumptions, the required sign of algebro-geometric invariant
	$$\Im\left(\frac{Z_V(E)}{Z_X(E)}\right)$$
	in $Z$-stability may change depending on the codimension of $V\subset X$. Thus our definition \cref{def:subsolution} and results \cref{thm:subsolutiondivisorstability,thm:subsolutionsubvarietystability} should be viewed as preliminary. In particular it is likely that the correct definition of strong subsolution should have a sign which depends on the codimension of the subvarieties being considered. We note that in all cases one will require \emph{non-degeneracy} of the resulting $\End E$-valued $(n-p,n-p)$-forms, and so any alternative definition will still imply the ellipticity of the $Z$-critical equation and the non-degeneracy of the (potentially negative or indefinite) K\"ahler form $\Omega_Z$ on $\calA(h)$. 
	
	Again we predict that the correct formalism for identifying these conditions is that of the derived category. Indeed we note that in the study of polynomial stability conditions \cite{bayer} the perverse sheaves identified as hearts of bounded $t$-structures for a given stability vector $\rho$ and perversity function $p$ are arranged \emph{precisely} so that quotient objects $\calF$ in the heart corresponding to sheaves supported in codimension $p$ have central charges $Z(\calF)$ all lying in the same half-plane. Under a suitable formulation of the $Z$-critical equation on such perverse sheaves for example, the required sign of the algebraic invariants identified by Jacob--Sheu should be identified by $(-1)^{p(d)}$ where $p(d)=-\floor{\frac{d}{2}}$ is the perversity function for the stability vector $\rho$ for $Z_{\mathrm{dHYM}}$ and $\codim V = n-d$.
\end{remark}

In \cref{fig:subsolutionconjectures} we sketch the conjectured relationships between the the various notions of solution, subsolution, and stability we have discussed away from the large volume limit. We suppress any mention of the derived category in the figure, acknowledging that it will be necessary to make sense of the proposed equivalences in general.

\begin{figure}[h]
	\centering
	
	\begin{tikzcd}[
		column sep={17em,between origins},
		row sep={6em,between origins},
		]
		\text{$\exists Z$-critical metric $h$} \arrow[double,dashed,Leftrightarrow]{r} \arrow[double,dashed]{d} \arrow[dashed,double,bend right=77]{dd} & \text{$E$ $Z$-stable} \arrow[double]{d} \arrow[double,bend left=77]{dd}\\
		\text{$h$ is a strong subsolution} \arrow[double,shift right]{r} \arrow[double]{d} & \text{$Z$-stable for subvarieties} \arrow[double,dashed,shift right]{l} \arrow[double]{d}\\
		\text{$h$ is a  subsolution}  \arrow[double,shift right]{r} & \text{$Z$-stable for divisors}  \arrow[double,dashed,shift right]{l}
	\end{tikzcd}
	\caption{The conjectural relationships between stability and solutions and subsolutions.}\label{fig:subsolutionconjectures}
\end{figure}
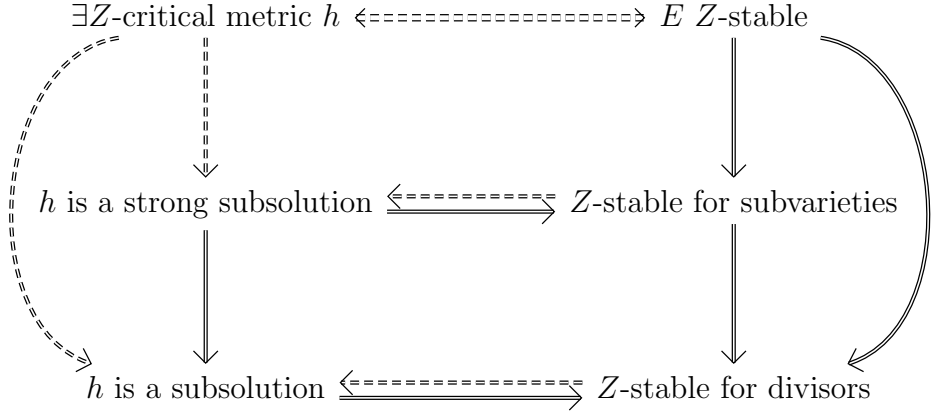

For posterity, we single out the most important analytical conjecture related to subsolutions in \cref{fig:subsolutionconjectures}.

\begin{conjecture}
	A solution to the $Z$-critical equation is a strong subsolution.
\end{conjecture}

Again we expect this to hold only in some critical phase range for the $Z$-critical equation.

\section{Existence on surfaces}

Away from the large volume limit, the analytical difficulties of the dHYM equation or $Z$-critical equation are significant. However in the case of a line bundle over a complex surface the equation admits a transformation to a Monge--Ampere-type equation and using Yau's proof of the Calabi conjecture an existence result can be proven. In the case of the dHYM equation this is a result due to Jacob--Yau \cite{jacob2017special}, and our proof follows a similar approach.

Let $L\to (X,\omega)$ be a holomorphic line bundle over a compact K\"ahler surface and let us denote $\alpha := \frac{i}{2\pi} F(h)$ for some Hermitian metric $h$ on $L$. Let $Z$ be a polynomial central charge for $(X,\omega)$. The Chern--Weil representative may be written
$$\tilde Z(h) = \rho_2 \omega^2 + \rho_1 \omega \wedge (\alpha + \tilde U_2) + \rho_0 \left(\tilde U_4 + 2\alpha \wedge \tilde U_2 + \frac{\alpha^2}{2}\right).$$

The $Z$-critical equation for $h$ is given by
\begin{multline}
	(\Re \rho_2 \Im \rho_1 - \Im \rho_2 \Re \rho_1) \left( [\omega]^2 \omega \wedge (\alpha + \tilde U_2) - (\deg L + [\omega].U_2) \omega^2 \right)\\ 
	+ (\Im \rho_2 \Re \rho_0 - \Re \rho_2 \Im \rho_0)( (U_4 + [\alpha].U_2 + \frac{[\alpha]^2}{2}) \omega^2 \\- [\omega]^2 (\tilde U_4 + \alpha \wedge \tilde U_2 + \frac{\alpha^2}{2})) \\
	+ (\Re \rho_0 \Im \rho_1 - \Im \rho_0 \Re \rho_1) ((U_4 + [\alpha].U_2 +\frac{[\alpha]^2}{2}) \omega \wedge (\alpha + \tilde U_2) \\
	- (\deg L + [\omega].U_2)(\tilde U_4 + \alpha \wedge \tilde U_2 + \frac{\alpha^2}{2})) = 0.\label{eq:Zcriticalsurface}
\end{multline}

Let us write \eqref{eq:Zcriticalsurface} in a different form by collecting powers of $\alpha$. Write the equation as
\begin{equation}\label{eq:Zcriticalsurfacesimplified}a \alpha^2 + \alpha \wedge \beta + \gamma = 0\end{equation}
for some $\beta$ and $\gamma$ and constant $c$. Explicitly, we have
\begin{align*}
	a &= -(\Im \rho_2 \Re \rho_0 - \Re \rho_2 \Im \rho_0)\frac{[\omega]^2}{2} - (\Re \rho_0 \Im \rho_1 - \Im \rho_0 \Re \rho_1) \frac{(\deg L + [\omega].U_2)}{2}.\\
	\beta &= (\Re \rho_2 \Im \rho_1 - \Im \rho_2 \Re \rho_1) [\omega]^2 \omega -(\Im \rho_2 \Re \rho_0 - \Re \rho_2 \Im \rho_0) [\omega]^2 \tilde U_2 \\
	&+ (\Re \rho_0 \Im \rho_1 - \Im \rho_0 \Re \rho_1)\left( (U_4 + [\alpha].U_2 +\frac{[\alpha]^2}{2}) \omega - (\deg L + [\omega].U_2) \tilde U_2\right).\\
	\gamma &= (\Re \rho_2 \Im \rho_1 - \Im \rho_2 \Re \rho_1) ([\omega]^2 \omega \wedge \tilde U_2 - (\deg L + [\omega].U_2) \omega^2) \\
	&+ (\Im \rho_2 \Re \rho_0 - \Re \rho_2 \Im \rho_0)  ((U_4 + [\alpha].U_2 + \frac{[\alpha]^2}{2}) \omega^2-[\omega]^2 \tilde U_4)\\
	&+(\Re \rho_0 \Im \rho_1 - \Im \rho_0 \Re \rho_1) ((U_4 + [\alpha].U_2 +\frac{[\alpha]^2}{2})\omega \wedge \tilde U_2 - (\deg L + [\omega].U_2) \tilde U_4).
\end{align*}

From now on we will make the assumption that $a\ne 0$. In the case of the dHYM equation where $\rho_2 = i, \rho_1 = 1, \rho_0 = -i$ and $U_2=0$, $a\ne 0$ except in the case where $\deg L = 0$, in which case the trivial metric is a dHYM metric.

Given a polynomial
$$p(x) = ax^2 + bx + c$$
with $a\ne 0$ we may rewrite it as 
$$p(x) = \frac{1}{a} \left( \frac{1}{2} p'(x)\right)^2 + c-\frac{1}{4a} b^2.$$
Applying this completion of the square to \eqref{eq:Zcriticalsurfacesimplified}, we can rewrite the $Z$-critical equation on a surface as the Monge--Ampere-type equation
\begin{equation}\label{eq:Zcriticalsurfacemongeampere} \left( a \alpha + \frac{1}{2} \beta\right)^2 = \frac{1}{4} \beta^2 - a\gamma.\end{equation}

\begin{definition}\label{def:volumehypothesis}
	We say $(L,\omega,\tilde U)$ satisfies the \emph{volume hypothesis} if $a\ne 0$ and 
	$$\frac{1}{4} \beta^2 - a\gamma > 0$$
	as a volume form on $X$.
\end{definition}

The importance of the volume hypothesis is contained in the following theorem.

\begin{theorem}\label{thm:existencesurfaces}
	Suppose $L\to X$ is a holomorphic line bundle and $Z$ is a polynomial central charge for which $(L,\omega,\tilde U)$ satisfies the volume hypothesis. Then the following are equivalent.
	\begin{enumerate}
		\item $L$ admits a $Z$-critical metric,
		\item $L$ admits a subsolution to the $Z$-critical equation,
		\item $L$ is $Z$-stable with respect to analytic curves $C\subset X$.
	\end{enumerate}
\end{theorem}
\begin{proof}
	First let us rephrase the subsolution condition. If $L$ admits a subsolution to the $Z$-critical equation, the subsolution condition asks that 
	$$2a\alpha + \beta > 0$$
	in the sense of $(n-1,n-1)$-forms. For two-forms on a surface this is simply asking that $a\alpha + \frac{1}{2}\beta$ is a K\"ahler form.
	
	Now suppose $L$ admits a $Z$-critical metric with curvature $\alpha$ solving \eqref{eq:Zcriticalsurfacemongeampere}. Defining $\chi = a\alpha + \frac{1}{2}\beta$ we have the equation
	$$\chi^2 = \left( \frac{1}{4} \beta^2 - a\gamma\right)$$
	where the right side is a positive volume form on $X$ by the volume hypothesis. Since $X$ is a complex surface, this implies $\chi$ or $-\chi$ is K\"ahler, and hence that $L$ admits a subsolution.\footnote{The case that $-\chi$ is K\"ahler should be excluded by the assumption that the polynomial central charge always lands in the upper half-plane, but as remarked on in \cref{rmk:jacobsheu}, for certain choices of central charge the subsolution condition should appropriately reverse sign to account for this discrepancy. In this case one simply applies the Demailly--Pǎun theorem to $-\chi$ to conclude the theorem holds.} So (i) implies (ii).
	
	On the other hand suppose $\chi$ is K\"ahler for some Hermitian metric $h$ on $L$, that is $L$ admits a subsolution. Then by the volume hypothesis, comparing volume forms on $X$ we may write
	$$\frac{1}{4} \beta^2 - a\gamma = e^F \chi^2$$
	for some function $F: X\to \RR$. Then \eqref{eq:Zcriticalsurfacemongeampere} may be rewritten
	$$(\chi + i\deldelbar \phi)^2 = e^F \chi^2.$$
	By Yau's solution to the Calabi conjecture \cite{yau}, this equation admits a solution whenever $e^F$ integrates to the appropriate constant (which follows in this case from the $Z$-critical equation always holding on the level of cohomology). Thus (ii) implies (i). 
	
	\cref{thm:subsolutiondivisorstability} shows (ii) implies (iii). In order to deduce that $Z$-stability with respect to analytic curves implies the subsolution condition, we must apply Demailly--Pǎun's theorem \cite{DP}, the generalisation of the Nakai--Moishezon criterion to K\"ahler classes. In this setting of a surface, the theorem states that that a class $[\chi]$ on $X$ is K\"ahler if and only if
	$$\int_C \chi > 0$$
	for all analytic curves $C\subset X$. Now if we assume $Z$-stability with respect to analytic curves, then we have
	$$\int_C \Im(e^{-i\varphi(L)} \tilde Z'(h)) > 0$$
	for every analytic curve $C\subset X$, so Demailly--Pǎun's theorem asserts that 
	$$[\Im(e^{-i\varphi(L)} \tilde Z'(h))] = [\chi]$$
	is a K\"ahler class on $X$. 
\end{proof}

\begin{remark}
	In the case of the dHYM equation on a line bundle $L$, the volume hypothesis is automatically satisfied, where it is equivalent to the condition that
	$$(1+(\cot(\varphi(L)))^2) \omega^2 > 0$$
	which is automatic if $\omega$ is a K\"ahler form. In general the volume hypothesis is a positivity condition on the data of the central charge $Z$. For example, consider the case $Z_{\Td}$ on $X$ a Calabi--Yau surface with trivial B-field. Then we have $\rho_2=i, \rho_1=1,\rho_0=-i$ and 
	$$U=\sqrt{\Td(X)} = 1 + \frac{c_2(X)}{24}.$$
	The condition $a\ne 0$ is equivalent to
	$$\deg L \ne 0.$$
	The condition 
	$$\frac{1}{4} \beta^2 - a \gamma > 0$$
	is equivalent to the cumbersome looking requirement that
	$$\frac{1}{4} \left( \left([\omega]^2 - \frac{c_2(X)}{24} + \frac{c_1(L)^2}{2}\right)^2 + \frac{(\deg L)^2}{2}\right) \omega^2 - \frac{(\deg L)^2}{2} \frac {\tilde c_2(X)}{24} > 0$$
	as a $(2,2)$-form on $X$. Indeed for fixed $L$ we can always scale $\omega$ large enough that this form will be positive.\footnote{However it is not necessary to go all the way to the large volume limit, since the size required of $[\omega]$ is bounded in terms of the topology of $L$ and $X$. It would be interesting to have an interpretation of this condition in terms of chambers of the stability manifold of the surface $X$.} Then \cref{thm:existencesurfaces} provides new examples of $Z_{\Td}$-critical metrics on line bundles in this case.
	
\end{remark}

\section{Moment map\label{sec:momentmap}}

In this section we will exhibit the $Z$-critical equation as a moment map on the space $\calA(h)$ of integrable unitary connections on a holomorphic vector bundle $E\to (X,\omega)$. 

Recall that the tangent space to $\calA(h)$ at a connection $A$ consists of skew-Hermitian endomorphisms $a\in \Omega^1(\End_{SH}(E,h))$ such that $(\delbar_A + a^{0,1})^2 = 0$. Let $Z$ be a polynomial central charge with representative data $\omega$ and $\tilde U$. Define a two-form on $\calA(h)$ by the expression
\begin{equation}\label{eq:Kahlerform}\Omega_Z(A)(a,b) = \int_X \trace \left[a, b , \Im(e^{-i\varphi(E)} \tilde Z'(A)) \right]_{\sym}\end{equation}
for $A\in \calA(h)$, $a,b\in T_A\calA(h)$. By the definition of the graded symmetrisation, this is a skew bilinear form.

Associated to $\Omega_Z$ is a Hermitian pairing. Recall that there is a complex structure $J$ on $\calA(h)$ defined by
$$J(a) = -i a^{1,0} + i a^{0,1}$$
for $a\in T_A\calA(h) = \Omega^1(\End_{SH}(E,h)).$ The two-form $\Omega_Z$ is compatible with $J$ in the sense that $J^* \Omega_Z = \Omega_Z$ (which can be verified using the fact that $\Im(e^{-i\varphi(E)} Z'(A))$ is type $(n-1,n-1)$). Therefore we can define a Hermitian pairing 
$$\langle a,b\rangle_A = \Omega_Z(a,J(b)) + i \Omega_Z(a,b).$$
Explicitly one may compute
$$\langle a,b \rangle_A = i \int_X \trace \left[ \Im(e^{-i\varphi(E)} \tilde Z'(A)), a^{1,0}, b^{0,1}\right]_{\sym}.$$

\begin{proposition}\label{prop:subsolutionnondegenerate}
	If $A$ is a subsolution of the $Z$-critical equation, then the two-form $\Omega_Z$ is non-degenerate at $A$.
\end{proposition}
\begin{proof}
	The subsolution condition at $A$ gives
	$$i\trace \left[ \Im(e^{-i\varphi(E)} \tilde Z'(A), u^*, u\right]_{\sym} > 0$$
	at every point $p\in X$. Integrating this gives the positivity condition
	$$i \int_X \trace  \left[ \Im(e^{-i\varphi(E)} \tilde Z'(A), u^*, u\right]_{\sym} > 0$$
	for any non-zero $u\in \Omega^{0,1}(\End E)$. Setting $u=a^{0,1}$ where $a\in T_A \calA(h)$ we obtain
	$$\langle a,a \rangle_A \ge 0$$
	for any $a\in T_A\calA(h)$, and equality occurs if and only if $a=0$. Therefore $\langle -, -\rangle_A$ is a positive-definite Hermitian inner product on $T_A\calA(h)$, and consequently its imaginary part $\Omega_Z$ is a non-degenerate two-form.
\end{proof}

Let us denote $\calA(h)^{Z} \subset \calA(h)$ the locus of subsolutions to the $Z$-critical equation, which is open since the subsolution condition is an open condition. Then \cref{prop:subsolutionnondegenerate} asserts that $\calA(h)^Z$ is an infinite-dimensional Hermitian manifold. Notice that this locus is invariant under the action of the complex gauge group $\calG^\CC$ (as can be easily checked from the definition of subsolution, using the presense of the trace). We will now show that $\Omega_Z$ is in fact a $\calG(h)$-invariant closed form on $\calA(h)^Z$, and therefore $\calA(h)^Z$ is an infinite-dimensional K\"ahler manifold. In fact in the following it is not necessary to restrict to the locus of subsolutions, and $\Omega_Z$ is closed and $\calG(h)$-invariant on all of $\calA(h)$.

\begin{proposition}\label{prop:Kahlerformclosed}
	The two-form $\Omega_Z$ is closed and $\calG(h)$-invariant on $\calA(h)$. Consequently it is a K\"ahler form on $\calA(h)^Z$.
\end{proposition}
\begin{proof}
	The $\calG(h)$-invariance of $\Omega_Z$ is clear from the presence of the trace in \eqref{eq:Kahlerform}. Indeed if $g\in \calG(h)$ then 
	$$g^* \Omega_Z(A)(a,b) = \Omega_Z(g\cdot A)(\Ad(g)a, \Ad(g)b)$$
	and using the fact that for a unitary gauge transformation we have
	$$F_{g\cdot A} = g F_A g^{-1}$$
	and 
	$$\Ad(g) a = g a g^{-1}$$
	then 
	$$ g^* \Omega_Z(A)(a,b) = \int_X \trace g \left[ \Im(e^{-i\varphi(E)} \tilde Z'(A), a, b \right]_{\sym} g^{-1}= \Omega_Z(A)(a,b).$$
	
	We will verify the closedness of $\Omega_Z$ in the space of all unitary connections for $h$. In particular we will consider vector fields $a,b,c\in \Gamma(T\calA(h))$ constant with respect to $A\in \calA(h)$. Therefore the Lie brackets vanish:
	$$[a,b]=[a,c]=[b,c]=0.$$
	To show the closedness of $\Omega_Z$ we will show
	$$(d\Omega_Z)(A)(a,b,c)=0.$$
	Since the Lie brackets vanish, we must verify 
	$$(d\Omega_Z)(A)(a,b,c) = d(\Omega_Z(A)(b,c))(a) - d(\Omega_Z(A)(a,c))(b) + d(\Omega_Z(A)(a,b))(c) = 0.$$
	In order to differentiate $F_A$ in the direction of $a$, we take the constant path $A+ta$. Then the curvature transforms as 
	$$F_{A+ta} = F_A + t d_A(a) + t^2 a\wedge a$$
	and so the curvature linearises as $d_A(a)$. In particular an arbitrary term in, for example, $d(\Omega_Z(A)(b,c))(a)$ is of the form
	$$\int_X \trace \left[ \omega^{n-j-k} \wedge \tilde U_k \wedge \left(\frac{i}{2\pi} F_A\right)^{j-2} \wedge \frac{i}{2\pi} d_A(a) \wedge b \wedge c \right]_{\sym}.$$
	Summing this expression over all cyclic permutations of $\{a,b,c\}$ with sign according to the graded symmetrisation we obtain 
	$$d(\Omega_Z)(A)(a,b,c) = \frac{i}{2\pi}\int_X \trace d_A \left[ \Im(e^{-i\varphi(E)} \tilde Z^{(2)}(A), a, b, c\right]_{\sym}$$
	where $\tilde Z^{(2)}(A)$ is the second formal derivative of $\tilde Z(A)$ with respect to $iF_A/2\pi$. Here we have used the Bianchi identity $d_A F_A = 0$ and the closedness of $\omega$ and $\tilde U$. 
	
	Now since $\trace d_A = d \trace$, Stokes' theorem implies 
	$$d(\Omega_Z)(A)(a,b,c) = 0$$
	so $\Omega$ is closed.
\end{proof}

\begin{remark}
	\cref{prop:Kahlerformclosed} clearly demonstrates the necessity of the graded symmetrisation introduced when studying subsolutions in \cref{def:symmetrisation}. The necessity of this symmetrisation is apparent in Leung's work on the moment map for the almost Hermite--Einstein equation described in \cref{sec:AHE}. In fact in that case one can also prove the closedness of the corresponding symplectic form by exploiting that it is the curvature of a connection on a universal determinant line bundle over $\calA(h)$, and the symmetrisation manifests in the formula for the curvature of the universal connection. 
\end{remark}

We will now show that the $Z$-critical equation is a moment map for the action of $\G(h)$ on $\calA(h)^Z$ with respect to the K\"ahler form $\Omega_Z$. It is customary to formally identify the dual space of $\Lie \G(h) \isom \End_{SH}(E,h)$ with $\Omega^{n,n}(\End_{SH}(E,h))$ using the non-degenerate pairing
\begin{equation}\label{eq:pairing}(\phi, \psi) \mapsto -\int_X \trace \phi \psi.\end{equation}

\begin{theorem}\label{thm:momentmap}
	The map
	\begin{align*}
		D_Z: \calA(h)^Z &\to \Omega^{n,n}(\End_{SH}(E,h))\\
		A&\mapsto 2\pi i \Im(e^{-\varphi(E)} \tilde Z(A))
	\end{align*}
	is a moment map for the $\calG(h)$-action on $(\calA(h)^Z,\Omega_Z)$. 
\end{theorem}
\begin{proof}
	The proof is a reasonably straight-forward adapation of the Collins--Yau moment map for the dHYM equation on a line bundle \cite[\S 2]{collins2018moment} and Leung's proof that the almost HE equation is a moment map \cite[\S 3]{leung1998symplectic}. Let $\phi\in \Lie \G(h)$ be a skew-Hermitian endomorphism of $E$. Then by the infinitesimal action obtained by differentiating the action of $\calG(h)$ on $\calA(h)$ (see \eqref{eq:gaugeaction}) the induced vector field is given by $d_A \phi$. We wish to verify the moment map identity
	$$d(D_Z, \phi)(A)(b) = - \Omega_Z(A) (d_A \phi, b)$$
	for any $b\in T_A \calA(h)$. On the right-hand side we compute
	\begin{align*}
		\Omega_Z (d_A \phi, b) &= \int_X \trace \left[ \Im(e^{-i\varphi(E)} \tilde Z'(A)) \wedge d_A \phi \wedge b \right]_{\sym}\\
		&= - \int_X \trace \left[ \Im(e^{-i\varphi(E)} \tilde Z'(A)) \phi \wedge d_A b \right]_{\sym}
	\end{align*}
	using integration by parts. Explicitly we have
	\begin{align*}
		0 &= \int_X d \trace \left[ \Im(e^{-i\varphi(E)} \tilde Z'(A) \wedge \phi  b \right]_{\sym}\\ 
		&= \int_X \trace d_A \left[ \Im(e^{-i\varphi(E)} \tilde Z'(A) \wedge \phi  b \right]_{\sym}\\
		&= \int_X \trace \left[ d_A \left(\Im(e^{-i\varphi(E)} \tilde Z'(A) \right) \wedge \phi  b \right]_{\sym} + \int_X \trace \left[ \left(\Im(e^{-i\varphi(E)} \tilde Z'(A) \right) \wedge d_A (\phi  b) \right]_{\sym}\\
		&= \int_X \trace \left[ \Im(e^{-i\varphi(E)} \tilde Z'(A)) \wedge d_A \phi \wedge b \right]_{\sym} + \int_X \trace \left[ \Im(e^{-i\varphi(E)} \tilde Z'(A)) \phi \wedge d_A b \right]_{\sym}
	\end{align*}
	where we have used in the first step Stokes' theorem, in the second step that $\trace d_A = d \trace$, in the third step the product rule for $d_A$ combined with the Bianchi identity $d_A F_A=0$ and closedness of $\omega$ and $\tilde U$ for each term appearing in $\Im(e^{-i\varphi(E)} \tilde Z'(A))$, and in the final step the product rule $d_A(\phi b) = d_A(\phi) \wedge b + \phi \wedge d_A(b)$.
	
	Now recall that if $F_A$ is a curvature form then in the direction of a tangent vector $b$, the curvature transforms as $F_{A+tb} = F_A + t d_A(b) + t^2 b\wedge b$. In particular computing the exterior derivative on the left-hand side similarly to in \cref{prop:Kahlerformclosed} we have, using our non-degenerate pairing \eqref{eq:pairing},
	\begin{align*}
		d(D_Z, \phi)(A)(b) &= -2\pi i \left.\deriv{}{t}\right|_{t=0} \int_X \trace \Im(e^{-i\varphi(E)} \tilde Z(A+tb)) \phi\\
		&= -2\pi i \int_X \trace \left[\Im (e^{-i\varphi(E)} \tilde Z'(A)) \wedge \frac{i}{2\pi} d_Ab \right]_{\sym} \phi\\
		&=  \int_X \trace \left[\Im (e^{-i\varphi(E)} \tilde Z'(A)) \wedge d_Ab \phi \right]_{\sym}.
	\end{align*}
	Thus we have
	$$d(D_Z,\phi)(A)(b) = -\Omega_Z(d_A \phi, b)$$
	as desired. The $\calG(h)$-equivariance of $D_Z$ is straightforward and similar to the computation of $\calG(h)$-invariance of $\Omega_Z$ in \cref{prop:Kahlerformclosed}.
\end{proof}

\begin{remark}
	Note that the subsolution condition was not used in the proof of \cref{thm:momentmap}. In particular the moment map condition is still formally satisfied on all of $\calA(h)$ for the possibly degenerate form $\Omega_Z$. That is, the form $\Omega_Z - D_Z$ is $\calG(h)$-equivariantly closed on $\calA(h)$. 
\end{remark}

\begin{remark}
	This moment map formalism is new even in the case of dHYM on a line bundle with non-zero B-field. The moment map formalism of Collins--Yau only considered the case of dHYM for $\alpha = \frac{i}{2\pi} F(h)$, although the calculations there could have been extended in a straightforward way to include the case of non-zero B-field. 
\end{remark}

\begin{remark}
	A moment map criterion of this form for gauge-theoretic equations $f(F_A)=0$ on vector bundles was predicted by Thomas \cite[Eq. 38.11]{clay}.
\end{remark}

\subsection{Variational functional\label{sec:variationalfunctional}}

As we observed in \cref{ch:preliminaries} associated to a moment map formalism is a Kempf--Ness-type functional $\calM$, which is formally defined by its first variation. Let us deduce the Kempf--Ness functional for the $Z$-critical equation as a moment map. 

We work now in a fixed complex gauge orbit inside $\calA(h)$, and take the view of fixing the holomorphic structure on $E\to (X,\omega)$ and varying the Hermitian metric $h$. First we define a functional $\calM(h,h')$ of a pair of Hermitian metrics. To do so take the $Z$-critical operator $D_Z(h)$ and replace any form
$$\gamma \wedge \left( \frac{i}{2\pi} F(h)\right)^k$$ 
with 
$$\gamma \wedge R_{k+1}(h,h')$$
where $R_k(h,h')\in \Omega^{k-1,k-1}$ is the $k$th Bott--Chern form (see for example \cite[\S 1.2]{donaldson1985anti} or \cite[\S 6.3]{kobayashi1987differential}). We interpret terms of the form $\gamma \otimes \id_E$ as $\gamma \wedge \left(\frac{i}{2\pi} F(h)\right)^0$. Recall that the Bott-Chern forms are characterised by the following properties:
\begin{enumerate}
	\item $R_k(h,h) = 0$ and $R_k(h,h') + R_k(h',h'') = R_k(h,h'')$,
	\item $\deriv{}{t} R_k(h_t, h') =  \left[h_t^{-1} \del_t h_t \left( \frac{i}{2\pi} F(h_t)\right)^{k-1}\right]_{\sym}$,
	\item $i \deldelbar R_k(h,h') = \left( \frac{i}{2\pi} F(h)\right)^k - \left( \frac{i}{2\pi} F(h')\right)^k$.
\end{enumerate}

Write $\tilde \calM_Z(h,h')$ for the resulting expression, and define a functional
$$\calM_Z(h,h') := \int_X \trace \tilde \calM_Z(h,h').$$

\begin{theorem}\label{thm:ZDonaldsonfunctionalfirstvariation}
	Fix a reference metric $\tilde h$. Then $h$ is a critical point of $\calM_Z(h,\tilde h)$ if and only if it is $Z$-critical.
\end{theorem}
\begin{proof}
	Let $h_t$ be a family of Hermitian metrics on $E$ with $h_0 = h$ and compute the first variation 
	$$\deriv{}{t} \calM(h_t,\tilde h) =  \int_X \trace\left( h_t^{-1} \del_t h_t \circ D_Z(h_t) \right).$$
	Near $t=0$ we may consider a path $h_t=\exp(tV)\cdot h$ in which case $h_t^{-1} \del_t h_t = V$. Using the non-degeneracy of the trace pairing we see the first variation vanishes for all such $V$ if and only if $D_Z(h) = 0$.
\end{proof}

\begin{theorem}
	The functional $\calM_Z$ is convex along geodesics on the locus $\Herm(E)^Z\subset \Herm(E)$ of subsolutions to the $Z$-critical equation.
\end{theorem}
\begin{proof}
	The condition for a path $h_t$ of metrics in $\calA(h)$ to be geodesic is that $\del_t (h_t^{-1} \del_t h_t) = 0$ (see \cite[Eq. 6.1.11]{kobayashi1987differential}). Note in particular this is satisfied for $h_t = \exp(tV) \cdot h$ and in fact all geodesics are locally of this form. We need to compute the second variation 
	$$\frac{d^2}{dt^2} \calM(h_t,\tilde h).$$
	From the expression in \cref{thm:ZDonaldsonfunctionalfirstvariation} and the geodesic condition applied to $h_t = \exp(tV) \cdot h$ we obtain
	\begin{align*}
		\left.\frac{d^2}{dt^2}\right|_{t=0} \calM(h_t,\tilde h) &= \int_X \trace V \left[ \Im(e^{-i\varphi(E)} \tilde Z'(h)) \wedge \frac{i}{2\pi} \delbar \del_h V \right]_{\sym}\\
		&=\int_X \trace  \left[ \Im(e^{-i\varphi(E)} \tilde Z'(h)) \wedge \frac{i}{2\pi} (\delbar \del_h V) V \right]_{\sym}.
	\end{align*}
	Here we have used the linearisation of the curvature under a change of metric identified in \cref{lem:linearisationcurvatureMetric}, which naturally produces the subsolution expression as discussed in \cref{sec:subsolutions}. Integrating by parts we obtain
	\begin{align*}
		\left.\frac{d^2}{dt^2}\right|_{t=0} \calM(h_t,\tilde h) &= \int_X \trace  \left[ \Im(e^{-i\varphi(E)} \tilde Z'(h)) \wedge \frac{i}{2\pi} (\del_h V) \wedge \delbar V \right]_{\sym}\\
		&= \frac{1}{2\pi} \int_X \trace  \left[ \Im(e^{-i\varphi(E)} \tilde Z'(h)) \wedge i (\del_h V)^* \wedge \delbar V \right]_{\sym}\\
		& > 0
	\end{align*}
	where we have set $u=\delbar V$ and used the fact that $(\del_h V)^* = \delbar V$ since $V$ is Hermitian, and finally the assumption that $h\in \Herm(E)^Z$ is a subsolution to obtain strict positivity.\footnote{Note that this expression is just $\frac{1}{2\pi}$ times the norm $\|\delbar_h V\|_Z^2$ for the Hermitian pairing $\langle -,-\rangle_Z$ defined in \cref{sec:momentmap}, which is non-degenerate on the locus of subsolutions.}
\end{proof}

Previously a functional of a similar form has appeared for the dHYM equation and was introduced by Collins--Yau \cite{collins2018moment}. They refer to the \emph{complexified Calabi--Yau functional} $\mathrm{CY}_\CC$ defined by the first variation
$$d\mathrm{CY}_\CC (\varphi)(\psi) = \int_X \psi(\omega + i (\alpha+i\deldelbar \varphi))^n$$
where $\alpha$ is some closed real $(1,1)$-form and $\varphi,\psi$ are smooth functions on $(X,\omega)$. Indeed we remark that the property (iii) for the Bott--Chern classes for a pair of Hermitian metrics $h,h'$ on a line bundle $L\to X$ shows that, for example, $R_1(h,h') = \varphi$ where $F(h') = F(h) - 2\pi \deldelbar \varphi$. Indeed replacing $\alpha = F(\tilde h)$ for some reference metric $h$ we see that $\calM_Z$ recovers $\mathrm{CY}_\CC$ in the case where $Z=Z_{\mathrm{dHYM}}$ is the dHYM central charge on a line bundle.

Another variational framework was proposed earlier by Jacob--Yau \cite{jacob2017special} (see also \cite[\S 2]{collins2018moment} for further discussion) where it is observed that dHYM metrics are absolute minimisers of the functional
$$V(\alpha) = \int_X |\tilde Z_{\mathrm{dHYM}}(\alpha)|.$$
That is, dHYM metrics define forms $\tilde Z(\alpha)$ with constant phase angle, and which minimise the corresponding radial norm. The functional $V$ is the ``Yang--Mills functional" in comparison to $\calM_Z$ being the ``Donaldson functional" of the problem. 

In general one might define a ``$Z$-Yang--Mills" functional as follows.

\begin{definition}
	Let $Z$ be a polynomial central charge and fix representative data $\omega, \tilde U$. Let $A$ be a Chern connection and write $\tilde Z(A) = \zeta \frac{\omega^n}{n!}$ for an endomorphism $\zeta\in \Gamma(\End E)$.
	Define the \emph{$Z$-Yang--Mills functional} by
	$$Z\YM(A) = \int_X |\zeta(A)|^2 \frac{\omega^n}{n!}$$
	where we use the trace pairing $\langle \zeta, \xi \rangle = \trace (\zeta \xi^*)$ on endomorphisms.
\end{definition}

\begin{theorem}
	A Chern connection $A$ in $\calA(h)$ is a minimum of the $Z\YM$ functional if and only if it is $Z$-critical and
	\begin{equation}\label{eq:Zcriticalrealpart}\Re(e^{-i\varphi(E)} \tilde Z(A)) = \lambda \id_E \otimes \omega^n\end{equation} 
	for some constant $\lambda\in \RR$. On the locus of $Z$-almost calibrated connections $\lambda > 0$.
\end{theorem}
\begin{proof}
	First note that we have
	$$Z\YM(A) = \int_X |\zeta(A)|^2 \frac{\omega^n}{n!} = \int_X |e^{-i\varphi(E)} \zeta(A)|^2 \frac{\omega^n}{n!}.$$
	Now decompose $e^{-i\varphi(E)} \zeta = \Re (e^{-i\varphi(E)} \zeta) + i \Im (e^{-i\varphi(E)} \zeta)$ and consider the inequality
	$$0 \le \int_X |e^{-i\varphi(E)} \zeta - \Re (e^{-i\varphi(E)} \zeta)|^2  \frac{\omega^n}{n!}.$$
	Then on the right hand side we have
	$$Z\YM(A) + \|\Re(e^{-i\varphi(E)} \zeta)\|^2 - \langle e^{-i\varphi(E)} \zeta, \Re(e^{-i\varphi(E)} \zeta)\rangle - \langle \Re(e^{-i\varphi(E)} \zeta), e^{-i\varphi(E)} \zeta\rangle$$
	and expanding the inner products and using the sesquilinear property this reduces to
	$$Z\YM(A) - \|\Re(e^{-i\varphi(E)} \zeta)\|^2$$
	so we obtain the inequality 
	$$ \|\Re(e^{-i\varphi(E)} \zeta)\|^2 \le Z\YM(A)$$
	with equality if and only if $A$ satisfies $\Im(e^{-i\varphi(E)} \zeta) = 0$, that is if $A$ is $Z$-critical. Now using the fact that $\Re(e^{-i\varphi(E)} \zeta)$ is Hermitian with respect to $h$, and applying Cauchy--Schwarz we obtain
	$$\left|\int_X \trace \Re(e^{-i\varphi(E)} \zeta) \frac{\omega^n}{n!} \right|^2 = |\langle \Re(e^{-i\varphi(E)} \zeta), \id_E \rangle|^2 \le \rk (E) \vol(X) \|\Re(e^{-i\varphi(E)} \zeta)\|^2.$$
	Thus we get the topological bound
	$$\frac{\Re(e^{-i\varphi(E)} Z(E))^2}{\rk (E) \vol(X)} \le Z\YM(A)$$
	with equality if and only if $A$ is $Z$-critical and $\Re(e^{-i\varphi(E)} \zeta)$ is a constant multiple of $\id_E$.\footnote{It might be interesting to the reader to think about this inequality in the case where $Z(E) = -\deg E + i \rk E$, in which case the $Z\YM$ functional is essentially $1 + \YM(A)$ and the bound on the left-hand side is essentially $1 + \frac{\deg E^2}{\rk E}$. Thus we recover the traditional lower bound on the Yang--Mills functional \cite[Thm. 4.3.9]{kobayashi1987differential} which is attained if and only if a connection is Hermitian Yang--Mills.} 
\end{proof}

\begin{remark}
	The above theorem suggests that the definition of a $Z$-critical connection may be modified to include the stronger Einstein-type condition on the real part of $e^{-i\varphi(E)} \tilde Z(A)$. Note that in the case of line bundles this second equation is automatically implied by the dHYM equation due to the rank being one, so for example does not appear in Jacob--Yau's analysis of the minima of the $Z\YM$ functional for the dHYM equation \cite{jacob2017special}. 
	
	It may be interesting to determine if a more careful perturbation argument for the existence of solutions proved in the large volume limit in \cref{ch:correspondence} could find a solution satisfying this stronger system of equations. One expects this to involve a more careful understanding of the errors introduced when solving the $Z$-critical equation and how they change the expression $\Re(e^{-i\varphi(E)} \tilde Z(A))$. For example note that to leading order in $k$ this is a constant multiple of $\id_E$, and to subleading order is a constant multiple of the Hermite--Einstein equation, so an approximate solution to the $Z$-critical equation is also an approximate solution to \eqref{eq:Zcriticalrealpart} to second order. If such an approximate solution can be improved (or is manifestly a higher order approximate solution due to the structure of the $Z$-critical equation) then a similar perturbation argument may solve the stronger system of equations.
\end{remark}

\begin{remark}
	Notice that in the proof of the moment map criterion in \cref{thm:momentmap}, one could have equally considered the ``complex moment map"
	$$\mu: A\mapsto e^{-i\varphi(E)}  \tilde Z(A)$$
	which formally satisfies the moment criterion with respect to the alternating form
	$$(a,b) \mapsto \int_X \trace \left[a, b, e^{-i\varphi(E)} \tilde Z'(A) \right]_{\sym}.$$
	This represents the ``complexified K\"ahler geometry" upgrade of the Yang--Mills functional and Atiyah--Bott symplectic structure of \cref{sec:bundles}, and in particular we note
	$$Z\YM = \|\mu\|^2$$
	and we have essentially ``taken the imaginary part" of this complex moment map in \cref{thm:momentmap}. It would be interesting to understand if the complexified K\"ahler geometry point of view for the Yang--Mills functional can be pushed further to obtain analogues of the many interesting facts about $\YM$ in the complexified setting.
\end{remark}

\section{Future directions\label{sec:Zcriticalfuturedirections}}

In this section we will discuss some future directions of research in the study of $Z$-critical metrics. 

\subsection{Modification due to Freed--Kapustin--Witten anomaly\label{sec:counterexample}}

In this section we will discuss a counterexample to \cref{conj:folklore} as it relates to the dHYM equation and Bridgeland stability, identified by Collins--Shi \cite{collins2020stability}. We will investigate how a modification to the stability condition which has been suggested in the physics literature appears to resolve this counterexample.

Recall that the existence of solutions to the $Z$-critical equation away from the large volume limit is obstructed both by subbundles and subvarieties. The way this is predicted to manifest in Bridgeland stability is due to destabilising quotient objects of the form
$$E\to E\otimes \calO_V$$
where $V\subset X$ is a proper subvariety of $X$, or more generally quotients 
$$E\to F$$
where $F$ is a torsion coherent sheaf.

Correspondingly, recall \cref{conj:collinsjacobyaudHYM} and its resolution in the supercritical phase \cref{thm:dhymexistencesupercritical} provide an algebro-geometric stability condition from subvarieties which is equivalent to the existence of a dHYM metric on a line bundle.  However, Collins--Shi have produced an explicit counterexample to \cref{conj:folklore} on $\Bl_p \CCPP^2$, a line bundle $L$ which is Bridgeland stable with respect to the dHYM stability condition described in \cref{ex:dhymcentralcharge}, but does not admit a solution to the dHYM equation \cite{collins2020stability}. In particular, the line bundle $L$ fails to satisfy the algebro-geometric stability condition equivalent to the dHYM equation, which depends on the invariant $Z_C(L)$, but \emph{is} Bridgeland stable, and in particular satisfies stability with respect to quotients $L\to L\otimes \calO_C$ for analytic curves $C\subset \Bl_p \CCPP^2$. 

The key point of this discrepancy is that the algebraic invariants
$$Z_V(L) := -\int_V e^{-i\omega} \Ch(L)\text{ and }Z_X(L\otimes \calO_V) := -\int_X e^{-i\omega} \Ch(L\otimes \calO_V)$$
do not agree. Indeed the difference can be easily computed from the Riemann--Roch theorem as follows.

\begin{lemma}\label{lem:cherncharacterstructuresheaf}
	Let $\iota: C\subset X$ be a smooth curve in a surface $X$. Then 
	$$\Ch(\iota_* \calO_C) = \iota_* \left(1 - \frac{c_1(K_C\otimes \rest{K_X}{C}^*)}{2}\right).$$
\end{lemma}
\begin{proof}
	We note that by the Grothendieck--Riemann--Roch formula, which applies even to non-projective compact complex manifolds by \cite{ATIYAH1962151}, for the embedding $\iota$ we have
	$$\Ch(\iota_* \calO_C) = \iota_* (\Td(N_C)^{-1} \Ch(\calO_C))$$
	where $N_C$ is the normal bundle to $C$. Recalling the short exact sequence
	$$0\to N_C^* \to \iota^* \Omega_X \to K_C \to 0$$
	and taking determinants we have
	$$\rest{K_X}{C} \isom N_C^* \otimes K_C$$
	so $N_C \isom K_C \otimes \rest{K_X}{C}^*$. Computing the inverse of the Todd class
	$$\Td(N_C)^{-1} = 1 - \frac{c_1(K_C)}{2} + \frac{c_1(\rest{K_X}{C})}{2}.$$
\end{proof}

By the lemma, we note that
$$Z_X(L\otimes \calO_C) = -\int_C e^{-i\omega} \Ch(L\otimes K_C^{-1/2} \otimes \rest{K_X}{C}^{1/2}) $$
and thus we see the discrepancy between the algebraic invariants arising from Bridgeland stability and the dHYM equation come from the class $-\frac{1}{2} ( c_1(K_C) - c_1(\rest{K_X}{C}))$. 

It has been suggested in the physics literature (see \cite{minasian1997k,katz2002d,caldararu2003d}, \cite[\S 5.4]{aspinwall2004d}, and \cite[\S 5.1.4, \S 5.3.3.4]{aspinwall2009dirichlet}) that taking into account the Freed--Kapustin--Witten anomaly (see for example \cite{freed1999anomalies,kapustin2000d} and the discussion there) and a detailed understanding of the Ext groups defining morphisms of intersecting D-branes, the coherent sheaf on $X$ corresponding to a D-brane on a submanifold $\iota: Y \into X$ with Chan--Paton bundle $E\to Y$ should not be $\iota_* E$, but $\iota_* (E\otimes K_Y^{1/2}).$

Note that in general the square root of the canonical bundle of an arbitrary $Y\subset X$ may not exist. Indeed recall the following fundamental theorem from spin geometry.

\begin{theorem}[For example \cite{hitchin1974harmonic}]
	A compact K\"ahler manifold $X$ is spin if and only if the canonical bundle $K_X$ admits a holomorphic square root.
\end{theorem}

Whilst not every complex submanifold is spin, they are all Spin\textsuperscript{C}, and the holomorphic square root $K_Y^{1/2}$ exists as a gerbe. In this case, the Freed--Witten anomaly in the presence of a topologically trivial B-field asserts exactly that if $\calE$ is a D-brane on $X$ which is locally free supported on a submanifold $Y$, and $Y$ is only Spin\textsuperscript{C}, then there exists a twisted bundle $E'$ (twisted by a gerbe) precisely so that
$$\iota^* \calE = E' \otimes K_Y^{-1/2}$$
where $\iota^* \calE$ is a regular vector bundle.

This suggests that the correct formalism for studying the dHYM equation or $Z$-critical equation in order to get a direct correspondence with Bridgeland stability as predicted in \cref{conj:folklore} is not connections on vector bundles, but connections twisted by Spin\textsuperscript{C} structures, or twisted Dirac operators. The key property these objects must have is that restriction to a submanifold requires a twisting which precisely introduces a factor of $\Td(N)^{-1}$ so that the natural algebro-geometric invariants arising out of the subsolution conditions such as \cref{def:subsolution} agree with those from Bridgeland stability.

Let us consider the case of a locally free sheaf $\calE$ on a manifold $(X,\omega)$. Then associated to $\calE$ is the gauge ``bundle" $E= \calE \otimes K_X^{-1/2}.$ Suppose we wish to consider an obstruction from a quotient $\calE \onto \calE \otimes \calO_V$ for some submanifold $\iota: V\into X$. Then the appropriate gauge ``bundle" corresponding to $\calE \otimes \calO_V$ is $E' = \calE \otimes K_V^{-1/2}$. 

Thus upon restriction of twisted bundles $\iota^* E = E' \otimes \det N_V^{-1/2}$. We propose that this twisting operation is the first approximation to the correct differential-geometric operation which captures stability with respect to quotients $E\to E\otimes \calO_V$. In order to justify this proposal, consider the following straightforward observation.

\begin{proposition}
	Let $C$ be a smooth curve on a Calabi--Yau surface $X$ and $E\to X$ a vector bundle. Then 
	$$Z_C(E\otimes K_C^{-1/2}) = Z(E\otimes \calO_C).$$
\end{proposition}
\begin{proof}
	First let us note that
	$\Ch(K_C^{-1/2}) = 1 - \frac{c_1(K_C)}{2}.$
	By the preceeding calculation in \cref{lem:cherncharacterstructuresheaf} for the Todd class this gives
	$$\Ch(K_C^{-1/2}) = \Td(N_C)^{-1}.$$
	This implies
	\begin{align*}
		Z(E\otimes \calO_C) &= -\int_X e^{-i\omega} \Ch(E \otimes \calO_C)\\
		&= -\int_C e^{-i\omega} \Ch(E) \Td(N_C)^{-1}\\
		&= -\int_C e^{-i\omega} \Ch(E\otimes K_C^{-1/2})\\
		&= Z_C(E\otimes K_C^{-1/2}).
	\end{align*}	
\end{proof}

By a similar computation as above but for Calabi--Yau threefolds, one obtains the following:

\begin{proposition}
	Let $X$ be a Calabi--Yau threefold, and $C\subset X$ a smooth curve. Then 
	$$Z_C(E\otimes K_C^{-1/2}) = Z(E\otimes \calO_C)$$
	for any vector bundle $E\to X$.
\end{proposition}
\begin{proof}
	This follows from the same calculation as a curve in a Calabi--Yau surface, but now the normal bundle $N_C$ is rank two. However only the first Chern class of the determinant $\det N_C$ enters due to dimensional reasons, and the same formula $\Ch(K_C^{-1/2}) = \Td(N_C)^{-1}$ holds.
\end{proof}

However, in dimension three we observe a failure of the anomaly condition to produce this matching for divisors. This suggests in general a more sophisticated understanding of the restriction of these twisted connections is necessary to accurately capture the stability condition.

\begin{proposition}
	Let $X$ be a Calabi--Yau threefold, and $D\subset X$ a divisor. Then 
	$$\Ch(\iota_* K_D^{1/2}) = \iota_* \left(1 + \frac{1}{24} c_1(K_S)^2\right).$$
	In particular 
	$$Z_D(E) \ne Z(E\otimes K_D^{1/2})$$
	in general for vector bundles $E\to X$. 
\end{proposition}
\begin{proof}
	The key point is to observe that 
	$$\Ch(\iota_* K_D^{1/2}) = \iota_* (\Td(N_D)^{-1} \Ch(K_D^{1/2}))$$
	and one may compute
	$$\Td(N_D)^{-1} = 1 - \frac{c_1(K_S)}{2} + \frac{c_1(K_S)^2}{6}$$
	whilst we have
	$$\Ch(K_D^{1/2}) = 1 + \frac{c_1(K_S)}{2} + \frac{c_1(K_S)^2}{8}.$$
	Thus we see 
	$$\iota_* (\Td(N_D)^{-1} \Ch(K_D^{1/2})) = \iota_*\left(1 + \left(\frac{1}{8} + \frac{1}{6} - \frac{1}{4}\right) c_1(K_S)^2\right).$$
\end{proof}

\subsection{Categorical K\"ahler geometry}

An upcoming proposal\footnote{The author thanks Pranav Pandit for explaining many of the aspects of this proposal to them.} of Haiden--Katzarkov--Kontsevich--Pandit (mentioned for example in \cite{haidensemistability}) suggests that Bridgeland stability conditions may be defined on the derived category through an analogue of geometric invariant theory in the derived setting. In particular the proposal, dubbed ``categorical K\"ahler geometry", hopes to make sense of the following:

\begin{itemize}
	\item A categorical notion of ``ample line bundle" on $\DbCoh(X)$. This should be essentially equivalent to the notion of a Bridgeland stability condition. 
	\item A categorical notion of ``K\"ahler metric" on $\DbCoh(X)$.
	\item Bridgeland stability of an object $\calE$ with respect to some condition $(Z,\calA)$ corresponds to a categorical GIT stability with respect to the ample line bundle.
	\item There exists a flow for the categorical K\"ahler metric on any object $\calE \in \DbCoh(X)$ which converges to a metric on the Harder--Narasimhan filtration of $\calE$.
\end{itemize}

In particular the proposal of HKKP relies on a notion of a space $\Met(\calE)$ of metrics for each object $\calE\in \DbCoh(X)$, and eventually on an analogue of the Donaldson--Uhlenbeck--Yau theorem which would allow a geometric flow to choose out the algebraically stable components of the Harder--Narasimhan filtration of $\calE$. 

This ambitious proposal was studied in an elementary form in \cite{haidensemistability} where the way that canonical Jordan--H\"older filtrations of semistable objects arise from geometric flows was studied. In particular they consider the the case of the Hermitian Yang--Mills flow on a compact Riemann surface, and flows for Hermitian metrics on quiver representations, and show how the asymptotic convergence rates of the flow cause the filtration to manifest.

Under any reasonable upcoming interpretation of HKKPs proposal, we note:
\begin{itemize}
	\item The categorical K\"ahler metric on $\DbCoh(X)$ depending on a stability condition $(Z,\calA)$ should be approximated by the K\"ahler forms $\Omega_Z$ on the spaces of integrable unitary connections $\calA(h)$ on vector bundles considered in \cref{sec:momentmap}.
	\item The categorical line bundle can be given a literal interpretation as a kind of Quillen determinant line bundle $\calL\to \calA(h)$ for $\Omega_Z$ such that $\Omega_Z \in c_1(\calL)$. Such a line bundle will generally only be a $\QQ$ or $\RR$-line bundle over $\calA(h)$ depending on the stability condition $(Z,\calA)$.
\end{itemize}
A first verification of the proposals of HKKP in this setting would be to repeat the analysis of canonical Jordan--H\"older filtrations for asymptotic $Z$-stability.

\subsection{Metrics on complexes and a variational proposal}

In order to develop a theory of $Z$-critical metrics for objects in the derived category, it is necessary to identify a well-behaved notion of \emph{Hermitian metric} on an object $\calE\in \DbCoh(X)$. A preliminary notion of such a metric has been identified by Burgos Gil--Freixas i Montplet--Li\textcommabelow{t}canu \cite{gil2012hermitian}.\footnote{The author thanks Mario Garcia-Fernandez for directing their attention to this work.}

\begin{definition}
	A \emph{Hermitian structure} $h$ on $\calE\in \DbCoh(X)$ is given by
	\begin{itemize}
		\item A quasi-isomorphism $q: E \dashrightarrow \calE$ in $\DbCoh(X)$ where $E$ is a complex of vector bundles, and
		\item a choice of Hermitian metric $h_i$ on $E_i$ for each term in the complex $E$. 
	\end{itemize}
\end{definition}

Note that when $X$ is a smooth projective variety every object $\calE\in \DbCoh(X)$ admits a quasi-isomorphism to a complex of locally free sheaves, so for such $X$ objects will have many Hermitian structures.

In \cite{gil2012hermitian} the notion of a Hermitian structure is upgraded to give a refinement $\overline{\DbCoh}(X)$ of the derived category consisting of complexes with Hermitian structures, with a well-behaved forgetful functor $\mathfrak{F}: \overline{\DbCoh}(X) \to \DbCoh(X)$. It is necessary to make precise the notion of equivalence of such structures. Essentially, one considers two Hermitian structures on an object $\calE$ to be equivalent if they differ by \emph{meager} complex, which is to first approximation an orthogonally split complex of Hermitian vector bundles.

One can then identify a ``space" of Hermitian structures on an object $\calE$.

\begin{definition}
	Define $$\Met(\calE) := \mathfrak{F}^{-1}(\calE) \subset \overline{\DbCoh}(X)$$ to be the collection of all Hermitian structures on an object $\calE$.
\end{definition}

This space is a torsor over a group $\mathrm{\mathbf{KA}}(X)$ of complexes of Hermitian vector bundles up to meager complexes. It should be formally thought of as an analogue of the set $\calG^\CC/\calG(h)$ parametrising Hermitian metrics on a vector bundle $E$.

One motivation for \cite{gil2012hermitian} was to define a notion of Bott--Chern class which makes sense for complexes of vector bundles and for coherent sheaves, for other applications. 

Recall that in \cref{sec:variationalfunctional} we defined a Kempf--Ness-type functional for the $Z$-critical equation using the Bott--Chern classes for a Chern connection on a vector bundle, whose critical points are precisely $Z$-critical metrics.

In \cite{gil2012hermitian} the Bott--Chern class of an additive genus $\varphi$ (such as the Chern character) is defined for a pair $h,h'$ of Hermitian structures on an object $\calE\in \DbCoh(X)$ as a class 
$$\tilde \varphi(h,h') \in \bigoplus_{n,p} \tilde \calD^{n-1} (X,p)$$
in Deligne cohomology. Such a class is represented by a Bott--Chern form
$$\varphi(h,h') \in \bigoplus_{n,p} \calD^{n-1} (X,p)$$
in the Deligne algebra of differential forms. 

We propose therefore a variational functional on $\Met(\calE)$ using the above formalism for which ``critical points" should be $Z$-critical metrics on an object $\calE$. 

\begin{definition}
	Let $(Z,\calA)$ be a Bridgeland stability condition on $\DbCoh(X)$ for a projective manifold $X$, such that $Z$ is given by a polynomial central charge $Z$ of the form \cref{def:polynomialcentralcharge}. 
	
	Let $\calE \in \calA$ and suppose $h,h'$ are two Hermitian structures on $\calE$. Define a functional
	$$\calM_Z(h,h') := \int_X \trace \tilde \calM_Z(h,h')$$
	where as in \cref{sec:variationalfunctional} we replace terms in the $Z$-critical equation with the corresponding Bott--Chern class for the $k$th Chern character of the complex $\calE$. 
\end{definition}

After making sense of such a functional in terms of the Bott--Chern classes for a complex $\calE$, one could define a $Z$-critical metric on the object $\calE$ to be a minimiser of $\calM_Z$ on $\Met(\calE)$.

\chapter{Correspondence in the large volume limit}
\label{ch:correspondence}

In this chapter we prove our main correspondence \cref{thm:maintheoremZstability}, that existence of $Z$-critical metrics in the large volume limit is equivalent to asymptotic $Z$-stability. 

The results of this chapter are joint work with Ruadha\'i Dervan and Lars Martin Sektnan \cite{dervan2021zcritical}, and have been reproduced here with an emphasis on the contributions of the author. In particular in the proof that existence implies stability, we will restrict to the case where the graded object $\Gr(E)$ of the semistable bundle $E\to X$ has just two components. This is a simplification of the general result proven in \cite{dervan2021zcritical}, however many of the technical difficulties already occur at this step, and in particular the importance of asymptotic $Z$-stability is apparent in this setting (in fact, its relevance is demonstrated more clearly in this simpler setting than in full generality). We refer to the joint work for the details of the full proof when $\Gr(E)$ is locally free with arbitrarily many components. We will briefly comment on the difficulties that manifest in this more general setting at the end of the chapter.

\section{Existence implies stability\label{sec:existenceimpliesstability}}

In this section we prove that existence of solutions to the $Z$-critical metric in the large volume limit implies asymptotic $Z$-stability. 

\begin{theorem}\label{thm:existenceimpliesstability}
	Let $E\to (X,\omega)$ be a holomorphic vector bundle over a compact K\"ahler manifold. If $E$ admits $Z_k$-critical metrics for all $k\gg 0$ which are uniformly bounded in $k$, then $E$ is asymptotically $Z$-stable with respect to holomorphic subbundles. 
\end{theorem}

We proceed following the general strategy for existence implying stability in gauge theory, using the principle that curvature decreases in subbundles. 

It will be convenient to rephrase the $Z$-critical equation in terms of the parameter $\oldepsilon = 1/k$ in the following. Recall that our definition of asymptotic $Z$-stability asks that for all holomorphic subbundles $S \subset E$, for $0 < \oldepsilon \ll 1$ we have strict inequality $$\varphi_{\oldepsilon}(S) < \varphi_{\oldepsilon}(E).$$ Our restriction to \emph{subbundles} rather than \emph{subsheaves} is natural given our hypothesis that $E$ is sufficiently smooth, and will arise naturally in the analysis to follow.

In our new notation, using $\oldepsilon$ rather than $k$, our central charge takes a slightly different form. Namely, define 
$$\ch^{\oldepsilon}(E) = \sum_{j=0}^n \oldepsilon^j \ch_j(E),\quad U^{\oldepsilon} = \sum_{j=0}^n \oldepsilon^j U_j$$
where we recall $U_j$ refers to the degree $2j$ component of $U$. The central charge associated to the input $(\omega, \rho, U)$ may then be written in terms of $\oldepsilon$ as
$$Z_{\oldepsilon}(E) = \int_X \sum_{d=0}^n \rho_{n-d} \omega^{n-d} \cdot \ch^{\oldepsilon}(E) \cdot U^{\oldepsilon},$$
and noting that we have $$\frac{1}{\oldepsilon^n} Z_{\oldepsilon}(E) = Z_k(E).$$ The $Z$-critical equation is simply produced as before, and we see a solution to the $Z_k$-critical equation is equivalent to a solution to the $Z_{\oldepsilon}$-critical equation at $\oldepsilon=1/k$.

\begin{remark}
	In \cref{sec:stabilityimpliesexistence} we will make the different substitution $\varepsilon^2 = 1/k$, which will avoid the introduction of fractional powers in that argument, but is not necessary here. To emphasise the differences in parameters, we use the notation $\oldepsilon = \varepsilon^2$.
\end{remark}

In this section we will take the view of Hermitian metrics on $E\to (X,\omega)$ rather than Chern connections. In addition to assuming that the vector bundle $E$ admits $Z_{\oldepsilon}$-critical metrics $h_{\oldepsilon}$ for all $0 < \oldepsilon \ll 1$, we make in \cref{thm:existenceimpliesstability} the assumption that these metrics $h_{\oldepsilon}$ are uniformly bounded as tensors in the $C^2$-norm (with respect to any fixed Hermitian metric on $E$). This extra assumption will be justified in \cref{sec:stabilityimpliesexistence}, where we prove the reverse direction that asymptotic $Z$-stability implies the existence of $Z_{\oldepsilon}$-critical connections for all $0 < \oldepsilon \ll 1$. There the metrics $h_{\oldepsilon}$ will be constructed as perturbations of the Hermite--Einstein metric on the bundle $\Gr(E)$, and hence uniform boundedness in $C^2$, or even in $C^{\infty}$ holds. The reason we employ the $C^2$-norm is the following.

\begin{lemma}\label{lemma:bounds} Suppose $h_{\oldepsilon}$ is a family of Hermitian metrics on a holomorphic vector bundle $E$ with uniformly bounded $C^2$-norm with respect to a fixed Hermitian structure $h_0$, and suppose $S \subset E$ is a holomorphic subbundle. Then the second fundamental forms $\gamma_{\oldepsilon}$ and the curvature forms $F_{\oldepsilon}$ are uniformly bounded in $C^1$ and $C^0$ respectively.
\end{lemma}

\begin{proof} This is an immediate consequence of the definitions. Recall that locally, the Chern connection of $h_{\oldepsilon}$ is given as $$A_{\oldepsilon}  = \partial h_{\oldepsilon}\cdot h_{\oldepsilon}^{-1},$$ while the curvature is given as $$F_{\oldepsilon} = d A_{\oldepsilon} + A_{\oldepsilon}\wedge A_{\oldepsilon}.$$ Thus the Chern connections $A_{\oldepsilon}$ are uniformly bounded in $C^1$ due to the uniform $C^2$-bound on $h_{\epsilon}$, and  the curvature forms  $F_{\oldepsilon}$ are uniformly bounded in $C^0$. 
	
	The holomorphic subbundle $S\subset E$ induces a short exact sequence
	\begin{center}
		\begin{tikzcd}
			0\arrow{r} & S \arrow{r} & E \arrow{r} & Q \arrow{r} & 0.
		\end{tikzcd}
	\end{center} The Hermitian metrics $h_{\oldepsilon}$ define an $\oldepsilon$-dependent orthogonal splitting of $E$ into a direct sum of complex vector bundles $$E \cong S \oplus Q,$$ where $Q = E/S$. Via this orthogonal splitting one writes the connections $A_{\oldepsilon}$ as $$A_{\oldepsilon}=\begin{pmatrix}
		A_{F, \oldepsilon} & \gamma_{\oldepsilon}\\
		-\gamma_{\oldepsilon}^* & A_{Q,{\oldepsilon}}
	\end{pmatrix}.$$ Here $$\gamma_{\oldepsilon}\in \Omega^{0,1}(X,\Hom(Q,S))$$ is the second fundamental form of the subbundle $F$ of $E$ with respect to the Hermitian structure $h_\oldepsilon$. Thus the uniform $C^2$-bound on $h_{\oldepsilon}$ induces a uniform $C^1$-bound on the second fundamental forms $\gamma_{\oldepsilon}$. \end{proof}

We are now ready to prove the main result of the Section, which uses some techniques of Leung in the study of almost Hermite--Einstein metrics   and Gieseker stability \cite[Proposition 3.1]{leung1997einstein}.

\begin{theorem}\label{thm:existencestabilitysubbundles}
	Suppose $E$ is irreducible and sufficiently smooth, and for every $0< \oldepsilon\ll 1$, $E$ admits a solution $h_{\oldepsilon}$ to the $Z$-critical equation such that the metrics $h_{\oldepsilon}$ are uniformly bounded in $C^2$ with respect to some fixed Hermitian metric $h_0$. Then $E$ is asymptotically $Z$-stable with respect to holomorphic subbundles.
\end{theorem}
\begin{proof} We follow the notation introduced  in  \cref{lemma:bounds}. The  $\oldepsilon$-dependent orthogonal splittings $E \cong S \oplus Q$ defined through the Hermitian metrics $h_{\oldepsilon}$ induce block matrix decompositions
	$$A=\begin{pmatrix}
		A_{S,\oldepsilon} & \gamma_{\oldepsilon}\\
		-\gamma_{\oldepsilon}^* & A_{Q,\oldepsilon}
	\end{pmatrix},\qquad F_{\oldepsilon} = \begin{pmatrix}
		F_{S, \oldepsilon} - \gamma_{\oldepsilon} \wedge \gamma_{\oldepsilon}^* & d_{A_{\oldepsilon}} \gamma_{\oldepsilon}\\
		-d_{A_{\oldepsilon}} \gamma_{\oldepsilon}^* & F_{Q,\oldepsilon} - \gamma_{\oldepsilon}^* \wedge \gamma_{\oldepsilon}
	\end{pmatrix},$$ where as above $\gamma_{\oldepsilon}$ is the second fundamental form of $F \subset E$ and $F_{\oldepsilon}$ is the curvature of $h_{\oldepsilon}$. We fix a reference Hermitian metric $h_0$ with which to measure the norms of the various tensors of interest; by our assumption of uniform boundedness, any of the $h_{\oldepsilon}$ suffice, all producing equivalent $L^2$-norms on tensors with values in $\End E$. Since we have assumed that  $E$ is irreducible, the second fundamental form is non-trivial, thus $$\|\gamma_{\oldepsilon}\| >0.$$ By the uniform boundedness obtained in \cref{lemma:bounds} we have \begin{equation}\label{eqn:uniformoldepsilonbounds}\lim_{\oldepsilon\to 0} \oldepsilon \|F_{\oldepsilon}\|_{C^0} = 0, \quad \lim_{\oldepsilon \to 0} \oldepsilon \|\gamma_{\oldepsilon}\|_{C^1} = 0.\end{equation}

	In order to show $E$ is asymptotically $Z$-stable with respect to $S$, we must show that
	$$\varphi_{\oldepsilon}(S) < \varphi_{\oldepsilon}(E)$$
	for all $0<\oldepsilon \ll 1$, where $\varphi_{\oldepsilon}(E) = \arg Z_{\oldepsilon}(E).$ This is equivalent to the inequality
	$$\Im\left( \frac{Z_{\oldepsilon} (S)}{Z_{\oldepsilon}(E)} \right) < 0,$$
	which in turn is equivalent to the inequality
	$$\Im\left( e^{-i\varphi_\oldepsilon(E)} Z_{\oldepsilon}(S) \right) < 0.$$
	Now since $h_{\oldepsilon}$ solves the $Z_{\oldepsilon}$-critical equation, we have
	$$\Im(e^{-i\varphi_{\oldepsilon}(E)} \tilde{Z}_{\oldepsilon}(h_{\oldepsilon})) = 0$$
	and so
	\begin{equation}\label{eqn:zeqnsubbundle}
		\Im(e^{-i\varphi_{\oldepsilon}(E)} \trace_S( \tilde Z_{\oldepsilon} (h_{\oldepsilon}))) = 0.
	\end{equation} Here $\tilde Z_{\oldepsilon} (h_{\oldepsilon})$ is an $\End(E)$-valued $(n,n)$-form, which restricts to an $\End(S)$-valued $(n,n)$-form via the splitting $E \cong S \oplus Q$ induced by $h_\oldepsilon$, and $ \trace_S( \tilde Z_{\oldepsilon} (h_{\oldepsilon}))$ is the induced $(n,n)$-form on $X$ obtained by taking trace. Note in particular that $\trace_S$ depends on $\oldepsilon$.	
	
	We will argue that for sufficiently small $\oldepsilon$, there is a positive $\oldepsilon$-dependent constant $C_{\oldepsilon}$ such that
	\begin{equation}\label{eqn:wanted}\int_X \Im(e^{-i\varphi_{\oldepsilon}(E)} \trace_S( \tilde Z_{\oldepsilon} (h_{\oldepsilon}))) = \Im(e^{-i\varphi_{\oldepsilon}(E)} Z_{\oldepsilon}(S)) + C_{\oldepsilon} \|\gamma_{\oldepsilon}\|^2,\end{equation}
	which will imply the result since the left hand side vanishes by \eqref{eqn:zeqnsubbundle}. Note that the leading order term in the $Z_{\oldepsilon}$-critical equation occurs at order $\oldepsilon$, so the constant $C_{\oldepsilon}$ will have lowest order $\oldepsilon$. 
	We begin by considering the order $\oldepsilon = \oldepsilon^1$ term. By \cref{lem:largevolume}, to leading order the $Z_{\oldepsilon}$-critical equation is given by the weak Hermite--Einstein equation. That is, there is a positive constant $c>0$ such that the leading order term of the $\oldepsilon$-expansion of the $Z_{\oldepsilon}$-critical equation is given by	
	$$c\left([\omega]^n \rk (E) \omega^{n-1} \wedge \left( \frac{i}{2\pi} F_{A_{\oldepsilon}} + \tilde U_2 \id_E\right) - \deg_U( E)  \omega^n\otimes \id_E\right).$$
	Thus we see the leading order $\oldepsilon^1$-term of $\Im(e^{-i\varphi_{\oldepsilon}(E)} \trace_S( \tilde Z_{\oldepsilon} (h_{\oldepsilon})))$  is given by
	\begin{align*}
		&c\left([\omega]^n \rk (E) \cdot \left( [\omega]^{n-1}. (\ch_1(S) + \rk(S) U_2)  - [\omega]^{n-1} .\frac{i}{2\pi} \trace(\gamma_{\oldepsilon} \wedge \gamma_{\oldepsilon}^*)\right) \right. \\&\quad \left. \phantom{\frac{i}{2\pi}} - \deg_U(E) \rk(S)  [\omega]^n\right) \\
		&= \Im(e^{-i\varphi_{\oldepsilon}(E)} Z_{\oldepsilon}(S))^1 -  c\frac{i}{2\pi}  (\rk E) [\omega]^n \cdot [\omega]^{n-1} . [\trace(\gamma_{\oldepsilon} \wedge \gamma_{\oldepsilon}^*)]\\
		&= \Im(e^{-i\varphi_{\oldepsilon}(E)} Z_{\oldepsilon}(S))^1 + C_1 \|\gamma_{\oldepsilon}\|^2.
	\end{align*} Here we have used the fact that, for some positive constant $C_1$
	$$c\omega^{n-1} \wedge \trace(\gamma_{\oldepsilon} \wedge \gamma_{\oldepsilon}^*) = -2\pi iC_1 |\gamma_{\oldepsilon}|^2 \omega^n,$$ and have written $\Im(e^{-i\varphi_{\oldepsilon}(E)} Z_{\oldepsilon}(S))^1$ to denote the $\oldepsilon^1$ term in the expansion. This is precisely the desired Equation \eqref{eqn:wanted} to leading order $\oldepsilon$, where we observe that, $C_{\oldepsilon} = \oldepsilon C_1 + O(\oldepsilon^2)$.  What we have crucially used here is that  $c>0$, which from \cref{lem:largevolume} follows from our assumption that $\Im(\rho_{n-1}/\rho_n) > 0,$ the crucial \emph{stability vector assumption} on our central charge $Z$. Curiously, at higher order in $\oldepsilon$ the lower order inequalities $\Im (\rho_{i-1}/\rho_i) > 0$ do not come into the argument, again reinforcing the observation that we only require this leading inequality in our work.
	
	We will now argue that at each higher order $\oldepsilon^{j}$, we obtain a similar expansion. Each of the terms appearing in the $Z$-critical equation at order $\oldepsilon^j$ involve differential forms of the form
	\begin{equation}\label{eqn:termform} C \oldepsilon^j\omega^{n-j} \wedge F_{A_{\oldepsilon}}^p \wedge \tilde U_{j-p}\end{equation}
	for $p$ possibly between $0$ and $j$ and $C$ some $\oldepsilon$-independent constant. 
	
	First let us note that if $p=0$, then this term is independent of the subbundle $S$ and is unaffected by our taking the $\trace_S$ in \eqref{eqn:zeqnsubbundle}, and so after integrating is absorbed by $\Im(e^{-i\varphi_{\oldepsilon}(E)} Z_{\oldepsilon}(S))^j$ on the right-hand side of \eqref{eqn:wanted}.
	
	Now if $p>0$, we need to understand the block matrix decomposition of a product of curvature terms
	\begin{align}
		&\underbrace{F_{A_{\oldepsilon}} \wedge \cdots \wedge F_{A_{\oldepsilon}}}_{p \text{ times}} \nonumber\\
		&= \begin{pmatrix}
			F_{S,{\oldepsilon}} - \gamma_{\oldepsilon} \wedge \gamma_{\oldepsilon}^* & d_A \gamma_{\oldepsilon}\\
			-d_A \gamma_{\oldepsilon}^* & F_{Q,{\oldepsilon}} - \gamma_{\oldepsilon}^* \wedge \gamma_{\oldepsilon}
		\end{pmatrix} \underbrace{\wedge \cdots \wedge}_{p \text{ times}} \begin{pmatrix}
			F_{S,{\oldepsilon}} - \gamma_{\oldepsilon} \wedge \gamma_{\oldepsilon}^* & d_A \gamma_{\oldepsilon}\\
			-d_A \gamma_{\oldepsilon}^* & F_{Q,{\oldepsilon}} - \gamma_{\oldepsilon}^* \wedge \gamma_{\oldepsilon}
		\end{pmatrix}.\label{eq:curvaturewedgeproduct}
	\end{align}
	We will be interested in the $\trace_S$ of the terms that appear in the top left block of this matrix decomposition. This will in general involve a term of the form 
	$$F_{S_{\oldepsilon}} \underbrace{\wedge \cdots \wedge}_{p \text{ times}} F_{S,{\oldepsilon}},$$
	which after taking $\trace_S$ and wedging with the differential forms as in \eqref{eqn:termform} and integrating, gives the required factor for $\Im (e^{-i\varphi_{\oldepsilon}} Z_{\oldepsilon}(S))^j$ in \eqref{eqn:wanted}, which is
	$$ \int_X C \oldepsilon^j\omega^{n-j} \wedge \trace_S F_{S_{\oldepsilon}}^p \wedge \tilde U_{j-p}.$$
	In addition to this desired term, we will also obtain other terms all involving at least one of 
	\begin{equation} \label{eqn:list} \gamma_{\oldepsilon}^* \wedge \gamma_{\oldepsilon},\quad  \gamma_{\oldepsilon} \wedge \gamma_{\oldepsilon}^*,\quad  d_{A_{\oldepsilon}} \gamma_{\oldepsilon},\quad  d_{A_{\oldepsilon}} \gamma_{\oldepsilon}^*,\quad F_{Q,{\oldepsilon}}\end{equation} 
	and also some possible factors of $F_{S,{\oldepsilon}}$. We wish to show that no matter what product we obtain containing these terms, the corresponding form can be absorbed by the factor $C_{\oldepsilon} \|\gamma\|^2 \omega^n$ in \eqref{eqn:wanted} after taking a trace and integrating.
	
	At order $\oldepsilon^j$ with a differential form of the form \eqref{eqn:termform}, our curvature component consists of a product of $p$ terms in the list \eqref{eqn:list} given above, with $p$ at most $j$. Following Leung's notation, let us call such a product $\calT_p$, so our form of interest is
	$$\omega^{n-d} \wedge \calT_p \wedge \tilde U_{j-p}.$$
	We will show that provided $\oldepsilon$ is sufficiently small, any such form is much smaller in norm than $C_1 \oldepsilon \|\gamma_{\oldepsilon} \|^2$ appearing in \eqref{eqn:wanted}, and therefore can be absorbed into this term whilst preserving the positivity of $C_{\oldepsilon} \|\gamma\|^2$. There are three cases to consider.
	
	\begin{enumerate}
		\item \emph{$\calT_p$ containing $\gamma_{\oldepsilon}^* \wedge \gamma_{\oldepsilon}$ or $\gamma_{\oldepsilon} \wedge \gamma_{\oldepsilon}^*$:}
		
		First notice that by the uniform estimates \eqref{eqn:uniformoldepsilonbounds}, $\oldepsilon$ times any term in the list \eqref{eqn:list}, or a term $F_{S,\oldepsilon}, F_{Q,\oldepsilon}$ tends to zero as $\oldepsilon\to 0$. Thus if we have a term of the form $\gamma_{\oldepsilon}^* \wedge \gamma_{\oldepsilon}$ or $\gamma_{\oldepsilon} \wedge \gamma_{\oldepsilon}^*$ in our product $\calT_p$, we have an expression
		$$\pm \oldepsilon^j \omega^{n-j} \wedge \calT'_{p-1} \wedge \tilde U_{j-p} \wedge \gamma_{\oldepsilon}^* \wedge \gamma_{\oldepsilon}$$
		where $\calT'_{p-1}$ consists of the remaining $p-1$ factors in $\calT_p$, and the sign depends on the order of $\gamma_{\oldepsilon}$ or $\gamma_{\oldepsilon}^*$ in our wedge term. But we can rewrite this as
		$$\pm  \omega^{n-j} \wedge \oldepsilon^{p-1} \calT'_{p-1} \wedge  \oldepsilon^{j-p} \tilde U_{j-p} \wedge \oldepsilon  \gamma_{\oldepsilon}^* \wedge \gamma_{\oldepsilon}.$$
		Then by our initial observation, $\oldepsilon^{p-1} \calT'_{p-1}$ tends to zero as $\oldepsilon\to 0$, so for $\oldepsilon$ taken sufficiently small we can estimate (after taking trace and integrating over $X$),
		$$\|\omega^{n-j} \wedge \calT_p \wedge \tilde U_{j-p}\| \le  c_{\oldepsilon} \oldepsilon\|\gamma_{\oldepsilon}\|^2$$
		where the constant $c_{\oldepsilon}$ depending on our factor $\oldepsilon^{p-1} \calT'_{p-1}$ is as small as we like provided we take $\oldepsilon$ sufficiently small. Such a term is therefore small in norm compared to $C_1 \oldepsilon \|\gamma_{\oldepsilon}\|^2$ for $\oldepsilon$ sufficiently small, as desired.
		\item \emph{$\calT_p$ containing $d_{A_{\oldepsilon}} \gamma_{\oldepsilon}$ or $d_{A_{\oldepsilon}} \gamma_{\oldepsilon}^*$:}
		
		If there is no $\gamma_{\oldepsilon}^* \wedge \gamma_{\oldepsilon}$ or $\gamma_{\oldepsilon} \wedge \gamma_{\oldepsilon}^*$ term in the product $\calT_p$, but there is a term of the form $d_{A_{\oldepsilon}} \gamma_{\oldepsilon}$ or $d_{A_{\oldepsilon}} \gamma_{\oldepsilon}^*$, then we may integrate by parts when computing \eqref{eqn:wanted}.  This shifts $d_{A_{\oldepsilon}}$ on to the other terms appearing in the differential form
		$$\omega^{n-j} \wedge \calT_p \wedge \tilde U_{j-p}.$$
		Using the Leibniz rule for the exterior covariant derivative $d_{A_{\oldepsilon}}$, we can deal with each possibility in cases. If $d_{A_{\oldepsilon}}$ is applied to a term of the form $F_S$ or $F_Q$ after integrating parts, or a form $\tilde U_{j-p}$ or $\omega^{n-j}$, then this will vanish by the Bianchi identity $d_{A_{\oldepsilon}} F_{\oldepsilon} = 0$ or closedness of $\omega$ and $\tilde U$. 
		
		Thus the only non-vanishing possibilities occur if, after integrating by parts, the $d_{A_{\oldepsilon}}$ is applied to a term of the form  $d_{A_{\oldepsilon}} \gamma_{\oldepsilon}$ or $d_{A_{\oldepsilon}} \gamma_{\oldepsilon}^*$. Now we recall that in fact $d_{A_{\oldepsilon}} \gamma = \del_{A_{\oldepsilon}} \gamma$ and similarly $d_{A_{\oldepsilon}} \gamma^* = \delbar_{A_{\oldepsilon}} \gamma^*$. Thus, for example, if we started with $\del_{A_{\oldepsilon}} \gamma$ and our product $\calT_p$ contains another term $\del_{A_{\oldepsilon}} \gamma$, after integrating by parts we would obtain $\gamma \wedge \del_{A_{\oldepsilon}}^2 \gamma = 0$, and similarly for when we have $\delbar_{A_{\oldepsilon}} \gamma^*$. Thus we reduce just to the case where we have a factor $\gamma \wedge \del_{A_{\oldepsilon}} \delbar_{A_{\oldepsilon}}\gamma^*$ or $\gamma^* \wedge \delbar_{A_{\oldepsilon}} \del_{A_{\oldepsilon}}\gamma$ in our product after integrating by parts.
		
		Using the fact that $\delbar_{A_{\oldepsilon}} \gamma_{\oldepsilon} = \del_{A_{\oldepsilon}} \gamma_{\oldepsilon}^* = 0$, and that $F_{A_{\oldepsilon}} \wedge \gamma_{\oldepsilon}= d_A^2 \gamma_{\oldepsilon} = (\del_{A_{\oldepsilon}} \delbar_{A_{\oldepsilon}} + \delbar_{A_{\oldepsilon}} \del_{A_{\oldepsilon}})  \gamma_{\oldepsilon}$, we will therefore obtain more curvature terms wedge terms involving $\gamma_{\oldepsilon}^* \wedge \gamma_{\oldepsilon}$ or $\gamma_{\oldepsilon} \wedge \gamma_{\oldepsilon}^*$. This lands us in the previous situation above, which we have already dealt with, so we are done in this case.
		
		\item \emph{$\calT_p$ not containing a term with a $\gamma_{\oldepsilon}$:}
		
		When there is no $\gamma_{\oldepsilon}$ term appearing, from the block matrix decomposition we see that the only non-zero terms which can appear are p-fold products of $F_{S,\oldepsilon}$ or $F_{Q,\oldepsilon}$. The latter terms have vanishing $\trace_S$ and the former terms were already accounted for, contributing to $\Im(e^{-i\varphi_{\oldepsilon}} Z_{\oldepsilon}(S))^j.$
	\end{enumerate}
	
	This shows that every form with curvature part given by a product $\calT_p$ is very small in norm relative to $C_1 \oldepsilon \|\gamma_\oldepsilon\|^2$ provided we choose $\oldepsilon>0$ small enough. Thus \eqref{eqn:wanted} holds and since the left hand side is zero, for such small $\oldepsilon$ we have
	$$\Im(e^{-i\varphi_{\oldepsilon}(E)} Z_{\oldepsilon}(S)) < 0$$
	so $E$ is asymptotically $Z$-stable with respect to the subbundle $S$. 
\end{proof}

\begin{remark}\label{rmk:leungproof} In Leung's proof of the above result for Gieseker stability and almost Hermite--Einstein metrics (see \cref{sec:AHE}) where they show that the existence of almost Hermite--Einstein metrics implies Gieseker stability with respect to subbundles \cite[Proposition 3.1]{leung1997einstein}, Leung states that the existence of almost Hermite--Einstein metrics actually implies Gieseker stability in general, with respect to not only subbundles but also subsheaves. Leung's proof of this statement relies on the claim that if $$f: S \to Q_0$$ is a morphism between coherent sheaves such that $S$ is Gieseker stable, $Q_0$ is slope stable and $S, Q_0$ have the same slope, then $f$ must be zero or an isomorphism \cite[page 530]{leung1997einstein}. It follows from general theory that such an $f$ must be zero or surjective, since $S$ is slope semistable, but it is possible for such an $f$ to not be injective. Thus one only obtains stability with respect to subbundles. A similar issue occurs in \cref{thm:existenceimpliesstability} where we have restricted to the same assumption. Since the perturbative result of stability implies existence also restricts to the case where the graded object $\Gr(E)$ is locally free, it is no further restriction to only have stability with respect to subbundles, as this is all which is required for the existence result in any case.
\end{remark}

\section{Stability implies existence\label{sec:stabilityimpliesexistence}}

In this section we prove the following existence result for $Z$-critical connections in the large volume limit.

\begin{theorem}\label{thm:stabilityimpliesexistence}
	Let $E\to (X,\omega)$ be a simple, semistable holomorphic vector bundle over a compact K\"ahler manifold, such that $\Gr(E)$ is locally free with two components. If $E$ is asymptotically $Z$-stable for a polynomial central charge $Z$, then $E$ admits a $Z_k$-critical metric for all $k\gg 0$. 
\end{theorem}

As mentioned in the introduction to this chapter, this is a simplification of the main result of \cite{dervan2021zcritical} which includes the case where $\Gr(E)$ is locally free with arbitrarily many components. We will comment on the difficulties that manifest in that case at the end of this section, and refer to the joint work for the details. In particular the existence result in that generality combined with \cref{thm:existenceimpliesstability} completes the proof of the main theorem \cref{thm:maintheoremZstability}. Let us recall the statement:

\begin{corollary}[\cref{thm:maintheoremZstability}]
	A simple, semistable vector bundle $E\to (X,\omega)$ with $\Gr(E)$ locally free is asymptotically $Z$-stable if and only if it admits $Z_k$-critical metrics for $k\gg 0$. 
\end{corollary}

The proof technique is based on the perturbation result of Sektnan--Tipler who considered the problem of constructing Hermite--Einstein metrics on pullbacks of slope stable vector bundles under holomorphic submersions \cite{sektnan2020hermitian}. The strategy is as follows:

\begin{enumerate}[label=Step \arabic*:]
	\item Construct approximate solutions to any order on a \emph{slope} stable vector bundle.
	\item On the graded object $\Gr(E)=E_1\oplus E_2$ of an asymptotically $Z$-stable bundle $E$, take the direct sum of approximate solutions on each stable factor $E_i$. Using general properties of the linearisation, after perturbing the complex structure from $\Gr(E)$ to $E$ one can correct any errors up to order $\varepsilon^{2q-1}$ where $q$ is the \emph{discrepancy order of $E_1$}, the first order at which the coefficient of $\varphi_\varepsilon(E_1)$ becomes strictly less than $\varphi_\varepsilon(E)$. 
	\item At the critical order $\varepsilon^{2q}$, the stability of $E$ with respect to the holomorphic subbundle $E_1$ allows one to cancel the terms which are not orthogonal to the kernel of the linearisation, and obtain an approximate solution to order $\varepsilon^{2q}$.
	\item Having corrected the error at order $\varepsilon^{2q}$, similarly to Step 2 one can cancel errors to all higher orders to obtain approximate solutions to arbitrary order.
	\item Using a Poincar\'e inequality, establish a bound on the inverse of the linearisation of the approximate solution. When the approximate solution is sufficiently good (precisely, at least order $\varepsilon^{4q+1}$) one may apply the quantitative implicit function theorem to obtain solutions for all $\varepsilon>0$ sufficiently small.
\end{enumerate}

In the following we will make the change of variables $\varepsilon^2 = \frac{1}{k}$ similarly to \cref{sec:existenceimpliesstability} but noting the square $\epsilon^2 = \oldepsilon$. The square here is taken to avoid fractional powers appearing in the construction of approximate solutions later in the argument.

We will label the $Z$-critical operator by
\begin{align*}
	D_\varepsilon:\quad  \calA(h) &\to \Gamma(X,\End_H(E,h))\\
	A &\mapsto \frac{\Im(e^{-i\varphi_{\epsilon}(E)} \tilde Z_\epsilon(A))}{\omega^n}
\end{align*}
and we will rescale this operator by a factor of $\epsilon^{4q-2} = k^{1-2n}$ so that the leading order, which by \cref{lem:largevolume} is the weak Hermite--Einstein operator, is order $O(\epsilon^0)$. 

At times we will want to consider the operator $D_\epsilon$ near a particular Chern connection $A$, that is, in a given orbit of $\calG^\CC$ in $\calA(h)$, and may use the notation
$$D_{\epsilon,A}: \Gamma(X,\End_H(E,h)) \to \Gamma(X,\End_H(E,h))$$
where we interpret the operator as $D_{\epsilon,A}(s) = D_\epsilon(\exp(s) \cdot A)$ for $s\in \Gamma(X,\End_H(E,h))$ a Hermitian endomorphism. In this case $D_{\epsilon,A}(0) = D_\epsilon(A)$. We will drop the subscript $A$ when the Chern connection is understood.

The linearisation of $D_\epsilon$ at some Chern connection $A$ will be denoted $P_{\epsilon,A}$, and similarly we will drop the $A$ if the particular Chern connection is understood. The inverse of the linearisation, when and where it exists, will be denoted $Q_\epsilon$.

In the following all estimates will be taken in $L^2$ with respect to a reference Hermitian metric $h_0$ on $E$. The $L^2$-norms with respect to any such choice are equivalent as norms on tensors with values in $\End E$, so it is immaterial which norm is used to estimate, but the natural choice is to choose this reference metric to be the Hermite--Einstein metric on $\Gr(E)$, which defines smoothly a metric on $E$ (under the assumption that $\Gr(E)$ is locally-free, and is therefore smoothly (but not holomorphically!) isomorphic to $E$).

\subsection{Step 1: The slope stable case\label{sec:step1}}

Before proceeding with the proof of \cref{thm:stabilityimpliesexistence} in the case where $\Gr(E)$ has two components, we will prove the existence result when $\Gr(E)$ has just one component, that is, when $E$ is itself slope stable (and therefore by \cref{cor:stabilityimpliesazs} is asymptotically $Z$-stable). 

In fact we give two proofs, firstly by applying the inverse function theorem to prove the existence of genuine $Z$-critical metrics on $E$. Secondly however, we construct explicitly approximate solutions of arbitrary order, which is necessary to have control over the expansion of the linearised operator at an approximate solution in the later steps of the argument when the graded object $\Gr(E)$ has two components.

First let us prove a genuine existence result in the slope stable case. By \cref{lem:largevolume}, the leading order term in the $Z$-critical equation is the \emph{weak} Hermite--Einstein condition. As discussed in \cref{sec:hermiteeinsteinmetrics} by applying a conformal change of metric one can transform between Hermite--Einstein and weak Hermite--Einstein metrics, so we can apply the Donaldson--Uhlenbeck--Yau theorem \cref{thm:DUY} freely to obtain a weak Hermite--Einstein metric on $E$ with the function $f\in C^{\infty}(X,\RR)$ defined by
$$f = -2\pi \left( \frac{\deg_U(E)}{(n-1)! [\omega]^n\rk(E)} - \contr_{\omega} \tilde U_2 \right),$$ arising through \cref{lem:largevolume}.
The function $f$ will be fixed throughout, so we will refer to a connection solving this equation as simply a weak Hermite--Einstein connection.

\begin{theorem}
	\label{thm:stabilityimpliesexistencestable}
	Suppose $E\to (X,\omega)$ is slope stable. Then for all $k\gg 0$, $E$ admits $Z_k$-critical connections.
\end{theorem}
\begin{proof}
	Since $E$ is slope stable, it admits a weak Hermite--Einstein metric $h_0$ by \cref{thm:DUY}. By \cref{lem:largevolume} the leading order term of the $Z$-critical equation is the  weak Hermite--Einstein equation, so we immediately have
	$$D_0 (h_0) = 0,\qquad \|D_{\epsilon} (h_0) \| \le C\epsilon^2$$
	for some constant $C$, where $D_0$ denotes the weak Hermite--Einstein operator. As observed in \cref{lem:linearisationcurvatureMetric}, when taking the point of view of changing the Hermitian metric on $E$ through the action of the gauge group, the linearisation of the weak Hermite--Einstein operator $D_0$ at $h_0$ is given (up to a constant factor) by $$P_0 = \Lap_{\delbar_{h_0}},$$ the bundle Laplacian on $\End E$ with respect to $h_0$. Since $E$ is stable, it is simple, and so the kernel of $P_0$ consists only of the constant endomorphisms, and $P_0$ is invertible orthogonal to this kernel. 
	
	We  pass to Banach spaces, and view our linearisation as an invertible operator
	$$P_{\epsilon} : L_{d+2,0}^2 (\End E) \to L_{d}^2(\End E)$$
	where $d\in \ZZ_{\ge0}$ is some non-negative integer and $L_{d,0}^2$ denotes the Sobolev of space trace-integral zero endomorphisms of $E$. Our previous discussion produces the estimate
	$$\|P_{\epsilon} - P_0\| \le C\epsilon^2$$
	for some $C$ independent of $\epsilon$, and since $P_0=\Lap_{\delbar_{h_0}}$ is invertible modulo constant endomorphisms and invertibility is an open condition, for $\epsilon$ sufficiently small, $P_{\epsilon}$ is also invertible on this space $L_{d+2,0}^2$. If $G$ denotes the inverse of $P_0$ and $Q_{\epsilon}$ denotes the inverse of $P_{\epsilon}$, then we also obtain a bound
	$$\frac{1}{C'} \|G \| \le \|Q_{\epsilon}\| \le C'\|G\|$$
	for some constant $C'$ independent of $\varepsilon$, provided $\varepsilon$ is sufficiently small. By the inverse function theorem for Banach spaces applied to the point $h_0$, there exists a neighbourhood of $D_{\epsilon}(h_0)$ in $L_d^2(\End E)$ with size independent of $\varepsilon$ which is in the image of the operator $D_{\epsilon}$ from $L_{d+2,0}^2$. In particular, since $D_{\epsilon}(h_0) \to 0$ as $\epsilon\to 0$, there exists some $\varepsilon_0>0$ such that for all $0<\epsilon<\epsilon_0$, there exists a solution
	$$D_{\varepsilon}(h_{\epsilon}) = 0$$
	for $h_{\epsilon} = \exp (V_{\epsilon}) h_0$ for some $V_{\epsilon}$ in the Sobolev space $L_{d+2,0}^2(\End E).$ But as was proven in \cref{lem:linearisation}, the $Z$-critical equation is elliptic for all $\epsilon$ sufficiently small, so by elliptic regularity this $h_{\epsilon}$ is actually smooth and hence is a genuine solution of the $Z_{\epsilon}$-critical equation. 
\end{proof}

\begin{remark}
	This result gives the first known examples of solutions to the higher rank deformed Hermitian Yang--Mills equation. After proving \cref{thm:stabilityimpliesexistence} for the case where $\Gr(E)$ has two components, we will also obtain solutions on the asymptotically $Z$-stable bundles identified in \cref{ex:zcritical}, which in particular do \emph{not} admit solutions to the Hermite--Einstein equation.
\end{remark}

Having proven a genuine existence result on slope stable vectors, we now backtrack to show the existence of approximate solutions of a specific form. This will both demonstrate the general technique of constructing approximate solutions we will make great use of in the next steps, and serves as the starting point for understanding approximate solutions when $\Gr(E)$ has two components.

\begin{proposition}\label{prop:approximatesolutionsstable}
	Let $E\to (X,\omega)$ be a slope stable vector bundle with Dolbeault operator $\delbar_E$ and weak Hermite--Einstein metric $h$, with associated Chern connection $A$. Let $r\in \ZZ_{\ge 0}$ be any non-negative integer. Then there exists Hermitian endomorphisms $f_2,f_4,\dots,f_{2r}$ such that 
	$$A_r := \exp\left( \sum_{j=1}^r f_{2j} \epsilon^{2j}\right) \cdot A$$
	satisfies
	$$D_\epsilon (A_r) = O(\epsilon^{2r+1}).$$
\end{proposition}
\begin{proof}
	Since $E$ is stable, it is simple. By \cref{lem:linearisation} the linearisation of $D_\epsilon$ at $h$ is, to leading order, given by the Laplacian $\Lap_0$ on $\End E$. Recall from \cref{cor:HermiteEinsteinelliptic} that the kernel of $\Lap_0$ is given by
	$$\ker \Lap_0 = H^0(X,\End E) = \CC \cdot \id_E.$$
	The equation $\Lap_0(f) = g$ is solvable whenever $g$ is orthogonal to the kernel of the Laplacian, and $f$ is a smooth endomorphism whenever $g$ is. This orthogonality on $\Gamma(\End_H(E,h))$ occurs when $$\int_X \trace g \, \omega^n = 0.$$
	By expanding out \eqref{eq:Zcriticaltraceintegral} in powers of $\epsilon$ we see that for every $j$, the $\epsilon^{2j}$ term in $D_\epsilon$ trace-integrates to zero, and is therefore orthogonal to the kernel of $\Lap_0$. 
	
	Now let $f_2$ be some Hermitian endomorphism of $E$. Then by \cref{lem:linearisation} and \cref{lem:linearisationcurvatureChern} we see $A_2 := \exp(\epsilon^2 f_2) \cdot A$ satisfies
	\begin{align*}
		D_\epsilon(A_2) &= D_\epsilon(A) + \epsilon^2 P_\epsilon(f_2) + O(\epsilon^4)\\
		&= D_\epsilon(A) + \epsilon^2 \Lap_0 (f_2) + O(\epsilon^4)
	\end{align*}
	where $P_\epsilon$ is the linearisation of $D_\epsilon$ at $A$ and in the second line the $O(\epsilon^4)$ term may depend on $f_2$ itself. Moreover again applying \eqref{eq:Zcriticaltraceintegral} to $A_2$ and expanding in $\epsilon$, we see that these additional error terms at order $\epsilon^4$ and higher depending on $f_2$ trace-integrate to zero.
	
	Now note that $D_\epsilon(A)=O(\epsilon^2)$ since $A$ is the weak Hermite--Einstein connection, and as remarked the $\epsilon^2$ term, $\sigma_2 \epsilon^2$ say, is orthogonal to $\ker \Lap_0$. Thus there exists some $f_2$ such that 
	$$\Lap_0 (f_2) + \sigma_2 = 0$$
	and we obtain $A_2$ with $D_\epsilon(A_2) = O(\epsilon^4)$. This completes the $r=1$ case of the proposition. Repeating the same argument at order $\epsilon^4$ by considering a correction $A_4 := \exp(f_4 \epsilon^4)\cdot A_2$ we obtain $A_4$, and so on inductively cancelling errors up to order $\epsilon^{2r}$, obtaining an approximate solution $A_r$.
\end{proof}

\begin{remark}
	In the above notion, the sum notation
	$$g_r = \exp\left( \sum_{j=1}^r f_{2j} \epsilon^{2j}\right)$$
	is used abusively to denote the automorphism
	$$g_r = \exp(f_{2j} \epsilon^{2j})  \exp(f_{2j-2} \epsilon^{2j-2})  \cdots \exp(f_2 \epsilon^2)$$
	and in general the $f_\ell$ do note commute so the expressions cannot be literally interchanged. This is simply a notational convention and in the construction of approximate solutions we will always mean the latter ordering of products of automorphisms. No confusion will arise as this ordering is already implicit in the arguments of constructing approximate solutions. 
\end{remark}

\begin{remark}
	Now with approximate solutions of order $A_r$ we could apply the inverse function theorem again to obtain genuine solutions, however the argument of \cref{thm:stabilityimpliesexistencestable} already shows that $A=A_0$ is a good enough approximate solution. Indeed carefully checking the quantitative perturbation argument in the proof of \cref{thm:stabilityimpliesexistence} in \cref{sec:step5} we see that one needs an approximate solution of order $4q+1$ where $q$ is the discrepancy order, which in the stable case is $q=0$ so $A_0$ is already good enough.
\end{remark}

\subsection{Deformations of complex structure}

We now begin the proof of \cref{thm:existenceimpliesstability} proper. To that end we fix a holomorphic vector bundle $E\to (X,\omega)$ which is asymptotically $Z$-stable. Then \cref{lemma:slopesemistable} implies that $E$ is slope semistable, so as discussed in \cref{sec:filtrations} $E$ admits a Jordan--H\"older filtration which we assume to have two steps: 
$$0\subset E_1 \subset E.$$
Then $E$ has graded object
$$\Gr(E) = E_1 \oplus E_2$$
where $E_2 = E/E_1$. We will work under the assumption that $\Gr(E)$ is locally free, so that $E_1$ and $E_2$ are vector bundles, and in this case $\Gr(E)$ is the unique torsion-free graded object associated to $E$. Recall that $E_1$ and $E_2$ have the same slope. The first consequence of asymptotic $Z$-stability in this setting is the following:

\begin{lemma}\label{lem:gradedcomponentsnotisomorphic}
	If $E$ is asymptotically $Z$-stable, then $E_1$ is not isomorphic to $E_2$.
\end{lemma}
\begin{proof}
	This follows from the see-saw property \cref{lem:seesawazs} for asymptotic $Z$-stability. In particular we have
	$$\varphi_k(E_1) < \varphi_k(E) < \varphi_k(E_2)$$
	for all $k\gg 0$, so we cannot have $E_1 \isom E_2$.
\end{proof}

As a consequence of the above lemma, we can very easily describe the automorphism group of the graded object.

\begin{lemma}
	The endomorphisms of $\Gr(E)$ are given by
	$$H^0(X,\End \Gr(E)) = H^0(X, \End E_1) \oplus H^0(X,\End E_2) = \CC \id_{E_1} \oplus \CC \id_{E_2}.$$
\end{lemma}
\begin{proof}
	By the slope stability of $E_1$ and $E_2$ any morphism $u: E_1 \to E_2$ is either zero or an isomorphism. By \cref{lem:gradedcomponentsnotisomorphic} since $E$ is asymptotically $Z$-stable we must have $u=0$. Therefore any automorphism of $\Gr(E)$ has only diagonal components with respect to the holomorphic splitting $\Gr(E) = E_1 \oplus E_2$. Then we recall \cref{prop:stabilitysimplicity} implies $E_1$ and $E_2$ are simple so all holomorphic endomorphisms are constant multiples of the identity.
\end{proof}

As discussed in \cref{rmk:turningoffextension}, we can view the short exact sequence
\begin{center}
	\ses{E_1}{E}{E_2}
\end{center}
as producing a degeneration of $E$ to $\Gr(E)$ by turning off the extension. Recall from \cref{sec:chernconnections} that we can describe this process differential-geometrically as follows.

The graded object $\Gr(E)$ of $E$ is slope polystable, and therefore admits a Hermite--Einstein metric $h$ by the Donaldson--Uhlenbeck--Yau theorem \cref{thm:DUY}. Viewing $h$ as a Hermitian metric on $E$ itself (by the smooth identification $E\isom \Gr(E)$), this defines a smooth splitting
$$E \isom E_1 \oplus E_2.$$
Let us fix a Dolbeault operator $\delbar_E$ on $E$ producing its holomorphic structure. Then with respect to this splitting $\delbar_E$ splits as
$$\delbar_E = \begin{pmatrix}
	\delbar_1 & \gamma \\
	0 & \delbar_2
\end{pmatrix}$$
where $\delbar_1$ and $\delbar_2$ are the Dolbeault operators defining the holomorphic structures on $E_1$ and $E_2$ respectively, and $\gamma\in \Omega^{0,1}(X,\Hom(E_2,E_1))$ is the second fundamental form of $E_1 \subset E$. Let us write
$$\delbar_0 := \begin{pmatrix}
	\delbar_1 & 0 \\
	0 & \delbar_2
\end{pmatrix}$$
for the Dolbeault operator on $\Gr(E)$ for which the Chern connection $A(h,\delbar_0)$ is Hermite--Einstein. Then 
$$\delbar_E = \delbar_0 + \gamma$$
and the integrability condition for $\delbar_E$ implies $\delbar_0 \gamma = 0$. Recall that as discussed in \cref{rmk:secondfundamentalformatiyah} this condition implies $[\gamma]$ defines a class in Dolbeault cohomology. If $\gamma$ is $\delbar_0$-cohomologous to zero then $E$ splits holomorphically as a direct sum $E\isom \Gr(E)$ and is slope polystable. In this case the direct sum of $Z$-critical connections on the $E_i$ afforded by \cref{thm:stabilityimpliesexistencestable} is $Z$-critical on $E$. From now on therefore we assume $[\gamma]\ne 0$ and that $E$ is \emph{strictly} slope semistable.

In this case we produce a deformation of complex structure from $\Gr(E)$ to $E$ by turning off the extension. Define
$$\delbar_t := \delbar_0 + t\gamma.$$
Such Dolbeault operators can be obtained from $\delbar_E$ from an automorphism
$$g_t = \begin{pmatrix}
	t \id_{E_1} & 0\\
	0 & \id_{E_2}
\end{pmatrix}$$
by
$$\delbar_t = g_t \circ \delbar_E \circ g_t^{-1}.$$
In particular for every $t\ne 0$ the operators $\delbar_t$ define isomorphic complex structures equivalent to $\delbar_E$, and at $t=0$ the structure splits into $\Gr(E)$. 

The Chern connections of $\delbar_t$ with respect to $h$ are 
$$\nabla_t = \nabla_0 + ta$$
where
$$a := \begin{pmatrix}
	0 & \gamma\\
	-\gamma^* & 0
\end{pmatrix}$$
and $\nabla_0$ is the Chern connection of $\delbar_0$. The curvature of $\nabla_t$ is
\begin{equation}\label{eq:curvatureFt}F_t = F_0 + t d_{\nabla_0} a + t^2 a\wedge a\end{equation}
where we note
$$a\wedge a = -\begin{pmatrix}
	\gamma \wedge \gamma^* & 0\\
	0 & \gamma^* \wedge \gamma
\end{pmatrix}.$$

Notice that the term $td_{\nabla_0} a$ is off-diagonal in the block matrix representation of $F_t$, where the induced connection on the endomorphism bundle is given by
$$\nabla_t^{\End E} = \nabla_0^{\End E} + t[a,-].$$

In the following we will fix the gauge of the Chern connection $A$ inside its $\calG^\CC$ orbit by imposing the Coulomb condition
\begin{equation}
	\label{eq:coulombgauge}
	\delbar_0^* \gamma = 0.
\end{equation}
It is well-known (see, for example, \cite[\S 7.2]{kobayashi1987differential} or \cite[Lem. 2.5]{buchdahl2020polystability}) that the second fundamental form can always be transformed by applying a unitary gauge transformation of $E$ to be of this form.

The key concept which we will need when working in the asymptotic regime is the notion of the order of discrepancy. 

\begin{definition}[Order of discrepancy]
	Normalise the $Z$-phase $\varphi_k(E)$ to have leading order term $O(k^0)$ in $k$. Let $\varphi_k(E)(j)$ denote the order $k^{-j}$ term in the expansion of $\varphi_k(E)$ for $k\gg 0$. Define the \emph{order of discrepancy} of a subbundle $S\subset E$ as the smallest $q$ such that
	$$\varphi_k(S)(q) < \varphi_k(E)(q).$$
\end{definition}

Note that when $E$ is asymptotically $Z$-stable, we have $\varphi_k(S)(j) = \varphi_k(E)(j)$ for all $j=0,\dots,q-1$. In the case where $E$ is slope stable, we have $q=0$ since $\varphi_k(E)(0) = \mu(E)$ so inequality occurs immediately.

\begin{remark}
	With our convention that $\varepsilon^2 = 1/k$, the order of discrepancy $q$ is equal to \emph{half} the smallest order at which $\varphi_\epsilon(E_1)$ becomes strictly less than $\varphi_\epsilon(E)$. Thus in the the following arguments the critical order will be $\epsilon^{2q}$.
\end{remark}

We record explicitly the following key feature of the order of discrepancy, which is more or less a rephrasing of \cref{lem:equivalentstabilitycondition}.

\begin{lemma}\label{lem:orderofdiscrepancy}
	If $q$ is the order of discrepancy of $E_1\subset E$ then the $\epsilon^{2q}$ term in the expansion of
	$$\Im(e^{-i\varphi_\epsilon(E)} Z_\epsilon(E_1))$$
	is strictly negative, and the lower order terms in $\epsilon$ vanish.
\end{lemma}
\begin{proof}
	Write $Z_\epsilon(E_1) = r_\epsilon(E_1) e^{\varphi_\epsilon(E_1)}$. Then 
	$$\Im(e^{-i\varphi_\epsilon(E)} Z_\epsilon(E_1)) = r_\epsilon(E_1) \Im (e^{i\varphi_\epsilon(E_1) - i\varphi_\epsilon(E)}).$$
	We have $\Im (\exp(i(\varphi_\epsilon(E_1) - \varphi_\epsilon(E)))) = \sin(\varphi_\epsilon(E_1) - \varphi_\epsilon(E))$ where the term $\varphi_\epsilon(E_1) - \varphi_\epsilon(E)$ is equal to $C\epsilon^{2q} + O(\epsilon^{2q+1})$ for some $C<0$. To leading order $\sin x = x + O(x^3)$ and $r_\epsilon(E_1) = C' \rk(E_1) [\omega]^n + O(\epsilon)$ for some $C'>0$ so we obtain an expansion
	$$\Im(e^{-i\varphi_\epsilon(E)} Z_\epsilon(E_1)) = C C' \rk(E_1) [\omega]^n \epsilon^{2q} + O(\epsilon^{2q+1})$$
	for $\epsilon>0$ sufficiently small, and since $C<0$ the we get the result.\footnote{We could have avoided this straightforward comparison of phases and slopes by \emph{defining} the order of discrepancy as the first order at which $\Im(e^{-i\varphi_\epsilon(E)} Z_\epsilon(E_1))$ became negative. The given definition is closer in style to the study of polynomial Bridgeland stability however.}
\end{proof}

\subsection{Step 2: Approximate solutions below the discrepancy order}

Before constructing approximate solutions, we need to understand the expansion of the linearised operator in $\epsilon$ for the Chern connections $g\cdot A_t$ which will occur at successive stages of constructing approximate solutions. Critically, so long as our perturbations are of order at least $\epsilon$, the \emph{leading} term in $P_\epsilon$ remains the bundle Laplacian $\Lap_0$ on $\Gr(E)$. 

\begin{lemma}\label{lem:bundlelaplacianexpansion}
	The bundle Laplacian for $A_t$ has the expansion
	$$\Lap_t = \Lap_0 + t L_1 + t^2 L_2$$
	where
	\begin{align*}
		L_1(u) &= i\contr_{\omega} (\del_0([\gamma,u]) - [\gamma^*, \delbar_0(u)] + \delbar_0([\gamma^*, u]) - [\gamma, \del_0(u)])\\
		L_2(u) &= i\contr_{\omega}([\gamma,[\gamma^*,u]] - [\gamma^*,[\gamma,u]]).
	\end{align*}
\end{lemma}
\begin{proof}
	Recall that on the endomorphism bundle we have
	$$\nabla_t = \nabla_0 + t[\gamma - \gamma^*, -]$$
	where $\gamma$ has type $(0,1)$. Then 
	\begin{align*}
		\Lap_t(u) &= i\contr_{\omega} (\del_t \delbar_t - \delbar_t \del_t) (u)\\
		&= i\contr_{\omega} \left((\del_0 -t[\gamma^*, -])(\delbar_0 + t[\gamma, -]) - (\delbar_0 + t[\gamma, -])(\del_0 -t[\gamma^*, -])\right)(u)\\
		&=\Lap_0(u) + ti\contr_{\omega}\left( \del_0([\gamma,u]) -[\gamma^*, \delbar_0(u)] +\delbar_0([\gamma^*, u]) - [\gamma, \del_0(u)]\right) \\
		&\, + t^2 i\contr_{\omega} \left( [\gamma, [\gamma^*, u]] - [\gamma^*, [\gamma, u]] \right),
	\end{align*} which proves the result.
\end{proof}

\begin{corollary}\label{cor:linearisationperturbed}
	Let $P_\epsilon$ be the linearisation of the $Z$-critical operator $D_\epsilon$ at a connection $g\cdot A_t$ for $g=\exp(s)$ and $t=\lambda \epsilon^q$. Then if $s=O(\epsilon)$ we have the expansion
	$$P_\epsilon =  C(\rk E)[\omega]^n \Lap_0 + O(\epsilon).$$
\end{corollary}
\begin{proof}
	By the same argument as \cref{lem:linearisation} applied at $g\cdot A$ we obtain
	$$P_\epsilon = C(\rk E)[\omega]^n \Lap_t + O(\epsilon)$$
	and then we apply \cref{lem:bundlelaplacianexpansion}.
\end{proof}

\begin{remark}
	This explains analytically the main difference between the stable case, where the leading term in $\epsilon$ of the linearisation of the $Z$-critical operator was invertible modulo constants, and the semistable case, where $\Lap_0$ has non-trivial kernel generated by the identity endomorphisms of $E_1,E_2$.
\end{remark}

Let us now begin to construct approximate solutions to the $Z$-critical equation.

\begin{proposition}\label{prop:approximatesolutionbelow}
	Suppose $E$ is asymptotically $Z$-stable with graded object $\Gr(E) = E_1 \oplus E_2$ and $E_1$ has order of discrepancy $q$. Pick $t=\lambda \epsilon^q$. Then there exists Hermitian endomorphisms $f_0,\dots,f_{2q-1}$ of $E$ such that if
	$$g:= \exp\left( \sum_{j=0}^{2q-1} f_j \epsilon^j\right)$$
	and we set $\tilde A_t := g\cdot A_t$ then
	$$D_\epsilon(\tilde A_t) = O(\epsilon^{2q}).$$
\end{proposition}
\begin{proof}
	Let $h$ denote the weak Hermite--Einstein metric on $\Gr(E)$ with its complex structure $\delbar_0$ and Chern connection $A_0$. By \cref{prop:approximatesolutionsstable} there exist block-diagonal Hermitian endomorphisms $\tilde f_2,\dots,\tilde f_{2q-2}$ such that, working separately on the two factors, we have
	$$\tilde A_0 = \exp \left( \sum_{j=1}^{q-1} \tilde f_{2j} \epsilon^{2j} \right) \cdot A_0$$
	satisfies $D_\epsilon(\tilde A_0) = O(\epsilon^{2q})$. 
	
	In order to obtain a Chern connection on $E$ we must perturb the complex structure from $\del_0$ to $\del_t$, obtaining a Chern connection $A_t$. Now consider 
	$$\tilde A_t^1 = \exp \left( \sum_{j=1}^{q-1} \tilde f_{2j} \epsilon^{2j} \right) \cdot A_t.$$
	Then we get two types of new contributions coming from the off-diagonal term $td_{A_0} a$ and diagonal term $t^2 a\wedge a$ in \eqref{eq:curvatureFt}. These contribute to $D_\epsilon(\tilde A_t^1)$ in the form of a term $\lambda \epsilon^q d_{A_0} a$ at order $\epsilon^q$, and conjugates of this term by the block-diagonal automorphisms $\exp (\tilde f_{2j} \epsilon^{2j})$ contributing off-diagonal errors at order $\epsilon^{q+2j+2\ell}$ for $j,\ell>0$. We also get a contribution $t^2 a\wedge a$ at order $\epsilon^{2q}$ which we can ignore at this step, which also contributes errors at orders $\epsilon^{2q+2j+2\ell}$ for $j,\ell>0$ after being conjugated by $\exp (\tilde f_{2j} \epsilon^{2j})$. From the expansion of wedge products of curvature terms similar to \eqref{eq:curvaturewedgeproduct} we also obtain new terms resulting from wedge products of $td_{A_0}a$ and $t^2 a\wedge a$ terms, all of which will occur at order $\epsilon^{2q}$ or higher.
	
	Let us now correct those error terms above which arise at orders below $\epsilon^{2q}$. From our working assumption that $E$ is \emph{strictly} slope semistable, we know that that discrepancy order of $E_1$, $q\ge 1$. By \cref{cor:linearisationperturbed} at $\tilde A_t^1$ the linearisation has leading order given (up to a constant factor) by the bundle Laplacian $\Lap_0$. The first error term occurs at order $\epsilon^q$ provided by $\lambda i\contr_{\omega} d_{A_0} a$. Since $d_{A_0} a$ is off-diagonal with respect to the block-matrix decomposition for $\Gr(E)=E_1\oplus E_2$, it is orthogonal to the kernel of $\Lap_0$ and its contraction is in the image of $\Lap_0$. In particular we can find an off-diagonal Hermitian endomorphism $f_q$ such that
	$$\Lap_0(f_q) + \lambda i\contr_{\omega} d_{A_0} a=0.$$
	Setting
	$$\tilde A_t^2 = \exp (\epsilon^q f_q) \cdot \tilde A_t^1$$
	we have $D_\epsilon(\tilde A_t^2) = O(\epsilon^{q+1}).$
	
	Now the correction at order $\epsilon^q$ has introduced new error terms at higher orders in $\epsilon$. However $f_q$ is off-diagonal, so the lowest order \emph{diagonal} errors introduced when perturbing to $\tilde A_t^2$ come from $\epsilon^q \Lap(f_q)$ and the conjugate of $\epsilon^q \lambda i\contr_{\omega} d_{A_0} a$ by $\exp(f_q \epsilon^q)$. These occur starting at order $\epsilon^{2q}$ so we can ignore them at this step (but we will consider them in \cref{prop:approximatesolutiondiscrepancy}). Thus again we only have an off-diagonal error at order $\epsilon^{q+2}$ which lies in the image of $\Lap_0$, and we can perturb from $\tilde A_t^2$ to $\tilde A_t^3$ to correct this error. Again higher order errors occur but are off-diagonal to at least order $\epsilon^{2q+2}$. Continuing inductively we produce an approximate solution $\tilde A_t = \tilde A_t^{q-1}$ satisfying
	$$D_\epsilon(\tilde A_t) = O(\epsilon^{2q}).$$
\end{proof}

\subsection{Step 3: Stability at the discrepancy order}

The next step is to correct the errors occuring at the critical order $\epsilon^{2q}$ coming from the discrepancy between the $Z$-slopes of $E_1$ and $E$. Whereas in the previous step \cref{prop:approximatesolutionbelow} we only had to deal with off-diagonal error terms arising from perturbing from $A_0$ to $A_t$, or from off-diagonal errors introduced after subsequent corrections, at order $\epsilon^{2q}$ we have new diagonal contributions, chiefly from $\lambda^2 \epsilon^{2q} a\wedge a$ in the expansion \eqref{eq:curvatureFt} of the curvature of $A_t$. 

Here we make connection with the algebraic condition of asymptotic $Z$-stability of $E$. This will allow us to cancel out the terms which are not in the image of $\Lap_0$ by a judicious choice of deformation rate $\lambda>0$ in $t=\lambda e^{q}$.

First let us introduce the projection operator 
$$\pi: \Gamma(\End_{H}(E,h)) \to \ker \Lap_0 = H^0(X,\End \Gr(E))$$ defined by
$$\pi(u) := \sum_{i=1}^2 \frac{1}{\rk E_i} \left( \int_X \trace_{E_i} (u) \omega^n\right) \id_{E_i}$$
where $\trace_{E_i}$ is defined by taking the block decomposition of $u$ with respect to the smooth splitting $E=E_1\oplus E_2$ and taking the trace of the $i$th diagonal block of $u$. This is, up to a constant factor $n!\vol(X)^{-1}$ the $L^2$-orthogonal projection onto $\ker \Lap_0$ with respect to the inner product
$$(u,v)_{L^2} = \int_X \trace (uv) \omega^n.$$

Further, define a distinguished endomorphism $\id_\pm$ of $E$ by
$$\id_\pm := \frac{1}{\rk E_1} \id_{E_1} - \frac{1}{\rk E_2} \id_{E_2}.$$
Then $\id_\pm$ is in the kernel of $\Lap_0$ and indeed is a holomorphic endomorphism of $\Gr(E)$. It is orthogonal to the subspace $\CC \cdot \langle \id_{E_1} + \id_{E_2} \rangle \subset \ker \Lap_0$, and in particular is orthogonal to $\CC \cdot \id_E$ and therefore is orthogonal to $\ker \Lap_t = H^0(X,\End E)$.

\begin{proposition}\label{prop:approximatesolutiondiscrepancy}
	Suppose $E$ is asymptotically $Z$-stable with graded object $\Gr(E) = E_1 \oplus E_2$ and $E_1$ has order of discrepancy $q$. Then there exists a $\lambda>0$ such that if $t=\lambda \epsilon^q$ then there exist Hermitian endomorphisms $f_1,\dots,f_{2q}$ such that if 
	$$g:= \exp\left( \sum_{j=1}^{2q} f_j \epsilon^j\right)$$
	then
	$$D_\epsilon(\tilde A_t) = O(\epsilon^{2q+1}).$$
\end{proposition}
\begin{proof}
	By \cref{prop:approximatesolutionbelow} there exist Hermitian endomorphisms $f_1,\dots,f_{2q-1}$ so that, in the notion above if we set $f_{2q}=0$, we have
	$$D_\epsilon(g\cdot A_t) = O(\epsilon^{2q}).$$
	We now need to analyse the error terms occuring at this critical order.
	The first component is an error term, say $\sigma_1$, which is a conjugate by $g$ of the term $\epsilon^{2q} \lambda^2 i \contr_{\omega} a \wedge a$ arising from \eqref{eq:curvatureFt}. Let us split this term up as
	$$\sigma_1 = \hat \sigma_1 + \pi(\sigma_1)$$
	where $\hat \sigma_1$ is orthogonal to $\ker \Lap_0$. 
	
	The second error term, say $\sigma_2$, is the total error term introduced in the product of the approximate solution of \cref{prop:approximatesolutionbelow}. As noted in the construction of the approximate solution at lower orders, the only contribution to order $\epsilon^{2q}$ that is not off-diagonal occurs in the construction of the connection $\tilde A_t^2$ when dealing with the error at order $\epsilon^q$. All other diagonal contributions occur at higher order in $\epsilon$. By considering the expansion in the proof of \cref{cor:linearisationperturbed}, this contribution is given by the $t$ coefficient in the expansion \cref{lem:bundlelaplacianexpansion} of the Laplacian $\Lap_t$ applied to the Hermitian endomorphism $f_q$ chosen to satisfy
	$$\Lap_0 (f_q) + \lambda i \contr_{\omega} d_{A_0} a = 0.$$
	Using the expansion of the Laplacian, we therefore have
	\begin{equation}\label{eq:sigma2pi}\pi(\sigma_2) = \lambda \pi (i \contr_{\omega}(\del_0([\gamma,f_q]) - [\gamma^*, \delbar_0(f_q)] + \delbar_0([\gamma^*, f_q]) - [\gamma,\del_0(f_q)])).\end{equation}
	Now first note that $f_q = \lambda (\beta - \beta^*)$ for some $\beta \in \Gamma(\Hom(E_2, E_1))$ because $a=\gamma - \gamma^*$ is of this form. Then we get \eqref{eq:sigma2pi} equals
	$$\lambda^2 \pi \left( i \contr_{\omega}\left( - \del_0([\gamma, \beta^* ]) -[\gamma^*, \delbar_0(\beta)] +\delbar_0([\gamma^*, \beta]) + [\gamma, \del_0(\beta^*)]\right) \right).$$
	Then using the fact that $\del_0 \beta^* = (\delbar_0 \beta)^*$ and $[\alpha,\beta]^* = -[\alpha^*, \beta^*]$ we have
	\begin{align*}
		- \del_0([\gamma, \beta^* ]) + [\gamma, \del_0(\beta^*)] &= \del_0 (  [\gamma^* , \beta]^*) + [\gamma , ( \delbar_0 (\beta) )^*] \\
		&= \left( \delbar_0 (  [\gamma^* , \beta]  -  [\gamma^* , ( \delbar_0 (\beta) )] \right)^*.
	\end{align*}
	Now, as $\delbar_0 (  [\gamma^* , \beta]  -  [\gamma^* , ( \delbar_0 (\beta) )]$ is purely imaginary,
	$$ \trace_{E_i} \left( \delbar_0 (  [\gamma^* , \beta]  -  [\gamma^* , ( \delbar_0 (\beta) )] \right)^* = - \trace_{E_i} \left( \delbar_0 (  [\gamma^* , \beta]  -  [\gamma^* , ( \delbar_0 (\beta) )] \right) . $$
	Thus 
	$$ \trace_{E_i} \left( - \del_0([\gamma, \beta^* ]) -[\gamma^*, \delbar_0(\beta)] +\delbar_0([\gamma^*, \beta]) + [\gamma, \del_0(\beta^*)]\right) =0,$$
	and so in fact
	$$ \pi (\sigma_2 ) =0.$$
	Thus we will be able to remove the error caused by $\sigma_2$ using $\Lap_0$.
	
	The third term error term at order $\epsilon^{2q}$, say $\sigma_3$, arises from the expansion of the $Z$-critical equation at order $\epsilon^{2q}$. Let us split this term up as
	$$\sigma_3 = \hat \sigma_3 + \pi(\sigma_3)$$
	where $\hat \sigma_3$ is orthogonal to the kernel of $\Lap_0$. Thus we have an expansion
	$$D_\epsilon(g\cdot A_t) = \epsilon^{2q}(\sigma_1 + \sigma_2 + \sigma_3) + O(\epsilon^{2q+1})$$
	where the only factors not orthogonal to $\ker \Lap_0$ are
	$$\pi(\sigma_1) + \pi(\sigma_3)$$
	where $\sigma_1$, a conjugate of $\lambda^2 i\contr_{\omega} a\wedge a$ depends on our deformation rate $\lambda > 0$. We will now show that assuming asymptotic $Z$-stability, we can choose $\lambda$ so that this sum vanishes.
	
	Now the projection $\pi(\sigma_1)$ is a positive multiple of
	$$\sum_{i=1}^2 \frac{1}{\rk E_i} \int_X \trace_{E_i} (\lambda^2 i\contr_{\omega} a\wedge a) \omega^n \cdot \id_{E_i}$$
	where we have used the trace to get rid of the conjugation by $g$. Then using the fact that $\trace \gamma \wedge \gamma^* = - \trace \gamma^* \wedge \gamma$ we see $\trace_{E_1}(a\wedge a) = - \trace_{E_2}(a\wedge a)$ so the above projection equals
	$$C\lambda^2 \left( \frac{1}{\rk E_1} \id_{E_1} - \frac{1}{\rk E_2} \id_{E_2} \right)$$
	for some positive constant $C$ depending on $\gamma$. On the other hand $\pi(\sigma_3)$ is a positive multiple of the sum
	$$\sum_{i=1}^2 \left[\Im(e^{-i\varphi_\epsilon(E)} Z_\epsilon(E_i) )\right]^{2q} \cdot \frac{\id_{E_i}}{\rk E_i}$$
	where $[-]^{2q}$ denotes the order $\epsilon^{2q}$-coefficient.
	Thus we need to find $\lambda>0$ solving the equation
	\begin{equation}\label{eqn:keyequation}
		\begin{cases}
			&\left[\Im(e^{-i\varphi_{\varepsilon}(E)} Z_{\varepsilon} (E_1)\right]^{2q} + C \lambda^2 = 0\\
			&\left[\Im(e^{-i\varphi_{\varepsilon}(E)} Z_{\varepsilon} (E_2)\right]^{2q} - C \lambda^2 = 0.
		\end{cases}
	\end{equation}
	This implies
	$$\lambda^2 = \frac{1}{C} \left[\Im(e^{-i\varphi_\epsilon(E)} Z_\epsilon(E_2)\right]^{2q}=-\frac{1}{C} \left[\Im(e^{-i\varphi_\epsilon(E)} Z_\epsilon(E_1)\right]^{2q}$$
	so we require the two conditions
	$$\left[\Im(e^{-i\varphi_\epsilon(E)} Z_{\varepsilon} (E_2)\right]^{2q} = - \left[\Im(e^{-i\varphi_\epsilon(E)} Z_{\varepsilon} (E_1)\right]^{2q}$$
	and
	$$ \left[\Im(e^{-i\varphi_\epsilon(E)} Z_{\varepsilon} (E_2)\right]^{2q} > 0.$$
	The first condition follows from the additivity of the central charge $Z$ for the short exact sequence
	\begin{center}
		\ses{E_1}{E}{E_2}
	\end{center}
	and the fact that $\Im(e^{-i\varphi_\epsilon(E)} Z_\epsilon(E)) = 0$ for all $\epsilon$. The second condition follows from \cref{lem:orderofdiscrepancy} and the see-saw property \cref{lem:seesawazs}, using the fact that $q$ is the order of discrepancy of the subbundle $E_1\subset E$ (or equivalently of the quotient $E\onto E_2)$. 
	
	Thus we can always choose some $\lambda>0$ solving \eqref{eqn:keyequation}. The remaining terms $\hat \sigma_1 + \sigma_2 + \hat \sigma_3$ are all orthogonal to the kernel of $\Lap_0$ and can be removed using a Hermitian endomorphism $f_{2q}$ as before.
\end{proof}

\subsection{Step 4: Approximate solutions to arbitrary order}

Having fixed the deformation rate $\lambda>0$ and constructed approximate solutions at the critical discrepancy order of $E_1$, we move on to the construction of approximate solutions to arbitrary order. 

\begin{lemma}
	\label{lem:tlaplacianprojection}
	$$\pi(\Lap_t(\id_{\pm})) = t^2 C \id_{\pm}$$
	for some non-zero constant $C$ independent of $t$.
\end{lemma}
\begin{proof}
	Using that $\id_{\pm}\in \ker \Lap_0$, let us compute
	\begin{align*}
		\Lap_t (\id_{\pm}) &=  ti\contr_{\omega}\left( \del_0([\gamma,\id_{\pm}]) -[\gamma^*, \delbar_0(\id_{\pm})] +\delbar_0([\gamma^*, \id_{\pm}]) - [\gamma, \del_0(\id_{\pm})]\right)\\
		&\, + t^2 i\contr_{\omega} \left( [\gamma, [\gamma^*, \id_{\pm}]] - [\gamma^*, [\gamma, \id_{\pm}]] \right).
	\end{align*}
	Note that we have $\del_0 \id_{\pm} = \delbar_0 \id_{\pm} = 0$ as $\nabla_0 = \del_0 + \delbar_0$ is a product connection and $\id_{\pm}$ consists of two constant components on the factors of $E=E_1 \oplus E_2$.  This reduces the $t$ term to
	$$ti\contr_{\omega} \left( [\del_0 \gamma, \id_{\pm}] + [\delbar_0 \gamma^*, \id_{\pm}]\right).$$
	Recall that the gauge $\delbar_0^* \gamma = 0$ from \eqref{eq:coulombgauge} implies $\contr_{\omega}(\del_0 \gamma) = \contr_{\omega}(\delbar_0 \gamma^*) = 0$, so this term in fact vanishes. The $t^2$ term can be simplified by computing
	$$[\gamma, [\gamma^*, \id_{\pm}]] - [\gamma^*, [\gamma, \id_{\pm}]] = \frac{2}{\rk E} (\gamma \wedge \gamma^* + \gamma^* \wedge \gamma).$$
	Computing the orthogonal projection we now have
	$$\pi(\Lap_t(\id_{\pm})) = \frac{2it^2}{\rk E} \sum_{i=1}^2 \frac{1}{\rk E_i} \left( \int_X \trace_{E_i} (\contr_{\omega} (\gamma \wedge \gamma^* + \gamma^* \wedge \gamma)) \omega^n \right) \id_{E_i}.$$
	We now use the fact that $\trace (\gamma\wedge \gamma^* + \gamma^* \wedge \gamma) = \trace_{E_1} (\gamma \wedge \gamma^*) + \trace_{E_2} (\gamma^* \wedge \gamma) = 0$, to conclude
	$$\pi(\Lap_t (\id_{\pm})) = t^2 \frac{2 i \int_X \trace_{E_1} (\contr_{\omega} (\gamma \wedge \gamma^*)) \omega^n}{\rk E} \id_{\pm}.$$
\end{proof}

\begin{proposition}\label{prop:linearisationidpm}
	Let $P_\epsilon$ denote the linearisation of the $Z$-critical operator $D_\epsilon$ at $A_t$ where $t=\lambda \epsilon^q$ for $\lambda>0$ and $q\ge 2$. Then $P_\epsilon(\id_\pm)$ is order $O(\epsilon^q)$ and furthermore
	$$\pi(P_\epsilon(\id_\pm)) = C\lambda^2 \epsilon^{2q} \id_\pm + O(\epsilon^{2q+1})$$
	for some non-zero constant $C$. This expansion also holds at a complex structure $g\cdot \del_t$ provided $g=\exp(s)$ for some $s$ which is $O(\epsilon)$ of the form used to construct an approximate solution.
\end{proposition}
\begin{proof}
	By \cref{cor:linearisationperturbed} we know 
	$$P_\epsilon=C(\rk E) [\omega]^n \Lap_t + O(\epsilon).$$
	
	Recall from \cref{sec:subsolutions} that the linearisation of the $Z_{\epsilon}$-critical equation at a Chern connection $A$ consists of a sum of terms of the form
	\begin{equation}\label{eq:linearisationterm}f\mapsto C \epsilon^{2\kappa} \omega^i \wedge F_A^{j-1} \wedge (\del_A \delbar_A - \delbar_A \del_A) s \wedge \tilde U_\ell\end{equation}
	where $C$ is constant, $\kappa\ge 0$, and $i+j+\ell = n$ (and graded symmetrisations of such terms). We wish to consider the case where $A=A_t$ also depends on $\epsilon$. By \cref{lem:tlaplacianprojection} after applying the linearisation to $\id_\pm$, we see that the contribution of $C\lambda^2 \epsilon^{2q} \id_{\pm}$ from the Laplacian $\Lap_t$ comes from the term above with $j=1$ and $\kappa=0$, so we can assume $j>1,\kappa>0$. The computation of \cref{lem:bundlelaplacianexpansion} shows that the difference
	$$(\del_{A_t} \delbar_{A_t} - \delbar_{A_t} \del_{A_t}) - (\del_{A_0} \delbar_{A_0} - \delbar_{A_0} \del_{A_0})$$
	equals
	\begin{multline*} 
		s \mapsto \lambda \varepsilon^q \left( \del_0([\gamma,s]) -[\gamma^*, \delbar_0(s)] +\delbar_0([\gamma^*, s]) - [\gamma, \del_0(s)]\right) \\
		+  \lambda^2 \varepsilon^{2q} \left( [\gamma, [\gamma^*, s]] - [\gamma^*, [\gamma, s]] \right).
	\end{multline*}
	The corresponding term of the form \eqref{eq:linearisationterm} in the linearisation is multiplied by $\epsilon^{2\kappa}$ for some $\kappa>0$, and we only wish to understand terms of order $\epsilon^{2q}$. Thus we only need to consider the contribution from the $\epsilon^q$ term
	\begin{equation}\label{eq:linearisationtterm}s \mapsto \lambda \varepsilon^q \left( \del_0([\gamma,s]) -[\gamma^*, \delbar_0(s)] +\delbar_0([\gamma^*, s]) - [\gamma, \del_0(s)]\right).\end{equation}
	When $s=\id_\pm$, this consists of constant multiples of $\del_0 \gamma$ and $\delbar_0 \gamma^*$, which in particular are off-diagonal. Moreover since the connection $A_t$ is a product of the weak Hermite--Einstein connections $E_1,E_2$ up to order $\epsilon^{q-1}$, the curvature $F_{A_t}$ is a diagonal up to order $\epsilon^{q-1}$. Using furthermore that
	$$\Lap_0 \id_\pm = 0$$
	we see that the term in \eqref{eq:linearisationterm} is a product of:
	\begin{itemize}
		\item a coefficient $C\epsilon^{2\kappa}$ for $\kappa>0$,
		\item an off-diagonal term of order $\epsilon^q$ coming from \eqref{eq:linearisationtterm} with $s=\id_\pm$, and
		\item a product of curvature terms $F_{A_t}^{j-1}$ which are diagonal up to and including order $\epsilon^{q-1}$. 
	\end{itemize}
	Thus the term in \cref{eq:linearisationterm} is off-diagonal to at least order $2\kappa+ q + q - 1 = 2q-1+2\kappa > 2q$. Thus every such term is in the kernel of $\pi$ up to and including to order $\epsilon^{2q}$, showing
	$$\pi(P_\epsilon(\id_\pm)) = C\lambda^2 \epsilon^{2q} \id_{\pm} + O(\epsilon^{2q+1}).$$
	
	Now suppose we make a perturbation $g\cdot A_t$ with $g$ an automorphism arising in the construction of approximate solutions. Since the perturbation $g$ is by diagonal automorphism up to order $\epsilon^{q-1}$, we still have that $F_{g\cdot A_t}^{j-1}$ is diagonal up to and including order $\epsilon^{q-1}$, so we just need to analyse 
	$$\del_{g\cdot A_t}\delbar_{g\cdot A_t} - \delbar_{g\cdot A_t} \del_{g\cdot A_t} - (\del_{A_0} \delbar_{A_0} - \delbar_{A_0} \del_{A_0}).$$
	For example we have
	$$\del_{g\cdot A_t} \delbar_{g\cdot A_t} = g \circ \del_{A_t} \circ g^{-2} \circ \delbar_{A_t} \circ g.$$
	Thus we see that provided $g=\exp(s)$ has $s=O(\epsilon)$ then for the perturbed connection $g\cdot A_t$ we have the exact same contribution \eqref{eq:linearisationtterm} at order $\epsilon^q$. Thus all the above assumptions are satisfied for the terms \eqref{eq:linearisationterm} when $A=g\cdot A_t$, so we obtain the same expansion for $\pi(P_\epsilon(\id_\pm))$ for the linearisation at $g\cdot A_t$ also.
\end{proof}

\begin{proposition}\label{prop:approximatesolutionarbitrary}
	Suppose $E$ is asymptotically $Z$-stable with graded object $\Gr(E) = E_1 \oplus E_2$ and $E_1$ has order of discrepancy $q$. Let $r\in \ZZ_{>0}$. Then there exists a $\lambda>0$ such that if $t=\lambda \epsilon^q$ then there exist Hermitian endomorphisms $f_1,\dots,f_{r}$ such that if 
	$$g:= \exp\left( \sum_{j=1}^{r} f_j \epsilon^j\right)$$
	then
	$$D_\epsilon(\tilde A_t) = O(\epsilon^{2r+1}).$$
\end{proposition}
\begin{proof}
	The proof is by induction. By \cref{prop:approximatesolutiondiscrepancy} we can assume $r>2q$ and this fixes our choice of $\lambda>0$. Write $A_{r-1} = g_{r-1} \cdot A_t$ for the approximate solution with $D_\epsilon(A_{r-1}) = O(\epsilon^{r})$. Then since each term in the expansion of $D_\epsilon(A_{r-1})$ is orthogonal to $\id_E$, we can write this as 
	$$D_\epsilon(A_{r-1}) = (c_r \id_\pm + \sigma_r) \epsilon^r + O(\epsilon^{r+1})$$
	where $c_r$ is a constant and $\sigma_r$ is orthogonal to $\Lap_0$.
	
	Since $r>2q$, the leading order in $\epsilon$ change in the expansion of
	$$D_\epsilon(\exp(\tau_r \epsilon^{r-2q} \id_\pm) \cdot A_{r-1})$$
	is given by the linearisation at $A_{r-1}$ applied to $s=\tau_r \epsilon^{r-2q} \id_\pm$. In particular by \cref{prop:linearisationidpm} we have
	$$\pi(D_\epsilon(\exp(s) \cdot A_{r-1})) = (c_r + \tau_r C) \epsilon^{2q} \id_\pm + O(\epsilon^{r+1})$$
	for some non-zero constant $C$ depending on $\lambda$. Thus one can choose $\tau_r$ such 
	$$c_r + \tau_r C = 0.$$
	Set $\tilde A_r := \exp(\tau_r \epsilon^{r-2q} \id_\pm) \cdot A_{r-1}$ for this choice of $\tau_r$. Then we have that $D_\epsilon(\tilde A_r)$ is orthogonal to $\ker \Lap_0$ up to order $\epsilon^{r+1}$. However we have only controlled for the projection onto $\ker \Lap_0$, and we may have introduced errors at lower orders in $\epsilon$ which are orthogonal to $\ker \Lap_0$. That is, we have some expansion
	$$D_\epsilon(\tilde A_r) = \sum_{j=0}^{2q-2} \tilde \sigma_j \epsilon^{r-2q+2+j} + O(\epsilon^{r+1})$$
	for some terms $\tilde \sigma_j$ orthogonal to $\ker \Lap_0$ with $\tilde \sigma_j$ diagonal for $j\le q-1$ because $\id_\pm$ is and the approximate solution $A_{r-1}$ is a product structure up to order $\epsilon^{q-1}$. Now we can repeat the standard procedure of constructing approximate solutions to successively cancel these errors which begin at order $\epsilon^{r-2q+2}$ using the linearisation. In particular there exists some $\tilde f_0$ such that 
	$$\Lap_0 (\tilde f_0) + \tilde \sigma_0 = 0$$
	so taking $\tilde A_r^1 = \exp(\tilde f_0 \epsilon^{r-2q+2}) \cdot \tilde A_r$ we remove the $\tilde \sigma_0$ error. This now introduces new errors at higher orders $\epsilon^{r-2q+3}$ and higher, but these errors stay orthogonal to $\id_\pm$ up to at least order $\epsilon^{r+2}$, as the leading order correction $P_\epsilon (\tilde\sigma_j)$ is orthogonal to $\id_\pm$ up to order $q$, since the connection $A_{r-1}$ is a product up to order $q$. Moreover the higher order corrections act at order at least $r-2q+2+j+2q > r+2$ and so do not effect the correction at lower orders (for a similar correction, see \cite[Prop. 5.16]{sektnan2020hermitian}). Thus we obtain a new expansion of the above form with a sum starting at $j=1$, where the $\tilde \sigma_j$ are still orthogonal to $\ker \Lap_0$. Successively removing all these errors in the same way we produce an approximate solution $\tilde A_r^{2q-2} = A_r$ with 
	$$D_\epsilon(A_r) = O(\epsilon^{r+1})$$
	as desired.
\end{proof}

\subsection{Step 5: Perturbing to a solution\label{sec:step5}}

Having constructed approximate solutions to any choice of order, we now wish to apply the inverse function theorem to obtain nearby solutions for all $\epsilon>0$ sufficiently small. To do so we apply a quantitative version of the inverse function theorem (which follows immediately from the standard proof, taking care to notice when the Lipschitz constant is being used to construct the contraction before applying the Banach fixed point theorem). In order to apply this theorem, it is necessary to obtain a bound on the inverse of the linearised operator $Q_\epsilon$ at an approximate solution $A_r$ with $r>q$, which is uniform in $\epsilon$ for $\epsilon>0$ sufficiently small. This will allow us to find neighbourhoods of $A_r$ of definite size upon which $D_\epsilon$ is surjective provided $r$ is taken large enough, and therefore obtain solutions for all $\epsilon$ sufficiently small.

The required bound is the following. Here the Sobolev space $L_{d,0}^2$ denotes the subspace of $L_d^2$ of trace-average zero sections of $\End E$ with respect to $h$ and $\omega$. These are precisely the sections orthogonal to $\id_E\in L_d^2$. 

\begin{proposition}\label{lem:inversebound}
	Suppose the discrepancy order of $E_1$ is $q$. For any $d,r\in \ZZ_{>0}$, let $P_\epsilon: L_{d+2,0}^2 \to L_{d,0}^2$ be the linearisation of the $Z$-critical operator $D_\epsilon$ at the connection $A_r$ constructed in \cref{prop:approximatesolutionarbitrary}. Then $P_\epsilon$ is invertible for all $\epsilon$ sufficiently small. If $Q_\epsilon: L_{d,0}^2 \to L_{d+2,0}^2$ denotes the inverse, then there exists $C>0$ depending on $d,r$ such that 
	$$\|Q_\epsilon\|_{L_{d,0}^2 \to L_{d+2,0}^2} \le C \epsilon^{-2q}.$$
\end{proposition}
\begin{proof}
	Let $s\in \Gamma_0(\End E)$ and write $s=\hat s + c_s \id_\pm$ where $\hat s$ is orthogonal to $\ker \Lap_0$. Consider the operator
	$$\tilde P_\epsilon : L_{d+2,0}^2 \to L_{d,0}^2$$
	defined by
	$$\tilde P_\epsilon(s) = \rk(E) [\omega]^n \Lap_0 (\hat s) + c_s C \lambda^2 \epsilon^{2q} \id_\pm$$
	where $C\ne 0$ is the constant given in \cref{prop:linearisationidpm}. This operator is invertible, as $\Lap_0(\hat s)$ remains orthogonal to $\ker \Lap_0$ when $\hat s$ is, and $\Lap_0$ is invertible orthogonal to its kernel. 
	
	It follows that $\|\tilde P_\epsilon(\hat s)\| \ge C\|\hat s\|$ for $\hat s$ in the orthogonal complement to $\ker \Lap_0$ by the Poincar\'e inequality for the bundle Laplacian $\Lap_0$. 
	
	On the other hand we have $\|\tilde P_\epsilon(c\id_\pm)\| \ge C \epsilon^{2q} \|c\id_\pm\|$ for some $C$, so combining these inequalities we obtain a bound
	$$\|\tilde P_\epsilon(s)\| \ge C\epsilon^{2q} \|s\|$$
	for some $C>0$ and any $s$ orthogonal to $\id_E$. 
	
	Recall that in the construction of approximate solutions, the perturbed solution is a product up to and including order $\epsilon^{q-1}$. 
	
	Now for $\hat s\in \ker \Lap_0^\perp$ we have that $\tilde P_\epsilon(\hat s)$ is the leading order term in the linearisation $P_\epsilon(\hat s)$ by \cref{cor:linearisationperturbed}. Thus we obtain a similar bound
	$$\|P_\epsilon(\hat s)\| \ge C\|\hat s\|$$
	for such $s$, provided $\epsilon>0$ is chosen sufficiently small. Similarly by \cref{prop:linearisationidpm} the leading order term in the projection of $P_\epsilon(\id_\pm)$ is given by $\tilde P_\epsilon(\id_\pm)$. Thus we obtain a bound
	$$\|\pi P_\epsilon(c\id_\pm)\| \ge C \epsilon^{2q} \|c \id_\pm\|$$
	for $\epsilon>0$ sufficiently small. Let us now consider the difference
	$$R_\epsilon = P_\epsilon - \tilde P_\epsilon.$$
	Then we have expansions
	$$R_\epsilon(\hat s) = \sum_{j=1}^{2q} \epsilon^j T_j(\hat s) + O(\epsilon^{2q+1})$$
	for $\hat s \in \ker \Lap_0^\perp$ and
	$$R_\epsilon(\id_\pm) = \sum_{j=0}^{2q} \epsilon^j \sigma_j + O(\epsilon^{2q+1})$$
	where $\pi \sigma_j = 0$ for all such $j$ by \cref{prop:linearisationidpm}.
	
	Now we observe that the image of $R_\epsilon(\hat s)$ lies in $\ker \Lap_0^\perp$ up to and including order $\epsilon^{q}$. Indeed the only relevant diagonal contribution occurs at order $\epsilon^{q}$ and is precisely the error term $\sigma_2$ occuring in \cref{prop:approximatesolutiondiscrepancy}, which was observed to satisfy $\pi \sigma_2 = 0$. The contribution to the linearisation from this term is the operator
	$$\hat s \mapsto \epsilon^{2q} i \contr_{\omega} d_0 ([[\gamma-\gamma^*,\sigma-\sigma^*], \hat s])$$
	which therefore lies in $\ker \Lap_0^\perp$ for any $\hat s$, since $\sigma \circ \gamma^* - \gamma \circ \sigma^*$ is traceless.
	
	To obtain the required bound, by self-adjointness it is equivalent to prove the bound
	$$\langle P_\epsilon(s), s\rangle \ge C \epsilon^{2q} \|s\|^2.$$
	By the above discussion the case where $s=\hat s$ is orthogonal to $\ker \Lap_0$ and $s=\id_\pm$ are concluded, so it remains to verify the bound for an arbitrary $s=\hat s + c_s \id_\pm$. Thus we need to verify
	$$\langle P_\epsilon(\hat s), c_s \id_\pm\rangle + \langle \hat s, P_\epsilon(c_s \id_\pm)\rangle \ge C \epsilon^{2q} \|s\|^2.$$
	By the above discussion concerning $R_\epsilon(\hat s)$ we know $P_\epsilon(\hat s)$ is orthogonal to $\id_\pm$ up to and including order $\epsilon^{q}$, and thus $|\langle P_\epsilon(\hat s), c_s\id_\pm \rangle| \ge C' \epsilon^{q+1} |c_s| \|s\|$. We now wish to show also that 
	$$\langle P_\epsilon(c_s \id_\pm), \hat s\rangle \ge C\epsilon^{q+1} |c_s|\|s\|.$$
	
	From the expansion of \cref{lem:bundlelaplacianexpansion} and \cref{cor:linearisationperturbed}, and the fact that the approximation solution $A_r$ is diagonal up to order $\epsilon^q$, we see the key contribution arises due to the inner product with the term at order $\epsilon^q$ in the linearisation of the curvature, which takes the form
	$$i \contr_{\omega} (\del_0([\gamma_\epsilon, c_s \id_\pm]) + \delbar_0([\gamma_\epsilon^*, c_s \id_\pm])).$$
	Since $\hat s$ is a Hermitian endomorphism, the diagonal contributions coming from the inner product with the above form sum to be traceless, and therefore the inner product at order $\epsilon^q$ vanishes. Thus we obtain the desired bound, and that there is some $C'>0$ such that
	$$\langle P_\epsilon(\hat s), c_s \id_\pm\rangle + \langle \hat s, P_\epsilon(c_s \id_\pm)\rangle \ge C' \epsilon^{q+1} |c_s| \|\hat s\|$$
	Now writing $\epsilon^{q+1}|c_s|\|\hat s\| = \epsilon \|s\| |\epsilon^q c_s|$ and completing the square we obtain
	\begin{align*}
		C'\epsilon^{q+1} |c_s| \|\hat s\| &= \frac{1}{2} C'\epsilon \left(  (|\epsilon^q c_s| + \|\hat s\|)^2 - (\epsilon^{2q} |c_s|^2 + \|\hat s\|^2) \right)\\
		&\ge -\frac{1}{2} C' \epsilon \left( (\epsilon^{2q} |c_s|^2 + \|\hat s\|^2) \right)
	\end{align*}
	Now this is a negative lower bound, but due to the factor of $\epsilon$ for $\epsilon$ sufficiently small this term can be absorbed into the bounds $\langle P_\epsilon(\hat s), \hat s \rangle \ge C \|\hat s\|^2$ and $\langle P_\epsilon (c_s \id_\pm), c_s \id_\pm\rangle \ge C \epsilon^{2q} |c_s|^2.$ Thus we obtain in total the bound
	$$\langle P_\epsilon(s), s\rangle \ge C \epsilon^{2q} \|s\|^2$$
	for some $C>0$. By self-adjointness of $P_\epsilon$ we obtain the required lower bound on the operator itself for any $s\in L_{d+2,0}^2$. This implies the associated upper bound on the inverse operator $Q_\epsilon$ as desired.
\end{proof}

We also need to more carefully understand when the non-linear part of the $Z$-critical operator $D_\epsilon$ is Lipschitz to utilise the inverse function theorem. We do so using the mean value theorem for Banach spaces as follows.

\begin{lemma}\label{lem:meanvalue}
	Let $\calM_{\epsilon,r}$ denote the non-linear part $D_{\epsilon,r} - P_{\epsilon,r}$ of the $Z$-critical operator at a connection $A_r$ constructed in \cref{prop:approximatesolutionarbitrary}. Then there are constants $c,C>0$ such that for all $\epsilon>0$ sufficiently small, we have that for $s_0,s_1\in L_{d+2,0}^2$ with $\|s_i\|\le c$, 
	$$\|\calM_{\epsilon,r}(s_0) - \calM_{\epsilon,r}(s_1)\|_{L_d^2} \le C(\|s_0\|_{L_{d+2}^2} + \|s_1\|_{L_{d+2}^2}) \|s_0-s_1\|_{L_{d+2}^2}.$$
\end{lemma}
\begin{proof}
	We apply the mean value theorem. Let $s_0,s_1\in B_0(c)\subset L_{d+2,0}^2$. Consider the path $s(t) = (1-t)s_0 + ts_1$. Then by the mean value theorem for Banach spaces there exists some $s^* = s(t^*)$ such that
	$$\calM_{\epsilon,r}(s_0) - \calM_{\epsilon,r}(s_1) = (D\calM_{\epsilon,r})_{s^*} (s_0-s_1).$$
	Now 
	$$(D\calM_{\epsilon,r})_{s^*} = (P_{\epsilon,r})_{s^*} - (P_{\epsilon,r})_0$$
	where $(P_{\epsilon,r})_{s}$ is the linearisation of the $Z$-critical operator $D_\epsilon$ at $\exp(s) \cdot A_r$. Thus we need to estimate
	$$\|(P_{\epsilon,r})_{s^*} - (P_{\epsilon,r})_0\|_{L_{d+2}^2}.$$
	Setting $g=\exp(s^*)$ and recalling \cref{lem:linearisationcurvatureChern} with $t=1$ we see that the order $O(\epsilon^0)$ term in the linearisation $(P_{\epsilon,r})_{s^*} - (P_{\epsilon,r})_0$ will consist of terms of the form
	$$V\mapsto \sum_{i=1}^{\infty} f_i(\del_{A_r},\delbar_{A_r}, s^*, V)$$
	where $f_i$ is a polynomial expression each term of which contains exactly $i$ factors of $s^*$, one copy of $\del_{A_r}$, $\delbar_{A_r}$ and $V$. Say we choose $c\le 1/2$ so $\|s^*\|\le 1/2$, then we have $\|s^*\|^i \le \|s^*\|$ so to order $\epsilon^0$, we quickly obtain the bound
	$$\|(P_{\epsilon,r})_{s^*} - (P_{\epsilon,r})_0\|^0\le C \|s^*\|$$
	for some $C$ depending on $A_r$. Choosing $\epsilon$ sufficiently small we can repeat this process for higher orders of $\epsilon$ in the expansion to obtain a bound
	$$\|(P_{\epsilon,r})_{s^*} - (P_{\epsilon,r})_0\|\le C \|s^*\|.$$
	Now since $s^*$ is a convex linear combination of $s_0$ and $s_1$,
	$$\|\calM_{\epsilon,r}(s_0) - \calM_{\epsilon,r}(s_1)\| \le C(\|s_0\| + \|s_1\|) \|s_0-s_1\|$$
	for some $C,c>0$ and all $\epsilon$ sufficiently small such that $\|s_0\|,\|s_1\|\le c$. 
\end{proof}

The above lemma gives a characterisation of the Lipschitz constant of the non-linear part of the operator $D_\epsilon$ on balls of decreasing radius $\rho<c$ around $0\in L_{d+2,0}^2$. Using this characterisation, we will apply the quantitative inverse function theorem to find a solution to the $Z$-critical equation. See for example \cite[Thm. 4.1]{fine}.

\begin{theorem}[Quantitative inverse function theorem]\label{thm:qift}
	Let $\Phi: V \to W$ be a differentiable map of Banach spaces $V,W$, with invertible linearisation $P=D\Phi$ at $0$ with inverse $Q$. Let
	\begin{itemize}
		\item $\delta'$ be the radius of the closed ball in $V$ such that $\Phi - P$ is Lipschitz of constant $\frac{1}{2\|Q\|}.$;
		\item $\delta = \frac{\delta'}{2\|Q\|}.$
	\end{itemize}
	Then for all $w\in W$ with $\|w-\Phi(0)\|<\delta$, there exists a $v\in V$ with $\Phi(v)=w$.
\end{theorem}

Finally we can complete the proof of the existence result.

\begin{proof}[Proof of \cref{thm:stabilityimpliesexistence}]
	We can now prove the existence of a solution to the $Z$-critical equation. We wish to find a root of $D_{\epsilon}$ near some approximate solution $A_r$ constructed in \cref{prop:approximatesolutionarbitrary} for all $\epsilon>0$ sufficiently small. By \cref{lem:meanvalue} we see that there exists some constant $C>0$ such that for all $\rho>0$ sufficiently small, the non-linear part $\calM_{\epsilon,r}$ of $D_{\epsilon}$ at $A_r$ is Lipschitz with Lipschitz constant $2C\rho$ on the ball of radius $\rho$. Moreover, by \cref{lem:inversebound} we have a lower bound
	$$C_r \epsilon^{2q} \le \frac{1}{2\|Q_{\epsilon,r}\|}$$
	for some constant $C_r>0$. Thus there exists a constant $C_r'>0$ such that the radius $\delta'$ of the ball upon which $\calM_{\epsilon,r}$ is Lipschitz with Lipschitz constant $\frac{1}{2\|Q_{\epsilon,r}\|}$ is bounded below by
	$$C_r' \epsilon^{2q} \le \delta'$$
	for a constant $C_r' = C_r/2C>0$.
	
	Therefore by the definition of $\delta$ and again applying the lower bound $\cref{lem:inversebound}$ we obtain the lower bound
	$$C_r''\epsilon^{4q} \le \delta$$
	for some constant $C_r''=C_r^2/2C>0$. 
	Now let us take $r=4q$. Then $\|D_{\epsilon}(A_r)\| \le C_r'''\epsilon^{4q+1}$ for some constant $C_r'''$ so, when $\epsilon>0$ is sufficiently small, $D_{\epsilon}(A_r)$ is contained within the ball of radius $C_r''\epsilon^{4q}$, and hence in the ball of radius $\delta$. 
	
	By the inverse function theorem \cref{thm:qift} for $\Phi = D_\epsilon$ we can therefore find a root of $D_\epsilon$ in $L_{d+2,0}^2$ when $\epsilon$ is sufficiently small. By \cref{lem:linearisation} the $Z$-critical operator is elliptic for such small $\epsilon$, so by elliptic regularity this solution is smooth.
\end{proof}

\subsection{Remarks on the general case}

In the proof of \cref{thm:stabilityimpliesexistence} we have only considered the case where $\Gr(E)$ has one or two components. Indeed in the case where $\Gr(E)$ has two components, \cref{prop:approximatesolutiondiscrepancy} demonstrates clearly how the asymptotic $Z$-stability assumption enters into the analysis of the construction of approximate solutions.

Considerable technical difficulties occur in the general case covered in \cite[\S 4]{dervan2021zcritical} for the following reason: When $\Gr(E)=E_1\oplus \cdots \oplus E_\ell$ the assumption of asymptotic $Z$-stability no longer guarantees that the locally free factors $E_i$ are pairwise non-isomorphic. Indeed $\Gr(E)$ may admit automorphisms which permute factors, so the kernel $\ker \Lap_0 = H^0(X,\End \Gr(E))$ becomes more complicated. This makes the above arguments more difficult in two ways:
\begin{itemize}
	\item The second fundamental form $\gamma$ specifying the deformation of complex structure from $\Gr(E)$ to $E$ has a more complicated off-diagonal shape, meaning extra care must be taken in the construction of approximate solutions at each stage.
	\item Extra care must be taken in the proof of the bound on the inverse $Q_\epsilon$ of the linearised operator \cref{lem:inversebound}, as the kernel does not just consist of sums of diagonal endomorphisms.
\end{itemize}
In order to control these extra factors, in \cite[\S 4.2]{dervan2021zcritical} a refinement of the Jordan--H\"older filtration depending on the stability condition $Z$ is constructed, so that the deformation of complex structure which is induced using this filtration has a particularly nice form. An inductive process on this filtration allows approximate solutions to arbitrary order to be constructed, and the bounds on the inverse $Q_\epsilon$ to be proven.

Once the same approximate solutions and bounds are established, the proof of \cref{thm:stabilityimpliesexistence} repeats without change to find $Z$-critical metrics.

\begin{remark}
	A version of the theory of $Z$-critical connections and asymptotic $Z$-stability has been developed by Dervan for varieties, where so-called ``$Z$-critical K\"ahler metrics" are perturbations of cscK metrics and asymptotic $Z$-stability converges to K-stability \cite{dervan2021stability}. Dervan uses the moment map formalism in that problem to reduce the perturbation result to a finite-dimensional GIT-type problem, and the technique seems to be quite general. In view of the moment map description of \cref{sec:momentmap} for the $Z$-critical equation, it would be interesting if such techniques could be adapted to the study of $Z$-critical metrics on bundles.
\end{remark}

	\part{K\"ahler Fibrations\label{part:fibrations}}
	
	\chapter{Background}
	
	In this chapter we will recall the notion of an \emph{optimal symplectic connection} on a K\"ahler fibration introduced by Dervan--Sektnan \cite{dervan2019optimal}, and, as according to \cref{principle}, the corresponding notion of stability of a fibration \cite{dervan2019moduli}.
	
	For our purposes, we will also describe the theory of stability and Hermite--Einstein connections on holomorphic principal bundles, which is a variant of, but closely related to, the theory for holomorphic vector bundles which has been described in \cref{sec:bundles}.
	
	\section{Optimal symplectic connections}
	
	To begin, we will recall the motivation of the work of Dervan--Sektnan who introduced optimal symplectic connections (OSCs). The spaces we will be considering are certain K\"ahler fibrations.
	
	\subsection{K\"ahler fibrations\label{sec:kahlerfibrations}}
	
	\begin{definition}[K\"ahler fibration]
		A \emph{K\"ahler fibration} consists of a surjective holomorphic submersion
		$$\pi: (X,\omega_X) \to (B,\omega_B)$$
		where $(B,\omega_B)$ is K\"ahler and $\omega_X$ is a closed $(1,1)$-form on the complex manifold $X$ such that the restriction to the fibre directions is non-degenerate. That is, $\omega_X$ is a K\"ahler form in the vertical directions. We call such a form \emph{relatively K\"ahler}. Denote the restriction $\rest{\omega_X}{X_b}$ to any fibre as $\omega_b$.
	\end{definition}

	In the proceeding theory, we will always make the following further assumptions:
	\begin{itemize}
		\item We will always notate that $\dim B = n$ and $\reldim X/B = m$, so that the dimension of any fibre $\dim X_b = m$.
		\item The spaces $X$ and $B$ will always be compact. In this case Ehresmann's lemma implies that $X\to B$ has the structure of a smooth fibre bundle, but the complex structure of the fibres may vary.
		\item We will always assume that the fibres $(X_b,\omega_b)$ are cscK manifolds so that $S(\omega_b)$ is constant for every $b$. We call relatively K\"ahler metrics satisfying this assumption \emph{relatively cscK metrics}.
		\item We will always consider the polarised setting where $\omega_B \in c_1(L)$ for some ample line bundle $L\to B$ and $\omega_X \in c_1(H)$ for some relatively ample line bundle $H\to X$. 
		\item We assume that the dimension $\dim \Aut (X_b,\omega_b)$ is independent of $b$. As a consequence, by the upcoming discussion in \cref{sec:holomorphypotentials} the space $\h_0(X_b, \omega_B)^\RR$ of real mean-zero holomorphy potentials on $X_b$ has dimension independent of $b\in B$.
	\end{itemize}

	Since the form $\omega_X$ is relatively symplectic, the study of K\"ahler fibrations is the complex analogue of the more general theory of symplectic fibrations (see for example \cite{mcduff2017introduction}). Let us emphasise some key features arising from the symplectic structure of the fibration $(X,\omega_X) \to (B,\omega_B)$. Since $\omega_X$ is non-degenerate in the direction of the vertical subbundle $\calV\subset TX$, there is an orthogonal complement $\calH\subset TX$ with respect to $\omega_X$, the \emph{horizontal subbundle}, such that $TX=\calV\oplus \calH$. This defines an Ehresmann connection on $X$ as a fibre bundle, with curvature $F_\calH \in \Omega^2(X, \calV)$ defined by
	$$F_\calH (u,v) = [u_\calH, v_\calH]_\calV = [u_\calH, v_\calH] - [u,v]_\calH.$$
	\begin{definition}[Symplectic curvature]
		Given two vectors $u,v\in T_b B$ then 
		$$F_\calH(u^\#, v^\#)\in \Symp(X_b, \omega_b)$$ is a symplectic vector field, where $u^\#$ and $v^\#$ are the horizontal lifts of $u,v$. Denote (abusively) 
		$$F_\calH \in \Omega^2(B, \Symp(\calV, \omega_X))$$
		the two-form on $B$ with values in fibrewise symplectic vector fields on $(X,\omega_X)$. This is the \emph{symplectic curvature} of $(X,\omega_X) \to B$. 
	\end{definition}
	In this symplectic setting we have the following remarkable theorem.
	\begin{theorem}[Minimal coupling, see {\cite[\S 1]{guillemin1996symplectic}, \cite[Lem. 3.2]{dervan2019optimal}}]\label{thm:minimalcoupling}
		The symplectic curvature $F_\calH$ always takes values in vertical Hamiltonian vector fields. Furthermore if $\mu^*: \Ham(\calV) \to C_0^\infty(X)$ denotes the map taking a vertical Hamiltonian vector field to its associated relative (mean zero) Hamiltonian function on $X$, and we abuse notation by identifying $\mu^*F_\calH$ with its pullback to the total space of $X$, then 
		$$\mu^* F_\calH = (\omega_X)_\calH + \pi^* \beta$$
		where $(\omega_X)_\calH$ is the horizontal component of $\omega_X$ and $\beta$ is some two-form on $B$. 
	\end{theorem}
	
	\subsection{Holomorphy potentials and automorphisms\label{sec:holomorphypotentials}}
	
	As remarked in the previous section, we will always assume that the fibres $(X_b,\omega_b)$ of the K\"ahler fibration are cscK manifolds. Let us now discuss the consequences of this for the symplectic curvature and the automorphisms of the fibration. To do so, we must recall a special class of Hamiltonian-type functions on a K\"ahler manifold, the holomorphy potentials.
	
	\begin{definition}
		Let $(Y,\omega)$ be a compact K\"ahler manifold. A \emph{holomorphy potential} $f$ on $Y$ is a function $f: Y \to \CC$ such that
		$$\delbar \nabla^{1,0} f = 0$$
		where $\nabla^{1,0}$ is the $(1,0)$-component of the Riemannian gradient of $f$. Any such function $f$ defines a holomorphic vector field 
		$$\xi_f = \nabla^{1,0} f,$$
		the $(1,0)$-part of the symplectic gradient of $f$. Denote the space of holomorphy potentials by
		$$\h := \ker \delbar \nabla^{1,0} : C^{\infty}(Y,\CC) \to \Omega^{0,1}(T^{1,0} Y).$$
	\end{definition}

	The holomorphy potentials generate holomorphic vector fields, which live in the Lie algebra $H^0(Y,TY)$ of the holomorphic automorphism group $\Aut(Y)$ of the K\"ahler manifold $Y$. Two holomorphy potentials which differ by a constant define the same vector field, so we fix this indeterminacy by restricting to the \emph{mean-zero} holomorphy potentials, which we denote $\h_0$.
	
	We also have assumed that the fibration is \emph{polarised}, so we have a pair $(X,H) \to (B,L)$ where $H$ is relatively ample and $L$ is ample. For a polarised compact K\"ahler manifold $(Y,\omega,H_Y)$, we can consider the \emph{reduced automorphism group} 
	$$\Aut(Y,H_Y)\subset \Aut(Y)$$ 
	of automorphisms of $Y$ which lift to $H_Y$. The Lie algebra 
	$$\Lie \Aut(Y,H_Y)$$ 
	can be identified with the non-zero holomorphic vector fields on $X$ which vanish somewhere. Such vector fields are precisely those which can be written as $\nabla^{1,0} f$ for some mean-zero holomorphy potential $f$, so we have an identification 
	$$\h_0 \isom \Lie \Aut(Y,H_Y).$$
	See \cite[\S 3.5]{gauduchon2010calabi} for more details. 
	
	If $\omega\in c_1(H_Y)$ is a K\"ahler form, then we obtain a group of holomorphic isometries $\Isom(Y,\omega) \subset \Aut(Y)$, and a reduced isometry group $\Isom(Y,\omega,H_Y)\subset \Aut(Y,H_Y)$ of holomorphic isometries which lift to $H_Y$. The Lie algebra $\Isom(Y,\omega,H_Y)$ is given by the space of holomorphic Killing vector fields which vanish somewhere on $Y$. The following theorem shows that such automorphisms can be described by holomorphy potentials when $(Y,\omega)$ is cscK.
	
	\begin{proposition}[Matsushima--Lichnerowicz theorem (see {\cite[\S 3.5,3.6]{gauduchon2010calabi}})]
		If $(Y,\omega,H_Y)$ is a polarised manifold and $\omega$ is cscK, then the reduced automorphism group $\Aut(Y,H_Y)$ is reductive, and the reduced holomorphic isometry group $\Isom(Y,\omega,H_Y)$ is a maximal compact subgroup. Under the identification 
		$$\h \isom \Lie \Aut (Y,H_Y)$$
		this corresponds to the decomposition
		$$\h_0 = \h_0^\RR \oplus i \h_0^\RR$$
		of (mean zero) holomorphy potentials, where $\h_0^\RR$ denotes the real mean-zero holomorphy potentials $f: Y\to \RR$. With respect to the isomorphism between holomorphy potentials and holomorphic vector fields, a purely imaginary holomorphy potential $f\in i \h_0^\RR$ generates a Killing vector field $\nabla^{1,0} f \in \Lie \Isom(Y, \omega_Y, H_Y)$. 
	\end{proposition}

	\begin{remark}
		\label{rmk:spaceofcscKmetrics}
		One may go further to describe the space of all cscK metrics in the class $c_1(H_Y)$ by constructing an identification with the homogeneous space $\Aut(Y,H_Y)/\Isom (Y,\omega_Y,H_Y)$ for a fixed cscK metric $\omega_Y$. One may further use that the Riemannian exponential map is a diffeomorphism on this space to deduce an isomorphism $$\h_0^\RR \isom \Aut(Y,H_Y)/\Isom (Y,\omega_Y,H_Y)$$ of the \emph{real} mean-zero holomorphy potentials with respect to $\omega_Y$ with the set of cscK metrics in the class $c_1(H_Y)$. See for example \cite[\S 2.2]{hallam2022thesis} for a more detailed discussion of this description.
	\end{remark}

	\subsection{The equation}
	
	Let us now turn to the question studied in \cite{dervan2019optimal}. We will briefly recall its origins in the \emph{adiabatic limit}, which occurs when we consider the de Rham cohomology class 
	$$[\omega_X + k \pi^* \omega_B] = c_1(H) + k c_1(L)$$
	on $X$, which for $k\gg 0$ is a K\"ahler class (in the following we will omit the $\pi^*$ when referring to $\omega_B$ on $X$). In this setting the natural question about canonical metrics is the following:
	
	\begin{question}\label{que:fibrations}
		When does an adiabatic K\"ahler class $[\omega_X + k \omega_B]$ admit a constant scalar curvature K\"ahler metric?
	\end{question}

	To answer this question, it is useful to compute the expansion of the scalar curvature $S(\omega_X + k \omega_B)$ in powers of $k$. We have (see \cite[Cor. 4.7]{dervan2019optimal})
	\begin{multline}\label{eq:adiabaticscalarcurvature}
		S(\omega_X+ k \omega_B) = S(\omega_b) + k^{-1}(S(\omega_B) + \contr_{\omega_B} \rho_\calH + \Lap_\calV (\contr_{\omega_B}(\omega_X)_\calH)) + O(k^{-2}).
	\end{multline}
	Here $S(\omega_b)$ is the function on $X$ whose restriction to a fibre $X_b$ is the scalar curvature of $\omega_b$. Since this appears as the leading order term in the expansion \eqref{eq:adiabaticscalarcurvature}, to first approximation in order to answer \cref{que:fibrations} we should require that $S(\omega_b)$ is a constant function on $X_b$ for every $b$, as we assumed in \cref{sec:kahlerfibrations}.

	The other terms in the expansion \eqref{eq:adiabaticscalarcurvature} are defined as follows:
	\begin{itemize}
		\item For a differential form $\beta$ on $X$, the horizontal component $\beta_\calH$ denotes the restriction to the horizontal distribution defined by $\omega_X$.
		\item The contraction $\contr_{\omega_B}$ is given by 
		$$\contr_{\omega_B} \beta = n \frac{\beta_\calH \wedge \omega_B^{n-1}}{\omega_B^n}$$
		where the quotient is taken in $\det \calH^*$.	When $\beta$ is pulled back from $B$, this is just the pullback of the regular contraction of $\beta$ with $\omega_B$ on $B$.
		\item The form $\rho$ denotes the \emph{relative Ricci curvature} of $\omega_X$ defined as follows. The form $\omega_X$ is K\"ahler in the vertical directions, and so induces a positive-definite Hermitian metric on $\calV \to X$. The induced Hermitian metric on the holomorphic line bundle $\det \calV\to X$ has curvature form $\rho$. In local coordinates one can write
		$$\rho = -\frac{i}{2\pi}\deldelbar \log \det (\omega_X)_\calV$$
		where the determinant is taken in $\calV^*$, so $\rho_\calV$ is equal to the Ricci curvature of $(X_b,\omega_b)$. Note that $\rho$ may have non-trivial horizontal component $\rho_\calH$, as $\deldelbar$ is being taken on the total space of $X$.
		\item The vertical Laplace operator $\Lap_\calV$ on functions $f$ on $X$ is defined by
		$$\Lap_\calV = \contr_\calV (i\deldelbar f)$$
		where $\contr_\calV$ is the vertical contraction with $\omega_X$, given by
		$$\contr_\calV \beta = m\frac{\beta_\calV \wedge \omega_X^{m-1}}{\omega_X^m}$$
		where the quotient is taken in $\det \calV^*$.
	\end{itemize}

	To write down the optimal symplectic connection, we need one more ingredient, a certain projection operator 
	$$p: C^\infty(X,\CC) \to C_E^{\infty}(X).$$ 
	Here $C_E^\infty(X)$ is the subspace of smooth functions $C^{\infty}(X,\CC)$ which restrict to mean-zero \emph{real} holomorphy potentials on each fibre $(X_b,\omega_b)$:
	$$C_E^\infty(X) := \left\{f \in C^\infty(X,\CC) \mid \rest{f}{X_b} \in \h_0(X_b,\omega_b)^\RR \text{ and } \int_{X_b} \rest{f}{X_b} \omega_b^m = 0 \text{ for all } b\in B\right\}.$$
	The vector space $E_b := \h_0(X_b,\omega_b)^{\RR}$ of real mean-zero holomorphy potentials on $(X_b,\omega_b)$ has dimension independent of $b$ by our last assumption in \cref{sec:kahlerfibrations} and the fact that the Lie algebra of reduced automorphisms is identified with holomorphy potentials as described in \cref{sec:holomorphypotentials}. Indeed these vector spaces $E_b$ form a real vector bundle over $B$ and a smooth section of $E\to B$ can be identified exactly with a smooth function in $C_E^{\infty}(X)$ (\cite[p. 13]{dervan2019optimal}). Furthermore it was explained by Hallam \cite{hallam2022thesis} that relatively cscK metrics in the same relatively K\"ahler class $[\omega_X]$ as $\omega_X$ can be identified with the smooth sections of $E$ (on each fibre this corresponds to the discussion in \cref{rmk:spaceofcscKmetrics}).
	
	The projection $p$ is the $L^2$ projection onto $C_E^{\infty}(X)$ with respect to the inner product
	$$(f,g) \mapsto \int_X fg \omega_B^n \wedge \omega_X^m$$
	defined on $(X,\omega_X) \to (B,\omega_B)$.
	
	\begin{remark}\label{rmk:ppullback}
		Note that if $f\in C^{\infty}(B)$ is a smooth function on the base, then $p(\pi^* f) = 0$. In particular if $\varphi$ is a smooth function on $B$ then $p(\contr_{\omega_B} \pi^* (i\deldelbar \varphi)) = 0$ by the observation that the horizontal contraction on $X$ for pulled back forms is simply the contraction on the base. 
	\end{remark}
	
	\begin{definition}\label{def:OSC}
		A relatively cscK metric $\omega_X$ on a compact K\"ahler fibration $\pi: X \to (B,\omega_B)$ is an \emph{optimal symplectic connection} if
		\begin{equation}\label{eq:OSC}p(\Lap_\calV \contr_{\omega_B} \mu^* F_\calH + \contr_{\omega_B} \rho_\calH) = 0.\end{equation}
	\end{definition}

	Notice by \cref{rmk:ppullback} and \cref{thm:minimalcoupling} that we could have equivalently written $(\omega_X)_\calH$ instead of $\mu^* F_\calH$, so the term inside the projection agrees with the terms at subleading order in \cref{eq:adiabaticscalarcurvature}. 

	\begin{remark}
		As has been observed in \cite[\S 3.5]{dervan2019optimal} an optimal symplectic connection on the projectivisation $\PP(E)$ of a holomorphic vector bundle arises precisely from a Hermite--Einstein metric on $E$. Moreover it follows simply from the relatively cscK assumption that, vacuously, an optimal symplectic connection on a fibration $X\to \{p\}$ over a point is simply a cscK metric on $X$. In this sense the notion of an optimal symplectic connection interpolates between cscK metrics on varieties and Hermite--Einstein metrics on vector bundles.
	\end{remark}
	
	Let us now recall the existence result of Dervan--Sektnan for cscK metrics in adiabatic K\"ahler classes.
	
	\begin{theorem}[\cite{dervan2019optimal}]\label{thm:adiabaticscKDervanSektnan}
		A compact polarised fibration $\pi: (X,H) \to (B,L)$ admits a cscK metric in the class $H + kL$ for all $k\gg 0$ sufficiently large whenever
		\begin{itemize}
			\item $X$ admits an optimal symplectic connection $\omega_X \in c_1(H)$,\footnote{Here the inclusion of the projection $p$ becomes clear after comparing with the expansion \eqref{eq:adiabaticscalarcurvature}. Indeed the space $C_E^{\infty}(X) \oplus \pi* C^{\infty}(B)$ is the kernel of the linearisation of the adiabatic scalar curvature, so the assumption that $\omega_X$ is OSC ensures obstructions to higher order corections to \eqref{eq:adiabaticscalarcurvature} vanish and the linearisation can be used to construct arbitrarily good approximate solutions.} and
			\item the base $(B,L)$ admits a twisted cscK metric $\omega_B\in c_1(L)$ such that
			$$S(\omega_B) - \contr_{\omega_B} \alpha = const$$
			where $\alpha = q^* \Omega_{WP}$ is the pullback of the Weil--Peterson metric from the moduli space of cscK manifolds.\footnote{This term appears since the function $\contr_{\omega_B} \rho_\calH$ has non-zero projection onto $\pi^* C^{\infty}(B)$ given precisely by the contraction of the Weil--Peterson metric on the base.}
		\end{itemize}
	\end{theorem}

	The form $\alpha$ can be defined via pullback using the existence of the moduli space of cscK manifolds with automorphisms by Dervan--Naumann \cite{dervan2018moduli}, although the form can be identified with a fibre integral
	$$\alpha = -\int_{X/B} \rho_\calH \wedge \omega_X^m$$
	without any reference to to the moduli space. The condition that $X$ admits an optimal symplectic connection is vacuous when the fibres of $(X,H)$ have discrete automorphisms, in which the above theorem follows from the work of Fine \cite{fine} when $X$ is a surface and $B$ is a curve.
	
	We also have the following uniqueness result due to Dervan--Sektnan and Hallam.
	
	\begin{theorem}[{\cite{dervan2019optimal,hallam2020geodesics}}]\label{thm:uniqueness}
		Suppose $\omega_X, \omega_X'$ are two cohomologous optimal symplectic connections on a K\"ahler fibration $\pi: X \to (B,\omega_B)$. Then there exists a holomorphic automorphism $g$ of the fibration $X\to B$ (a biholomorphic $g: X \to X$ such that $\pi \circ g = \pi$) and a function $\varphi$ on $B$ such that
		$$\omega_X = g^* \omega_X' + \pi^* (i\deldelbar \varphi).$$
		Since the forms are cohomologous, the automorphism $g$ can always be taken in the identity component $\Aut_0(\pi)$ of the automorphism group $\Aut(\pi)$ of the fibration.
	\end{theorem}

	One interpretation of this uniqueness result is that optimal symplectic connections give \emph{canonical} relatively K\"ahler metrics on compact K\"ahler fibrations.
	
	\begin{remark}
		The preceeding notion of an optimal symplectic connection has been generalised by Ortu \cite{ortu2022optimal} to smooth deformations $Y\to (B,\omega_B)$ of a relatively cscK fibration $(X,\omega_X) \to (B,\omega_B)$. In this case an extra term appears in \eqref{eq:OSC} related to the deformation of complex structure of the fibres, and a version of \cref{thm:adiabaticscKDervanSektnan} is proven on the deformed fibration $Y$. Algebraically such deformations should correspond to fibrations with only K-semistable fibres (rather than relatively cscK fibrations, which by \cref{conj:YTD} should correspond to fibrations with K-polystable fibres). Since K-semistability is an open condition, fibrations of this form should be more amenable to the construction of moduli.
	\end{remark}

	\begin{remark}
		In view of the Kempf--Ness picture of \cref{sec:GIT}, we note that Hallam \cite{hallam2020geodesics} has introduced a relative version of the Mabuchi functional of K-stability (see \cref{sec:kahlermetrics}) which acts as a Kempf--Ness functional for the optimal symplectic connection equation. However an interpretation of the equation in terms of a moment map is still an open problem.
	\end{remark}

	\section{Stability of fibrations\label{sec:stabilityoffibrations}}
	
	Let us now describe, as suggested by \cref{principle}, the algebro-geometric theory corresponding to the preceding notion of a canonical metric on a fibration. Just as the OSC equation generalises the Hermite--Einstein equation in the case of projective bundles, this theory will be closely related to the slope stability described in \cref{sec:stability} for vector bundles.
	
	We begin with a polarised fibration $\pi: (X,H) \to (B,L)$ where $L$ is ample and $H$ is relatively ample. We will work in the setting where $X$ and $B$ are smooth, and where the fibres $(X_b,H_b)$ of $X$ are K-polystable. This agrees (assuming the Yau--Tian--Donaldson conjecture \cref{conj:YTD}) with the assumption of the fibres being cscK which appeared in our discussion of optimal symplectic connections. In order to identify stability, we will first define a notion of test configuration for a fibration analogous to \cref{def:testconfiguration}.
	
	\begin{definition}[Fibration degeneration]
		A \emph{fibration degeneration} of $\pi: (X,H) \to (B,L)$ of exponent $k$, for $k\gg 0$, is a scheme $p: \calX \to B\times \CC$ over $\CC$, and a relatively ample line bundle $\calH$ on $\calX$ such that
		\begin{itemize}
			\item The morphism $p: \calX\to B\times \CC$ is flat,
			\item $(\calX_t,\calH_t) \isom (X,H^k)$ for $t\ne 0$,
			\item there is a $\CC^*$ action on $\calX$ lifting to $\calH$ which covers the standard action on $B\times \CC$ (which is trivial on $B$).
		\end{itemize}
	\end{definition}

	We say the fibration degeneration is a \emph{product} if there is a $\CC^*$-equivariant isomorphism $\calX \isom X \times \CC$ where the action on $X$ is given by a one-parameter subgroup of the group $\Aut(\pi)$ of relative automorphisms of the fibration $(X,H)\to (B,L)$. Furthermore if this one-parameter subgroup is trivial so that $\CC^*$ acts on $X\times \CC$ only on the second factor, we say the fibration degeneration is \emph{trivial}.

	If there is a $\CC^*$-equivariant isomorphicm $\calX\isom X\times \CC$ as fibrations over $B\times \CC$ we call the fibration degeneration a \emph{product} degeneration. If furthermore $\calH \isom H$ we call it a trivial 

	\begin{remark}
		Any such fibration degeneration is equivalent to a one-parameter subgroup of $\GL(N_{j,k}+1)$ acting on $\Hilb(\PP(\calU_{j,k}))$ where $\calU_{j,k}$ is the universal family of an appropriate quot scheme for which $\Hilb(\PP(\calU_{j,k}))$ parametrises fibrations over $B$. See \cite[Lem 3.2]{dervan2019moduli}.
	\end{remark}

	Note that for every $j\gg 0$, we have a genuine test configuration $(\calX,jL+\calH)$ for the polarised variety $(X,jL+H^k)$ of exponent one, in the sense of \cref{def:testconfiguration}. Using the standard definition of Donaldson--Futaki invariant \cref{def:donaldsonfutakiinvariant} we can expand
	\begin{equation}\label{eq:DFfibration}\DF(\calX,jL+\calH) = j^n W_0(\calX,\calH) + j^{n-1} W_1(\calX,\calH) + O(j^{n-2})\end{equation}
	in powers of $j$. 
	
	\begin{definition}[Stability of fibration]\label{def:stabilityoffibration}
		We say a fibration $\pi: (X,H)\to (B,L)$ is 
		\begin{itemize}
			\item \emph{semistable} if $W_0(\calX,\calH) \ge 0$ for all fibration degenerations $(\calX,\calH)$ and $W_1(\calX,\calH) \ge 0$ whenever $W_0(\calX,\calH)=0$.
			\item \emph{polystable} if it is semistable and whenever $W_0=W_1=0$, there exists an open subset $U\subset B$ of complement codimension at least 2 such that $\rest{\calX,\calH}{U}$ normalises to a product fibration degeneration over $U$,\footnote{The necessity of this condition was pointed out by Hallam \cite{hallam2022thesis} to resolve the existence of certain fibration degenerations identified by Hattori \cite{hattori2022fibration} which destabilise any fibration. This condition is analogous to considering only torsion-free coherent subsheaves in the theory of slope stability of vector bundles, or considering ``almost trivial" test configurations in K-stability.}
			\item \emph{stable} if it is semistable and whenever $W_0=W_1=0$, there exists an open subset $U\subset B$ of complement codimension at least 2 such that $\rest{\calX,\calH}{U}$ is the trivial degeneration.
		\end{itemize}
	\end{definition}

	\begin{remark}\label{rmk:fibrationgiesekerstability}
		Recently an alternative notion of stability called $\mathfrak{f}$-stability has been introduced by Hattori \cite{hattori2022fibration}, which is the ``asymptotic Chow stability" version of the above notion of stability of fibrations. This stability asks that whenever $W_i(\calX,\calH) = 0$ for $i=0,\dots,i_0$ then $W_{i_0+1}(\calX,\calH) \ge 0$. In particular it follows quickly that stability in the above sense implies $\mathfrak{f}$-stability, which in turn implies semistability of the fibration. 
	\end{remark}
	
	We will now describe the standard method of producing fibration degenerations of a given polarised fibration $\pi: (X,H)\to (B,L)$, which was indeed given as the definition of a fibration degeneration in \cite{dervan2019moduli}. 
	
	By the relative ampleness of $H$, for $k\gg 0$ the dimension of $H^0(X_b,H_b^k)$ is constant over $b\in B$. Indeed for such $k\gg 0$ by the flatness of $\pi$ we obtain vector bundles
	$$V_k := \pi_* H^k$$
	over $B$ with fibre $\rest{V_k}{b} = H^0(X_b,H_b^k)$. To produce degenerations of $X$, we will take degenerations of $V_k$ as a vector bundle. As discussed in \cref{rmk:turningoffextension}, this is one perspective that one may understand the stability of vector bundles in the language of K-stability.
	
	\begin{definition}
		Given a vector bundle $E\to B$, a \emph{vector bundle degeneration} is a torsion-free coherent sheaf $\calE \to B\times \CC$, flat over $\CC$, such that
		\begin{itemize}
			\item $\calE$ is an equivariant sheaf with respect to the standard action of $\CC^*$ on $B\times \CC$, 
			\item the general fibre $\calE_t$ is isomorphic to $E\to B$ for all $t\ne 0$.
		\end{itemize}
	\end{definition}

	Let $\calE$ be some vector bundle degeneration of $V_k$. Then we may take the relative $\Proj$ of the sheaf $\calE$ to obtain a projective variety 
	$$\PP(\calE) := \underline{\Proj}(\Sym \calE)$$
	with a morphism $p: \PP(\calE) \to B$ and a relatively ample line bundle $\calO(1)$. The fibre of $\PP(\calE)$ over $b \in B$ is the projective space of \emph{quotients} $\PP(\calE(b))$ of the vector space $\calE \otimes_{\calO_B} k(b)$. In \cref{sec:Fibrationsfuturedirections} we will consider a simpler case where $\calE$ is a locally-free vector bundle degeneration, in which case $\PP(\calE)$ is simply the regular projective bundle of quotients.
	
	By the relative Kodaira embedding for the relatively ample line bundle $H\to X$, one naturally obtains an embedding
	$$X\into \PP(V_k)$$
	for $k\gg 0$ large enough that $H^k$ is relatively very ample. Thus we obtain a subscheme
	$$X\times \CC^* \subset \PP(\calE)$$
	and define
	$$\calX = \overline{X\times \CC^*}$$
	to be the closure of $X\times \CC^*$ inside the projectivisation $\PP(\calE)$ of the vector bundle degeneration $\calE$. Being the projective closure of $X\times \CC^*$, $\calX$ has equidimensional fibres and therefore is flat over the one-dimensional base $\CC$. The $\CC^*$ action on $\calE$ produces a $\CC^*$ action on $\calX$, and we obtain a $\CC^*$-equivariant morphism
	$$p: \calX \to B$$
	and a relatively ample line bundle $\calH := \iota^* \calO(1)$ where $\iota: \calX \into \PP(\calE)$ is the inclusion. Then $(\calX,\calH)$ is a fibration degeneration of $(X,H)$ of exponent $k$.
	
	As discussed in \cref{rmk:turningoffextension} the simplest vector bundles degenerations are those obtained by starting with a subsheaf $\calF \subset V_k$ and ``turning off the extension class" $e\in \Ext^1(V_k/\calF, \calF)$ defining the extension
	\begin{center}
		\ses{\calF}{V_k}{V_k/\calF}.
	\end{center}
	This produces a vector bundle degeneration $\calE\to \CC$ for which the general fibre $\calE_t$ is given by $V_k \to B$ for $t\ne 0$, and the central fibre is $\calE_0 = \calF \oplus V_k/\calF \to B.$ In the case where, for example, $\calF$ is a holomorphic subbundle, then the induced test configuration on each fibre $(X_b,H_b)$ corresponds to a test configuration arising from deformation to the normal cone of a linear subspace, a special form of those appearing in slope K-stability (see \cref{rmk:slopeKstability}).

	Let discuss further the coefficients $W_0$ and $W_1$ appearing in the expansion \eqref{eq:DFfibration}. First we recall $W_0(\calX,\calH)$.
	
	\begin{proposition}[{\cite[Lem. 2.33]{dervan2019moduli}} or {\cite[Lem. 4.8]{hattori2022fibration}}]
		Let $(\calX,\calH)$ be a fibration degeneration of $(X,H)$. Then for general $b\in B$, we have
		$$W_0(\calX,\calH) = {m+n\choose n} L^n \cdot \DF(\calX_b,\calH_b).$$
	\end{proposition}
	
	By the assumption that the fibres of $(X,H)\to (B,L)$ are K-polystable, we automatically have $W_0(\calX,\calH) \ge 0$ for any fibration degeneration of $(X,H)$. Further, from the definition of stability of a fibration \cref{def:stabilityoffibration} the fibration degenerations of most interest are those for which the induced test configuration $(\calX_b,\calH_b)$ of a generic fibre $(X_b,H_b)$ normalises to a product test configuration (so that $\DF(\calX_b,\calH_b) = 0$). 
	
	\begin{remark}
		In fact, if the fibres of $(X,H)\to (B,L)$ have trivial automorphisms $\Aut(X_b,H_b)=0$ so that they are K-stable (and not just K-polystable), then $W_0(\calX,\calH)>0$ for any fibration degeneration which does not normalise to the trivial test configuration on a generic fibre. by the above proposition, and therefore the fibration is stable. This agrees with the observation that any such fibration with trivial automorphisms of the fibres admits an optimal symplectic connection (the condition being vacuous in that setting).
	\end{remark}
	
	The subleading order coefficient $W_1(\calX,\calH)$ admits an intersection-theoretic expansion identified by Dervan--Sektnan, by expanding the intersection formula for the Donaldson--Futaki invariant \eqref{eq:DFintersectiontheory} in powers of $j$ in this setting. We will only recall this formula in the simplified setting of a Fano fibration, so we have a polarisation $(X,-K_{X/B})\to (B,L)$. Here we compactify the test configuration $(\calX,jL+\calH)$ over $\PP^1$ trivially at infinity and obtain
	\begin{equation}\label{eq:W1Fanofibration}{n+m\choose n-1}^{-1} W_1(\calX,\calH) = \frac{m}{m+2} L^{n-1}.\calH^{m+2} + \frac{1}{m+1} \gamma L^n.\calH^{m+1} + L^{n-1}.\calH^{m+1}.K_{\calX/B\times\PP^1}\end{equation}
	where 
	$$\gamma = \frac{L^{n-1}.(-K_{X/B})^{m+1}}{L^n.(-K_{X/B})^m}.$$
	As discussed for example in \cite[\S 1.2]{codogni2018positivity}, $-\gamma$ is the degree of the CM line bundle over the base $B$ induced by the family $(X,-K_{X/B})\to (B,L)$ of Fano varieties. In particular by the positivity of the CM line bundle, $\gamma \le 0$ and if $X\to B$ is isotrivial then in fact $\gamma=0$. This is the case for example for projective bundles as we will see in \cref{sec:isotrivialstability}.
	
	Finally let us state the central conjecture relating stability of fibrations and optimal symplectic connections is the following.	

	\begin{conjecture}[\cite{dervan2019moduli}]\label{conj:fibrations}
		A polarised fibration $(X,H)\to (B,L)$ is polystable if and only if it admits an optimal symplectic connection.
	\end{conjecture}

	Progress towards this conjecture was made by Dervan--Sektnan, who showed that the existence of an optimal symplectic connection implies semistability of the fibration \cite{dervan2019moduli}. This was improved to polystability with respect to certain product-type fibration degenerations by Hallam using geodesic analysis \cite{hallam2020geodesics}.
	
	\section{Principal bundles}
	
	We will now recall the theory of stability of principal bundles over compact K\"ahler manifolds and the associated notion of a Hermite--Einstein connection. 
	This theory closely mirrors the theory of holomorphic vector bundles discussed in \cref{sec:bundles}, and in particular relies on the analogue of a Chern connection on a principal bundle.
	
	\subsection{Hermite--Einstein connections}
	
	Consider now a holomorphic principal $G$-bundle $P\to (B,\omega_B)$ over a compact K\"ahler manifold, where $G$ is a reductive complex Lie group. 
	\begin{definition}
		A \emph{Hermitian structure} on $P$ is a choice of reduction of structure group
		$$\sigma: B \to P/K$$
		of $P$ to a principal $K$-bundle $P_\sigma \to (B,\omega_B)$ where $K$ is a maximal compact subgroup of $G$, such that $K^\CC = G$.
	\end{definition}

	The key example of a Hermitian structure occurs when $G=\GL(r,\CC)$ and $K=\U(r)$. Then the quotient $\GL(r,\CC)/\U(r)$ is identified with the space of Hermitian inner products on $\CC^r$, and a reduction of structure group $\sigma: B\to P/K$ is a smooth choice of Hermitian inner product on every fibre of the standard associated holomorphic vector bundle $E := P \times_G \CC^r$. That is, $\sigma$ is just the data of a Hermitian metric on $E$. 
	
	In general we note that when $G$ is reductive, the quotient space $G/K$ is contractible and there are many sections $\sigma$ of the quotient bundle $P/K$. Thus there always exists Hermitian structures on holomorphic principal bundles with reductive structure group.

	\begin{definition}
		A principal bundle connection $A\in \Omega^1(P, \g)$ on $P$ is said to be:
		\begin{itemize}
			\item \emph{Complex} if $A\in \Omega^{1,0}(P,\g)$ is of type $(1,0)$.
			\item \emph{Unitary with respect to a Hermitian structure $\sigma$} if there exists a principal bundle connection $A_\sigma$ on the reduction of structure group $P_\sigma$ such that under the induced associated bundle construction
			$$P = P_\sigma \times_K G,$$
			$A_\sigma$ pushes forward to $A$.
		\end{itemize}
	\end{definition}

	The analogue of the existence and uniqueness of Chern connections on vector bundles is the following.
	
	\begin{proposition}[{\cite[Thm. IX.10.1]{kobayashi1963foundations}}]
		Given a Hermitian structure $\sigma: B\to P/K$ on a holomorphic principal $G$-bundle $P$ with reductive structure group, there is a unique complex connection $A$ on $P$ unitary with respect to $\sigma$. This is called the \emph{Chern connection} of $\sigma$ on $P$. 
	\end{proposition}

	In the case where $G=\GL(r,\CC)$ and $K=\U(r)$, the Chern connection on $P$ induces exactly the standard Chern connection on $E=P\times_G \CC^r$ with respect to the Hermitian metric induced by $\sigma$.
	
	\begin{definition}
		A complex unitary metric $A$ on a Hermitian holomorphic principal $G$-bundle $P\to B$ with Hermitian structure $\sigma$ is \emph{Hermite--Einstein} if
		$$\contr_{\omega_B} F_A = \tau$$
		where $F_A\in \Omega^{1,1}(B,\ad P)$ is the curvature of $A$ and $\tau$ is a covariantly constant, central section $\tau \in \Gamma(\ad P)$. 
	\end{definition}

	In order to clarify this definition, we note that the Lie algebra bundle $\ad P\to B$ has fibre $\g$ the Lie algebra of $G$. Whilst $\ad P$ is in general a non-trivial vector bundle on $B$ with connection induced by $A$, it contains a trivial central subbundle $Z\subset \ad P$, and a section $\tau \in \Gamma(Z)$ is covariantly constant with respect to the induced connection if and only if it is constant with respect to the standard trivialisation (the pushdown of the trivialisation of $P\times Z{z}(\g)$ under $\pi$). Such a section pulls back to a constant function $P\to Z(\g)$ into the centre of the Lie algebra on the total space of $P$.
	
	In the case where $G=\GL(r,\CC)$ and $K=\U(r)$, the centre $Z(\g)$ of the Lie algebra $\End(\CC^r)$ of $G$ is simply $\CC\cdot \id_{\CC^r}$ and $\ad P = \End E$ where $E$ is the standard associated bundle $E=P\times_G \CC^r$. In this setting $Z\subset \ad P$ is the trivial bundle generated by $\CC\cdot \id_E$ and $\tau$ is constant if $\tau = \lambda \id_E$ for some $\lambda \in \CC$ determined by the topology of $P$. Thus a Hermite--Einstein connection $A$ on $P$ induces a Hermite--Einstein connection on $E$ in the sense of \cref{sec:bundles} discussed previously.
	
	\begin{remark}\label{rmk:PrincipalHEdirectsum}
		Suppose if $G=\GL(r,\CC) \times \GL(r',\CC)$. Then a central element of the Lie algebra $\g$ is of the form
		$$\begin{pmatrix}
			\lambda \id_{\CC^r} & 0\\
			0 & \mu \id_{\CC^{r'}}
		\end{pmatrix}$$
		and if one considers the product standard representation then we have $$P\times_G (\CC^r \oplus \CC^{r'}) = E\oplus F$$
		where $E$ and $F$ are the standard associated bundles for the two factors of $G$. Then a Hermite--Einstein metric on $P$ is equivalent to an \emph{extremal} Yang--Mills metric on $E\oplus F$ in the sense of \cref{rmk:highercriticalym}. 
	\end{remark}
	
	\subsection{Stability\label{sec:stabilityprincipalbundle}}
	
	By \cref{principle} associated to the above extremal notion on a principal bundle we expect a stability theory closely analogous to the slope stability of vector bundles. Such a theory was first developed by Ramanathan on compact Riemann surfaces \cite{ramanathan1975stable,ramanathan1996moduli,ramanathan1996moduli2} and a general theory has been developed for projective manifolds and compact K\"ahler manifolds (see for example \cite{ramanathan1988einstein,anchouche2001einstein}). 
	
	Let $(B,\omega)$ be a compact K\"ahler manifold and $P\to B$ a holomorphic principal $G$-bundle. Let $Q\subset G$ be a parabolic subgroup. We wish to consider reductions of structure group $\sigma: U\to P/Q$ where $U\subset B$ is an open subset such that $\codim(X\backslash U) \ge 2$.
	
	\begin{definition}
		The principal bundle $P$ is \emph{(semi)stable} with respect to the reduction of structure group $\sigma$ to a \emph{maximal} parabolic subgroup $Q$ if
		$$\deg \sigma^* \calV(P/Q) > 0\quad \text{(resp. }\ge)$$
		where $\calV(P/Q)$ is the vertical tangent bundle of $P/Q$. The bundle $P$ is \emph{(semi)stable} if it is (semi)stable with respect to reductions to all maximal parabolic subgroups defined over Zariski open subsets with complement codimension greater than one.
	\end{definition}

	We will omit the precise definition of polystability of a principal bundle, where one must precisely encode the idea that $P$ splits as a direct sum of stable principal bundles of the same slope. See \cite[Def. 3.5]{anchouche2001einstein}. 

	\begin{remark}
		Note that $\sigma^* \calV(P/Q) \to U$ is a vector bundle over $U\subset X$ so one must be careful in defining the degree. It is not necessarily the case that $\sigma^* \calV(P/Q)$ (or indeed as is necessary here, even just the determinant of this bundle) extends to all of $X$. However one can show that when $\calV(P/Q)$ has arisen from the reduction of structure group to a parabolic subgroup $Q\subset G$, the determinant line bundle always extends uniquely and the degree is well-defined \cite[\S 1]{ramanathan1988einstein}.
	\end{remark}

	\begin{remark}
		The restriction only to \emph{maximal} parabolic subgroups is analogous to the restriction of testing stability of vector bundles only on single subsheaves as opposed to filtrations thereof. Indeed a parabolic subgroup $Q\subset G$ is specified exactly as the stabiliser of a flag 
		$$0\subset V_1 \subset \cdots \subset V_\ell = V$$
		inside a vector space $V$ for which $\rho: G \to \GL(V)$ is some faithful representation. A maximal parabolic corresponds to a one-step filtration
		$$0 \subset V_1 \subset V.$$
		It suffices to consider reductions to such parabolic subgroups to obtain a well-behaved theory for stability of principal bundles.
	\end{remark}

	\begin{remark}
		In the case where $G=\GL(r,\CC)$, the stability of $P\to (B,\omega_B)$ is precisely equivalent to the slope stability of $E=P\times_G \CC^r$ the standard associated bundle. In general the stability of $P$ implies the \emph{polystability} of the adjoint vector bundle $\ad P$ with fibre $\g$. In fact $P$ is semistable if and only if $\ad P$ is semistable, and a consequence of the Hitchin--Kobayashi correspondence in this setting shows that $P$ is also polystable if and only if $\ad P$ is polystable, so one can in this sense subsume the theory of stability of principal bundles into the corresponding theory of vector bundles. 
	\end{remark}

	The analogue of the Donaldson--Uhlenbeck--Yau theorem in this setting is the following (which is proved by appealing to the regular theorem \cref{thm:DUY} for vector bundles after identifying the correct associated vector bundle of $P$).
	
	\begin{theorem}[\cite{ramanathan1988einstein,anchouche2001einstein}]\label{thm:DUYPrincipalbundle}
		A holomorphic principal $G$-bundle on a compact K\"ahler manifold $(B,\omega_B)$ admits a Hermite--Einstein connection if and only if it is $[\omega_B]$-polystable.
	\end{theorem}

	\section{Product fibrations}

	To conclude our background, we briefly discuss the above constructions in the case of product fibrations. The following holds for any product of compact K\"ahler fibrations.
	
	\begin{proposition}\label{prop:products}
		If $\omega_X$ and $\omega_{X'}$ are optimal symplectic connections on two fibrations $(X,H),(X',H')\to (B,L)$, then the product metric on the fibred product $Z=X\times_B X'$ is an optimal symplectic connection.
	\end{proposition}
	\begin{proof}
		This is simply a matter of verifying that the various terms appearing in the optimal symplectic connection equation split with respect to fibred products in the expected way. Let us first note that the curvature $F_\calH$ of the product metric $\omega=\omega_X+\omega_{X'}$ is the direct sum $F_\calH = F_X + F_{X'}$ where $F_X, F_{X'}$ denote the curvatures of $\omega_X, \omega_{X'}$ on $X$ and ${X'}$ respectively. One also observes that
		$$d\mu^* F_\calH = \iota_{F_\calH} \omega = \iota_{F_\calH} (\omega_X+ \omega_{X'}) = \iota_{F_X} \omega_X + \iota_{F_{X'}} \omega_{X'}$$
		from which we conclude that $\mu^* F_\calH = \mu_X^* F_X + \mu_{X'}^* F_{X'}$. Similarly it is straight forward to see $\Lap_\calV = \Lap_X + \Lap_{X'}$ where $\Lap_X$ and $\Lap_{X'}$ denote the vertical Laplacians on $X$ and $X'$ respectively. .
		
		If the dimension of the fibres of $X$ and ${X'}$ are $m_X, m_{X'}$ respectively, for the projection operator $p$ applied to a sum function $\varphi=\varphi_X + \varphi_{X'}$ where $\varphi_X, \varphi_{X'}$ depend only on the fibre coordinates of $X$ and of ${X'}$ respectively, we have
		\begin{align*}
			p(\rest{\varphi}{b}) &= \rest{\varphi}{b} - \frac{\int_{X_b\times {X'}_b} \rest{\varphi}{b} (\omega_X + \omega_{X'})^{m_X+m_{X'}}}{\int_{X_b\times {X'}_b} (\omega_X + \omega_{X'})^{m_X+m_{X'}}}\\
			&= \rest{\varphi}{b} - \frac{\int_{X_b\times {X'}_b} \rest{\varphi}{b} \omega_X^{m_X} \wedge \omega_{X'}^{m_{X'}}}{\int_{X_b\times {X'}_b} \omega_X^{m_X} \wedge \omega_{X'}^{m_{X'}}}\\
			&= \left( \rest{\varphi_X}{b} - \frac{\vol {X'}_b}{\vol {X'}_b}\frac{\int_{X_b} \rest{\varphi_X}{b} \omega_X^{m_X}}{\int_{X_b} \omega_X^{m_X}}\right) + \left( \rest{\varphi_{X'}}{b} - \frac{\vol X_b}{\vol X_b}\frac{\int_{{X'}_b} \rest{\varphi_{X'}}{b} \omega_{X'}^{m_{X'}}}{\int_{{X'}_b} \omega_{X'}^{m_{X'}}}\right)\\
			&=p_X(\rest{\varphi_X}{b}) + p_{X'}(\rest{\varphi_{X'}}{b}),
		\end{align*}
		where in the second to last step we have used Fubini's theorem.
		
		Finally, let us note that if $\rho_X$ and $\rho_{X'}$ denote the relative Ricci forms of $\omega_X$ and $\omega_{X'}$, then we have
		\begin{align*}
			\rho &= i \deldelbar \log \det (\omega_X + \omega_{X'})\\
			&= i \deldelbar \log \det \omega_X \det \omega_{X'}\\
			&= i \deldelbar \log \det \omega_X + i \deldelbar \log \det \omega_{X'}\\
			&= \rho_X + \rho_{X'}.
		\end{align*}
		Here we have used that the determinant of the product metric $\omega_X + \omega_{X'}$ is the product of the determinants, as can be seen by using the block matrix decomposition of the metric in local product coordinates, and note that any derivative in $\deldelbar$ in the $X$ fibre direction will vanish on $\log \det \omega_{X'}$, which depends only on the ${X'}$ fibre coordinates, and vice versa.
		
		In conclusion, since $\Lap_X \mu_{X'}^* F_{X'} = 0$ and $\Lap_{X'} \mu_X^* F_X=0$, we observe that
		\begin{multline*}
			p(\Lap_\calV \contr_{\omega_B} \mu^* F_\calH + \contr_{\omega_B} \rho_\calH)\\
			= p(\Lap_{\calV,X} \contr_{\omega_B} \mu_X^* F_X + \contr_{\omega_B} \rho_{H,X}) + p(\Lap_{\calV,{X'}} \contr_{\omega_B} \mu_{X'}^* F_{X'} + \contr_{\omega_B} \rho_{\calH,{X'}})
		\end{multline*}
		and thus we have the result.
	\end{proof}
	
	The above proposition combined with \cref{thm:maintheoremfibrations} produces a wealth of examples of optimal symplectic connections. Namely all fibred products of projectivisations of polystable vector bundles admit optimal symplectic connections, which in fact already follows from the above proposition and previous work about the existence of optimal symplectic connections on projective bundles \cite[\S 3.5]{dervan2019optimal}.
	
	This is somewhat curious from the perspective of vector bundles, as the direct sum of polystable vector bundles is not necessarily polystable (and therefore does not necessarily admit a Hermite--Einstein metric) unless the vector bundles have the same slope. Indeed the process of taking a direct sum and then projectivisation does not commute with taking projectivisation and then fibred product, in regards to the existence of optimal symplectic connections. 
	
	\begin{figure}[h]
		\centering
		\begin{tikzcd}
			E,F \arrow{r} \arrow{d} & E\oplus F \arrow{d}\\
			\PP(E)\times_B \PP(F) \arrow[leftrightarrow, dashed, "/"{anchor=center,sloped}]{r}& \PP(E\oplus F)
		\end{tikzcd}
		\caption{The existence of optimal symplectic connections on fibred products of projective bundles is not equivalent to the projectivisations of direct sums of bundles.}\label{fig:projectivebundlecommutatitive}
	\end{figure}

	On the other hand, a product of cscK manifolds is always cscK without regard for any topological matching criteria, so in this sense the study of optimal symplectic connections is closer to the study of cscK metrics. 
	
	Let us clarify this lack of commutativity from the perspective of principal bundles. If $\fr(E)$ denotes the frame bundle of a vector bundle $E$, which is a principal $\GL(\rk E,\CC)$-bundle, then the fact that $\PP(E)\times_B \PP(F)$ admits an optimal symplectic connection when $E$ and $F$ are polystable corresponds to the fact that $\fr(E)\times_B \fr(F)$ is a polystable principal bundle by our main result \cref{thm:maintheoremfibrations}. By the Hitchin--Kobayashi correspondence for principal bundles, this is known to be equivalent to the slope polystability of the slope zero vector bundle $$\ad (\fr(E)\times_B \fr(F)) = \ad \fr(E) \oplus \ad \fr(F) = \End(E) \oplus \End(F).$$ On the other hand if $E$ and $F$ have different slopes, so $\mu(E)<\mu(F)$ without loss of generality, then the principal bundle $\fr(E\oplus F)$ has adjoint bundle $\End(E\oplus F) = \End(E) \oplus \Hom(E,F)\oplus \End(F)$ which is not polystable, and so $\fr(E\oplus F)$ is not polystable either, which agrees with the fact that $\PP(E\oplus F)$ does not admit an optimal symplectic connection. 
	
	It is known (see \cite{ramanan1984some}) that if $P$ is a polystable principal $G$-bundle and $\rho: G\to H$ is a representation which sends the connected component of the identity of the center, $Z_0(G)$, to the corresponding component $Z_0(H)$, then the associated principal $H$-bundle $P\times_\rho H$ is also polystable. Now the direct sum structure of $E\oplus F$ provides a reduction of structure group from $\fr(E\oplus F)$ to $\fr(E)\times_B \fr(F)$, however the associated homomorphism $\GL(\rk E,\CC) \times \GL(\rk F,\CC) \to \GL(\rk E + \rk F,\CC)$ does not map the centre into centre, as a central element
	$$\begin{pmatrix}
		\lambda \id_E & 0\\
		0 & \mu \id_F
	\end{pmatrix}$$
	does not commute inside the larger group $\GL(\rk E + \rk F,\CC)$ unless $\lambda = \mu$. This explains the lack of commutativity in \cref{fig:projectivebundlecommutatitive}.
	
	\begin{remark}
		By the above discussion if $\PP(E) \times_B \PP(F)$ admits an OSC then $E\oplus F$ admits an extremal Yang--Mills metric. One therefore expects the projectivisation $\PP(E\oplus F)$ to admit an \emph{extremal symplectic connection} in the sense of \cite[Def. 3.15]{dervan2019optimal}. In particular one should expect the diagram \cref{fig:projectivebundlecommutatitive} to commute if we instead only consider extremal symplectic connections instead of optimal symplectic connections.
	\end{remark}
	
	Finally note that by the uniqueness of optimal symplectic connections \cref{thm:uniqueness}, if a holomorphic fibration $(Z,H_Z)$ admits a fibred product decomposition into factors $(X,H_X)$ and $(X',H_X')$ admitting optimal symplectic connections, then any optimal symplectic connection on $Z$ in $c_1(H_Z) = c_1(H_X\boxtimes H_X')$ must be of fibred product form after the action of some relative automorphism and pullback of a two-form from the base. Thus we might introduce the following notion.
	
	\begin{definition}
		A \emph{reducible polarised holomorphic fibration} $(Z,H_Z)$ is a fibration admitting a fibred product decomposition of the above form. A polarised holomorphic fibration is irreducible if it is not reducible.
	\end{definition}

	For the notion of reducibility to be useful, it should be the case that if $\omega_Z$ is an optimal symplectic connection on $Z$, then it can be transformed into product form where $\omega_X$ and $\omega_{X'}$ are optimal symplectic on $X$ and $X'$. The corresponding algebro-geometric prediction is that the (semi/poly)stability of $(Z,H_Z)$ implies the (semi/poly)stability of $(X,H_X)$ and $(X',H_{X'})$.\footnote{In the case of K-stability this follows from a straightfoward calculation of the Donaldson--Futaki invariant of a fibred product of test configurations (taken over $\CC$). A similar calculation should produce this algebro-geometric fact for fibration degenerations, provided more care is taken in the decompositions of spaces of sections $H^0(\calZ_0, jL + \calH_Z)$ with respect to a fibred product decomposition $\calZ_0 = \calX_0 \times_B \calX_0'$. This would justify the notion of reducibility in view of \cref{conj:fibrations}. Note the corresponding metric property: a cscK metric on a product is necessarily a product of cscK metrics; is not at all obvious.}
	
\chapter{Isotrivial fibrations}
\label{ch:isotrivialfibrations}

In this chapter we will study the notion of an optimal symplectic connection in the case of isotrivial fibrations. Our main result relates the existence of such connections to the existence of Hermite--Einstein connections on holomorphic principal bundles. Using the \cref{thm:DUYPrincipalbundle} this gives a characterisation of the existence of optimal symplectic connections in terms of a purely algebro-geometric stability condition in the isotrivial case, proving an instance of \cref{principle}.

The main section \cref{sec:existenceOSCisotrivial} is the content of the paper \cite{mccarthy2022canonical}.

We will also discuss several future direction of interesting, including a discussion of stability for isotrivial fibrations, and a formalism in terms of principal bundles in non-isotrivial setting.

\section{Existence of optimal symplectic connections\label{sec:existenceOSCisotrivial}}

Recall that a K\"ahler fibration $(X,\omega_X) \to (B,\omega_B)$ is called \emph{isotrivial} if it is a holomorphic fibre bundle. By the theorem of Fischer--Grauert this is equivalent to asking that the fibres of $X\to B$ are all biholomorphic \cite{fischergrauert}. In this section we prove the following.

\begin{theorem}[\cref{thm:introductionOSCHE}]\label{thm:maintheoremfibrations}
	Let $P\to (B,\omega_B)$ be a holomorphic principal $G$-bundle with maximal compact $K\subset G$ and suppose $G$ acts by biholomorphisms on a cscK manifold $(Y,\omega_Y)$ such that $K$ acts by holomorphic isometries. Then the symplectic connection $\omega_X$ on the associated bundle $$X=P\times_G Y = P_\sigma \times_K Y\to (B,\omega_B)$$ induced by a complex unitary connection $A$ on $P$ is an optimal symplectic connection whenever $A$ is a Hermite--Einstein connection.
	
	Moreover, if $P$ arises as the bundle of relative automorphisms of a given compact, isotrivial K\"ahler fibration $(X,\omega_X)$ with cscK fibres, then the induced principal bundle connection $A$ on $P$ is Hermite--Einstein if and only if $\omega_X$ is optimal symplectic.
\end{theorem}

We will be interested in smooth isotrivial relatively cscK fibrations $(X,\omega_X)\to (B,\omega_B)$. Let $H$ be relatively ample on $X$ and suppose $\omega_X\in c_1(H)$. Then for the model fibre $(Y,\omega_Y,H_Y)$ of $X$, where $\omega_Y$ is cscK, the automorphism group of $(Y,H_Y)$ is reductive \cite[\S 2.4, \S 8.1]{gauduchon2010calabi}. 

To summarise, let $G_0=\Aut_0(Y,H_Y)$ denote the the connected component of the identity of the group of holomorphic automorphisms of $Y$ which lift to $H_Y$, and let $K_0=\Isom_0(Y,\omega_Y,H_Y)$ denote the subset of $\Aut_0(Y,H_Y)$ of holomorphic isometries of $\omega_Y$. Then $K_0\subset G_0$ is a maximal compact subgroup. The Lie algebra $\h = \Lie(G_0)$ of complex holomorphy potentials of $Y$ can be identified with the holomorphic vector fields on $Y$ which vanish at least once. The Lie algebra $\k= \Lie(K_0) \subset \h$ is identified with the \emph{real} holomorphy potentials, and $\h = \k \oplus J\k$ where $J$ is the complex structure on $Y$. Integrating up to the Lie group, $G_0=K_0^\CC$, so $G_0$ is reductive.

In the case of fibrations the above description of the automorphism group has the following consequence.

\begin{lemma}\label{lemma:principalbundle}
	A smooth polarised isotrivial relatively cscK fibration $$\pi: (X,\omega_X,H)\to (B,\omega_B,L)$$ with model fibre $(Y,\omega_Y,H_Y)$ arises as the associated bundle to a reductive holomorphic principal $G_0=\Aut_0(Y,H_Y)$-bundle $P$ which admits a reduction of structure group $\sigma: B\to P/K_0$ to a principal $K_0=\Isom_0(Y,\omega_Y,H_Y)$-bundle $P_\sigma$ such that
	$$X = P\times_{G_0} Y = P_\sigma \times_{K_0} Y.$$
\end{lemma}
\begin{proof}
	Since $X\to B$ is a holomorphic fibre bundle, it admits a holomorphic system of local trivialisations, say $\{(U_\alpha, \varphi_\alpha)\}.$ Fix any $b_0\in B$ and any $\beta$ with $b_0\in U_\beta$. Define a model cscK metric $\omega_Y := \rest{{\varphi_\beta}_*}{b} \omega_X$ and polarisation $H_Y := \rest{{\varphi_\beta}_*}{b} H$. Let $G_0 = \Aut_0(Y, H_Y)$ and $K_0 = \Isom_0(Y, \omega_Y, H_Y)$. 
	
	By the uniqueness of cscK metrics up to automorphisms, for any local trivialisation $\psi$ on $V$ for $X$ and any $b\in V$, there exists a holomorphic automorphism $g_b$ of $Y$ taking $\rest{\psi_* \omega_X}{b}$ to $\omega_Y$. Since $\psi_* \omega_X$ and $\omega_Y$ are cohomologous, this may be taken to lie inside the reduced automorphism group of biholomorphisms which lift to the line bundle $H_Y$ for which $\omega_Y \in c_1(H_Y)$. Performed for the covering by the $U_\alpha$, this defines a system of local sections of a principal $G_0$-bundle over $B$. The cocycle condition for this system of sections follows from that of the trivilising functions $\varphi_\alpha$ for this cover. 
	
	Furthermore, if $b\in U_\alpha$ then under the identification of $X_b$ with $Y$ with respect to the biholomorphism $\varphi_\alpha$, the automorphism group $K_0$ for $Y$ can be identified with a conjugate of $\Isom_0(X_b, \rest{\omega_X}{b}, H_b)$ in $\Aut_0(X_b, H_b)$. The cocycle condition for this system of local trivialisations guarantees that on overlaps of the $U_\alpha$ the isometry groups of $b\in U_{\alpha}\cap U_{\beta}$ are mapped to the same conjugate. This defines a smooth section $\sigma: B \to P/K_0$ specifying the desired reduction of structure group to $K_0$.
\end{proof}

\subsection{Induced optimal symplectic connections\label{sec:inducedconnection}}

We suppose now that we have an isotrivial fibration arising as the associated holomorphic fibre bundle to some holomorphic principal bundle with reductive structure group, which reduces to a principal bundle for a maximal compact subgroup of the structure group. As we observed, every isotrivial relatively cscK fibration arises in this way, although in the following we make no assumption that the associated principal bundle has structure group the identity component of the automorphism group of the model fibre.

Explicitly, fix a smooth polarised variety $(Y,H_Y)$ with a constant scalar curvature K\"ahler metric $\omega_Y\in c_1(H_Y)$. Assume that a reductive group $G$ acts linearly on $(Y,H_Y)$ and furthermore that there exists a maximal compact subgroup $K\subset G$ for which the restriction of the $G$ action to $K$ preserves the K\"ahler metric $\omega_Y$, that is, assume that $K$ acts by holomorphic isometries on $Y$. Such an action of $K$ on $(Y,\omega_Y)$ is always Hamiltonian, where the Hamiltonian function of any induced vector field is given by the real mean-zero holomorphy potential with respect to $\omega_Y$ (see \cref{sec:holomorphypotentials}). Let us denote by $\mu: Y \to \k^*$ a corresponding moment map. Finally let us assume that $P$ admits a reduction of structure group $\sigma$ to $K$, and let $P_\sigma$ denote the principal $K$-bundle which induces $P$. Associated to this data is a holomorphic fibre bundle
$$\pi: X = P\times_G Y = P_\sigma \times_K Y \to B,$$
associated to $P$. Given such an associated bundle, there is an induced symplectic connection on $X$ given by the cscK metric on $Y$ and a complex unitary connection on $P$. This follows essentially from working with the smooth principal $K$-bundle $P_\sigma$ and applying a theorem of Weinstein (see for example \cite[Thm. 6.3.3]{mcduff2017introduction}). We reproduce the details here to emphasise the relationship between $P_\sigma$ and $P$ in our setting and show that the resulting form $\omega_X$ is a $(1,1)$-form on $X$.

\begin{proposition}\label{prop:inducedsymplecticconnection}
	Given the set up above, any choice of complex unitary connection $A$ on $P$ induces a relatively K\"ahler metric $\omega_X\in c_1(H)$ on $X$.
\end{proposition}
\begin{proof}
	Let $v_\xi$ denote the induced vector field on $Y$ from some $\xi\in \k$ under the action of $K$ on $Y$. Then $v_\xi$ preserves the K\"ahler structure of $Y$. 
	
	Let $y\in Y, \hat y \in T_y Y, \eta \in \k$. Then we have the two identities
	$$\ip{d\mu(y) \hat y, \xi} = \omega_Y(v_\xi (y), \hat y),$$
	$$\ip{\mu(y), [\xi,\eta]} = \omega_Y(v_\xi(y),v_\eta(y)).$$
	The first follows from the definition of a moment map, and the second from the infinitesimal equivariance condition on the moment map.
	
	Let us abuse notation by writing $A=A_\sigma$ to denote the connection on the reduction of structure group $P_\sigma$ of $P$, and let $F\in \Omega^2(P_\sigma;\k)$ denote the curvature form of $A_\sigma$. Then for $p\in P_\sigma, \xi \in \k, v_1,v_2\in T_p P_\sigma$ we also have the standard expressions
	$$A_p(p\xi) = \xi,$$
	$$F_p(v_1,v_2) = (dA)_p(v_1,v_2) + [A_p(v_1),A_p(v_2)].$$
	Define a projection operator $\pi_A: TP_\sigma\times TY \to TY$ by $\pi_A(v,\hat y) = \hat y + v_{A_p(v)}(y).$ Let us define a two-form $\hat \omega_X\in \Omega^2(P_\sigma \times Y)$ by
	$$\hat \omega_X := \omega_Y - d\ip{\mu,A}.$$
	Then note that $\hat \omega_X$ is closed on $P_\sigma \times Y$.
	
	One may write 
	\begin{equation}\label{eqn:formequality}
		\hat \omega_X = \pi_A^* \omega_Y - \ip{\mu, F}.
	\end{equation}
	Indeed, using the identities above and the definition of $\pi_A$, we compute
	\begin{align*}
		&(\pi_A^* \omega_Y - \ip{\mu, F})_{(p,y)} ((v_1,\hat y_1), (v_2,\hat y_2))\\
		&= \omega_Y(\hat y_1 + v_{A_p(v_1)}(y), \hat y_2 + v_{A_p(v_2)} (y)) - \ip{\mu(y), F_p(v_1,v_2)}\\
		&= \omega_Y (\hat y_1, \hat y_2) + \omega_Y (v_{A_p(v_1)}(y), \hat y_2) - \omega_Y (v_{A_p(v_2)}(y), \hat y_1)\\
		&\qquad + \omega_Y (v_{A_p(v_1)}(y), v_{A_p(v_2)}(y))\\
		&\qquad - \ip{\mu(y), (dA)_p(v_1,v_2)} - \ip{\mu(y), [A_p(v_1),A_p(v_2)]}\\
		&=\omega_Y (\hat y_1, \hat y_2) + \ip{d\mu(y)\hat y_2, A_p(v_1)} - \ip{d\mu(y) \hat y_1, A_p(v_2)}\\
		&\qquad + \ip{\mu(y), [A_p(v_1),A_p(v_2)]}\\
		&\qquad - \ip{\mu(y), (dA)_p(v_1,v_2)} - \ip{\mu(y), [A_p(v_1),A_p(v_2)]}\\
		&= \omega_Y(\hat y_1, \hat y_2) - d\ip{\mu, A}_p ((v_1, \hat y_1), (v_2,\hat y_2)).
	\end{align*}
	
	From \eqref{eqn:formequality} it follows that the two-form $\hat \omega_X$ is $K$-invariant and horizontal for the quotient map $P_\sigma \times Y \to P_\sigma \times_K Y = X$. In particular a vector $(v,\hat y)$ is vertical with respect to this projection precisely if $v$ is a vertical tangent vector to $P$ and $\pi_A(v,\hat y) = 0$. From this it follows immediately that $\iota_{(v,\hat y)} \hat \omega_X = 0$. Since $\hat \omega_X$ is also closed this holds infinitesimally, and $\hat \omega_X$ is basic. 
	
	Additionally, the moment map condition for $\mu$ and $\omega_Y$ implies the $K$-equivariant closedness of $\omega - \mu$, and combined with the Bianchi identity this implies the equivariant closedness of $\tilde \omega_X$.
	
	Therefore the $K$-invariant and equivariantly closed two-form $\tilde \omega_X$ descends to a closed two-form $\omega_X$ on the quotient $X = P_\sigma \times_K Y$. The explicit expression 
	$$\omega_X((v_1,\hat y_1),(v_2, \hat y_2)) = \omega_Y(\hat y_1 + v_{A_p(v_1)}(y), \hat y_2 + v_{A_p(v_2)} (y)) - \ip{\mu(y), F_p(v_1,v_2)}$$
	shows that $\rest{\omega_X}{b} = \omega_Y$ since a vertical vector of $X$ takes the form $(0, \hat y)$, and  $(\omega_X)_{\calH} = \mu^* F_\calH$, since a horizontal vector is given by $(v,0)$ where $v$ is horizontal on $P$. Since we assumed that the initial connection $A$ was complex, the curvature $F$ has type $(1,1)$, and so $\omega_X$ is a closed $(1,1)$-form on $X$.
\end{proof}

In order to investigate when the symplectic connection $\omega_X$ induced on $X$ is optimal, we will use the following fact about compact K\"ahler manifolds which will simplify the optimal symplectic connection equation in the isotrivial setting.

\begin{lemma}[See for example {\cite[Lem. 28]{szekelyhidi2012blowing}}]\label{lemma:riccihamiltonian}
	If $h$ is a Hamiltonian function for a K\"ahler metric $\omega$ on a compact K\"ahler manifold, with Hamiltonian vector field $v$, then $\Lap h$ is formally the Hamiltonian function for the (not necessarily symplectic) two-form $\Ric \omega = \rho$ with the same vector field.
\end{lemma}
\begin{proof}
	\begin{align*}
		2\iota_v \rho &= \iota_v (dJd \log \det \omega)\\
		&= \calL_v (Jd \log \det \omega) - d\iota_v (Jd \log \det \omega)\\
		&= -d (\calL_{Jv} \log \det \omega)\\
		&= d\Lambda \calL_{Jv} \omega.
	\end{align*}
	Here we have used that $v$ preserves $J$ and $\omega$, where $\Lambda$ is the trace with respect to $\omega$. But $\calL_{Jv} \omega = -2i\deldelbar h$ so $\iota_v \rho = d\Lap h$. 
\end{proof}

The key argument which demonstrates how the optimal symplectic connection equation simplifies for isotrivial fibrations is the following.

\begin{proposition}\label{lemma:horizontalcomponentricciform}
	Given a complex unitary connection on a holomorphic principal bundle $P\to (B,\omega_B)$ with reductive fibre $G$, and an associated K\"ahler fibration $X=P\times_G (Y,\omega_Y)$ with cscK fibre, the relative Ricci form $\rho$ of the induced symplectic connection $\omega_X$ is related to the curvature of the connection $A$ by
	$$p(\contr_{\omega_B} \rho_\calH) = p(\Lap_\calV \contr_{\omega_B} \mu^* F_\calH).$$
\end{proposition}
\begin{proof}
	As above, suppose we have an action  $K\acts (Y,\omega_Y)$ of a real compact Lie group by holomorphic isometries on a K\"ahler manifold $Y$. By assumption this action admits an equivariant moment map $\mu: Y \to \k^*$. Let us define a comoment-type map 
	$$\nu^* : \k \to C_0^{\infty}(Y)$$
	by the composition
	$$\nu^* = \Lap \circ \mu^*.$$
	A restatement of \cref{lemma:riccihamiltonian} implies that the comoment map $\nu^*$ satisfies the standard moment map criterion with respect to the Ricci form $\Ric \omega_Y = \rho_Y$. That is,
	$$d\nu^*(\xi) = i_{v_\xi} \rho_Y.$$
	Additionally, since $K$ acts by isometries on $Y$, the Laplacian is $K$-equivariant as a morphism $C_0^{\infty} (Y) \to C_0^{\infty} (Y)$, and therefore the composition $\nu^* = \Lap \circ \mu^*$ is a $K$-equivariant map from $\k$ to $C_0^{\infty} (Y)$ with respect to the adjoint action of $K$ on $\k$. Differentiating this condition at the identity in $K$ gives the infinitesimal equivariance condition
	$$\nu^*([\xi, \eta]) = \rho_Y (v_\xi, v_\eta)$$
	for the comoment map $\nu^*$, which explicitly gives the interesting geometric formula
	$$\rho_Y(v_\xi, v_\eta) = \Lap(\omega_Y(v_\xi, v_\eta))$$
	for induced vector fields $v_\xi,v_\eta$. 
	
	Whilst the comoment map $\nu^*$ is not a genuine comoment map for a \emph{symplectic} form, it satisfies the same formal properties with respect to the $K$ action relative to $\rho_Y$ as $\mu^*$ does relative to $\omega_Y$. In particular the argument of \cref{prop:inducedsymplecticconnection} repeats without change for the differential form
	$$\tilde \tau_A := \rho_Y - d(\nu^* A) = \pi_A^* \rho_Y - \nu^* F$$
	on the product $P_\sigma \times Y$. Thus there exists a closed $(1,1)$-form $\tau_A$ on $X=P_\sigma \times_K Y$ with the property that $$(\tau_A)_\calV = (\rho)_\calV$$
	where $\rho$ is the relative Ricci form of $\omega_X$ itself. Additionally the explicit formula for $\tau_A$ reveals that
	$(\tau_A)_\calH = \nu^* F_\calH$.
	
	By \cite[Lem. 3.9]{dervan2019optimal}, if two closed $(1,1)$-forms agree when restricted to the vertical directions of a fibration, then their horizontal components are equal up to pullback from the base. In particular we have
	$$\rho_\calH + \pi^* \beta= \nu^* F_\calH$$
	for some two-form $\beta$ on $B$. 
	
	Let us now observe that after contracting with $\omega_B$ we have
	$$\contr_{\omega_B} \rho_\calH + \pi^* f= \contr_{\omega_B} \nu^* F_\calH = \Lap_\calV \contr_{\omega_B} \mu^* F_\calH$$
	for some function $f=\contr_{\omega_B} \beta$ on $B$. This second equality follows from the observation that $F_\calH$ on $X$ actually arises from a two-form defined on $B$, which is the very same $F_A\in \Omega^2(B, \ad P_\sigma)$ defining the curvature of the connection $A$ on $P_\sigma$ (after composing with the Lie algebra homomorphism from $\k$ to $\Ham(\calV)$). Since the two-form $\omega_B$ contracting $\mu^* F_\calH$ is also pulled back from the base, we can consider $\contr_{\omega_B} F_\calH$ as a section of the bundle of Hamiltonian vector fields on each fibre over $B$, and we have $\contr_{\omega_B} \mu^* F_\calH = \mu^* \contr_{\omega_B} F_\calH$ and similarly for $\nu^*$. Here we also use the identity
	$$\Lap_\calV \mu^* s = \nu^* s$$
	where $s: B\to \Ham(\calV)$ is a section of the relative Hamiltonian vector field bundle of the fibres of $X$ over $B$. This identity \emph{defines} $\nu^*$ for a general K\"ahler fibration, but in this case follows immediately from the descent of the comoment maps $\mu^*$ and $\nu^*$ to $X$ with respect to the diagonal action on $P\times Y$. 
	
	To conclude, we note by \cref{rmk:ppullback} that since the projection $p$ is invariant under the addition of a contraction of a form pulled back from the base, we have
	$$p(\Lap_\calV \contr_{\omega_B} \mu^* F_\calH) = p(\contr_{\omega_B} \nu^* F_\calH) = p(\contr_{\omega_B} \rho_\calH).$$
\end{proof}

\begin{theorem}\label{theorem:maintheorembody}
	If a complex unitary connection $A$ on $P$ is Hermite--Einstein with respect to the Hermitian structure $\sigma$ defining the reduction of structure group to $K$, then the induced symplectic connection $\omega_X$ on $P$ is an optimal symplectic connection. 
\end{theorem}
\begin{proof}
	Let $(X,\omega_X)\to (B,\omega_B)$ be an isotrivial K\"ahler fibration with cscK fibres and base, with symplectic connection $\omega_X$ induced from a holomorphic principal bundle $P$. Then by \cref{lemma:horizontalcomponentricciform} the optimal symplectic connection equation for $\omega_X$ reduces to
	$$p(\Lap_\calV \contr_{\omega_B} \mu^* F_\calH) = 0.$$
	Suppose now that $\omega_X$ is an optimal symplectic connection, and that $\contr_{\omega_B} \mu^* F_\calH = h$ for some smooth function $h\in C^{\infty}(X)$. Then $h$ restricts to a mean-zero holomorphy potential on each fibre of $X$, because the isotrivial fibration $X$ arises from a holomorphic principal bundle and the curvature takes values in fibrewise holomorphic vector fields. Since $p$ is the orthogonal projection onto such relative holomorphy potentials, we have
	\begin{align*}
		0 &= \ip{h, p(\Lap_\calV h)}\\
		&= \int_X h p(\Lap_\calV h) \omega_X^m \wedge \omega_B^n\\
		&= \int_X h \Lap_\calV h \omega_X^m \wedge \omega_B^n\\
		&= \int_X |\nabla h|^2 \omega_X^m \wedge \omega_B^n.
	\end{align*}
	Thus the holomorphy potential $h$ is covariantly constant, and in fact zero as $\mu^*$ lands in the mean-zero holomorphy potentials. Thus we have $p(\Lap_\calV h) = 0$ if and only if $h$ is zero. In particular the optimal symplectic connection equation is equivalent to
	$$\contr_{\omega_B} \mu^* F_\calH = 0.$$
	
	Now if $\omega_X$ arose from a principal bundle connection $A$, then the correspondence between the Lie algebra of automorphisms and holomorphy potentials tells us that the above equation is equivalent to asking 
	$$\contr_{\omega_B} F_A = \tau$$
	for some central section of $\ad P$, so if $A$ is Hermite--Einstein then the induced symplectic connection $\omega_X$ is an optimal symplectic connection. 
\end{proof}

The Hitchin--Kobayashi correspondence for principal bundles \cref{thm:DUYPrincipalbundle} allows us to interpret the above theorem in terms of the algebraic geometry of $P$. 

\begin{corollary}\label{corollary:stableprincipal}
	If a principal bundle $P\to (B,L)$ is polystable over a cscK base, and if the structure group $G$ of $P$ acts linearly on a polarised variety $(Y,H_Y)$ and admits a restriction to a $K$ action for a maximal compact subgroup which acts on $Y$ by isometries with respect to a cscK metric $\omega_Y\in c_1(H_Y)$, then the associated fibration $(X,H)\to (B,L)$ admits an optimal symplectic connection.
\end{corollary}

By \cref{thm:adiabaticscKDervanSektnan} one can use \cref{theorem:maintheorembody} to generate new examples of cscK metrics in adiabatic K\"ahler classes on the total space of holomorphic principal bundles.

\begin{example}
	Let $P\to (B,L)$ be a non-trivial, stable principal $\SL(2,\CC)$-bundle over a polarised variety $(B,L)$ of dimension at least two (every such principal bundle is trivial in dimension one, and the construction reduces to a product cscK metric in that case). Such a principal bundle could be constructed as the frame bundle of a non-trivial stable rank two holomorphic vector bundle over $(B,L)$ with trivial determinant. For the existence of such bundles on any projective algebraic surface see for example \cite{gieseker1988construction}. Assume $(B,L)$ admits a cscK metric and has discrete automorphism group. Let $(Y,-K_Y)$ denote the Mukai--Umemura threefold discussed in \cref{sec:mukaiumemura}, which admits an action of $\SL(2,\CC)$ and a K\"ahler--Einstein metric satisfying the assumptions of \cref{theorem:maintheorembody} \cite[\S 5]{donaldson2008kahler}. Then the associated bundle $$X= P\times_{\SL(2,\CC)} Y$$
	admits an optimal symplectic connection by \cref{theorem:maintheorembody} and \cref{corollary:stableprincipal} and since the base has discrete automorphisms and $P$ is simple, the total space of the fibre bundle has discrete automorphisms and by \cref{thm:adiabaticscKDervanSektnan} admits cscK metrics in adiabatic K\"ahler classes. 
\end{example}

The construction demonstrated above is general, and produces a wide variety of new examples of cscK metrics on the total space of holomorphic fibre bundles.

\begin{remark}
	It may be interesting to ask the question of when a polarised fibration $(X,H)\to (B,L)$ admits cscK metrics in \emph{non-adiabatic} K\"ahler classes $kL+H$ for $k$ not necessarily very large. In a special setting where $B=B_1\times \cdots \times B_\ell$ and $P$ is a product principal $(\CC^*)^n$-bundle arising from $\CC^*$-bundles on each $B_i$, Delcroix--Simon \cite{delcroix2022effective} have identified criteria for which associated isotrivial fibrations with toric fibre admit cscK metrics in K\"ahler classes on the total space. Away from the adiabatic limit, it is necesary to assume, in addition to the existence of a Hermite--Einstein connection on $P$ and a cscK metric on the base, that the toric fibre has a \emph{weighted} cscK metric. 
	
	It would be interesting to understand if existence of cscK metrics in non-adiabatic classes can be understood in terms of a weighted cscK condition on the fibre for other principal bundle constructions such as the one considered above, which is not necessarily of the special toric form considered by Delcroix--Simon.
\end{remark}

\subsection{Induced Hermite--Einstein structures\label{section:OSCimpliesHE}}

Let $\pi: (X,\omega_X,H)\to (B,\omega_B,L)$ be an isotrivial relatively cscK fibration arising from a principal bundle $\varpi: P\to (B,L)$ as described by \cref{lemma:principalbundle}. In this section we will describe how to pass from the symplectic connection $\omega_X$ on $X$ to a principal bundle connection $A$ on $P$, and show that when $\omega_X$ is optimal, the induced principal bundle connection $A$ is Hermite--Einstein.

\subsubsection{Alternative description of $P$\label{section:alternativeprincipalbundle}}
First we proceed by giving an alternative invariant description of the principal bundle $P$ associated to the isotrivial fibration $X$. This description is inspired by the case of infinite-dimensional principal bundles of symplectomorphisms for a symplectic fibration \cite[Rmk. 6.4.11]{mcduff2017introduction}. 

Let $(Y,\omega_Y,H_Y)$ denote the model fibre of the isotrivial fibration. Then the fibre $P_b$ of $P$ over a point $b\in B$ is given by the set of all biholomorphisms $f: Y\to X_b$ isotopic to the identity which also lift to the linearisations $H_Y$ and $\rest{H}{b}$. This set is a $G_0$-torsor for the group $G_0 = \Aut_0(Y,H_Y)$ acting by precomposition, which defines the right $G_0$-action on $P$. 

The tangent space $T_f P\subset \Gamma(f^* TX)$ for some $f\in P_b$ is given by all vector fields $v\in \Gamma(\rest{TX}{X_b})$ such that $d(\pi \circ f)v : Y \to T_b B$ is constant and for which the vertical part of $v$ with respect to the symplectic connection $\omega_X$ is holomorphic. The vector fields which preserve the K\"ahler structure of the fibration further satisfy the compatibility condition that the one-form
$$\omega_X(v, df(-))\in \Omega^1(Y)$$
is closed, which implies the vector field preserves the symplectic form $\rest{\omega_X}{b}$.

The vertical vectors $V_f \subset T_f P$ consist of all vector fields of the form $v=df \circ u: Y\to \rest{TX}{X_b}$ for some $u\in \h(Y,H_Y)$ a holomorphic vector field on $Y$ which generates an automorphism lifting to $H_Y$. 

To define horizontal vectors, note that using the Ehresmann connection defined by $\omega_X$, any vector $v_0 \in T_b B$ admits a unique horizontal lift to a vector field $v_0^\# \in \Gamma(\rest{TX}{X_b})$. Define the horizontal vectors $H_f\subset T_f P$ as the vector fields of the form $v = v_0^\# \circ f$ for some $v_0^\#: X_b \to \rest{TX}{X_b}$ for $v_0\in T_b B$. 

To observe the splitting, note that any vector field $v\in T_f P$, when viewed as a vector field on $X_b$, admits a splitting with respect to $\omega_X$, and the condition that $d(\pi \circ f)v$ is constant is exactly the statement that the horizontal component of $v$ is of the form $v_0^\#$ for some $v_0\in T_b B$. 

Thus the connection $\omega_X$ on $X$ induces a connection on $P$. We note that the equivariance of $H\subset TP$ with respect to the action of $G_0 = \Aut_0 (Y,H_Y)$ follows from the fact that $dR_g (v) = v_0^\# \circ (f\circ g)$ for $v=v_0^\# \circ f\in T_f P$ and $g\in G_0$, so $dR_g (H_f) \subset H_{f\cdot g}$.

Again we note that if we restricted to holomorphic isometries and vector fields preserving the K\"ahler structure then the same construction above would afford us a principal bundle connection on the principal $K_0$-bundle $P_\sigma$ which is the reduction of structure group for $P$ as in \cref{lemma:principalbundle}. This connection on $P_\sigma$ induces the connection on $P$ under the associated bundle construction, as can be seen easily by noting that the induced Ehresmann connection under the inclusion $K_0 \hookrightarrow G_0$ simply views a horizontal vector $v=v_0^\# \circ f$ for $f:Y \to X_b$ a holomorphic isometry as a vector $v$ for $f: Y \to X_b$ a biholomorphism, forgetting the isometry. This clearly maps the horizontal subspaces for $P_\sigma$ into those for $P$.

\subsubsection{Curvature of the connection on $P$}

We have described how a symplectic connection $\omega_X$ on $X$ induces a principal bundle connection on $P$, which we denote by $A$. The curvature of $A$ is a two-form on $B$ with values in $\ad P$, a Lie algebra bundle with fibre $\h$, defined by
$$F_A(v_1,v_2) = [v_1^\#,v_2^\#]^{\text{vert}}\in \h(X_b,\rest{\omega_X}{b},\rest{H}{b}))$$
where $v_1,v_2\in T_b B$. 

From the above construction we can see that the horizontal distribution inside $TP$ which defines the connection $A$ induces the same horizontal distribution on $X$, given by the orthogonal complement of $\omega_X$. Then using the construction of \cref{sec:inducedconnection}\footnote{Here we must choose some moment map on the fibre, for example by fixing a model fibre $(X_b,\omega_b)$ and defining a moment map $\mu^*$ on the Lie algebra of holomorphic vector fields lifting to $H_B$ by taking the holomorphy potential of mean zero.} we obtain a new, possibly different relatively K\"ahler metric $\omega_X'$ on $X$ which has the same horizontal distribution as $\omega_X$, and with the property that $\rest{\omega_X'}{b} = \rest{\omega_X}{b}$ for every $b$. Thus by \cite[Lem. 3.9]{dervan2019optimal} we have $\omega_X' = \omega_X + i\deldelbar \pi^* \varphi$ for some function $\varphi$ on $B$. Since $\omega_X$ is OSC by assumption, so too is $\omega_X'$ and by the direct calculation in \cref{sec:inducedconnection} applied to $\omega_X'$ we have that the curvature of the connection on $P$ induced by $\omega_X'$ satisfies the Hermite--Einstein equation. Since this is just the same connection $A$ on $P$ induced by $\omega_X$, we are done. Thus we obtain:

\begin{theorem}\label{thm:inducedhermiteeinstein}
	An optimal symplectic connection $\omega_X$ on an isotrivial relatively cscK fibration $(X,H)\to (B,\omega_B,L)$ induces a Hermite--Einstein metric on the associated principal bundle $P$ of relative automorphisms described in \cref{lemma:principalbundle} and \cref{section:alternativeprincipalbundle}.
\end{theorem}

We remark that from the construction of $P$ it is clear that a holomorphic automorphism of $P$ is the same data as a holomorphic fibre bundle automorphism of $X$, and the uniqueness of the Hermite--Einstein connection $A$ on $P$ up automorphism shows that the induced optimal symplectic connection on $X$ is unique up to automorphisms (and pullback of a two-form from $B$). This recovers a special case of the uniqueness result \cref{thm:uniqueness} of Dervan--Sektnan and Hallam.

\section{Future directions\label{sec:Fibrationsfuturedirections}}

\subsection{Stability\label{sec:isotrivialstability}}

The \cref{principle}-analogue of the main theorem \cref{thm:maintheoremfibrations} of the previous section is the following.

\begin{conjecture}\label{conj:stabilityofisotrivial}
	Suppose $P\to (B,L)$ is a holomorphic principal bundle with reductive fibre $G$, and that $G$ acts linearly on a K-polystable variety $(Y,H_Y)$. Let $X=P\times_G Y$ by the associated isotrivial fibration with the relatively ample polarisation $H=(P\times H_Y)/G$ where $G$ acts diagonally with respect to the lift of the action from $Y$ to $H_Y$. If $(X,H)\to (B,L)$ is (semi/poly)stable, then $P$ is (semi/poly)stable.
\end{conjecture}

In order to approach this conjecture, it is first necessary to identify how fibration degenerations for $(X,H)\to (B,L)$ relate to the stability of $P$. Recall that to test the stability of $P$ one considers reductions of structure group $\sigma: B\to P/Q$ to maximal parabolic subgroups $Q\subset G$. Such subgroups $Q\subset G$ should be thought of as those stabilising a (one step, in the maximal case) flag with respect to some representation $\rho: G \to \GL(V)$. To understand this process, let us consider another characterisation of parabolic subgroups.

A subgroup $Q\subset G$ of a reductive group is parabolic if and only if there exists a one-parameter subgroup $\lambda: \CC^* \into G$ such that
$$Q = \{g \in G \mid \lim_{t\to0} \lambda(t) g \lambda(t)^{-1} \text{ exists}\}.$$
Given a parabolic subgroup $Q=Q(\lambda)$ and a representation $\rho: G \to \GL(V)$, the induced action of $\lambda$ on $V$ produces a filtration into weight spaces of the $\CC^*$ action. If $w_1,\dots,w_r$ are the weights of the $\CC^*$ action with $w_1>\dots>w_r$ then we obtain a filtration
$$0 = V^{(0)} \subset V^{(1)} \subset \cdots \subset V^{(r)} = V$$
with the property that the induced action of $\lambda$ on $V^{(i)}/V^{(i-1)}$ has weight $r_i$. In particular if $Q(\lambda)$ is a maximal parabolic then it will act with just two weights $w_1,w_2$ and fix a flag 
$$0 \subset V^{(1)} \subset V$$
inside $V$, where $\lambda$ acts on $V^{(1)}$ with weight $w_1$ and $V/V^{(1)}$ with weight $w_2$.

Now suppose $G$ acts linearly on a polarised variety $(Y,H_Y)$. Then there is an induced action on the space of sections $H^0(Y,H_Y^k)$ for each $k>0$ and a choice of parabolic subgroup $Q(\lambda) \subset G$ induces a filtration 
$$0 \subset V_k^{(1)} \subset \cdots \subset V_k^{(r-1)} \subset H^0(Y,H_Y^k).$$

Using the representation $\rho: G \to \GL(H^0(Y,H_Y^k))$ we see that the bundles $V_k = \pi_* H^k$ for the associated fibration $(X,H)\to (B,L)$ may be obtained as associated bundles
$$P\times_G H^0(Y,H_Y^k) \isom \pi_* H^k = V_k.$$
A choice of parabolic reduction of structure group $\sigma: B \to P/Q$ gives the associated bundle $V_k$ the structure of a $Q$-bundle, and so the above construction on each fibre produces a filtration
$$0\subset V_k^{(1)} \subset \cdots \subset V_k^{(r-1)} \subset V_k$$
of $V_k$. 

Working under the simplified setting where the reduction of structure group $\sigma: B \to P/Q$ is defined over the entirety of $B$ (recalling as in \cref{sec:stabilityprincipalbundle} that in general we must allow $\sigma$ only to be supported on open subsets $U\subset B$ with complement codimension two) then each term $V_k^{(i)}$ in the induced filtration of the vector bundle $V_k$ for $k\gg0$ will also be locally free.\footnote{In general we expect a filtration by coherent subsheaves.} In this setting we can easily define the vector bundle degeneration given by turning off the extention (\cref{rmk:turningoffextension}) by smoothly splitting
$$V_k \isom \tilde V_k = \bigoplus_i V_k^{(i)}/V_k^{(i-1)}$$
and considering the extension class $\gamma \in H^1(B, \End \tilde V_k)$ whose $\Hom(V_k^{(i)}/V_k^{(i-1)}, V_k^{(i-1)})$-component is given by the extension class of the short exact sequence
\begin{center}
	\ses{V_k^{(i-1)}}{V_k^{(i)}}{V_k^{(i)}/V_k^{(i-1)}}.
\end{center}
Scaling $t\gamma$ and allowing $t\to 0$ we obtain a degeneration $\calE\to \CC$ of $V_k$ of the standard form described in \cref{sec:stabilityoffibrations} and therefore a fibration degeneration $(\calX,\calH)$ of $(X,H)\to (B,L)$. 

\begin{remark}
	The fibration degenerations of this form have been described explicitly in terms of transition functions by Hallam \cite[Ex. 6.15]{hallam2020geodesics} where they form examples of product-type degenerations. There it is proven that if $(X,H)\to (B,L)$ admits an optimal symplectic connection, then it is stable with respect to such product-type degenerations.
\end{remark}

A more precise version of \cref{conj:stabilityofisotrivial} is the following.

\begin{conjecture}\label{conj:stabilityisotrivialrefined}
	Suppose $\sigma: B \to P/Q$ is a reduction of structure group to a maximal parabolic. Suppose the induced test configuration $(\calX_b,\calH_b)$ of a generic fibre $(X_b,H_b) \isom (Y,H_Y)$ normalises to a product, so that $W_0(\calX,\calH) = 0$. Then
	$$W_1(\calX,\calH) = C \deg \sigma^* \calV(P/Q)$$
	for some constant $C>0$ and if $\sigma$ destabilises $P$, then $(\calX,\calH)$ destabilises $(X,H)\to (B,L)$. 
\end{conjecture}

We make the following further conjecture which is the analogue of \cref{thm:inducedhermiteeinstein} in our setting.

\begin{conjecture}\label{conj:principalfibration}
	Suppose $P\to (B,L)$ is the bundle of relative automorphisms of a polarised isotrivial K\"ahler fibration $(X,\omega_X,H)\to (B,L)$ as described in \cref{lemma:principalbundle}. Then $P$ is (un/semi/poly)stable if and only if $(X,H)\to (B,L)$ is (un/semi/poly)stable.
\end{conjecture}

\subsubsection{The case of projective bundles}

In the case of projective bundles, parts of \cref{conj:principalfibration} follow from the work of Ross--Thomas using their notion of slope K-stability \cite{ross2006obstruction}. In that setting they consider fibrations of the form $(\PP(E),\calO(1)) \to (B,L)$ and study K-stability in adiabatic classes $jL + \calO(1)$ for $j\gg 0$ with respect to test configurations arising as deformation to the normal cone of subschemes $Z=\PP(F)\subset \PP(E)$ arising as the projectivisation of saturated coherent subsheaves $F\subset E$. In particular they prove in this case that if $E$ is unstable then $(\PP(E),\calO(1))$ is an unstable fibration, and moreover if $B$ is a curve they show that if $(\PP(E),\calO(1))$ is a (semi/poly)stable fibration then $E$ is (semi/poly)stable. In this section we will rephrase this theory using the language of stability of fibrations of \cref{sec:stabilityoffibrations}.

By using Ross--Thomas's formula for the Donaldson--Futaki invariant for a deformation to the normal cone, it is possible to compute explicitly the Donaldson--Futaki invariant for certain fibration degenerations of $\PP(E)$. Namely if $F\subset E$ is a subbundles and $Z=\PP(F) \subset \PP(E)$ is the projective subbundle, then the deformation to the normal cone (see \cite[\S 4]{ross2006obstruction}) $(\calX,\calH_c)$ of $(\PP(E),H)$ has Donaldson--Futaki invariant
\begin{equation}\label{eq:DFslopeKstability}a_0 \DF (\calX, jL+\calH_c) = a_1 \int_0^c a_0(x)\, dx - a_0 \int_0^c \left( a_1(x) + \frac{a_0'(x)}{2} \right) dx.\end{equation}
Here we will take $H=-K_{\PP(E)/B}$ to be the relative anticanonical bundle of the projective bundle instead of the linearisation $\calO(1)$ considered by Ross--Thomas.\footnote{Note that $H$, and indeed $\calO(1)$, are already relatively very ample.} In particular this gives us a Fano fibration, and as discussed in \cref{sec:stabilityoffibrations} the expressions for $W_1$ simplify in that setting. Recall we have
$$K_{\PP(E)/B} = -(\rk E) \calO(1) - \pi^* \det E$$
and so 
$$V_k = \pi_* (-K_{\PP(E)/B})^k = \Sym^{k \rk E} E^* \otimes (\det E)^k.$$
We study the setting in which $\PP(F) \subset \PP(E) \subset \PP(V_k)$.
\begin{remark}
	Note that $\mu(V_k)=0$ for all $k$ using this linearisation of the projective bundle, which corresponds to the vanishing degree of the CM line bundle as mentioned in \cref{sec:stabilityoffibrations}.
\end{remark}
The coefficients $a_i(x)$ appearing in \eqref{eq:DFslopeKstability} are computed as the coefficients in the expansion of
$$\chi(\Bl_{\PP(F)} \PP(E), L^{jk} \otimes H^k(-xk \calE)) = a_0(x,j) k^n + a_1(x,j)k^{n-1} + O(k^{n-2})$$
where $xk\in \ZZ$ and $\calE$ is the exceptional divisor of the blow up $\Bl_{\PP(F)} \PP(E)$. The coefficients $a_i(x,j)$ are polynomials in $x$ of degree at most $n+m-i$ where $\dim B = n$ and $\dim \PP(E_b) = m$ so that $\rk E = m+1$. 

The weight of the $\CC^*$ action on the central fibre of the deformation to the normal cone for $\PP(F)\subset \PP(E)$ can be computed in terms of the $a_i(x)$ by
\begin{align*}
	b_0 &= \int_0^c a_0(x) \, dx - ca_0,\\
	b_1 &= \int_0^c \left( a_1(x) + \frac{a_0'(x)}{2} \right) dx - ca_1,
\end{align*}
and then the expression above for the Donaldson--Futaki invariant follows from the definition 
$$\DF (\calX, jL+\calH_c) = \frac{b_0a_1 - a_0b_1}{a_0}.$$
Note here $c\in (0,\epsilon(Z)]$ is a parameter for the relatively ample line bundle $\calH_c$ on the deformation to the normal cone. We wish to consider the critical case in which $c=\epsilon(Z)$ so that the induced test configuration on the fibres of the projective bundle normalise to a product. The relative Seshadri constant in this case is $c=\rk E$. This follows from Ross--Thomas who show that the relative Seshadri constant for the projectivisation of a subbundle with respect to the linearisation $\calO(1)$ is $\varepsilon = 1$, and in our case the linearisation on the fibre reduces to $\calO(\rk E)$. At such a choice, the resulting fibration degeneration has $W_0(\calX,\calH_c) = 0$. 

In our setting one may compute that 
\begin{multline}\label{eq:W1projectivebundle}
	W_1(\calX,\calH_c) = \frac{m}{2(m+1)!(n-1)!} \frac{1}{m+2} \left( -\frac{1}{\calE} L^{n-1}.(H-c\calE)^{m+2}\right)\\
	- \frac{1}{2m!(n-1)!}\frac{1}{m+1} \left( -\frac{1}{\calE} L^{n-1} . (H-c\calE)^{m+1} . (H-p\calE)\right)
\end{multline}
where $p=\codim \PP(F)$ is just $\rk E - \rk F$, and $\calE$ denotes the exceptional divisor in $\Bl_{\PP(F)} \PP(E)$. Notice that comparing to the expansion \eqref{eq:W1Fanofibration} we can identify the $L^n.\calH^{m+1}$ and $L^{n-1}.\calH^{m+1}.K_{\calX/B\times \PP^1}$ terms appearing in the intersection-theoretic formula for $W_1$. This third term vanishes due to working with an isotrivial Fano fibration and the fact that the degree $L^{n-1} . H^{m+1}$ of the CM line bundle vanishes, as already noted previously.

We can relate the above expansion of $W_1$ to the topology of $E$ and $F$ using the following characterisation of the Segre classes. Recall (see for example \cite[Prop. 9.13]{eisenbud20163264}) that if $\PP(F)\subset \PP(E)$ then the normal bundle $N\to \PP(F)$ can be expressed as
$$N=N_{\PP(F)/\PP(E)} = \calO_{\PP(E)}(1) \otimes \pi^* E/F.$$
The Segre classes of $N$ appear in the intersection formulae
$$\int_{\Bl_{\PP(F)} \PP(E)} L^{n-1}.H^i.\calE^{m-i+1} = (-1)^{m-i} \int_{\PP(F)} L^{n-1}.s_{m-p-i+1}(N).\rest{H}{\PP(F)}^i.$$
For the critical exponent such that $m-i+1=p$ we obtain on the right-hand side the zeroth Segre class $s_0(N)=1$. The higher Segre classes are defined by the relation $s(E)c(E) = 1$ where $s(E)$ and $c(E)$ are the total Segre and Chern classes. In particular one has the general expression
$$s_q(E\otimes L) = \sum_{j=0}^q (-1)^{q-j} {\rk E - 1 + q \choose \rk E - 1 + j} s_j(E) c_1(L)^{q-j}$$
for a vector bundle $E$ and line bundle $L$. In our case we take the vector bundle $\pi^* E/F$ to compute
$$s_{m-p-i+1}(N) = \sum_{j=0}^{m-p-i+1} (-1)^{m-p-i+1-j} {m-i\choose p-1+j} s_j(\pi^* E/F) c_1(\calO_{\PP(F)}(1))^{m-p-i+1-j}.$$
A laborious calculation applying the above expression to \eqref{eq:W1projectivebundle} produces that if
$$A=-\frac{m}{m+2} \left( \frac{1}{\calE} L^{n-1} . (H-c\calE)^{m+2}\right) + \frac{1}{\calE} L^{n-1}.(H-c\calE)^{m+1}.(H-p\calE),$$
which is a positive constant multiple of $W_1$, we obtain
$$A=B (\rk E)^{\rk E} (\rk E \deg F - \rk F \deg E)$$
where $B$ is the constant
\begin{multline*}
	B=\sum_{i=0}^{\rk F} (-1)^{\rk F - i} {\rk E + 1\choose i} {\rk E -1-i\choose \rk E -1- \rk F }\cdot \\ \frac{\rk E - i}{\rk E - \rk F} \left( -\rk E + 1 + \frac{i \rk E  + (\rk E - \rk F)(\rk E + 1 - i)}{\rk E}\right).
\end{multline*}
Thus in this case (taking the frame bundle of $E\to (B,L)$ as the principal bundle $P$ and the reduction of structure group inducing the holomorphic subbundle $F$) \cref{conj:stabilityisotrivialrefined} can be resolved provided $B>0$ for all $0 < \rk F < \rk E$. Indeed this follows \emph{a posteriori} from the work of Ross--Thomas.

\begin{remark}
	In view of the notion of $\mathfrak{f}$-stability introduced by Hattori (see \cref{rmk:fibrationgiesekerstability}) it would be interesting to see if the above formalism can be upgraded to study the Gieseker stability of vector bundles as it relates to the $\mathfrak{f}$-stability of the associated projective bundles. In particular does the Donaldson--Futaki invariant $\DF(\calX,jL+\calH)$ admit an expension in powers of $j$ for which the higher-order coefficients correspond in some way to the lower-order polynomial coefficents in the difference of Gieseker slopes appearing in \cref{def:giesekerstability}. It has been shown by Ross--Keller \cite{keller2014note} for example that the projectivisation of a Gieseker unstable rank two vector bundle on a surface is not asymptotically Chow stable, and using some analytical tools that Gieseker stability can recover asymptotic Chow stability of the projectivisation.
\end{remark}

\subsection{Principal bundles for non-isotrivial fibrations}

The characterisation of optimal symplectic connections in \cref{sec:existenceOSCisotrivial} in terms of connections on principal bundles uses the isotrivial structure of the fibration critically. However, in the case of symplectic fibrations (see for example \cite[Rmk. 6.4.11]{mcduff2017introduction}) it is possible to associate to a fibration $(X,\omega_X) \to B$ an infinite-dimensional principal bundle of relative symplectomorphisms $\calP \to B$ of the fibres of $X$. The symplectic connection $\calH$ on the fibration $X\to B$ induces, using the same arguments as in \cref{section:OSCimpliesHE}, a principal bundle connection on $\calP$. One can ask if the optimal symplectic connection condition on $\omega_X$ in the case where $X\to B$ is actually a K\"ahler fibration can be characterised in terms of the connection on the principal bundle $\calP$.

Although \emph{a priori} this principal bundle does not admit any holomorphic structure, it is possible to identify a holomorphic principal bundle associated to certain polarised K\"ahler fibrations $(X,\omega_X,H) \to (B,\omega_B,L)$. Indeed under the assumption that the fibres $(X_b,\omega_b)$ are cscK and that the automorphism groups $\Aut(X_b,H_b)$ are isomorphic for every $b$, then one can define a principal bundle $P\to B$ whose fibre over $b$ is given by 
$$P_b = \{f: G \to \Aut(X_b,H_b)\mid f \text{ an isomorphism}\}$$
where $G=\Aut(X_{b_0},H_{b_0})$ is a fixed model of the automorphism group.

\begin{remark}
	It suffices to assume, as we have in \cref{sec:kahlerfibrations} that the dimension $\dim \Aut(X_b,H_b)$ is independent of $b$. At least when $B$ is connected, this implies that $\Aut(X_b,H_b) \isom \Aut(X_{b'},H_{b'})$ for all $b,b'\in B$ by the rigidity of deformations of the reductive group $\Aut(X_b,H_b)$. 
\end{remark}

Upon this holomorphic principal bundle with reductive structure group one may construct a connection arising out of a relatively K\"ahler metric $\omega_X$ as in \cref{section:OSCimpliesHE}. 

\begin{question}
	Is there a characterisation of the optimal symplectic connection condition in terms of the connection on the principal bundle $P$ of relative automorphisms of the fibration?
\end{question}

In particular it is expected that just as in the isotrivial case, the term
$$\Lap_\calV \contr_{\omega_B} \mu^* F_\calH$$
appearing in the OSC equation should relate to the curvature $\contr_{\omega_B} F_A$ of the connection on $P$, but one expects that the Ricci term $\contr_{\omega_B} \rho_\calH$, which is no longer necessarily proportional to the first term, will contribute non-trivially to the resulting curvature equation on the principal bundle. For example, does there exist a characterisation of the difference
$$T = \contr_{\omega_B} \rho_\calH - \Lap_\calV \contr_{\omega_B} \mu^* F_\calH$$
in terms of the local deformation of complex structure of the fibres of the fibration $X\to B$?

Due to an observation of Hallam \cite{hallam2022thesis}, the space of relatively cscK metrics cohomologous to a given relatively cscK metric $\omega_X$ on a K\"ahler fibration can be identified with the sections of the vector bundle $E\to B$ of real (mean-zero) relative holomorphy potentials. In particular $E\to B$ here arises as the real adjoint bundle $\ad P_\sigma$ with respect to the reduction of structure group $\sigma: B \to P/K$ for $K=\Aut(X_b, \omega_b, H_b)$ defined by the choice of $\omega_X$. Thus we may expect that the term $\contr_{\omega_B} \rho_\calH$ for some relatively cscK metric $\omega_X$ may admit an interpretation in terms of sections of the real adjoint bundle $\ad P_\sigma$, and the optimal symplectic connection equation on $X$ should be related to a coupled Hermite--Einstein type equation for a pair $(A,s)$ where $s: B\to E=\ad P_\sigma$.

	\renewcommand{\bibname}{References}

	\cleardoublepage
	\phantomsection
	\addcontentsline{toc}{chapter}{References}
	\bibliographystyle{bibliography/alpha-abbrvsort}
	
	\bibliography{bibliography/thesis}

\end{document}